\documentclass[reqno,centertags,12pt,a4paper]{amsart}
\usepackage{dsfont}

\usepackage{amscd}
\usepackage{amsmath}
\usepackage{amssymb}
\usepackage{amsfonts}
\usepackage{hyperref}
\usepackage{enumerate}
\usepackage{latexsym}
\usepackage{cases}
\usepackage{amsthm}
\usepackage{graphicx}
\usepackage{verbatim}
\usepackage{mathrsfs}
\usepackage{color}
\allowdisplaybreaks \numberwithin{equation}{section}

\newtheorem{theorem}{Theorem}[section]

\newtheorem{lemma}{Lemma}[section]
\newtheorem{corollary}[theorem]{Corollary}
\newtheorem{proposition}[theorem]{Proposition}
\newtheorem{remark}[theorem]{Remark}
\newcommand{\normmm}[1]{{\left\vert\kern-0.25ex\left\vert\kern-0.25ex\left\vert #1
		\right\vert\kern-0.25ex\right\vert\kern-0.25ex\right\vert}}

\newcommand{\Z}{\mathds{Z}}
\newcommand{\C}{\mathds{C}}

\newcommand{\R}{\mathds{R}}

\newcommand{\B}{\mathbb{B}}

\newcommand{\W}{\mathbb{W}}
\newcommand{\LL}{\mathbb{L}}

\newcommand{\Rr}{\mathfrak{R}}
\newcommand{\bb}{\boldsymbol}
\DeclareMathOperator{\Div}{div}
\allowdisplaybreaks
\begin{document}
%%%%%%%%%%%%%%%%%%%%%%%%%%%%%%%%%%%%%%%%%%%%%%%%%%%%%%%%%%%%%%%%%%%%%%%%%%%%%%%%%%%%%%%%%%%%%%%%%%%%
\title[ $L^p$-$L^q$ estimates]{On the $C_0$ semigroup generated by the Oseen operator around a steady flow  exterior to a rotating obstacle}

\author[J. Li, C. Miao and  X. Zheng]{Jingyue Li, Changxing Miao and Xiaoxin Zheng}

\address{Institute of Applied Physics and Computational Mathematics, Beijing 100088, P.R. China}
\email{m\_lijingyue@163.com}

\address{ Institute of Applied Physics and Computational Mathematics, P.O. Box 8009, Beijing 100088, P.R. China.}
\email{miao\_changxing@iapcm.ac.cn}
\address{ School of Mathematics and Systems Science, Beihang University, Beijing 100191, P.R. China}
\email{xiaoxinzheng@buaa.edu.cn}

\begin{abstract}
We consider the motion of an incompressible viscous fluid filling the whole space exterior to a moving with rotation and translation obstacle. We show that the Stokes operator around the steady flow in the exterior of this obstacle generates a $C_0$-semigroup in $L^p$ space and then develop a series of $L^p$-$L^q$ estimates of such semigroup. As an application, we give out the stability of such steady flow when the initial disturbance in $L^3$ and the steady flow are sufficiently small.
 \end{abstract}
\maketitle
%%%%%%%%%%%%%%%%%%%%%%%%%%%%%%%%%%%%%%%%%%%%%%%%%%%%%%%%%%%%%%%%%%%%%%%%%%%%%%%%%%%%%%%%%%%%%%%%%%%
\section{Introduction}\label{INTR}
%%%%%%%%%%%%%%%%%%%%%%%%%%%%%%%%%%%%%%%%%%%%%%%%%%%%%%%%%%%%%%%%%%%%%%%%%%%%%%%%%%%%%%%%%%%%%%%%%%%%
\setcounter{section}{1}\setcounter{equation}{0}

We consider the motion of an incompressible viscous fluid  filling the whole space exterior to a rigid body $\mathcal{B}$ moving with rotation and translation. It is nature to describe the motion of
the fluid from a frame of reference $\mathcal S$ attached to $\mathcal{B}$ since the region occupied  by the fluid will becomes time-independent in  $\mathcal S$. However, since  $\mathcal{B}$ may rotate,
the frame $\mathcal{S}$ is no longer inertial, and accordingly we have to modify the classical Navier-Stokes equations  in order to take into account the fictitious forces. Assume that the angular velocity $\bb \omega$ of $\mathcal{B}$
 with respect to the initial frame is constant in time, this amounts to adding Coriolis force $2\bb \omega\times \bb v$
 to the left side of the classical Naiver-Stokes equations when the centrifugal force $\bb \omega\times(\bb \omega\times  x)$ is absorbed in the pressure.

In mathematics, let $\bb{v}=\bb{v}(x,t)$ be the velocity of the fluid with respect to  $\mathcal S$,
and $\theta=\theta(x,t)$ be the original pressure modified by adding the $-\frac12(\bb \omega\times  x)^2$. The motion of such fluid in the frame $\mathcal S$ can be described by
\begin{equation}\label{NS0}
\left\{\begin{aligned}
&\partial_t \bb{v}-\Delta \bb{v}+2\bb{\omega}\times \bb{v}+\bb{v}\cdot\nabla \bb{v}+\nabla \theta=\bb{f}\quad \text{in }\Omega\times (0,\infty),\\
&\Div \bb{v}=0\quad \text{in }\Omega\times[0,\infty),\\
&\bb{v}|_{\partial\Omega}=\bb{w}_*,\quad \lim_{|x|\rightarrow\infty}(\bb{v}+\bb{v}_{\infty})=\bb{0},\,\;\bb{v}_{\infty}\triangleq\bb{v}_{\rm trans}+
\bb{\omega}\times  {x},\\
&\bb{v}|_{t=0}=\bb{v}_0,
\end{aligned}\right.
\end{equation}
where $\Omega$ is the fixed region occupied by the fluid  with $\partial\Omega\in C^3$, $\bb f$ is a external force, and $\bb w_*$ is a prescribed velocity at $\partial\Omega$.  $\bb{v}_{\rm trans}$ denotes the translation velocity of the center of mass of the rigid body with respect to $\mathcal S$. A most significate situation, considered in this paper, is that both $\bb{v}_{\rm trans}$ and  $\bb{\omega}$ are constant vectors.

Make transformations: $\bb v\to \bb{v}+\bb v_{\infty}$ and $\bb{w}^*\to \bb{w}^*+\bb{v}_{\infty}$, stilled denoted by $\bb{v}$ and $\bb{w}^*$, respectively. After a simple calculation, the $-\bb{\omega}\times\bb{v}_{\infty}$ can be formally absorbed in the pressure, still denoted by $\theta$,  and then system \eqref{NS0} becomes
\begin{equation}\label{NS0'}
\left\{\begin{aligned}
&\partial_t \bb{v}-\Delta \bb{v}-\bb{v}_{\rm trans}\cdot\nabla \bb{v}-(\bb{\omega}\times {x})\cdot\nabla\bb{v}+\bb{\omega}\times\bb{v}\\
&\qquad\qquad\qquad\qquad\qquad\qquad\qquad+\bb{v}\cdot\nabla \bb{v}+\nabla \theta=\bb{f}\quad \text{in }\Omega\times (0,\infty),\\
&\Div \bb{v}=0\quad \text{in }\Omega\times[0,\infty),\\
&\bb{v}|_{\partial\Omega}=\bb{w}^*,\quad \bb{v}\to \bb{0}\,(| {x}|\to\infty),\\
&\bb{v}|_{t=0}=\bb{v}_0.
\end{aligned}\right.
\end{equation}
To consider the stationary flow of \eqref{NS0'}, we assume that $\bb{w}_*$ and $\bb{f}$ only depend on spatial variable throughout the paper. The steady-state counterpart of \eqref{NS0'} thus becomes
\begin{equation}\label{S0}
\left\{\begin{aligned}
&-\Delta \bb{w}-\bb{v}_{\rm trans}\cdot\nabla \bb{w}-(\bb{\omega}\times  {x})\cdot\nabla \bb{w}+\bb{\omega}\times \bb{w}+\bb{w}\cdot\nabla \bb{w}+\nabla \theta_s=\bb{f}\quad \text{in }\Omega,\\
&\Div \bb{w}=0\quad \text{in }\Omega,\\
&\bb{w}|_{\partial\Omega}=\bb{w}_*,\quad \bb{w}\to \bb{0}\,(| {x}|\to\infty),
\end{aligned}\right.
\end{equation}
  Roughly speaking, we call a stationary flow  $\bb{w}$ with a ``good" decay at infinite  as a \emph {physically reasonable solution},
  which is first introduced by Finn in \cite{Fin59a,Fin59b} for $\bb{\omega}=\bb{0}$.  More precisely, a Leary solution $\bb{w}$ satisfies
\begin{enumerate}
 \item[{\rm (i)}]  It is unique for `` small" data;
\vskip0.15cm
  \item[{\rm (ii)}]   $|\bb{w}|=O(|x|^{-1})$ as $|x|\to\infty$. Furthermore, in both cases $\bb{v}_{\rm trans}\neq \bb{0}$ with $\bb{\omega}=\bb{0}$, and $\bb{v}_{\rm trans}\cdot\bb{\omega}\neq \bb{0}$, the flow $\bb{w}$ must exhibit an infinite wake extending in the direction opposite to $\bb{v}_{\rm trans}$ and $(\bb{v}_{\rm trans}\cdot\bb{\omega})\bb{\omega}$, respectively.
\end{enumerate}
It is natural to arise a problem when the class of the physically reasonable solutions can be seen as a limit of corresponding nonstationary solutions as $t\to \infty$, which will be called a stability problem. The problem is significant because there are some interesting feature if $\bb{\omega}\neq \bb{0}$ including hyperbolic aspect caused by the presence of spin.
Assume that  $\bb{w}$ is  a  physically reasonable solution of  \eqref{S0}. Let $\bb{u}=\bb{v}-\bb{w}$, $P=\theta-\theta_s$ and $\bb{u}_0=\bb{v}_0-\bb{w}$. The stability of $\bb{w}$
can be reduced to finding  a unique solution $\bb{u}(x,t)$ satisfying
\begin{equation}\label{P0}
\left\{\begin{aligned}
&\partial_t \bb{u}-\Delta \bb{u}-\bb{v}_{\rm trans}\cdot\nabla \bb{u}-(\bb{\omega}\times  {x})\cdot\nabla \bb{u}+\bb{\omega}\times \bb{u}\\
&\qquad\quad\quad\quad+\bb{w}\cdot\nabla \bb{u}+\bb{u}\cdot\nabla \bb{w}+\bb{u}\cdot\nabla \bb{u}+\nabla P=\bb{0}\quad \text{in }\Omega\times (0,\infty),\\
&\Div \bb{u}=0\quad \text{in }\Omega\times [0,\infty),\\
&\bb{u}|_{\partial\Omega}=\bb{0},\quad \bb{u}\to \bb{0}\,(| {x}|\to\infty),\\
& \bb{u}_{t=0}= \bb{u}_0.
\end{aligned}\right.
\end{equation}
such that $\bb{u}(x,t) \to \bb{0}$ as $t\to\infty$ .

The stability problem was widely studied in  $L^2(\Omega)$ for the irrotational case $\bb{\omega}=\bb{0}$ by \cite{BM92,Hey70,Hey72,Hey80,Mas75,MS88}, under the small assumption on $\bb{w}$ and $\bb{u}_0$. Roughly, Heywood \cite{Hey70,Hey72} first proved problem \eqref{P0} admits a unique solution $\bb{u}$ converging to $\bb{0}$ in $L^2_{\rm loc}(\Omega)$ and  $\dot{W}^{1,2}(\Omega)$ as  $t\to \infty$. Masuda \cite{Mas75} and Heywood \cite{Hey80}  obtained an algebraic decay in time of  $L^{\infty}(\Omega)$ of a weak solution $\bb{u}$ to \eqref{P0}.  Miyakawa-Sohr \cite{MS88}  showed that any weak solution to \eqref{P0} satisfying the strong energy inequality tends to $\bb{0}$ in $L^2(\Omega)$ as $t\to \infty$. Borchers-Miyakawa \cite{BM92} further gave  an algebraic convergence rate of $\|\bb u\|_{L^2(\Omega)}$ if $\bb{v}_{\rm trans}=\bb{0}$ and a logarithmic convergence rate if $\bb{v}_{\rm trans}\not=\bb{0}$ as $t\rightarrow\infty$.
However, Finn \cite{Fin60} pointed out that $\bb{w}\in L^2(\Omega)$ if and only if
\begin{equation*}
\int_{\partial\Omega}\big(T[\bb{w},\theta_s]-(\bb{w}-\bb{v}_{\rm trans})\otimes (\bb{w}-\bb{v}_{\rm trans})
+\bb{v}_{\rm trans}\otimes(\bb{w}-\bb{v}_{\rm trans})\big)\cdot\bb{n}\mathrm d\sigma(x)=0,
\end{equation*}
where $T_{j,k}[\bb{w},\theta_s]\triangleq-\delta_{j,k}\theta_s+\partial_jw_k+\partial_kw_j$.
Thus, it seems more reasonable to study problem \eqref{P0} in non-square integral spaces at least $L^{3,\infty}$ to which $\bb{w}$ belongs, under the small assumption of $\bb{u}_0$ and $\bb{w}$.

In this direction, the stability problem was settled mainly depending on the so-called $L^p$–$L^q$ estimates of the semigroup. The stability problem was first studied in the irrotational case where $\bb{\omega}=\bb{0}$. When $\bb{v}_{\rm trans}=\bb{0}$, Borchers-Miyakawa \cite{BM95} proved the stability  in $L^{3,\infty}(\Omega)$ of the steady flow $\bb{w}$, obtained in \cite{BM95} and satisfying
\begin{equation}\label{w-D2}
 |\nabla^k\bb{w}|\leq C| {x}|^{-1-k}\,(k=0,1),
\end{equation}
 under the size of $\bb{u}_0$ and $\bb{w}$ depend on $3<q<\infty$. More precisely, they proved that \eqref{P0} admits a unique solution $\bb{u}$ in $L^{3,\infty}(\Omega)$ satisfying
\begin{equation}\nonumber
\left\{\begin{aligned}
&\|\bb{u}\|_{L^{3,\infty}(\Omega)}\to 0\qquad\quad\;\; \; \text{ as } t\to \infty,\\
&\|\bb{u}\|_{L^r(\Omega)}\lesssim t^{-\frac12+\frac3{2r}},\quad \quad\,  \forall\, t>0,\;3<r<q<\infty.
\end{aligned}\right.\end{equation}
The point of their proof was to establish the $L^p$–$L^q$ estimates of the semigroup generated by the Stokes operator around the steady flow $\bb{w}$ in $\Omega$ with Dirichlet zero boundary condition.
On the other hand, the case $\bb{v}_{\rm trans}\neq \bb{0}$ was considered by Shibata \cite{Shi99} and Enomoto-Shibata \cite{ES05} in $L^3(\Omega)$. Shibata \cite{Shi99} proved that the unique solution $\bb u$ of  problem \eqref{P0} in $L^3(\Omega)$  such that
\begin{align}
&\|\nabla \bb{u}\|_{L^3(\Omega)}\lesssim t^{-\frac12} ,\quad\|\bb{u}(t)\|_{L^q(\Omega)}\lesssim t^{-\frac12+\frac3{2q}}\quad\forall 3\leq q<\infty,\nonumber\\
&\|\bb{u}(t)\|_{L^{\infty}(\Omega)}\lesssim (t^{-\frac12}+t^{-1+\frac3{2r}})\label{est1-1}
\end{align}
if the $L^3$ norm of $\bb{u}_0$ and the constant $C_{\delta}$ in the following relation
\begin{equation}\label{w-D3}
|\nabla^k\bb{w}|\leq C_{\delta}|x|^{-1-\frac{k}2}(1+|\bb{v}_{\rm trans}|s_{\bb{v}_{\rm trans}}(x))^{-\frac{k}2-\delta}\,(k=0,1),\quad \forall\,0<\delta< \tfrac14
\end{equation}
are very small. Here $s_{\bb{v}_{\rm trans}}(x)\triangleq |x|- {x}\cdot\bb{v}_{\rm trans}/|\bb{v}_{\rm trans}|$, and $r$ is a number satisfying $3<r<\infty$. Note that, the size of the small assumption on $\bb{u}_0$ and $\bb{w}$ no longer depends on $q$.  Unlike Borchers-Miyakawa \cite{BM95}, Shibata in \cite{Shi99} used the $L^p$-$L^q$ estimates of the Oseen semigroup established in \cite{KS98}, and then viewed the linear term $\bb{w}\cdot\nabla \bb{u}+\bb{u}\cdot\nabla \bb{w}$ as a perturbation from the Oseen semigroup by splitting the integral of the Duhamel term on account of better properties of $\bb{w}$ and $\nabla \bb{w}$. Later, Enomoto-Shibata \cite{ES05} proved that the classical $L^{\infty}$-estimate also holds for the Oseen semigroup and then used it to refine the $L^{\infty}$-decay rate in \eqref{est1-1} to the shaper $t^{-\frac12}$, only needing that $\bb{w}$ satisfies the summability property
\begin{equation}\label{decay-estimate-1}
\bb{w}\in L^{3+}(\Omega)\cap  L^{3-}(\Omega), \quad \;\, \nabla \bb{w}\in L^{\frac32+}(\Omega)\cap  L^{\frac32-}(\Omega),
\end{equation}
which is weaker than \eqref{w-D3}. This result shows  the stability of the steady flow obtained by \cite{Shi99} in $L^3$ which satisfies \eqref{w-D3}.

For the rotational case $\bb{\omega}\neq \bb{0}$, we assume that
$$\bb{v}_{\rm trans}=v_{\rm trans}\bb{e}\, \;\;\text{\rm and}\;\; \bb{\omega}=\omega\bb{e}_1,\,\quad\; v_{\rm trans}\geq 0,\;\;\omega>0,\;\;\bb{e}_1\triangleq(1,0,0).$$
Let $\mathfrak{R}\triangleq v_{\rm trans}\bb{e}\cdot\bb{e}_1$.
By the Mozzi-Chasles transformation,  systems \eqref{NS0} and \eqref{S0} become
\begin{equation}\label{N}
\left\{\begin{aligned}
&\partial_t\bb{v}-\Delta \bb{v}-\mathfrak{R}\partial_1\bb{v}-{\omega}((\bb{e}_1\times  {x})\cdot\nabla -\bb{e}_1\times) \bb{v}+\bb{v}\cdot\nabla \bb{v}+\nabla \theta=\bb{f}\; \text{in }\Omega\times (0,\infty),\\
&\Div \bb{v}=0\quad \text{in }\Omega\times [0,\infty),\\
&\bb{v}|_{\partial\Omega}=\bb{w}^*,\quad \bb{v}\to \bb{0} \,(| {x}|\to\infty),\quad\bb{v}|_{t=0}=\bb{v}_0,
\end{aligned}\right.
\end{equation}
and
\begin{equation}\label{S}
\left\{\begin{aligned}
&-\Delta \bb{w}-\mathfrak{R}\partial_1\bb{w}-{\omega}((\bb{e}_1\times  {x})\cdot\nabla -\bb{e}_1\times) \bb{w}+\bb{w}\cdot\nabla \bb{w}+\nabla \theta_s=\bb{f}\quad \text{in }\Omega,\\
&\Div \bb{w}=0\quad \text{in }\Omega,\\
&\bb{w}|_{\partial\Omega}=\bb{w}^*,\quad \bb{w}\to \bb{0} \,(| {x}|\to\infty),
\end{aligned}\right.
\end{equation}
 respectively. Accordingly, problem \eqref{P0} can be rewritten as
  \begin{equation}\label{P}
\left\{\begin{aligned}
&\partial_t \bb{u}-\Delta \bb{u}-\mathfrak{R}\partial_1\bb{u}-{\omega}(\bb{e}_1\times {x}\cdot\nabla-\bb{e}_1\times) \bb{u}\\
&\quad\quad\quad\quad\quad+\bb{w}\cdot\nabla \bb{u}+\bb{u}\cdot\nabla \bb{w}+\bb{u}\cdot\nabla \bb{u}+\nabla P=\bb{0}\quad \text{in }\Omega\times (0,\infty),\\
&\Div \bb{u}=0\quad \text{in }\Omega\times [0,\infty),\\
&\bb{u}|_{\partial\Omega}=\bb{0},\quad \bb{u}\to \bb{0}\,(| {x}|\to\infty),\quad \bb{u}|_{t=0}= \bb{u}_0.
\end{aligned}\right.
\end{equation}
When $\Rr=0$, Hishida-Shibata \cite{HS09} proved problem \eqref{P} admits a unique global solutions in $L^{3,\infty}(\Omega)$ satisfying
\[\|\bb{u}\|_{L^r(\Omega)}\lesssim t^{-\frac12+\frac3{2r}},\quad \quad \forall 3<r<q\]
with $3<q<\infty$  if the $L^{3,\infty}$-norms of  $\bb{u}_0$ and $\bb{w}$, depending on $q$, are small enough. Such result implies the stability of the stationary flows $\bb{w}$ obtained by \cite{Gal03} in $L^{3,\infty}$.
The key  consists of two points, one is  the $L^p$-$L^q$ estimates of the Stokes semigroup with rotating effect generated by a principal part of the linearized operator of \eqref{P}, the another is a clever interpolation technique due to Yamazaki \cite{Ya00} which enables them to deal with the term $\Div(\bb{u}\cdot\nabla \bb{w}+\bb{w}\cdot\nabla \bb{u})$ as a perturbation from this semigroup.  For $\Rr\neq 0$, Shibata \cite{Shi08} obtained the same results as \cite{ES05} for problem \eqref{P} with $\bb{w}$ satisfying \eqref{decay-estimate-1}. This shows the stability of the steady flows with \eqref{decay-estimate-1} obtained in Theorem 4.4 of \cite{GK11}.  For such steady flows, Galdi-Kyed \cite{GK11} further proved that they satisfy anisotropic pointwise decay estimates with wake structure, that is,
\begin{proposition}[\cite{GK11}]\label{GK-Prop-1}
Let $0<|\Rr|\le \Rr^*$, $0<\omega<\omega^*$  and $\Omega\subset \R^3$ be an exterior domain of class $C^2$. There exists a constant $\eta=\eta_{\Omega,\Rr^*,\omega^*}>0$ such that if $\bb{f}=\Div F\in L^2(\Omega)$ with compact support and $\bb w_*\in W^{3/2,2}(\partial\Omega)$ satisfy
\[\sup_{x\in \Omega}((1+|x|)^2|\bb F|)+\|\bb{f}\|_{L^2(\Omega)}+\|\bb w_*\|_{W^{3/2,2}(\partial\Omega)}<\eta,\]
then problem \eqref{S} possesses a unique solution $\bb{w}$ in  $\dot{W}^{1,2}(\Omega)\cap L^6(\Omega)$. Moreover, this solution satisfies
\begin{equation}\label{S.1}
|\nabla^k\bb{w}(x)|\leq C_{\varepsilon}|x|^{-1-\frac{k}2}(1+|\mathfrak{R}|s_{\mathfrak{R}}(x))^{-\frac12+\varepsilon}\,( k=0,1),\quad \forall\varepsilon\in \big(0,\tfrac12\big),
\end{equation}
where $s_{\Rr}(x)\triangleq |x|+\tfrac{\Rr}{|\Rr|}x_1$.
\end{proposition}

As can be seen from the previous discussion, the stability of steady flows strongly depends on the $L^p$-$L^q$ estimates of 
the semigroup generated by different forms of linear operators.  In this sense, to establish $L^p$-$L^q$ estimate of the semigroup is of independent interest. 
This paper is devoted to showing that the Oseen operator around a steady flow satisfying  \eqref{S.1} in the exterior of a rotating obstacle generates a $C_0$ semigroup, that is, the solution map $\bb{u}_0\mapsto\bb{u}(t)$ satisfying the following system
\begin{equation*}%\label{L}
\left\{\begin{aligned}
&\partial_t \bb{u}-\Delta \bb{u}-\mathfrak{R}\partial_1\bb{u}-{\omega}((\bb{e}_1\times {x})\cdot\nabla-\bb{e}_1\times) \bb{u}+\bb{u}\cdot\nabla \bb{w}+\bb{w}\cdot\nabla \bb{u}=-\nabla P \\
&\Div \bb{u}=0,\\
&\bb{u}|_{\partial\Omega}=\bb{0},\quad \bb{u}\to \bb{0}\,(|x|\to\infty), \quad \bb{u}|_{t=0}=\bb{u}_0.
\end{aligned}\right.
\end{equation*}
in $\Omega\times(0,\infty)$ defines a $C_0$ semigroup. And then we establish a series of $L^p$-$L^q$ estimates  of this $C_0$ semigroup. As an application, we can show the stability of the steady flow satisfying \eqref{S.1} in the sense of $\lim_{t\to\infty}\|\bb{u}(t)\|_{L^3(\Omega)}=0$ and
\[t^{\frac12-\frac3{2q}}\|\bb{u}(t)\|_{L^q(\Omega)}+t^{\frac12}\|\nabla \bb{u}\|_{L^3(\Omega)}\leq C,\;\;\forall \;3\leq q< \infty.\]
if the constant $C_{\varepsilon}$ in \eqref{S.1} and the $L^3$-norm of $\bb{u}_0$ are small enough.

\subsection{Notations.}
To state main results more precisely, we will outline some notations used throughout the paper.   $\overline{D}$ and $D^c$ mean the closure and complement of the domain  $D$ of $\R^3$, respectively. Given a vector or matrix $A$, $A^{T}$ means the transpose of $A$.  We denote $C_{a,b,\cdots}$ as a positive constant
depending only on the quantities $a,b,\cdots$, and  nonessential constants $C$ and $C_{a,b,\cdots}$ may change from line to line. In addition, we denote $\lesssim$ and $\lesssim_{a,b,\cdots}$ as $\leq C$ and $\leq C_{a,b,\cdots}$, respectively.
As usual, we use the following differential symbols:
\begin{align*}
&\partial_t=\partial/\partial_t,\;\;\partial_j=\partial/\partial x_j, \;\; \nabla=(\partial_1,\partial_2,\partial_3),\;\;\Delta=\partial_1^2+\partial_2^2+\partial_3^2,\\ &\partial^{\alpha}_x=\partial^{\alpha_1}_1\partial^{\alpha_2}_2
\partial^{\alpha_3}_3,\;\;\nabla^j=\{\partial^{\alpha}_x, \,|\alpha|=j\geq 2\},\;\;\alpha=(\alpha_1,\alpha_2,\alpha_3),\;\;|\alpha|\triangleq\alpha_1+
\alpha_2+\alpha_3,
\end{align*}
Moreover, we employ the following  special sets:
\[B_{r}=\{x\in\R^3\,|\,|x|\leq r\},\quad B_{r_1,r_2}=\{x\in\R^3\,|\,r_1\leq |x|\leq r_2\},\quad \Omega_{r}=\Omega\cap B_r.\]
In particular, $R>0$ denotes a fixed number such that $\Omega^c\subset B_R$.

To distinguish with  scale functions, we  shall use  bold-face letter to denote three dimensional vector valued functions,   and  the black-board bold letters
to denote the corresponding function spaces, i.e.
\begin{align*}
\mathbb{C}^{\infty}_0(D)\;&=\{\bb{u}=(u_1,u_2,u_3)\,|\,u_j\in C^{\infty}_0(D),\,j=1,2,3\},\\
\mathbb{C}^{\infty}_{0,\sigma}(D)&=\{\bb{u}\in \mathbb{C}^{\infty}_0(D)\,|\,\Div \bb{u}=0\},\\
\mathbb{L}^p(D)\;\,\,&=\{\bb{u}=(u_1,u_2,u_3)\,|\,u_j\in L^p(D),\,j=1,2,3\},\quad \|\bb{u}\|_{\mathbb{L}^p(D)}=\sum^3_{j=1}\|u_j\|_{L^p(D)},\\
\mathbb{J}^p (D)\;\;\,&=\text{\rm the closure in }  \mathbb{L}^p(D)\text{ \rm of } \mathbb{C}^{\infty}_{0,\sigma}(D),\quad 1<p<\infty,\\
\mathbb{L}^p_{\ell}(D)\;\;&=\{\bb{f}\in \mathbb{L}^p(D)\,|\,\bb{f}(x)=\bb{0} \;\;\text{\rm in } B^c_{\ell}\}.
\end{align*}
etc, if $D$ is any domain of $\R^3$.
For two Banach space $X$ and $Y$, $(\mathcal{L}(X,Y), \|\cdot\|_{\mathcal{L}(X,Y)})$ denotes the Banach space of all bounded linear operators from $X$ into $Y$  and we set $\mathcal{L}(X)=\mathcal{L}(X,X)$.   In addition, $\mathscr{A}(I,X)$ means the set of all $X$-valued holomorphic functions in $I$, and $C(I;X)$ ($C_b(I;X)$) the set of all $X$-valued (bounded) continuous functions in $I$.

We now recall the well-known Helmholtz decomposition in \cite{Miy82,SS92}:
\begin{equation}\label{HD}
\mathbb{L}^p(\Omega)=\mathbb{J}^p (\Omega)\oplus \mathbb{G}^p(\Omega),\quad \mathbb{G}^p(\Omega)\triangleq\{\nabla\varphi\in \mathbb{L}^p(\Omega)\,|\, \varphi\in L^p_{{\rm loc}}(\bar{\Omega})\}.
\end{equation}
Let  $\mathcal{P}_{\Omega}$ be the  projector operator from $\mathbb{L}^p(\Omega)$ to $\mathbb{J}^p (\Omega)$. For every $ \Rr,\omega\in \R$, we define the Oseen operator with rotating effect
\begin{equation}\label{Op-L}
\mathcal{L}_{\Rr,{\omega}}=\mathcal{P}_{\Omega}L_{\mathfrak{R},{\omega}},\quad\;
L_{\mathfrak{R},{\omega}}\triangleq -\Delta -\Rr\partial_1-{\omega}\big((\bb{e}_1\times  {x})\cdot\nabla -\bb{e}_1\times \big),\end{equation}
and its perturbed operator around $\bb{w}$ satisfying \eqref{S.1}
\begin{equation}\label{Ope}
\mathcal{L}_{\mathfrak{R},{\omega},\bb{w}}
=\mathcal{P}_{\Omega}L_{\mathfrak{R},{\omega},\bb{w}}
= \mathcal{P}_{\Omega}\big(L_{\Rr,{\omega}}+B_{\bb{w}}\big),\;\;
B_{\bb{w}}\triangleq
\big(\bb{u}\cdot \nabla \bb{w}+\bb{w}\cdot\nabla \bb{u}\big)
\end{equation}
where
$$D_p(\mathcal{L}_{\mathfrak{R},{\omega}})=D_p(\mathcal{L}_{\mathfrak{R},{\omega},\bb{w}})=\big\{\bb{u}\in {\W}^{2,p}(\Omega)\cap \mathbb{J}^p (\Omega)\,\big|\, \bb{u}|_{\partial\Omega}=\bb{0}, (\bb{e}_1\times {x})\cdot\nabla \bb{u}\in \mathbb{L}^p(\Omega)\big\}.$$
  For the uniformity of notations, we denote $\mathcal{L}_{\mathfrak{R},{\omega}}$ by $\mathcal{L}_{\mathfrak{R},{\omega},\bb 0}$.

%%%%%%%%%%%%%%%%%%%%%%%%%%%%%%%%%%%%%%%%%%%%%%%%%%%%%%%%%%%%%%%%%%%%%%%%%%%%
\subsection{Main results}
Let $D=\Omega$ or $\R^3$, and then for each $0<\varepsilon< \frac12$, define
$$\normmm{\bb{g}}_{\varepsilon,D} =\sup_{x\in D}\big((1+|x|)|\bb{g}(x)|
+(1+|x|)^{\frac32}|\nabla\bb{g}(x)|\big)
(1+|\mathfrak{R}|s_{\mathfrak{R}}(x))^{\frac12-\varepsilon}.$$
Shibata \cite{Shi10} proved that  $-\mathcal{L}_{\Rr,{\omega},\bb{0}}$ generates a $C_0$-semigroup $\{T_{\Rr,{\omega},\bb{0}}(t)\}_{t\geq 0}$ in $\mathbb{J}^p(\Omega)$ such that
\begin{equation*}
\|\nabla^jT_{\Rr,{\omega},\bb{0}}(t)\bb{u}_0\|_{\mathbb{L}^p(\Omega)}\lesssim_{\gamma} e^{\gamma t}t^{-j/2}\|\bb{u}_0\|_{\mathbb{L}^p(\Omega)},\,\,\; j=0,1,2,
\end{equation*}
for some $\gamma>0$. This estimate shows
\begin{align*}
\int^{\alpha}_0\|\mathcal{P}_{\Omega}B_{\bb{w}}T_{\Rr,{\omega},\bb{0}}(t)\|_{\mathcal{L}(\mathbb{J}^p(\Omega))}\,\mathrm{d}t\lesssim_{\gamma, \Rr,{\omega}}\alpha^{\frac12}\normmm{\bb{w}}_{\varepsilon,\Omega},\;\;0<\alpha<1.
\end{align*}
This fact together with  $D_p(\mathcal{L}_{\Rr,{\omega},\bb{0}})\subset D_p(\mathcal{P}_{\Omega}B_{\bb{w}})$ yields that  $-\mathcal{L}_{\Rr,{\omega},\bb{w}}$ generates a $C_0$-semigroup $\{T_{\Rr,{\omega},\bb{w}}(t)\}_{t\geq 0}$ in $\mathbb{J}^p(\Omega)$ by the perturbation theorem in \cite{Had05}.
  In the same way. its dual operator $-\mathcal{L}^*_{\Rr,{\omega},\bb{w}}$ generates  a $C_0$-semigroup $\{T^{*}_{\Rr,{\omega},\bb{w}}(t)\}_{t\geq 0}$ in $\mathbb{J}^p(\Omega)$.
  Summing up, we have

\begin{proposition}\label{TH1}
For every $p\in (1,\infty)$, $-\mathcal{L}_{\mathfrak{R},{\omega},\bb{w}}$ and $-\mathcal{L}^*_{\mathfrak{R},{\omega},\bb{w}}$ generate $C_0$-semigroups $\{T_{\mathfrak{R},{\omega},\bb{w}}(t)\}_{t\geq 0}$ and $\{T^*_{\mathfrak{R},{\omega},\bb{w}}(t)\}_{t\geq 0}$ in $\mathbb{J}^p (\Omega)$, respectively. In particular,  $\{T_{\mathfrak{R},0,\bb{w}}(t)\}_{t\geq 0}$ and $\{T^*_{\mathfrak{R},0,\bb{w}}(t)\}_{t\geq 0}$ are analytic semigroups in $\mathbb{J}^p (\Omega)$.
\end{proposition}

\begin{remark}\rm \label{Rem1-1}
When ${\omega}=0$, for every $\delta>0$ and $\bb{u}\in D_p(\mathcal{P}_{\Omega}\Delta)$, we observe \begin{align*}
\|\mathcal{P}_{\Omega}(-\mathfrak{R}\partial_1\bb{u}+\bb{w}\cdot
\nabla\bb{u}+\bb{u}\cdot\nabla\bb{w})\|_{\mathbb{L}^p}\leq & (|\mathfrak{R}|+\normmm{\bb{w}}_{\varepsilon})\|\bb{u}\|_{\mathbb{W}^{1,p}(\Omega)}\\
\leq & \delta\|\bb{u}\|_{{\W}^{2,p}(\Omega)}+C_{\delta,\mathfrak{R}, \normmm{\bb{w}}_{\varepsilon}}\|\bb{u}\|_{\mathbb{L}^p(\Omega)}\\
\leq &C\delta\|\mathcal {P}_{\Omega}\Delta\bb{u}\|_{\mathbb{L}^p(\Omega)}+C_{\delta,\mathfrak{R},\normmm{\bb{w}}_{\varepsilon}}\|\bb{u}\|_{\mathbb{L}^p(\Omega)}
\end{align*}
by  Corollary to Theorem 1.7 of \cite{Miy82}. This, together with Theorem X.54 in \cite{RS75}, shows that $\mathcal{L}_{\mathfrak{R},0,\bb{w}}$
 generates a holomorphic semigroup in $\mathbb{J}^p (\Omega)$. So does $\mathcal{L}^*_{\mathfrak{R},0,\bb{w}}$.
\end{remark}

\begin{remark}\rm\label{Rem1-2}
We can not expect to control $(\bb{e}_1\times  {x})\cdot\nabla$ by  $-\Delta$ since the drift operator $(\bb{e}_1\times  {x})\cdot\nabla$ is a variable coefficient growing at large distance.  This implies that the nonstationary  problem associated to  $\mathcal{L}_{\mathfrak{R},{\omega},\bb{w}}$ contains hyperbolic features if ${\omega}\neq 0$.  Hence, the operator $\mathcal{L}_{\mathfrak{R},{\omega},\bb{w}}\,(\omega\neq 0)$ can only generates a $C_0$ semigroup in $\mathbb{J}^p (\Omega)$. This fact was verified rigorously by Farwig-Neustupa in \cite{FN08}, which proved  that the essential spectrum of the operator $\mathcal{L}_{\mathfrak{R},{\omega},\bb{0}}$ coincides with
\begin{equation}\label{spec}
\bigcup_{\ell\in\Z}\Big\{\sqrt{-1}{\omega}\ell+\{\lambda\in \C\,|\,\mathfrak{R}^2\mathop{\rm Re}\,\lambda+(\mathop{\rm Im}\, \lambda)^2>0\}\Big\}.
\end{equation}
\end{remark}

Now we are in position to state the main results.
\begin{theorem}\label{TH2} Assume that $0<\Rr_*\leq |\mathfrak{R}|\leq \Rr^*$, $|{\omega}|\leq {\omega}^*$ and $1<p<\infty$. Let $\varepsilon\in (0,\frac12)$ if $p\geq \frac65$ otherwise $\varepsilon \in (0,  3-\frac3p)$.
Then there exists a constant $\eta=\eta_{p,\mathfrak{R}_*,\mathfrak{R}^*,{\omega}^*}>0$, such that if
$\normmm{\bb{w}}_{\varepsilon,\Omega}<\eta$,
then for $\bb{f}\in \mathbb{J}^p (\Omega)$,
\begin{align}
&\|T_{\mathfrak{R},{\omega},\bb{w}}(t)\bb{f}\|_{\mathbb{L}^q(\Omega)},\,
\|T^*_{\mathfrak{R},{\omega},\bb{w}}(t)\bb{f}\|_{\mathbb{L}^q(\Omega)}\leq Ct^{-\frac32(\frac1p-\frac1q)}\|\bb{f}\|_{\mathbb{L}^p(\Omega)},\;\;p\leq q<\infty\label{sem-1}\\
&\|\nabla T_{\mathfrak{R},{\omega},\bb{w}}(t)\bb{f}\|_{\mathbb{L}^q(\Omega)},\,
\|\nabla T^*_{\mathfrak{R},{\omega},\bb{w}}(t)\bb{f}\|_{\mathbb{L}^q(\Omega)}\leq Ct^{-\frac12-\frac32(\frac1p-\frac1q)}\|\bb{f}\|_{\mathbb{L}^p(\Omega)},\,\, p\leq q\leq 3\label{sem-2}
\end{align}
with $C=C_{\mathfrak{R}_*,\mathfrak{R}^*,{\omega}^*}$.
\end{theorem}

In the light of Theorem \ref{TH1}, we reduce problem \eqref{P} to the integral equation \begin{equation}\label{eq.Inte}
\bb{u}=T_{\mathfrak{R},{\omega},\bb{w}}(t)\bb{u}_0+\int^t_0 T_{\mathfrak{R},{\omega},\bb{w}}(t-\tau)\mathcal{P}_{\Omega}(\bb{u}(\tau)\cdot\nabla) \bb{u}(\tau)\,\mathrm{d}\tau,
\end{equation}
With the help of Theorem \ref{TH2},  we can easily deduce the following result for \eqref{eq.Inte} by the classical Kato method, which implies the stability in $L^3$ of the steady flow satisfying \eqref{S} and \eqref{S.1}.
\begin{theorem}\label{TH3}
Assume that $0<\Rr_*\leq |\mathfrak{R}|\leq \Rr^*$,   $|{\omega}|\leq {\omega}^*$ and
 $\varepsilon\in (0,\tfrac12)$. Let $\bb{u}_0\in \mathbb{J}^3(\Omega)$. Then  there exists a constant $\eta=\eta_{\mathfrak{R}_*,\mathfrak{R}^*,{\omega}^*}>0$ such that if
\begin{equation}\label{c-w3}
\|\bb{u}_0\|_{\mathbb{L}^3(\Omega)}+\normmm{\bb{w}}_{\varepsilon,\Omega}\leq \eta
\end{equation}
then problem \eqref{eq.Inte} admits a unique global solution $\bb{u}$ satisfying
\[\bb{u}\in C_b([0,\infty);\mathbb{J}^3(\Omega)), \quad t^{\frac12}\nabla \bb{u}(t)\in C_b([0,\infty);\mathbb{L}^3(\Omega))\]
such that
\begin{align}
&\|\bb{u}(t)\|_{L^3(\Omega)}\to 0,\quad \text{as}\; t\to\infty.\label{sta-2}\\
&t^{\frac12-\frac3{2q}}\|\bb{u}(t)\|_{L^q(\Omega)}+t^{\frac12}\|\nabla \bb{u}\|_{L^3(\Omega)}\leq C,\;\;\forall \;3\leq q< \infty.\label{sta-1}
\end{align}
\end{theorem}

We would like to give the sketch proof of Theorem \ref{TH2}.  Due to \eqref{spec},  the traditional way to establish  $L^p$-$L^q$ estimates of semigroups no longer hold for $T_{\Rr,{\omega},\bb{w}}(t)$. So
we will adopt the domain decomposition to study
\begin{equation}\label{E-eq}
\left\{\begin{aligned}
&\partial_t \bb{u}+L_{\mathfrak{R},{\omega},\bb{w}}\bb{u}+\nabla P=\bb{0},\quad \Div \bb{u}=0\quad \text{\rm in } \Omega\times (0,\infty),\\
&\bb{u}|_{\partial\Omega}=\bb{0},\quad \bb{u}|_{t=0}=\bb{u}_0\in \mathbb{J}^p(\Omega)
\end{aligned}\right.
\end{equation}
and  then establish $L^p$-$L^q$ estimates of $T_{\Rr,{\omega},\bb{w}}(t)$. Roughly speaking,
let $\varphi\in C^{\infty}_0(B_{R+2})$ be a bump function with  $\varphi=1$ in $B_{R+1}$, and define $$\widetilde{\bb{v}}_0=(1-\varphi)\bb{u}_0+\B[\nabla \varphi\cdot\bb{u}_0]\in \mathbb{J}^p(\R^3)$$
 where  $\B$ is a Bogovski\v{i}'s operator, see  Lemma \ref{Lem.Bogo} below.
Let $\psi\in C^{\infty}_0(B_{R+1})$ satisfy $0\leq \psi\leq 1$ and $\psi=1$ in $B_{R+1/2}$.
Suppose that $(\widetilde{\bb{v}}(t),\widetilde{\theta}(t))$
solves
\begin{equation}\label{E-eq-add-2}
\left\{\begin{aligned}
&\partial_t \widetilde{\bb{v}}+L_{\mathfrak{R},{\omega},\overline{\bb{w}}}\widetilde{\bb{v}}+\nabla \widetilde{\theta}=\bb{0},\quad \Div \widetilde{\bb{v}}=0\quad \text{in }\R^3\times( 0,\infty),\\
&\widetilde{\bb{v}}|_{t=0}=\widetilde{\bb{v}}_0,\quad \int_{\Omega_{R+3}}\widetilde{\theta}(t)\,\mathrm{d}x=0
\end{aligned}\right.
\end{equation}
where  $\overline{\bb{w}}=(1-\psi)\bb{w}+\B[\nabla \psi\cdot\bb{w}]$ is
an extension of $\bb{w}$ to  $\R^3$ such that
\begin{equation}\label{S-1}
\overline{\bb{w}}=\bb{w}\quad \text{in } B^c_{R+1},\quad \Div \overline{\bb{w}}=0,\quad\normmm{\overline{\bb{w}}}_{\varepsilon,\R^3}\leq C\normmm{\bb{w}}_{\varepsilon,\Omega}.
\end{equation}
We decompose initial data $\bb{u}_0$ as follows:
\[\bb{u}_0=\bb{v}_0+\widetilde{\bb{u}}_0,\quad \bb{v}_0=(1-\varphi)\widetilde{\bb{v}}_0+\B[\nabla \varphi\cdot \widetilde{\bb{v}}_0].\]
This yields the following decomposition of the solution $\bb{u}$
\begin{equation*}
\left\{\begin{aligned}&\bb{u}=\bb{v}+\widetilde{\bb{u}},\;\;\; \bb{v}(t)=(1-\varphi)\widetilde{\bb{v}}(t)+\B[\nabla \varphi\cdot \widetilde{\bb{v}}(t)],\\
&P=\theta+\widetilde P, \quad\; \theta=(1-\varphi)\widetilde{\theta},\end{aligned}\right.\end{equation*}
 where $(\bb{v}, \theta)$ and $(\widetilde{\bb{u}}, \widetilde P)$ satisfy
\begin{equation*}
\left\{\begin{aligned}
&\partial_t \bb{v}+L_{\mathfrak{R},{\omega},\bb{w}}\bb{v}+\nabla \theta=\bb{F}(t),\quad \Div \bb{v}=0\quad \text{in }\Omega\times (0,\infty),\\
&\bb{v}|_{\partial\Omega}=\bb{0},\quad \bb{v}|_{t=0}=\bb{v}_0
\end{aligned}\right.
\end{equation*}
and
\begin{equation}\label{E-eq-add-3}
\left\{\begin{aligned}
&\partial_t \widetilde{\bb{u}}+L_{\mathfrak{R},{\omega},\bb{w}} \widetilde{\bb{u}}+\nabla \widetilde{P}=-\bb{F}(t),\quad \Div \widetilde{\bb{u}}=0\quad \text{in }\Omega\times( 0,\infty),\\
&\widetilde{\bb{u}}|_{\partial\Omega}=\bb{0},\quad \widetilde{\bb{u}}|_{t=0}=\widetilde{\bb{u}}_0\in \mathbb{L}^p_{R+2}(\Omega),
\end{aligned}\right.
\end{equation}
respectively. Here
\begin{align*}
\bb{F}(t)=&-\widetilde{\theta}\nabla \varphi+(\Delta \varphi)\widetilde{\bb{v}}+2(\nabla \varphi)\cdot\nabla \widetilde{\bb{v}}+\mathfrak{R}(\partial_1\varphi)\widetilde{\bb{v}}+{\omega}\big((\bb{e}_1
\times  {x})\cdot\nabla\varphi\big)\widetilde{\bb{v}}\\
&+(\bb{w}\cdot\nabla \varphi)\widetilde{\bb{v}}-\B[\nabla\varphi\cdot (L_{\mathfrak{R},{\omega},\bb{w}}\bb{v}+\nabla \widetilde{\theta})]+L_{\mathfrak{R},{\omega},\bb{w}}\B[\nabla\varphi\cdot \widetilde{\bb{v}}].
\end{align*}
Hence, to prove the $L^p$-$L^q$ estimates of $T_{\Rr,{\omega},\bb{w}}(t)$, it suffices to show the decay estimates with respect to $t$ of the solution maps of \eqref{E-eq-add-2} and \eqref{E-eq-add-3}.

We observe that equations \eqref{E-eq-add-2} with ${\omega}\neq 0$ can be reduced to the case ${\omega}=0$  under the rotation transformation \eqref{trans}.  Hence, we only prove the $L^p$-$L^q$ estimates of the solution map associated to the case $\omega=0$, which can be obtained by making use of
$L^p$-$L^q$ estimates of Oseen semigroups and the decay estimates \eqref{S-1} of $\overline{\bb{w}}$ and  viewing the additional linear term invoking $\overline{\bb{w}}$ as a perturbation from the Oseen semigroup by splitting the integral of the Duhamel term on account of \eqref{S-1}.

Now we turn to the study of decay estimates on $t$ of the solution map of \eqref{E-eq-add-3}. For this, we only to study the following problem
\begin{equation}\label{E-eq-add-4}
\left\{\begin{aligned}
&\partial_t \bb{u}+L_{\mathfrak{R},{\omega},\bb{w}} \bb{u}+\nabla {P}=\bb{0},\quad \Div \bb{u}=0
\quad\text{\rm in }\Omega\times (0,\infty),\\
&{\bb{u}}|_{\partial\Omega}=\bb{0},\quad {\bb{u}}|_{t=0}=\mathcal{P}_{\Omega}\bb{u}_0,\quad \bb{u}_0\in \mathbb{L}^p_{R+2}(\Omega)
\end{aligned}\right.
\end{equation}
by the homogenization principle since   $\bb{u}_0,\,\bb{F}(t)\in \mathbb{L}^p_{R+2}(\Omega)$. From $\bb{u}_0\in  \mathbb{L}^p_{R+2}$, we can expect that $(\lambda I+\mathcal{L}_{\mathfrak{R},{\omega},\bb{w}})^{-1}
\mathcal{P}_{\Omega}\bb{u}_0$ and corresponding pressure operator both process decay properties with respect to ${\rm Re }\lambda> 0$ in $\LL^p(\Omega)$ such that
\begin{equation}\label{est-1}
\bb{u}=T_{\mathfrak{R},{\omega},\bb{w}}(t)\mathcal{P}_{\Omega}\bb{u}_0=
\lim_{\ell\to\infty}\int^{\gamma+i\ell}_{\gamma -i\ell}e^{\gamma t}(\lambda I+\mathcal{L}_{\mathfrak{R},{\omega},\bb{w}})^{-1}
\mathcal{P}_{\Omega}\bb{u}_0\,\mathrm{d}\lambda,\;\;\; \gamma \geq 1.
\end{equation}
In addition, we can prove $(\lambda I+\mathcal{L}_{\mathfrak{R},{\omega},\bb{w}})^{-1}
\mathcal{P}_{\Omega}\bb{u}_0$ in $\LL^p(\Omega_{R+3})$ has some decay estimates with respect to ${\rm Re }\lambda\geq 0$, which enable us  show the key estimates, the local energy decay of $T_{\mathfrak{R},\omega,\bb{w}}(t)$, by following the idea in \cite{Iw89,KS98}.
This leads us to study the behavior  with respect to $\lambda$ of  $(\lambda I+\mathcal{L}_{\mathfrak{R},{\omega},\bb{w}})^{-1}\mathcal{P}_{\Omega}$ acting on $\LL^p_{R+2}(\Omega)$.

For this propose, we adopt the  ``splitting-gluing" argument to study the resolvent problem
\begin{equation}\label{eq-1-0}
(\lambda I+L_{\mathfrak{R},{\omega},\bb{w}})\bb{u}+\nabla P=\bb{f}\in \LL^p_{R+2}(\Omega),\quad \Div \bb{u}=0 \quad\text{in }\Omega, \quad \bb{u}|_{\partial\Omega}=\bb{0}.
\end{equation}
Roughly  speaking, we first split \eqref{eq-1-0} into a resolvent problem in a bounded domain
\begin{equation}\label{eq-1-2}\left\{\begin{aligned}
&(\lambda I+L_{\mathfrak{R},{\omega},\bb{w}})\bb{u}+\nabla P=\bb{f}\in \LL^p(\Omega_{R+3}),\\
& \Div \bb{u}=0 \quad\text{in }\Omega_{R+3}, \quad \bb{u}|_{\partial\Omega_{R+3}}=\bb{0}
\end{aligned}\right.
\end{equation}
and a resolvent problem in whole space
\begin{equation}\label{eq-1-3}
(\lambda I+L_{\mathfrak{R}, {\omega},\overline{\bb{w}}})\bb{u}+\nabla P=\bb{f}\in\LL^p_{R+2}(\R^3),\quad \Div \bb{u}=0 \quad\text{in }\R^3.
\end{equation}

For  \eqref{eq-1-2}, we can view it as a perturbation of the resolvent problem of the usual Stokes operator in $\Omega_{R+3}$. Thus we construct the solution operators $(\mathcal{R}^I_{\mathfrak{R},{\omega},\bb{w}}(\lambda),\mathring{\mathcal{Q}}_{\Omega_{R+3}}+\mathring{\Pi}^I_{\mathfrak{R},{\omega},\bb{w}}(\lambda))$ and establish their  decay estimates on $\lambda$,  see Theorem \ref{TH3-1} for details below.

For \eqref{eq-1-3}, we don't deal with it by the perturbation argument since $(\bb{e}_1\times {x})\cdot \nabla $ does not subordinate to $\Delta$.  Fortunately, we observe that the solution map to the corresponding nonstationary problem defines a $C_0$-semigroup $\{T^G_{\mathfrak{R},{\omega},\overline{\bb{w}}}(t)\}_{t\geq 0}$ in $\mathbb{J}^p(\R^3)$.  This enables us define the resolvent operator $\mathcal{R}^G_{\mathfrak{R},{\omega},\overline{\bb{w}}}(\lambda)$ of \eqref{eq-1-3} via the Laplace transform of $T^G_{\mathfrak{R},{\omega},\overline{\bb{w}}}(t)\mathcal{P}_{\R^3}$ and corresponding pressure operator $(\mathring{\mathcal{Q}}_{\R^3}+\mathring{\Pi}^G_{\mathfrak{R},{\omega},\overline{\bb{w}}}(\lambda))$ by Helmholtz decomposition.   To investigate the behavior on ${\rm Re}\lambda \geq 0$ of $\mathcal{R}^G_{\mathfrak{R},{\omega},\overline{\bb{w}}}(\lambda)$ acting on $\mathbb{L}^p_{R+2}(\R^3)$, we formally write
$$\mathcal{R}^G_{\mathfrak{R},{\omega},\overline{\bb{w}}}(\lambda)=\sum^{\infty}_{j=0}
(\mathcal{R}^G_{\mathfrak{R},{\omega},\bb{0}}(\lambda)B_{\overline{\bb{w}}})^j
\mathcal{R}^G_{\mathfrak{R},{\omega},\bb{0}}(\lambda).$$
Making use of the multiplier associated with $T^G_{\mathfrak{R},{\omega},\bb{0}}(t)\mathcal{P}_{\R^3}$ and the better pointwise estimates \eqref{S-1} of $\overline{\bb{w}}$,  we give out decay estimates of $\mathcal{R}^G_{\mathfrak{R},{\omega},\overline{\bb{w}}}(\lambda)$ with respect to  ${\rm Re}\,\lambda>0$ in $\mathcal{L}(\LL^p_{R+2}(\R^3),\mathbb{W}^{2,p}(\R^3))$.

To study the behavior near ${\rm Re}\lambda =0$ of $\mathcal{R}^G_{\mathfrak{R},{\omega},\overline{\bb{w}}}(\lambda)$ acting on $\mathbb{L}^p_{R+2}(\R^3)$, we first use the pointwise estimates of the kernel function of $T^G_{\mathfrak{R},\omega,\bb{0}}(t)\mathcal{P}_{\R^3}$ to construct a iterative scheme on account of the domain decomposition, and so prove $L^p$ estimates in small scale and decay estimates in large scale of $\mathcal{R}^G_{\mathfrak{R},{\omega},\overline{\bb{w}}}(\lambda)$ acting on  $\mathbb{L}^p_{R+2}(\R^3)$ which are uniformly  with respect to  ${\rm Re}\,\lambda>0$. Further, we establish a ``tree self-similar" iteration by the so called `` self-similar iteration"  and then obtain the decay estimates on ${\rm Re}\lambda \geq 0$ of $\mathcal{R}^G_{\mathfrak{R},{\omega},\overline{\bb{w}}}(\lambda)$ and $(\mathring{\mathcal{Q}}_{\R^3}+\mathring{\Pi}^G_{\mathfrak{R},{\omega},\overline{\bb{w}}}(\lambda))$ from $\mathbb{L}^p_{R+2}(\R^3)$ to local $\mathbb{L}^p$ spaces.

Next, we glue the solutions of  problems \eqref{eq-1-2} and \eqref{eq-1-3} to obtain the solution of problem \eqref{eq-1-0}. More precisely, Let $\bb{f}\in \mathbb{L}^p_{R+2}(\Omega)$ and denote  $\bb{f}_0$ by the zero extension to $\R^3$  of $\bb{f}$ and $\bb{f}_{\Omega_{R+3}}$ by the restriction on $\Omega_{R+3}$ of $\bb{f}$.
 By  the Bogovski\v{i} operator, we  construct the parametrix $(\Phi_{\mathfrak{R},{\omega},
\bb{w}}(\lambda)\bb{f},\Psi_{\mathfrak{R},{\omega},\bb{w}}(\lambda)\bb{f})$ to \eqref{eq-1-0}
\begin{equation*}
\left\{\begin{aligned}
&\Phi_{\mathfrak{R},{\omega},\bb{w}}(\lambda)\bb{f}=
(1-\varphi)\mathcal{R}^G_{\mathfrak{R},{\omega},\bb{w}}(\lambda)\bb{f}_0+\varphi \mathcal{R}^I_{\mathfrak{R},{\omega},\bb{w}}(\lambda)\bb{f}_{\Omega_{R+3}}\\
&\quad\quad\quad\quad\quad\quad+\B[\nabla \varphi\cdot (\mathcal{R}^G_{\mathfrak{R},{\omega},
\overline{\bb{w}}}(\lambda)\bb{f}_0-\mathcal{R}^I_{\mathfrak{R},{\omega},\bb{w}}(\lambda)\bb{f}_{\Omega_{R+3}})],\\
&\Psi_{\mathfrak{R},{\omega},\bb{w}}(\lambda)\bb{f}=(1-\varphi)(\mathring{\mathcal{Q}}_{\R^3}
+\mathring{\Pi}^G_{\mathfrak{R},{\omega},\overline{\bb{w}}}(\lambda))\bb{f}_0+\varphi(\mathring{
\mathcal{Q}}_{\Omega_{R+3}}+\mathring{\Pi}^I_{\mathfrak{R},{\omega},\bb{w}}(\lambda))\bb{f}_{\Omega_{R+3}}
\end{aligned}\right.
\end{equation*}
such that
\begin{equation*}
\left\{\begin{aligned}
&(\lambda I+L_{\mathfrak{R},{\omega},\bb{w}})\Phi_{\mathfrak{R},{\omega},\bb{w}}(\lambda)\bb{f}+\nabla \Psi_{\mathfrak{R},{\omega},\bb{w}}(\lambda)\bb{f}=(I+T+K_{\mathfrak{R},{\omega},\bb{w}}(\lambda))\bb{f}
\quad \text{in }\Omega,\\
&\Div \Phi_{\mathfrak{R},{\omega},\bb{w}}(\lambda)\bb{f}=0\quad\text{in }\Omega,\quad \Phi_{\mathfrak{R},{\omega},\bb{w}}(\lambda)\bb{f}|_{\partial\Omega}=\bb{0},
\end{aligned}\right.
\end{equation*}
where $T+K_{\mathfrak{R},{\omega},\bb{w}}(\lambda)$ is an compact operator in $\mathcal{L}(\LL^p_{R+2}(\Omega))$  and
 $K_{\mathfrak{R},{\omega},\bb{w}}(\lambda)$ tends to zero as $|\lambda|\to \infty$.
So we can show the operator $(I+T+K_{\mathfrak{R},{\omega},\bb{w}}(\lambda))$, from $\LL^p_{R+2}(\Omega)$ to itself, is reversible. Then, the resolvent operator $(\lambda I+\mathcal{L}_{\mathfrak{R},{\omega},\bb{w}})^{-1}
\mathcal{P}_{\Omega}\triangleq \mathcal{R}_{\Rr,\omega,\bb{w}}(\lambda)$ to \eqref{eq-1-0} and corresponding pressure operator $\Pi_{\Rr,\omega,\bb{w}}(\lambda)$ can be given by the following formulas:
\begin{align*}
&\mathcal{R}_{\Rr,\omega,\bb{w}}(\lambda)=\Phi_{\mathfrak{R},{\omega},\bb{w}}(\lambda)(I+T+K_{\mathfrak{R},{\omega},\bb{w}}(\lambda))^{-1},\\
&\Pi_{\Rr,\omega,\bb{w}}(\lambda)(\lambda)=\Psi_{\mathfrak{R},{\omega},\bb{w}}(\lambda)(I+T+K_{\mathfrak{R},{\omega},\bb{w}}(\lambda))^{-1}.
\end{align*}
These equalities help us give out the decay estimates of $\mathcal{R}_{\Rr,\omega,\bb{w}}(\lambda)\bb{f}$ and $\Pi_{\Rr,\omega,\bb{w}}(\lambda)\bb{f}$ with respect to ${\rm Re }\lambda>0$, and of $\mathcal{R}_{\Rr,\omega,\bb{w}}(\lambda)\bb{f}$ with respect to ${\rm Re }\lambda=0$ in $\LL^p(\Omega_{R+3})$.

The rest of the paper is organized as follows. In Section 2, we  show the $L^p$-$L^q$ estimates of $T^G_{\mathfrak{R},{\omega},\overline{\bb{w}}}(t)$ and decay estimates of $(\mathcal{R}^G_{\Rr,{\omega},\overline{\bb{w}}}(\lambda), \mathring{\Pi}^G_{\Rr,{\omega},\overline{\bb{w}}}(\lambda))$ with respect to ${\rm Re}\lambda>0$ acting on $\LL^p_{R+2}(\R^3)$. In Section 3, we investigate the behavior of $(R^I_{\Rr,{\omega},\bb{w}}(\lambda),\mathring{\Pi}^I_{\Rr,{\omega},\bb{w}}(\lambda))$ with regard to $\lambda$. In Section 4, we study the solvability of \eqref{eq-1-0} via constructing its parametrix. In Section 5, we show the behavior on $t$ of $T_{\mathfrak{R},{\omega},\bb{w}}(t)\mathcal{P}_{\Omega}$ acting on $\LL^p_{R+2}(\Omega)$.
Section 6 is devoted to the proof of Theorem \ref{TH2}. In the section 7,  we prove Theorem \ref{TH3}.  For the sake of readers, we give some useful technique lemmas or the well known results in Section 8.

%%%%%%%%%%%%%%%%%%%%%%%%%%%%%%%%%%%%%%%%%%%%%%%%%%%%%%%%%%%%%%%%%%%%%%%%%%%%%%%%%
\section{The resolvent problem in $\R^3$}
%%%%%%%%%%%%%%%%%%%%%%%%%%%%%%%%%%%%%%%%%%%%%%%%%%%%%%%%%%%%%%%%%%%%%%%%%%
\setcounter{section}{2}\setcounter{equation}{0}

In this section, we mainly study the resolvent problem
\begin{equation}\label{RPW}
\lambda \bb{u}+L_{\mathfrak{R},{\omega},\overline{\bb{w}}}\bb{u}+\nabla P=\bb{f},\quad \Div \bb{u}=0\quad \text{in }\R^3.
\end{equation}
First,  we define  $\mathcal{P}_{\R^3}$ (Leray projection operator)
and $\mathcal{Q}_{\R^3}$ as follows:
\begin{align}
&(\mathcal{P}_{\R^3}\bb{f})_i\triangleq \mathcal{F}^{-1}(\mathbb{P}(\xi)\hat{\bb f})_i=\sum^3_{j=1}\mathcal{F}^{-1}\Big(\big(\delta_{ij}
-\frac{\xi_i\xi_j}{|\xi|^2}\big)\hat{f}_j\Big),\quad i=1,2,3,\label{HD1}\\
&\mathcal{Q}_{\R^3}\bb{f}\triangleq\mathcal{F}^{-1}\Big(\frac{\sum_{j=1}^{3}
\xi_j\hat{f}_j(\xi)}{i|\xi|^2}\Big).\label{HD2}
\end{align}
It is well-known that $\Div \mathcal{P}_{\R^3}\bb{f}=0$ and
\begin{align}
\|\mathcal{P}_{\R^3}\bb{f}\|_{\mathbb{L}^p(\R^3)}+\|\nabla \mathcal{Q}_{\R^3}\bb{f}\|_{\mathbb{L}^p(\R^3)}\leq \|\bb{f}\|_{\mathbb{L}^p(\R^3)}.\label{P-est}
\end{align}
 For $\bb{f}\in \mathbb{L}^p(\R^3)$, we have the Helmholtz decomposition
\begin{equation}\label{HDG}
\bb{f}=\mathcal{P}_{\R^3}\bb{f}+\nabla \mathring{\mathcal{Q}}_{\R^3}\bb{f},
\end{equation}
which is the unique in the sense of $\int_{\Omega_{R+3}}\mathring{\mathcal{Q}}_{\R^3}\bb{f}\,\mathrm{d}x=0$. Here
\begin{equation}\label{Pi-G-1}
\mathring{\mathcal{Q}}_{\R^3}\bb{f}\triangleq\mathcal{Q}_{\R^3}\bb{f}-\frac1{|\Omega_{R+3}|}\int_{\Omega_{R+3}}\mathcal{Q}_{\R^3}\bb{f}\,\mathrm{d}x.
\end{equation}
We define the operator associated to problem \eqref{RPW} as follows
\begin{equation}\label{Op-G}
\left\{\begin{aligned}
&\mathcal{L}_{\mathfrak{R},{\omega},\overline{\bb{w}},\R^3}
=\mathcal{P}_{\R^3}L_{\Rr,{\omega},\overline{\bb{w}}}=L_{\Rr,{\omega},
\bb{0}}+\mathcal{P}_{\R^3}B_{\overline{\bb{w}}},\\
&D_p(\mathcal{L}_{\mathfrak{R},{\omega},\overline{\bb{w}},\R^3})
=\big\{\bb{u}\in {\W}^{2,p}(\R^3)\cap \mathbb{J}^p (\R^3)\,\big|\, (\bb{e}_1\times  {x})\cdot\nabla \bb{u}\in \mathbb{L}^p(\R^3)\big\}.
\end{aligned}\right.
\end{equation}
 Here, we used the fact that $\mathcal{P}_{\R^3}L_{\Rr,{\omega},\bb{0}}=L_{\Rr,{\omega},\bb{0}}\mathcal{P}_{\R^3}$

As we all know, $\mathcal{L}_{\Rr,{\omega},\bb{0}, \R^3}$ generates a $C_0$ semigroup $\{T^G_{\Rr,{\omega},\bb{0}}(t)\}_{t\geq 0}$ in $\mathbb{J}^p(\R^3)$ such that
\begin{equation*}
\|\nabla^jT^G_{\Rr,{\omega},\bb{0}}(t)\|_{\mathcal{L}(\mathbb{J}^p(\R^3),\mathbb{L}^q(\R^3))}\leq C_{p,q}t^{-\frac32(\frac1p-\frac1q)-\frac{j}2},\,\,1\leq q\leq p\leq \infty,\,\,j\leq 2,
\end{equation*}
see \cite{Shi10} for details. As a consequence of the perturbation theorem in \cite{Had05},
 we deduce  that $\mathcal{L}_{\Rr,{\omega},\overline{\bb{w}},\R^3}$ and  its dual operator $\mathcal{L}^*_{\Rr,{\omega},\overline{\bb{w}},\R^3}$ generate $C_0$-semigroups $\{T^G_{\Rr,{\omega},\overline{\bb{w}}}(t)\}_{t\geq 0}$ and $\{T^{G^*}_{\Rr,{\omega},\overline{\bb{w}}}(t)\}_{t\geq 0}$ in $\mathbb{J}^p(\R^3)$, respectively.

\subsection{$L^p$-$L^q$ estimates of $T^{G}_{\Rr,{\omega},\overline{\bb{w}}}(t)$
 and $T^{G^*}_{\Rr,{\omega},\overline{\bb{w}}}(t)$}\quad
 Consider
\begin{equation}\label{NSW}
\begin{cases}
\partial_t \bb{u}+L_{\mathfrak{R},{\omega},\overline{\bb{w}}}\bb{u}+\nabla P=\bb{0},\quad \Div \bb{u}=0\quad \text{\rm in}\; \R^3\times(0,\infty),\\
\bb{u}|_{t=0}=\bb{u}_0\in \mathbb{J}^p (\R^3).
\end{cases}
\end{equation}
Set
\begin{equation}\label{trans}
\left\{\begin{aligned}
&y=\mathcal{O}({\omega})x,\quad \widetilde{\bb{u}}(y,t)=\mathcal{O}({\omega}t)\bb{u}(\mathcal{O}^T({\omega}t)y,t),\\
&\widetilde{\bb{w}}(y,t)=\mathcal{O}({\omega}t)\overline{\bb{w}}(\mathcal{O}^T({\omega}t)y),\quad \widetilde{P}(y,t)=P(\mathcal{O}^T({\omega}t)y,t)
\end{aligned}\right.
\end{equation}
with
\[\mathcal{O}({\omega}t)=\left(
  \begin{array}{ccc}
    1 & 0 & 0 \\
    0 & \cos{{\omega}t} & -\sin{{\omega}t} \\
    0 & \sin{{\omega}t} & \cos{{\omega}t} \\
  \end{array}
\right).\]
Note that
\begin{equation}\label{est-w}
\normmm{\widetilde{\bb{w}}(t)}_{\varepsilon,\R^3}=\normmm{\overline{\bb{w}}}_{\varepsilon,
\R^3}\leq C\normmm{\bb{w}}_{\varepsilon,\Omega}\quad \text{for every }t\geq 0,
\end{equation}
then  problem \eqref{NSW} is equivalent to the nonautonomous system
\begin{equation}\label{NSW-1}
\left\{\begin{aligned}
&\partial_t \widetilde{\bb{u}}+L(t)\widetilde{\bb{u}}+\nabla \widetilde{P}=\bb{0},\quad \Div \widetilde{\bb{u}}=0\quad \text{\rm in }\R^3\times(0,\infty),\\
&\widetilde{\bb{u}}|_{t=0}=\bb{u}_0\in \mathbb{J}^p (\R^3)
\end{aligned}\right.
\end{equation}
where
$$L(t)=L_{\Rr,0,0}+B_{\widetilde{\bb{w}}(t)},\quad B_{\widetilde{\bb{w}}(t)}\triangleq B_{\overline{\bb{w}}}|_{{\rm replacing}\,\overline{\bb{w}} \,{\rm by}\, \widetilde{\bb{w}}(t)}.$$
From the parabolic evolution system theory  in \cite{Ama87, Lun95}, we conclude that $-\mathcal{P}_{\R^3}L(t)$ and its dual operator generate unique evolution operators $\{G(t, s)\}_{0\leq s\leq t}$ and $\{G^*(t, s)\}_{0\leq s\leq t}$ in $\mathbb{J}^p(\R^3)$, respectively, satisfying for $j\leq 2$
\begin{equation}\label{est-envo}
\|\nabla^jG(t, s)\|_{\mathcal{L}(\mathbb{J}^p(\R^3),\LL^p(\R^3))},\,\|\nabla^jG^*(t, s)\|_{\mathcal{L}(\mathbb{J}^p(\R^3),\LL^p(\R^3))}\lesssim_{\omega}e^{C_{\omega}(t-s)}(t-s)^{-\frac{j}2}.
\end{equation}
See Lemma \ref{Lem-evolu} below for details.
 This gives
\begin{equation}\label{sem-G-0}
\left\{\begin{aligned}
&\bb{u}=T^G_{\mathfrak{R},{\omega},\overline{\bb{w}}}(t)\bb{u}_0=\mathcal{O}^{T}({\omega}t)(G(t, 0)\bb{u}_0)(\mathcal{O}({\omega}t)x,t),\\
&T^{G^*}_{\Rr,\omega,\overline{\bb{w}}}(t)\bb{v}_0=\mathcal{Q}^T(-{\omega}t)
(G^*(t,0)\bb{v}_0)(\mathcal{Q}(-{\omega}t)x,t).
\end{aligned}\right.
\end{equation}
So it suffices to show the $L^p$-$L^q$ estimates of $G(t,s)$ and $G^*(t,s)$. It is well known that
\begin{align}\nonumber
T^{G}_{\Rr,0,\bb{0}}(t)\bb{f}(x)=&(4\pi t)^{-\frac32}\int_{\R^3}e^{\frac{|x+\mathfrak{R}t\bb{e}_1-y|^2}{4t}}\bb{f}(y)\,
\mathrm{d}y\\
=&(2\pi)^{-3}\int_{\R^3}e^{-(|\xi|^2-i\mathfrak{R}\xi_1)t}\hat{\bb{f}}
(\xi)e^{i {x}\cdot\xi}\,\mathrm{d}\xi\label{est2-1-1}
\end{align}
satisfies
\begin{equation}\label{est2-1-2}
\|\nabla^jT^{G}_{\Rr,0,\bb{0}}(t)\|_{\mathcal{L}(\mathbb{J}^p(\R^3),\mathbb{L}^q(\R^3))}\leq C t^{-\frac32(\frac1p-\frac1q)-\frac{j}2},\,\;j\leq 2,\;\,1\leq q\leq p\leq \infty,
\end{equation}
for details see Chapter $\text{VIII}$ in \cite{Gal11}. This helps us to deduce the following propositions.
\begin{proposition}\label{Pro-2-1}
 Let  $\varepsilon\in (0,\frac12)$, $p\in(1,\infty)$ and $\bb{f}\in \mathbb{J}^p (\R^3)$. Then, there exists a constant $\eta=\eta_{p,\varepsilon}>0$ such that if $
\normmm{\bb{w}}_{\varepsilon,\Omega}<\eta$, then
\begin{align}
&\|G(t,s)\bb{f}\|_{\mathbb{L}^q
(\R^3)}\leq C(t-s)^{-\frac32(\frac1p-\frac1q)}\|\bb{f}\|_{\mathbb{L}^p(\R^3)},\qquad\quad p\leq q\leq \infty,\label{semi-11}\\
&\|\nabla G(t,s)\bb{f}\|_{\mathbb{L}^q(\R^3)}\leq C(t-s)^{-\frac12-\frac32(\frac1p-\frac1q)}\|\bb{f}\|_{\mathbb{L}^p(\R^3)},\quad\, p\leq q\leq 3.\label{semi-11-1}
\end{align}
\end{proposition}

\begin{proof}
By Duhamel principle, we write $\widetilde{\bb{u}}(t,s)=G(t,s)\bb{f}$ as
\begin{align*}
\widetilde{\bb{u}}(t,s) =&T^{G}_{\Rr,0,\bb{0}}(t-s)\bb{f}+\int^t_sT^{G}_{\Rr,0,\bb{0}}(t-\tau)
\mathcal{P}_{\R^3}B_{\widetilde{\bb{w}}(\tau)}\widetilde{\bb{u}}(\tau,s)\,\mathrm{d}\tau \\
\triangleq& T^{G}_{\Rr,0,\bb{0}}(t-s)\bb{f}+L\widetilde{\bb{u}}(t,s).
\end{align*}
It is obvious that form \eqref{est2-1-2}
\begin{equation}\label{est2-1-3}
\|\nabla^k T^{G}_{\mathfrak{R},0,\bb{0}}(t-s)\bb{f}\|_{\mathbb{L}^q(\mathds{R}^3)}\leq (t-s)^{-\frac{j}2-\frac32(\frac1p-\frac1q)}\|\bb{f}\|_{\mathbb{L}^p(\mathds{R}^3)},\quad k=0,1,\;1<p\leq q\leq \infty.
\end{equation}
Thus, we next to prove this proposition by splitting the integral of the Duhamel term $L\widetilde{\bb{u}}(t,s)$ on account of \eqref{est-w}.
\vskip 0.2cm
\noindent{\bf \underline{Proof of \eqref{semi-11}}}.\quad We first estimate $\|L\widetilde{\bb{u}}(t,s)\|_{\LL^p(\R^3)}$. By H\"{o}lder's inequality and Lemma \ref{Lem.In}, we observe from \eqref{est-w} that
\begin{equation}\label{est2-1-4}
\|\widetilde{\bb{w}}(\tau)\|_{\LL^r(\R^3)}+\|\nabla \widetilde{\bb{w}}(\tau)\|_{\LL^{\ell}(\R^3)}\lesssim_{r,\ell} \normmm{\bb{w}}_{\varepsilon,\Omega},\quad \forall
r\in (\tfrac6{3-2\varepsilon},\infty],\; \ell \in (\tfrac3{2-\varepsilon},\infty].
\end{equation}
Hence we deduce by \eqref{est2-1-2} that for $q_0\in (\max(p,\frac32),\infty)$ satisfying $\frac1p-\frac1{q_0}< \frac13$,
\begin{align}
\|L\widetilde{\bb{u}}(t,s)\|_{\mathbb{L}^p(\mathds{R}^3)}
\lesssim& \int^t_s (t-\tau)^{-1+\frac32(\frac1p-\frac1{q_0})}\|\widetilde{\bb{w}}(\tau)\|_{\mathbb{L}^3(\mathds{R}^3)}\|\widetilde{\bb{u}}(\tau,s)\|_{\LL^{q_0}(\R^3)}\,\mathrm{d}\tau \\
\lesssim & \normmm{\bb{w}}_{\varepsilon,\Omega}\sup_{s\leq \tau\leq t}(\tau-s)^{-\frac32(\frac1p-\frac1{q_0})}\|\widetilde{\bb{u}}(\tau,s)\|_{\mathbb{L}^{q_0}(\mathds{R}^3)}\label{est2-1-6}
\end{align}
where we have used the fact that
\begin{equation}\label{est2-1-5}
T^{G}_{\mathfrak{R},0,\bb{0}}(t-\tau)
\mathcal{P}_{\mathds{R}^3}B_{\widetilde{\bb{w}}(\tau)}\widetilde{\bb{u}}(\tau,s)=\nabla\cdot T^{G}_{\mathfrak{R},0,\bb{0}}(t-\tau)
\mathcal{P}_{\mathds{R}^3} \big(\widetilde{\bb{w}}(\tau)\otimes \widetilde{\bb{u}}(\tau,s)+\widetilde{\bb{u}}(\tau,s)\otimes \widetilde{\bb{w}}(\tau)\big).\end{equation}

Next, we estimate $(t-s)^{\frac3{2p}}\|L\widetilde{\bb{u}}(t,s)\|_{\LL^{\infty}(\R^3)}$. When $t-s<2$,  by \eqref{est2-1-2},  \eqref{est2-1-4} and \eqref{est2-1-5}, we obtain for  $q_1\in (\max(3,p),\infty)$ satisfying  $\frac1p-\frac1{q_1}<\frac23$,
\begin{align}
&(t-s)^{\frac3{2p}}\|L\widetilde{\bb{u}}(t,s)\|_{\mathbb{L}^{\infty}(\mathds{R}^3)}\nonumber \\
\lesssim&_p(t-s)^{\frac3{2p}}\int^t_s (t-\tau)^{-\frac12-\frac3{2q_1}}\|\widetilde{\bb{w}}(\tau)\|_{\mathbb{L}^{\infty}(\mathds{R}^3)}\|\widetilde{\bb{u}}(\tau,s)\|_{\LL^{q_1}(\R^3)}\,\mathrm{d}\tau\nonumber\\
\lesssim&_p(t-s)^{\frac12}\normmm{\bb{w}}_{\varepsilon,\Omega}\sup_{s\leq \tau\leq t}(\tau-s)^{\frac32(\frac1p-\frac1{q_1})}\|\widetilde{\bb{u}}(\tau,s)\|_{\mathbb{L}^{q_1}(\mathds{R}^3)}\nonumber\\
\lesssim&_p\normmm{\bb{w}}_{\varepsilon,\Omega}\sup_{s\leq \tau\leq t}(\tau-s)^{\frac32(\frac1p-\frac1{q_1})}\|\widetilde{\bb{u}}(\tau,s)\|_{\mathbb{L}^{q_1}(\mathds{R}^3)}\label{est2-1-8}.
\end{align}
When $t-s>2$, we make the following decomposition
\begin{align*}
L\widetilde{\bb{u}}(t,s)=\Big[\int_s^{\frac{t+s}2}+\int_{\frac{t+s}2}^{t-1}+\int_{t-1}^t\Big]T^G_{\mathfrak{R},0,\bb{0}}(t-\tau)
\mathcal{P}_{\mathds{R}^3}B_{\widetilde{\bb{w}}(\tau)}\widetilde{\bb{u}}(\tau,s)\,\mathrm{d}\tau.\end{align*}
and then get  by \eqref{est2-1-2},  \eqref{est2-1-4} and \eqref{est2-1-5}
\begin{align*}
&(t-s)^{\frac3{2p}}\|L\widetilde{\bb{u}}(t,s)\|_{\mathbb{L}^{\infty}(\mathds{R}^3)} \\
\lesssim&_{p,\varepsilon}(t-s)^{\frac3{2p}}\Big(\int^{\frac{t+s}2}_s  (t-\tau)^{-1-\frac3{2q_1}}\|\widetilde{\bb{w}}(\tau)\|_{\mathbb{L}^3(\mathds{R}^3)}\|\widetilde{\bb{u}}(\tau,s)\|_{\LL^{q_1}(\R^3)}\,\mathrm{d}\tau\\
&+\int^{t-1}_{\frac{t+s}2} (t-\tau)^{-\frac12-\frac3{2r_1}}\|\widetilde{\bb{w}}(\tau)\|_{\mathbb{L}^{r_1}(\mathds{R}^3)}\|\widetilde{\bb{u}}(\tau,s)\|_{\LL^{\infty}(\R^3)}\,\mathrm{d}\tau\\
&+\int^{t}_{t-1} (t-\tau)^{-\frac12-\frac3{2r_2}}\|\widetilde{\bb{w}}(\tau)\|_{\mathbb{L}^{r_2}(\mathds{R}^3)}\|\widetilde{\bb{u}}(\tau,s)\|_{\LL^{\infty}(\R^3)}\,\mathrm{d}\tau\Big)\\
\lesssim&_{p,\varepsilon}\normmm{\bb{w}}_{\varepsilon,\Omega}\Big[(t-s)^{-1+\frac32(\frac1p-\frac1{q_1})} \int^{\frac{t+s}2}_s  (\tau-s)^{-\frac32(\frac1p-\frac1{q_1})}\,\mathrm{d}\tau\sup_{s\leq \tau\leq t}(\tau-s)^{\frac32(\frac1p-\frac1{q_1})}\|\widetilde{\bb{u}}(\tau,s)\|_{\mathbb{L}^{q_1}(\mathds{R}^3)}\\
&+\Big(\int^{t-1}_{\frac{t+s}2}  (t-\tau )^{-\frac12-\frac3{2r_1}}\,\mathrm{d}\tau+\int^{t}_{t-1} (t-\tau)^{-\frac12-\frac3{2r_2}}\,\mathrm{d}\tau\Big)\sup_{s\leq \tau\leq t}(\tau-s)^{\frac3{2p}}\|\widetilde{\bb{u}}(\tau,s)\|_{\mathbb{L}^{\infty}(\mathds{R}^3)}\Big]\\
\lesssim&_{p,\varepsilon}\normmm{\bb{w}}_{\varepsilon,\Omega}(\sup_{s\leq \tau\leq t}(\tau-s)^{\frac32(\frac1p-\frac1{q_1})}\|\widetilde{\bb{u}}(\tau,s)\|_{\mathbb{L}^{q_1}(\mathds{R}^3)}+\sup_{s\leq \tau\leq t}(\tau-s)^{\frac3{2p}}\|\widetilde{\bb{u}}(\tau,s)\|_{\mathbb{L}^{\infty}(\mathds{R}^3)})\end{align*}
with  $r_1\in (\tfrac6{3-2\varepsilon},3)$ and $ r_2\in (3,\infty)$.
This estimate, together with \eqref{est2-1-3}, \eqref{est2-1-6} and \eqref{est2-1-8} and the interpolation inequality between $\LL^p(\R^3)$ and $\LL^{\infty}(\R^3)$, yields that
\begin{align*}
&\sup_{s\leq t\leq \infty}\|\widetilde{\bb{u}}(t,s)\|_{\mathbb{L}^{p}(\mathds{R}^3)}+\sup_{s\leq t\leq \infty}(t-s)^{\frac3{2p}}\|\widetilde{\bb{u}}(t,s)\|_{\mathbb{L}^{\infty}(\mathds{R}^3)}\\
 \leq &C_p\|\bb{f}\|_{\mathbb{L}^p(\mathds{R}^3)}+C_{\varepsilon,p} \normmm{\bb{w}}_{\varepsilon,\Omega}
\big(\sup_{s\leq t\leq \infty}\|\widetilde{\bb{u}}(t,s)\|_{\mathbb{L}^{p}(\mathds{R}^3)}+\sup_{s\leq t\leq \infty}(t-s)^{\frac3{2p}}\|\widetilde{\bb{u}}(t,s)\|_{\mathbb{L}^{\infty}(\mathds{R}^3)}\big).
\end{align*}
Hence, we deduce \eqref{semi-11} if $C_{\varepsilon,p}\normmm{\bb{w}}_{\varepsilon,\Omega}<1$.
\vskip 0.3cm
\noindent{\underline{\bf Proof of \eqref{semi-11-1}}}.\;
Thanks to  \eqref{semi-11} and the fact: $G(t,s)=G(t,s_0)G(s_0,s)$ for all $ s\leq s_0\leq t$,  we only need prove \eqref{semi-11-1} with $1<p=q\leq 3$.

{\bf Case 1: $p< 3$}.\;  When $t-s\leq 2$, by \eqref{est-w}, \eqref{est2-1-2} and Lemma \ref{Lem.Hardy}, we deduce
\begin{align}\nonumber
&(t-s)^{\frac12}\|\nabla L\widetilde{\bb{u}}(t,s)\|_{\mathbb{L}^p(\mathds{R}^3)}\lesssim\|\nabla L\widetilde{\bb{u}}(t,s)\|_{\mathbb{L}^p(\mathds{R}^3)}\\
\lesssim&\int^t_{s}(t-\tau)^{-\frac12}\big(\|\widetilde{\bb{w}}(\tau)\|_{\mathbb{L}^{\infty}(\mathds{R}^3)}\|\nabla \widetilde{\bb u}(\tau,s)\|_{\mathbb{L}^p(\mathds{R}^3)}+\big\||x|\nabla \widetilde{\bb{w}}(\tau)\big\|_{\mathbb{L}^{\infty}(\mathds{R}^3)}\big\| \tfrac{\widetilde{\bb u}(\tau,s)}{|x|}\big\|_{\mathbb{L}^p(\mathds{R}^3)}\big)\,\mathrm{d}\tau\nonumber\\
\lesssim& \normmm{\bb{w}}_{\varepsilon,\Omega}\sup_{s\leq\tau\leq t}(\tau-s)^{\frac12}\|\nabla \widetilde{\bb{u}}(\tau,s)\|_{\mathbb{L}^p(\mathds{R}^3)}.\label{est2-1-10}
\end{align}
Similarly, when $t-s>2$, we have for $\alpha \in(0,\min(\frac13,\frac12-\varepsilon))$, $r>\frac{\alpha}3$,
\begin{align*}
&(t-s)^{\frac12}\|\nabla L\widetilde{\bb{u}}(t,s)\|_{\mathbb{L}^p(\mathds{R}^3)}\\
\lesssim&_{p,\varepsilon}(t-s)^{\frac12}\Big[\int^{t-1}_s (t-\tau)^{-1-\frac3{2r}}\big\||x|^{1-\alpha}s_{\mathfrak{R}}(x)^{\alpha}\widetilde{\bb{w}}(\tau)\big\|_{\mathbb{L}^{r}(\mathds{R}^3)}\big\| \tfrac{\widetilde{\bb u}(\tau,s)}{|x|^{1-\alpha}s_{\mathfrak{R}}(x)^{\alpha}}\big\|_{\mathbb{L}^p(\mathds{R}^3)}\,\mathrm{d}\tau\\
&+\int^t_{t-1}(t-\tau)^{-\frac12}\big(\|\widetilde{\bb{w}}(\tau)\|_{\mathbb{L}^{\infty}(\mathds{R}^3)}\|\nabla \widetilde{\bb u}(\tau,s)\|_{\mathbb{L}^p(\mathds{R}^3)}+\big\||x|\nabla \widetilde{\bb{w}}(\tau)\big\|_{\mathbb{L}^{\infty}(\mathds{R}^3)}\big\| \tfrac{\widetilde{\bb u}(\tau,s)}{|x|}\big\|_{\mathbb{L}^p(\mathds{R}^3)}\big)\,\mathrm{d}\tau\Big]\\
\lesssim&_{p,\varepsilon} \normmm{\bb{w}}_{\varepsilon,\Omega}\sup_{s\leq \tau\leq t}(\tau-s)^{\frac12}\|\nabla \widetilde{\bb{u}}(\tau,s)\|_{\mathbb{L}^p(\mathds{R}^3)}.
\end{align*}
where we have used the fact that for every $\delta>1$ and $0<\rho<1$
\begin{align}
\int^{t-1}_s(t-s)^{-{\delta}}(\tau-s)^{-\rho}\,\mathrm{d}\tau\leq& (t-s)^{-\delta}\int^{\frac{t-s}2}_s(\tau-s)^{-\rho}\,\mathrm{d}\tau+(t-s)^{-\rho}\int^{t-1}_{\frac{t-s}2}(t-s)^{-{\delta}}\,\mathrm{d}\tau\nonumber\\
\lesssim&(t-s)^{1-\delta-\rho}+(t-s)^{-\rho}\lesssim (t-s)^{-\rho},\quad \forall t-s>2.\label{est2-1-14-add-1}
\end{align}
This, together with \eqref{est2-1-3} and \eqref{est2-1-10}, yields
$$\sup_{s\leq t\leq \infty}(t-s)^{\frac12}\|\nabla \widetilde{\bb{u}}(t,s)\|_{\mathbb{L}^p(\R^3)}
\leq C\|\bb{f}\|_{\mathbb{L}^p(\mathds{R}^3)}+C'_{\varepsilon,p}\normmm{\bb{w}}_{\varepsilon,\Omega}\sup_{s\leq t\leq \infty}(t-s)^{\frac12}\|\nabla \widetilde{\bb{u}}(t,s)\|_{\LL^p(\R^3)}$$
which proves \eqref{semi-11-1} with $1<p=q<3$ if $C'_{\varepsilon,p}\normmm{\bb{w}}_{\varepsilon,\Omega}<1$.

{\bf Case 2: $p=3$}.\;
When $t-s<2$, by \eqref{est2-1-3} and \eqref{est2-1-4}, we have
\begin{align*}
&(t-s)^{\frac12}\|\nabla L\widetilde{\bb{u}}(t,s)\|_{\LL^3(\R^3)}\lesssim\|\nabla L\widetilde{\bb{u}}(t,s)\|_{\LL^3(\R^3)}\\
\lesssim& \int^t_s(t-\tau)^{-\frac12}\big(\|\widetilde{\bb{w}}(\tau)\|_{\LL^{\infty}(\R^3)}\|\nabla \widetilde{\bb u}(\tau,s)\|_{\LL^3(\R^3)}+\|\nabla\widetilde{\bb{w}}(\tau)\|_{\LL^{3}(\R^3)}\|\widetilde{\bb u}(\tau,s)\|_{\LL^{\infty}(\R^3)}\big)\,\mathrm{d}\tau\\
\lesssim& \normmm{\bb{w}}_{\varepsilon,\Omega}(\sup_{s\leq\tau\leq t}(\tau-s)^{\frac12}\|\nabla \widetilde{\bb{u}}(\tau,s)\|_{\LL^3(\R^3)}+\sup_{s\leq\tau\leq t}(\tau-s)^{\frac12}\|\widetilde{\bb{u}}(\tau,s)\|_{\LL^{\infty}(\R^3)}).
\end{align*}
When $t-s>2$,  by \eqref{est2-1-3}, \eqref{est2-1-4} and \eqref{est2-1-14-add-1}, we have for $\ell_0\in (6/(3-2\varepsilon),3)$,
\begin{align*}
&(t-s)^{\frac12}\|\nabla L\widetilde{\bb{u}}(t,s)\|_{\mathbb{L}^3(\mathds{R}^3)}\\
\lesssim&(t-s)^{\frac12}\Big[\int^{t-1}_{s}(t-\tau)^{-\frac12-\frac3{2\ell_0}}\|\widetilde{\bb{w}}(\tau)\|_{\mathbb{L}^{\ell_0}(\mathds{R}^3)}\|\widetilde{\bb u}(\tau,s)\|_{\mathbb{L}^{\infty}(\mathds{R}^3)}\,\mathrm{d}\tau\\
&+\int^{t}_{t-1}(t-\tau)^{-\frac12}\big(\|\widetilde{\bb{w}}(\tau)\|_{\mathbb{L}^{\infty}(\mathds{R}^3)}\|\nabla \widetilde{\bb u}(\tau,s)\|_{\mathbb{L}^3(\mathds{R}^3)}+\|\nabla\widetilde{\bb{w}}(\tau)\|_{\mathbb{L}^{3}(\mathds{R}^3)}\|\widetilde{\bb u}(\tau,s)\|_{\mathbb{L}^{\infty}(\mathds{R}^3)}\big)\,\mathrm{d}\tau\Big]\\
\lesssim& \normmm{\bb{w}}_{\varepsilon,\Omega}\big(\sup_{s\leq\tau\leq t}(\tau-s)^{\frac12}\|\nabla \widetilde{\bb{u}}(\tau,s)\|_{\mathbb{L}^3(\mathds{R}^3)}+\sup_{s\leq\tau\leq t}(\tau-s)^{\frac12}\|\widetilde{\bb{u}}(\tau,s)\|_{\mathbb{L}^{\infty}(\mathds{R}^3)}\big).
\end{align*}
 Thus, collecting the above two estimates, \eqref{semi-11} and \eqref{est2-1-3}, we deduce
\begin{align*}
&\sup_{s\leq t\leq \infty}(t-s)^{\frac12}\|\nabla\widetilde{\bb{u}}(t,s)\|_{\LL^3(\R^3)}\\
\leq& C(1+\normmm{\bb{w}}_{\varepsilon,\Omega})\|\bb{f}\|_{\mathbb{L}^p(\mathds{R}^3)}+C''_{\varepsilon}\normmm{\bb{w}}_{\varepsilon,\Omega}\sup_{s\leq t\leq \infty}(t-s)^{\frac12}\|\nabla \widetilde{\bb{u}}(t,s)\|_{\LL^3(\R^3)}.
\end{align*}
This yields  \eqref{semi-11-1} with $p=q=3$ if $C''_{\varepsilon}\normmm{\bb{w}}_{\varepsilon}<1$ and so ends  proof of this proposition.
\end{proof}

 \begin{proposition}\label{Pro-2-2}
Let $\varepsilon\in (0,\frac12)$, $p\in (1,\infty)$ and $q\in [p,\infty]$. Then  there exists a constant $\eta=\eta_{p,\varepsilon}>0$ such that if $\normmm{\bb{w}}_{\varepsilon,\Omega}<\eta$, then for $\bb{f}\in \mathbb{J}^p (\R^3)$, $k=0,1$,
\begin{equation*}
\|\nabla^k G^*(t,s)\bb{f}\|_{\mathbb{L}^q(\R^3)}\leq C(t-s)^{-\frac{k}2-\frac32(\frac1p-\frac1q)}\|\bb{f}\|_{\mathbb{L}^p(\R^3)},\quad (k,q)\neq (1,\infty).\end{equation*}
\end{proposition}
\begin{proof}
Thanks to Proposition \ref{Pro-2-1} and  the duality argument, we only need to prove
\begin{align}
&\|\nabla G^*(t,s)\bb{f}\|_{\mathbb{L}^p(\mathds{R}^3)}\leq C(t-s)^{-\frac{k}2}\|\bb{f}\|_{\mathbb{L}^p(\mathds{R}^3)},\label{est2-1-12}\\
&\|G^*(t,s)\bb{f}\|_{\mathbb{L}^{\infty}(\mathds{R}^3)}\leq C(t-s)^{-\frac3{2p}}\|\bb{f}\|_{\mathbb{L}^p(\mathds{R}^3)}.\label{est2-1-13}
\end{align}

 Let  $B^*_{\widetilde{\bb{w}}(t)}$ be the dual operator of $B_{\widetilde{\bb{w}}(t)}$.
  We write  $\widetilde{\bb{v}}(t,s)\triangleq G^*(t,s)\bb{f}$ into
\begin{align*}
\widetilde{\bb{v}}(t,s)=&T^{G}_{-\mathfrak{R},0,\bb{0}}(t-s)\bb{f}-\int^t_s
T^{G}_{-\mathfrak{R},0,\bb{0}}(t-\tau)\mathcal{P}_{\R^3}B^*_{\widetilde{\bb{w}}(\tau)}\widetilde{\bb{v}}(\tau,s)\,\mathrm{d}\tau\\
=&T^{G}_{-\mathfrak{R},0,\bb{0}}(t-s)\bb{f}+[L\widetilde{\bb{v}}(t,s)
\end{align*}
where we used  $T^{G^*}_{\mathfrak{R},0,\bb{0}}(t-s)=T^{G}_{-\mathfrak{R},0,\bb{0}}(t-s)$.
Note that, for every $\bb{\varphi}\in \mathbb{C}^{\infty}_{0,\sigma}(\R^3)$,
\[\big\langle \mathcal{P}_{\mathds{R}^3}B^*_{\widetilde{\bb{w}}(\tau)}\widetilde{\bb{v}}(\tau,s),
\bb{\varphi}\big\rangle
=\big\langle\nabla \widetilde{\bb{v}}(\tau,s), \big(\widetilde{\bb{w}}(\tau)\otimes \bb{\varphi}+\bb{\varphi}\otimes\widetilde{\bb{w}}(\tau)\big)\big\rangle.\]
So, by \eqref{est2-1-4} we get for  $m,q\in (1,\infty)$ and $r\in (\frac6{3-2\varepsilon},\infty]$ satisfying $\frac1m=\frac1r+\frac1q<1$
\begin{align}
\big\|\mathcal{P}_{\mathds{R}^3}B^*_{\widetilde{\bb{w}}(\tau)}\widetilde{\bb{v}}(\tau,s)\big\|_{\mathbb{L}^{m}(\mathds{R}^3)}\lesssim &\|\widetilde{\bb{w}}(\tau)\|_{\mathbb{L}^{r}(\mathds{R}^3)}\|\nabla \widetilde{\bb{v}}(\tau,s)\|_{\mathbb{L}^{q}(\mathds{R}^3)}\nonumber\\
\lesssim&_{r} \normmm{\bb{w}}_{\varepsilon,\Omega}\|\nabla \widetilde{\bb{v}}(\tau,s)\|_{\mathbb{L}^{q}(\mathds{R}^3)}\label{est2-1-14}
\end{align}

Suppose \eqref{est2-1-12} holds,  we have from \eqref{semi-11} and the duality argument that
\[\|\nabla \widetilde{\bb{v}}(t,s)\|_{\LL^q(\R^3)}\lesssim (t-s)^{-\frac12-\frac32(\frac1p-\frac1q)}\|\bb{f}\|_{\LL^p(\R^3)},\quad\forall\; 1<p\leq q<\infty,\;\; t>s\geq 0.\]
Hence, choosing $q_0\in [p,\infty)$ such that $q_0> \frac32$ and $\frac1p-\frac1{q_0}<\frac13$, we deduce by  \eqref{est2-1-2}  and \eqref{est2-1-14} with $(r,q)=(3,q_0)$ that for all $t>s\geq0$
\begin{align*}
&(t-s)^{\frac3{2p}}\|\widetilde{\bb{v}}(t,s)\|_{\mathbb{L}^{\infty}(\mathds{R}^3)}\\
\lesssim &_{p}\|\bb{f}\|_{\mathbb{L}^p(\mathds{R}^3)}+(t-s)^{\frac3{2p}}\int^t_s(t-\tau)^{-\frac12-\frac3{2q_0}}\|\mathcal{P}_{\R^3}B^*_{\widetilde{\bb{w}}(\tau)}\widetilde{\bb{v}}(\tau,s)\|_{\mathbb{L}^m(\R^3)}\,\mathrm{d}\\
\lesssim &_{p}\|\bb{f}\|_{\mathbb{L}^p(\mathds{R}^3)}+\normmm{\bb{w}}_{\varepsilon,\Omega}(t-s)^{\frac3{2p}}\int^t_s(t-\tau)^{-\frac12-\frac3{2q_0}}\|\nabla \widetilde{\bb{v}}(\tau,s)\|_{\mathbb{L}^{q_0}(\mathds{R}^3)}\,\mathrm{d}\tau\lesssim_{p} \|\bb{f}\|_{\mathbb{L}^p(\mathds{R}^3)}
\end{align*}
with $\frac1m=\frac1r+\frac1{q_0}$, which proves \eqref{est2-1-13}.

Next, we prove \eqref{est2-1-12}. Let $q'_0$ be the number such that $\frac1{q_0}+\frac1{q'_0}=1$. When $t-s\leq 2$, by \eqref{est2-1-2} and \eqref{est2-1-14} with $(r,q)=(\infty,q_0)$, we have
\begin{align}\nonumber
&(t-s)^{\frac12+\frac32(\frac1p-\frac1{q_0})}\|\nabla L\widetilde{\bb{v}}(t,s)\|_{\mathbb{L}^{q_0}(\mathds{R}^3)}\lesssim (t-s)^{\frac32(\frac1p-\frac1{q_0})}\|\nabla L\widetilde{\bb{v}}(t,s)\|_{\mathbb{L}^{q_0}(\mathds{R}^3)}\\
\lesssim& (t-s)^{\frac32(\frac1p-\frac1{q_0})}\int^t_s (t-\tau)^{-\frac12}\|\mathcal{P}_{\R^3}B^*_{\widetilde{\bb{w}}(\tau)}\widetilde{\bb{v}}(\tau,s)\|_{\mathbb{L}^{q_0}(\mathds{R}^3)}\,\mathrm{d}\tau\nonumber\\
\lesssim& (t-s)^{\frac32(\frac1p-\frac1{q_0})}\int^t_s (t-\tau)^{-\frac12}\normmm{\bb{w}}_{\varepsilon,\Omega}\|\nabla \widetilde{\bb{v}}(\tau,s)\|_{\mathbb{L}^{q_0}(\mathds{R}^3)}\,\mathrm{d}\tau\nonumber\\\lesssim&\normmm{\bb{w}}_{\varepsilon,\Omega}\sup_{s\leq \tau\leq t}(\tau-s)^{\frac12+\frac32(\frac1p-\frac1{q_0})}\|\nabla \widetilde{\bb{v}}(\tau,s)\|_{\mathbb{L}^{q_0}(\mathds{R}^3)}.\label{est2-1-16}
\end{align}
When $t-s>2$, making the  decomposition:
$$L\widetilde{\bb{v}}(t,s)=
\Big[\int_s^{t-1}+\int_{t-1}^t\Big] T^G_{-\mathfrak{R},0,\bb{0}}(t-\tau)\mathcal{P}_{\mathds{R}^3}B^*_{\widetilde{\bb{w}}(\tau)}\widetilde{\bb{v}}(\tau,s)\,\mathrm{d}\tau,$$
by \eqref{est2-1-2} and \eqref{est2-1-14} with $q=q_0$ and $r=r_0\in (\max(\frac6{3-2\varepsilon}),q'_0),3)$ or $r=\infty$, we get from \eqref{est2-1-14-add-1}
 \begin{align*}
 &(t-s)^{\frac12+\frac32(\frac1p-\frac1{q_0})}\|\nabla L\widetilde{\bb{v}}(t,s)\|_{\mathbb{L}^{q_0}(\mathds{R}^3)}\\
 \lesssim&_{\varepsilon,p}  (t-s)^{\frac12+\frac32(\frac1p-\frac1{q_0})}\Big[\int^{t-1}_s(t-\tau)^{-\frac12-\frac3{2{r_0}}}
 \|\mathcal{P}_{\mathds{R}^3}B^*_{\widetilde{\bb{w}}(\tau)}\widetilde{\bb{v}}(\tau,s)\|_{\mathbb{L}^{m}(\mathds{R}^3)}\,\mathrm{d}\tau\\
 &+\int^t_{t-1}(t-\tau)^{-\frac12}\|\mathcal{P}_{\mathds{R}^3}B^*_{\widetilde{\bb{w}}(\tau)}\widetilde{\bb{v}}(\tau,s)\|_{\mathbb{L}^{q_0}(\mathds{R}^3)}\,\mathrm{d}\tau\Big]\\
\lesssim&_{\varepsilon,p}  \normmm{\bb{w}}_{\varepsilon,\Omega}(t-s)^{\frac12+\frac32(\frac1p-\frac1{q_0})}\Big[\int^{t-1}_s(t-\tau)^{-\frac12-\frac3{2{r_0}}}
 \|\nabla \widetilde{\bb{v}}(\tau,s)\|_{\mathbb{L}^{q_0}(\mathds{R}^3)}\,\mathrm{d}\tau\\
 &+\int^t_{t-1}(t-\tau)^{-\frac12}\|\nabla \widetilde{\bb{v}}(\tau,s)\|_{\mathbb{L}^{q_0}(\mathds{R}^3)}\,\mathrm{d}\tau\Big]\\
 \lesssim&_{\varepsilon,p}\normmm{\bb{w}}_{\varepsilon,\Omega}\sup_{s\leq \tau\leq t}(\tau-s)^{\frac12+\frac32(\frac1p-\frac1{q_0})}\|\nabla \widetilde{\bb{v}}(\tau,s)\|_{\mathbb{L}^{q_0}(\mathds{R}^3)}.
 \end{align*}
with $\frac1m=\frac1{r_0}+\frac1{q_0}$. This estimate, combining with \eqref{est2-1-3} and \eqref{est2-1-16}, gives that
\begin{align*}
&\sup_{s\leq t\leq \infty}(t-s)^{\frac12+\frac32(\frac1p-\frac1{q_0})}\|\nabla \widetilde{\bb{v}}(t,s)\|_{\mathbb{L}^{q_0}(\mathds{R}^3)}\\
\leq &C \|\bb{f}\|_{\mathbb{L}^p(\mathds{R}^3)}+ C_{\varepsilon,p}\normmm{\bb{w}}_{\varepsilon,\Omega}
\sup_{s\leq t\leq \infty}(t-s)^{\frac12+\frac32(\frac1p-\frac1{q_0})}\|\nabla \widetilde{\bb{v}}(t,s)\|_{\mathbb{L}^{q_0}(\mathds{R}^3)}.
\end{align*}
Thus, we obtain
\begin{equation}\label{est2-1-17}
\sup_{s\leq t<\infty}(t-s)^{\frac12+\frac32(\frac1p-\frac1{p_0})}\|\nabla \widetilde{\bb{v}}(t,s)\|_{\mathbb{L}^{p_0}(\mathds{R}^3)}\leq C\|\bb{f}\|_{\mathbb{L}^p(\mathds{R}^3)}.
\end{equation}
 if $ C_{\varepsilon,p}\normmm{\bb{w}}_{\varepsilon,\Omega}<1$. Hence, we have by \eqref{est2-1-2} and  \eqref{est2-1-14} with $(r,q)=(3,q_0)$
\begin{align*}
&(t-s)^{\frac12}\|\nabla \widetilde{\bb{v}}(t,s)\|_{\mathbb{L}^p(\mathds{R}^3)}\\
\lesssim &\|\bb{f}\|_{\mathbb{L}^p(\mathds{R}^3)}+(t-s)^{\frac12}\int^t_s(t-\tau)^{-1+\frac32(\frac1p-\frac1{q_0})}\|\mathcal{P}_{\R^3}B^*_{\widetilde{\bb{w}}(\tau)}\widetilde{\bb{v}}(\tau,s)\|_{\mathbb{L}^{m}(\mathds{R}^3)}\,\mathrm{d}\tau\\
\lesssim &\|\bb{f}\|_{\mathbb{L}^p(\mathds{R}^3)}+\normmm{\bb{w}}_{\varepsilon,\Omega}(t-s)^{\frac12}\int^t_s(t-\tau)^{-1+\frac32(\frac1p-\frac1{q_0})}\|\nabla \widetilde{\bb{v}}(t,s)\|_{\mathbb{L}^{p_0}(\mathds{R}^3)}\,\mathrm{d}\tau\lesssim \|\bb{f}\|_{\mathbb{L}^p(\mathds{R}^3)}.
\end{align*}
with $\frac1m=\frac13+\frac1{q_0}$. This proves \eqref{est2-1-13}. So Proposition \ref{Pro-2-2} is proved.
\end{proof}

Thanks to \eqref{sem-G-0},  as a directly consequence of Proposition \ref{Pro-2-1}-\ref{Pro-2-2}, we have
 \begin{theorem}\label{TH2-2'}
Let $\varepsilon\in (0,\frac12)$, $p\in(1,\infty)$ and $\bb{f}\in \mathbb{J}^p(\R^3)$. Then there exists a constant $\eta=\eta_{\varepsilon,p}>0$ such that if $\normmm{\bb{w}}_{\varepsilon}\leq \eta$,
then
\begin{equation}\label{sem-G}
\left\{\begin{aligned}
&\|T^G_{\mathfrak{R},{\omega},\overline{\bb{w}}}(t)\bb{f}\|_{\mathbb{L}^q(\R^3)}\leq Ct^{-\frac32(\frac1q-\frac1p)}\|\bb{f}\|_{\mathbb{L}^{p}(\R^3)}, \quad\quad\;\;\;\;p\leq q\leq \infty,\\
&\|\nabla T^G_{\mathfrak{R},{\omega},\overline{\bb{w}}}(t)\bb{f}\|_{\mathbb{L}^q(\R^3)}\leq Ct^{-\frac{1}2-\frac32(\frac1q-\frac1p)}\|\bb{f}\|_{\mathbb{L}^{p}(\R^3)},\quad\; p\leq q\leq 3,
\end{aligned}\right.
\end{equation}
and for every $p\leq q\leq \infty$ and $k=0,1$
\begin{equation}\label{sem-G'}
\|\nabla^k T^{G^*}_{\mathfrak{R},{\omega},\overline{\bb{w}}}(t)\bb{f}\|_{\mathbb{L}^q(\R^3)}\leq Ct^{-\frac{k}2-\frac32(\frac1q-\frac1p)}\|\bb{f}\|_{\mathbb{L}^{p}(\R^3)},\qquad\, (k,p)\neq (1,\infty).
\end{equation}
\end{theorem}

\subsection{The resolvent estimates}\; Now we consider the resolvent problem \eqref{RPW}. Let
\begin{equation}\label{eq.R}
\mathcal{R}^G_{\mathfrak{R},{\omega},\overline{\bb{w}}}(\lambda)=\int^{\infty}_0 e^{-\lambda t} T^G_{\mathfrak{R},{\omega},\overline{\bb{w}}}(t)\mathcal{P}_{\R^3}\mathrm{d}t, \quad \mathcal{R}^{G^*}_{\mathfrak{R},{\omega},\overline{\bb{w}}}(\lambda)=\int^{\infty}_0 e^{-\lambda t} T^{G^*}_{\mathfrak{R},{\omega},\overline{\bb{w}}}(t)\mathcal{P}_{\R^3}\mathrm{d}t.
\end{equation}
For every  $\bb{f}\in \mathbb{L}^p(\R^3)$, $\bb{u}=\mathcal{R}^G_{\mathfrak{R},
{\omega},\overline{\bb{w}}}(\lambda)\bb{f}$ and $\bb{v}=\mathcal{R}^{G^*}_{\mathfrak{R},
{\omega},\overline{\bb{w}}}(\lambda)\bb{f}$ solve
\begin{align*}
(\lambda I  +\mathcal{L}_{\mathfrak{R},{\omega},\overline{\bb{w}},\R^3})\bb{u}=\mathcal{P}_{\R^3}\bb{f},\quad (\lambda I  +\mathcal{L}^*_{\mathfrak{R},{\omega},\overline{\bb{w}},\R^3})\bb{v}=\mathcal{P}_{\R^3}\bb{f}
\end{align*}
 uniquely, respectively.  Thus, we have
 \begin{equation}\label{resolv-1}
\mathcal{R}^G_{\mathfrak{R},{\omega},\overline{\bb{w}}}(\lambda)=(\lambda I+\mathcal{L}_{\mathfrak{R},{\omega},\overline{\bb{w}},\R^3})^{-1}\mathcal{P}_{\R^3},
 \;\; \mathcal{R}^{G^*}_{\mathfrak{R},{\omega},\overline{\bb{w}}}(\lambda)=(\lambda I+\mathcal{L}^*_{\mathfrak{R},{\omega},\overline{\bb{w}},\R^3})^{-1}\mathcal{P}_{\R^3}.
 \end{equation}
Set
\begin{align}
&\Pi^G_{\mathfrak{R},{\omega},\overline{\bb{w}}}(\lambda)\bb{f}=\mathcal{Q}_{\R^3}
(B_{\overline{\bb{w}}}\mathcal{R}^G_{\mathfrak{R},{\omega},\overline{\bb{w}}}(\lambda)\bb{f}),\quad\mathring{\Pi}^G_{\mathfrak{R},{\omega},\overline{\bb{w}}}(\lambda)\bb{f}=\mathring{\mathcal{Q}}_{\R^3}(B_{\overline{\bb{w}}}\mathcal{R}^G_{\mathfrak{R},{\omega},\overline{\bb{w}}}(\lambda)\bb{f}).\label{Pi-G}
\end{align}
 By the Helmholtz decomposition \eqref{HDG}, we conclude
$$\bb{u}=\mathcal{R}^G_{\mathfrak{R},{\omega},
\overline{\bb{w}}}(\lambda)\bb{f}, \;\;\;P=\mathring{\mathcal{Q}}_{\R^3}\bb{f}+\mathring{\Pi}^G_{\mathfrak{R},{\omega},\overline{\bb{w}}}
(\lambda)\bb{f}$$
 solve problem \eqref{RPW} with $\int_{\Omega_{R+3}}P\,\mathrm{d}x=0$ uniquely.
\vskip0.2cm

For $\theta\in (0,\tfrac{\pi}2)$ and $\ell,\gamma>0$, define
\begin{align*}
&\Sigma_{\theta}=\big\{\lambda\in \C\setminus 0\,\big|\,|\arg\lambda|\leq \pi-\theta\big\},\quad \Sigma_{\theta,\ell}=\big\{\lambda\in \Sigma_{\theta}\,\big|\,|\lambda|\geq \ell\big\},\\
&\C_{+}=\big\{\lambda\in \C\,\big|\,\mathop{\rm Re}\lambda>0\big\},\quad \C_{+\gamma}=\big\{\lambda\in \C_+\,\big|\,\mathop{\rm Re}\lambda\geq \gamma\big\}.
\end{align*}
The resolvent operator $\mathcal{R}^G_{\mathfrak{R},{\omega},\overline{\bb{w}}}(\lambda)$ has  the following properties.
\begin{theorem}\label{TH2-1}
Assume that  $0<|\mathfrak{R}|\leq \mathfrak{R}^*$, $|{\omega}|\leq {\omega}^*$ and $p\in(1,\infty)$. Let $\theta\in (0,\tfrac{\pi}2)$, $\gamma>0$ and $\varepsilon\in (0,\frac12)$ if $p\geq \frac65$ otherwise $\varepsilon \in (0,\frac{3p-3}p)$. Set $\ell_0=(\mathfrak{R}^*)^2/\sin^2(\theta/2)$. Then, there exists a constant $\eta=\eta_{p,\theta,\gamma,\mathfrak{R}^*,{\omega}^*}>0$ such that if
$\normmm{\bb{w}}_{\varepsilon,\Omega}<\eta$,
then
\[\mathcal{R}^G_{\mathfrak{R},{\omega},\overline{\bb{w}}}(\lambda)\in \mathscr{A}(\C_+,\mathcal{L}(\mathbb{L}^p_{R+2}(\R^3),{\W}^{2,p}(\R^3)))\]
with the decomposition
\begin{equation}\label{decom-RG}
\mathcal{R}^G_{\mathfrak{R},{\omega},\overline{\bb{w}}}(\lambda)=
(\lambda I-\Delta-\mathfrak{R}\partial_1)^{-1}\mathcal{P}_{\R^3}
+\mathcal{R}^{G,1}_{\mathfrak{R},{\omega},\overline{\bb{w}}}(\lambda)
+\mathcal{R}^{G,2}_{\mathfrak{R},{\omega},\overline{\bb{w}}}(\lambda)
\end{equation}
such that \begin{equation*}\left\{\begin{aligned}
&\mathcal{R}^{G,1}_{\mathfrak{R},{\omega},\overline{\bb{w}}}(\lambda)\in \mathscr{A}(\Sigma_{\varepsilon,\ell_0},\mathcal{L}(\mathbb{L}^p_{R+2}(\R^3),{\W}^{2,p}(\R^3))),\\
&\mathcal{R}^{G,2}_{\mathfrak{R},{\omega},\overline{\bb{w}}}(\lambda)\in \mathscr{A}(\C_+,\mathcal{L}(\mathbb{L}^p_{R+2}(\R^3),{\W}^{2,p}(\R^3))),\\
&(\lambda I-\Delta-\mathfrak{R}\partial_1)^{-1}\mathcal{P}_{\R^3}\in \mathscr{A}(\Sigma_{\varepsilon,\ell_0},\mathcal{L}(\mathbb{L}^p(\R^3),{\W}^{2,p}(\R^3))).\end{aligned}\right.\end{equation*}
Moreover,  we have for every  $|\beta|\leq 2$
\begin{align}
&\|\partial^{\beta}_x\mathcal{R}^{G,1}_{\mathfrak{R},{\omega},\overline{\bb{w}}}(\lambda)\|_{\mathcal{L}(\mathbb{L}^p_{R+2}(\R^3),\mathbb{L}^p(\R^3))}\leq C_{\theta,R,\mathfrak{R}^*,{\omega}^*}|\lambda|^{-\frac{3-|\beta|}2},\qquad\,\lambda\in \Sigma_{\theta,\ell_0},\label{est.glo1}\\
&\|\partial^{\beta}_x\mathcal{R}^{G,2}_{\mathfrak{R},{\omega},\overline{\bb{w}}}(\lambda)\|_{\mathcal{L}(\mathbb{L}^p_{R+2}(\R^3),\mathbb{L}^p(\R^3))}\leq  C_{\gamma,R,\mathfrak{R}^*,{\omega}^*}|\lambda|^{-\frac{5-|\beta|}2+\delta},\quad\, \lambda\in \C_{+\gamma},\,\,0<\delta<\tfrac12,\label{est.glo2}\\
&\|\partial^{\beta}_x(\lambda I-\Delta-\mathfrak{R}\partial_1)^{-1}\mathcal{P}_{\R^3}\|_{\mathcal{L}(\mathbb{L}^p(\R^3))}\leq C_{\theta,\mathfrak{R}^*}|\lambda|^{-\frac{2-|\beta|}2},\qquad \lambda\in \Sigma_{\theta,\ell_0}.\label{est.glo3}
\end{align}
\end{theorem}
\begin{proof}
In view of \eqref{resolv-1}, we formally write
\begin{equation}\label{eq.form}
\mathcal{R}^G_{\mathfrak{R},{\omega},\overline{\bb{w}}}(\lambda)=
\sum^{\infty}_{j=0}(-(\lambda I+\mathcal{L}_{\mathfrak{R},{\omega},\bb{0},\R^3})^{-1}\mathcal{P}_{\R^3}B_{\overline{\bb{w}}})^j(\lambda I+\mathcal{L}_{\mathfrak{R},{\omega},\bb{0},\R^3})^{-1}\mathcal{P}_{\R^3}.
\end{equation}
We will divide into  two steps to prove Theorem \ref{TH2-1}.
\vskip 0.2cm
\noindent\underline{\textbf{Step 1. Analysis of $(\lambda I+\mathcal{L}_{\mathfrak{R},{\omega},\bb{0},\R^3})^{-1}\mathcal{P}_{\R^3}$}}.\; Since $G(t,0)=T^G_{\Rr,0,\bb{0}}(t)$ provided $\overline{\bb{w}}=\bb{0}$,   we have from \eqref{sem-G-0}-\eqref{est2-1-1} and \eqref{eq.R}-\eqref{resolv-1} that for $\bb{g}\in \mathbb{L}^p(\R^3)$
\begin{align}
&(\lambda I+\mathcal{L}_{\mathfrak{R},{\omega},\bb{0},\R^3})^{-1}\mathcal{P}_{\R^3}\bb{g}\nonumber\\
=&\int^{\infty}_0 e^{-\lambda t}\mathcal{O}^T({\omega}t)(T^G_{\Rr,0,\bb{0}}(t)\mathcal{P}_{\R^3}\bb{g})
(\mathcal{O}({\omega}t)x)\,\mathrm{d}t\nonumber\\
=&\frac{1}{(2\pi)^3}\int^{\infty}_0\int_{\R^3}e^{-(\lambda+|\xi|^2-i\mathfrak{R}\xi_1)t}
\mathcal{O}^{T}({\omega}t)\mathbb{P}(\mathcal{O}({\omega}t)\xi)\hat{\bb{g}}
(\mathcal{O}({\omega}t)\xi)e^{ix\cdot\xi}\,\mathrm{d}\xi\mathrm{d}t.\label{est2-2-1}
\end{align}
 Integrating by parts $N$-times, we get
\begin{align}
&(\lambda I+\mathcal{L}_{\mathfrak{R},{\omega},\bb{0},\R^3})^{-1}\mathcal{P}_{\R^3}\bb{g}\nonumber\\
=&\frac{1}{(2\pi)^3}\sum^{N-1}_{j=0}\int_{\R^3}\frac{e^{ix\cdot
\xi}}{(\lambda+|\xi|^2-i\mathfrak{R}\xi_1)^{j+1}}
\partial^j_t\Big(\mathcal{O}^{T}({\omega}t)\mathbb{P}(\mathcal{O}({\omega}t)\xi)
\hat{\bb{g}}(\mathcal{O}({\omega}t)\xi)\Big)\Big|_{t=0}\,\mathrm{d}\xi\nonumber\\
&+\frac{1}{(2\pi)^3}\int^{\infty}_0\int_{\R^3}\frac{e^{-(\lambda+|\xi|^2-i\mathfrak{R}\xi_1)t}e^{ix\cdot\xi}}
{(\lambda+|\xi|^2-i\mathfrak{R}\xi_1)^N}\partial^N_t\Big(\mathcal{O}^{T}({\omega}t)
\mathbb{P}(\mathcal{O}({\omega}t)\xi)\hat{\bb{g}}(\mathcal{O}({\omega}t)\xi)\Big)
\,\mathrm{d}\xi\mathrm{d}t\nonumber\\
\triangleq&\mathcal{A}^N_{1}(\lambda)\bb{g}
+\mathcal{A}^N_{2}(\lambda)\bb{g}.
\label{eq-A}\end{align}
By Leibniz rule, we have
\begin{align*}
&\partial^j_t\big(\mathcal{O}^{T}({\omega}t)\mathbb{P}(\mathcal{O}({\omega}t)\xi)
\hat{\bb{g}}(\mathcal{O}({\omega}t)\xi)\big)\\
=&{\omega}^j\sum^j_{k=0}|\xi|^{k}
\sum_{|\alpha|=k,\alpha_1=0}
d^{j}_{\alpha}(\sin {\omega}t,\cos {\omega}t, \xi/|\xi|)(\partial^{\alpha}_{\xi}\hat{\bb{g}})(\mathcal{O}({\omega}t)\xi)
\end{align*}
where $d^{j}_{\alpha}(a,b,\upsilon)$ are some $3\times 3$ matrices of polynomials with respect to $a,b$ and $\upsilon=(\upsilon_1,\upsilon_2,\upsilon_3)$. This equality gives
\begin{align*}
\partial^j_t\big(\mathcal{O}^{T}({\omega}t)\mathbb{P}(\mathcal{O}({\omega}t)\xi)
\hat{\bb{g}}(\mathcal{O}({\omega}t)\xi)
\big)\Big|_{t=0}={\omega}^j\sum^j_{k=0}|\xi|^{k}\sum_{|\alpha|=k,\alpha_1=0}d^{j}_{\alpha}(0,1, \xi/|\xi|)(\partial^{\alpha}_{\xi}\hat{\bb{g}})(\xi).
\end{align*}
Thus, we can rewrite
\begin{align*}
&\mathcal{A}^N_{1}(\lambda)\bb{g}=\sum^{N-1}_{j=0}{\omega}^j\sum^j_{k=0}
\sum_{|\alpha|=k,\alpha_1=0}\frac{1}{(2\pi)^3}\int_{\R^3}\frac{|\xi|^{k}d^{j}_{\alpha}(0,1, \xi/|\xi|)}{(\lambda+|\xi|^2-i\mathfrak{R}\xi_1)^{j+1}}
(\partial^{\alpha}_{\xi}\hat{\bb{g}})(\xi)e^{ix\cdot
\xi}\,\mathrm{d}\xi,\\
&\mathcal{A}^N_{2}(\lambda)\bb{g}={\omega}^N\sum^N_{k=0}\sum_{|\alpha|=k,\alpha_1=0}
\frac{1}{(2\pi)^3}
\int^{\infty}_0\int_{\R^3}\frac{|\xi|^{k}e^{-(\lambda+|\xi|^2-i\mathfrak{R}\xi_1)t}}
{(\lambda+|\xi|^2-i\mathfrak{R}\xi_1)^N}
\\
&\quad \quad\quad\quad\quad\quad \quad d^{N}_{\alpha}(\sin {\omega}t,\cos {\omega}t, \mathcal{O}^T({\omega}t)\xi/|\xi|)(\partial^{\alpha}_{\xi}\hat{\bb{g}})(\xi)e^{i\mathcal{O}({\omega}t) x\cdot\xi}\,\mathrm{d}\xi\mathrm{d}t.
\end{align*}
Obviously,
\begin{equation}\label{A^1_1}
\mathcal{A}^1_{1}(\lambda)=(\lambda I-\Delta-\Rr\partial_1)^{-1}\mathcal{P}_{\R^3}.
\end{equation}

From Lemma $1$ in \cite{Shi10}, one has
$$|\lambda+|\xi|^2-i\mathfrak{R}\xi|\geq C_{\theta,\gamma,\mathfrak{R}^*}(|\lambda|+|\xi|^2),\quad \lambda\in \Sigma_{\theta,\ell_0}\cup\C_{+\gamma}$$
which implies
\begin{equation}\label{est2-2-2}
|\partial^{\nu}_{\xi}(\lambda+|\xi|^2-i\mathfrak{R}\xi_1)^{-m}|\leq C_{\theta,\gamma,\Rr^*}(|\lambda|+|\xi|^2)^{-m-(|\nu|/2)},\quad \lambda\in\Sigma_{\theta,\ell_0}\cup\C_{+\gamma}.
\end{equation}
This, combining with that
\begin{equation}\label{est2-2-3}
|\partial^{\nu}_{\xi}d^{j}_{\alpha}(\sin {\omega}t,\cos {\omega}t, \mathcal{O}^T({\omega}t)\xi/|\xi|)|\leq C|\xi|^{-|\nu|},
\end{equation}
gives for  $|\beta|\leq 2$
\begin{equation}\label{est2-2-4}
\Big|\partial^{\nu}_{\xi}\Big(\frac{(i\xi)^{\beta}|\xi|^{k}d^{j}_{\alpha}
(0,1,\xi/|\xi|)}{(\lambda+|\xi|^2-i\mathfrak{R}\xi_1)^{j+1}}\Big)\Big|\leq \frac{C_{j,k,\theta,\gamma,\mathfrak{R}^*}}{|\lambda|^{j+1-((|\beta|+k)/2)}}
|\xi|^{-|\nu|}, \quad\lambda\in \Sigma_{\theta,\ell_0}\cup \C_{+\gamma}.
\end{equation}
Thus, we obtain by Fourier multiplier theorem that for every $\lambda \in\Sigma_{\theta,\ell_0}\cup\C_{+\gamma}$
\begin{align}
&\|\partial^{\beta}_x\mathcal{A}^1_{1}(\lambda)\bb{g}\|_{
\mathbb{L}^p(\R^3)}\leq \frac{C_{\theta,\gamma,\mathfrak{R}^*,{\omega}^*}}{|\lambda|^{1-(|\beta|/2)}}
\|\bb{g}\|_{\mathbb{L}^p(\R^3)},\label{est2-2-5}\\
&\|\partial^{\beta}_x(\mathcal{A}^N_{1}(\lambda)-\mathcal{A}^1_{1}(\lambda))\bb{g}\|_{
\mathbb{L}^p(\R^3)}
\leq \frac{C_{N,\theta,\gamma,\mathfrak{R}^*,{\omega}^*}}{|\lambda|^{(3-|\beta|)/2}}
\sum_{0\leq |\alpha|\leq N-1}\|x^{\alpha}\bb{g}\|_{\LL^p(\R^3)}, \quad N\geq 2.\label{est2-2-6}
\end{align}
Further, we observe for $|\mu|=1$
\begin{align*}
& x^{\mu}\partial^{\beta}_x\int_{\R^3}\frac{|\xi|^{k}d^{j}_{\alpha}(0,1, \xi/|\xi|)}{(\lambda+|\xi|^2-i\mathfrak{R}\xi_1)^{j+1}}
(\partial^{\alpha}_{\xi}\hat{\bb{g}})(\xi)e^{ix\cdot
\xi}\,\mathrm{d}\xi\\
=&i\int_{\R^3}\partial^{\mu}_{\xi}\Big(\frac{(i\xi)^{\beta}|\xi|^{k}}
{(\lambda+|\xi|^2-i\mathfrak{R}\xi_1)^{j+1}}
d^{j}_{\alpha}(0,1, \xi/|\xi|)\Big)(\partial^{\alpha}_{\xi}\hat{\bb{g}})(\xi)e^{ix\cdot
\xi}\,\mathrm{d}\xi\\
&+i\int_{\R^3}\frac{(i\xi)^{\beta}|\xi|^k}{(\lambda+|\xi|^2-i\mathfrak{R}\xi_1)^{j+1}}
d^{j}_{\alpha}(0,1, \xi/|\xi|)(\partial^{\alpha+\mu}_{\xi}\hat{\bb{g}})(\xi)e^{ix\cdot
\xi}\,\mathrm{d}\xi.
\end{align*}
Hence, by \eqref{est2-2-2}-\eqref{est2-2-3}, we deduce for every $|\mu|=1$  and $|\beta|=1,2$
\begin{align*}
\Big|\partial^{\nu}_{\xi}\partial^{\mu}_{\xi}\Big(\frac{(i\xi)^{\beta}|\xi|^{k}d^{j}_{\alpha}(0,1, \xi/|\xi|)}
{(\lambda+|\xi|^2-i\mathfrak{R}\xi_1)^{j+1}}
\Big)\Big|\leq \frac{C_{j,k,\theta,\mathfrak{R}^*,\gamma}}{|\lambda|^{j+1-((|\beta|-1+k)/2)}}
|\xi|^{-|\nu|}, \quad\lambda\in \Sigma_{\theta,\ell_0}\cup\C_{+\gamma}.
\end{align*}
Combining this estimate with \eqref{est2-2-4}, we get by Fourier multiplier theorem
\begin{align*}
 &\Big\|x^{\mu}\partial^{\beta}_x\int_{\R^3}\frac{|\xi|^{k}d^{j}_{\alpha}(0,1, \xi/|\xi|)}{(\lambda+|\xi|^2-i\mathfrak{R}\xi_1)^{j+1}}
(\partial^{\alpha}_{\xi}\hat{\bb{g}})(\xi)e^{ix\cdot
\xi}\,\mathrm{d}\xi\Big\|_{\LL^p(\R^3)}\\
\leq& \frac{ C_{j,k,\theta,\gamma,\Rr^*}\|x^{\alpha}\bb{g}\|_{\LL^p(\R^3)}}{|\lambda|^{j+(3/2)-((|\beta|+k)/2)}}+\frac{ C_{j,k,\theta,\gamma,\Rr^*}\|x^{\alpha+\mu}
\bb{g}\|_{\LL^p(\R^3)}}{|\lambda|^{j+1-((|\beta|+k)/2)}},\quad\lambda\in \Sigma_{\theta,\ell_0}\cup\C_{+\gamma}.
\end{align*}
This equality  implies for every $|\mu|=1$ and $|\beta|=1,2$
\begin{equation}\label{est2-2-7}
\|x^{\mu}\partial^{\beta}_x\mathcal{A}^N_{1}(\lambda)\bb{g}\|_{\LL^p(\R^3)}\leq \frac{C_{\theta,\gamma,N,\Rr^*,{\omega}^*}}{|\lambda|^{1-(|\beta|/2)}}\sum_{0\leq |\alpha|\leq N}\|x^{\alpha}\bb{g}\|_{\LL^p(\R^3)},\quad\lambda\in \Sigma_{\theta,\ell_0}\cup\C_{+\gamma}.
\end{equation}

On the other hand, since
\[|\partial^{\nu}_{\xi}e^{-(\lambda+|\xi|^2-i\Rr\xi_1)t}|\leq \sum^{|\nu|}_{\ell=0}t^{\ell}(|\xi|^2+\Rr^2)^{\ell}e^{-(\mathop{\rm Re}\lambda+|\xi|^2)t}\]
and  $r^{s}e^{-r}\leq C_s,\,s\geq 0$, we have
\[|\partial^{\nu}_{\xi}e^{-(\lambda+|\xi|^2-i\Rr\xi_1)t}|\leq C_{\Rr^*,\gamma^{-1}}|\xi|^{-|\nu|}e^{-(\mathop{\rm Re}\lambda+|\xi|^2)t},\quad\lambda\in \C_{+\gamma}\]
which, together with \eqref{est2-2-2}-\eqref{est2-2-3},  yields  for $|\beta|\leq 2$
\begin{align*}
&\Big|\partial^{\nu}_{\xi}\Big(\frac{(i\mathcal{O}^T({\omega}t)\xi)^{\beta}
e^{-(\lambda+|\xi|^2-i\mathfrak{R}\xi_1)t}|\xi|^{k}d^{N}_{\alpha}(\sin {\omega}t,\cos {\omega}t,\mathcal{O}^T({\omega}t)\xi/|\xi|)}{(\lambda+|\xi|^2
-i\mathfrak{R}\xi_1)^{N}}\Big)\Big|\\
\leq& \left\{\begin{aligned}
&C_{\mathfrak{R}^*,\gamma^{-1}}\tfrac{t^{-(|\beta|+k)/2}}{|\lambda|^{N}}e^{-t\mathop{\rm Re}\lambda }|\xi|^{-|\nu|},\quad \text{if }|\beta|+k\leq 1,\\
&C_{\mathfrak{R}^*,\gamma^{-1}}\tfrac{t^{-1+\delta}}
{|\lambda|^{N+1-\delta-((|\beta|+k)/2)}}e^{-t\mathop{\rm Re}\lambda}|\xi|^{-|\nu|},\quad 0<\delta<\tfrac12, \;\;\text{if } |\beta|+k\geq 2.
\end{aligned}\right.
\end{align*}
Hence, by Fourier multiplier theorem and the fact that
\[\int^{\infty}_0 t^{-s}e^{-t\mathop{\rm Re}\lambda}\,\mathrm{d}t=(\mathop{\rm Re}\lambda)^{-1+s}\Gamma(1-s)\quad \text{if }s<1 \text{ and }\lambda\in \C_+,\]
we deduce for $0<\delta< \tfrac12$ and $\lambda\in  \C_{+\gamma}$
\begin{equation}\label{est2-2-8}
\|\partial^{\beta}_{x}\mathcal{A}^1_{2}(\lambda)\bb{g}\|_{
\mathbb{L}^p(\R^3)}\leq \sum_{0\leq |\alpha|\leq 1}\|x^{\alpha}\bb{g}\|_{\LL^p(\R^3)}\left\{\begin{aligned}
&C_{\gamma,\mathfrak{R}^*,{\omega}^*}|\lambda|^{-1},\quad |\beta|=0\\
&C_{\gamma,\delta,\mathfrak{R}^*,{\omega}^*}|\lambda|^{-\frac{3-|\beta|}2+\delta},\quad |\beta|=1,2.
\end{aligned}\right.
\end{equation}
and
\begin{align}
&\|\partial^{\beta}_{x}\mathcal{A}^N_{2}(\lambda)\bb{g}\|_{
\mathbb{L}^p(\R^3)}\leq \frac{C_{N,\gamma,\delta,\mathfrak{R}^*,{\omega}^*}}{|\lambda|^{\frac{N+2-|\beta|}2-\delta}}
\sum_{0\leq |\alpha|\leq N}\|x^{\alpha}\bb{g}\|_{\LL^p(\R^3)},\quad N\ge2.\label{est2-2-9}
\end{align}

\textbf{Step 2. Proof of \eqref{decom-RG}-\eqref{est.glo3}}\; Here, we argee that $|\beta|\leq 2$, $\bb{f}\in \LL^p_{R+2}(\R^3)$ and the constants appearing in this step  depends on $R,\Rr^*,\omega^*$.
 Let
\begin{equation}\label{eq.w}
B_{\overline{\bb{w}}}\bb{f}\triangleq B_{1,\overline{\bb{w}}}\bb{f}
+B_{2,\overline{\bb{w}}}\bb{f},\quad B_{1,\overline{\bb{w}}}\bb{f}\triangleq\overline{\bb{w}}\cdot\nabla \bb{f},\quad B_{2,\overline{\bb{w}}}\bb{f}\triangleq\bb{f}\cdot\nabla \overline{\bb{w}}.
\end{equation}
 We rewrite  by \eqref{eq-A}
\begin{align*}
\mathcal{R}^G_{\mathfrak{R},{\omega},\overline{\bb{w}}}(\lambda)
=(\lambda I-\Delta-\Rr\partial_1)^{-1}\mathcal{P}_{\R^3}+\mathcal{R}_1(\lambda)
+\mathcal{R}_2(\lambda)+\mathcal{R}_3(\lambda)+\mathcal{R}_4(\lambda)+\mathcal{R}_5(\lambda),
\end{align*}
where
\begin{align*}
\mathcal{R}_1(\lambda)\triangleq&\sum^{\infty}_{j=0}\big(\big(-\mathcal{A}^1_{1}(\lambda)
-\mathcal{A}^1_{2}(\lambda)\big)B_{\overline{\bb{w}}}\big)^j
\mathcal{A}^{4}_{2}(\lambda),\\
%%%%%%%%%%%%%
\mathcal{R}_2(\lambda)\triangleq&\mathcal{A}^{4}_{1}(\lambda)
-\mathcal{A}^1_{1}(\lambda)
+\sum^{\infty}_{j=1}\big(-\mathcal{A}^1_{1}(\lambda)B_{\overline{\bb{w}}}\big)^j
\mathcal{A}^{4}_{1}(\lambda),\\
%%%%%%%%%%%%%%%%%
\mathcal{R}_3(\lambda)\triangleq&\sum^{\infty}_{j=3}\sum_{\bb{\alpha}\in \{1,2\}^j,\bb{\alpha}\neq(1,\ldots,1)} \Big(\prod^j_{i=1}(-\mathcal{A}^1_{\alpha_i}
(\lambda)B_{\overline{\bb{w}}})\Big)
\mathcal{A}^{4}_{1}(\lambda),\\
%%%%%%%%%%%%%%%%%%%%%
\mathcal{R}_4(\lambda)\triangleq& \big(\mathcal{A}^1_{1}
(\lambda)+\mathcal{A}^1_{2}(\lambda)\big)B_{\overline{\bb{w}}}
\mathcal{A}^1_{2}(\lambda)B_{\overline{\bb{w}}}\mathcal{A}^{4}_{1}
(\lambda)+\mathcal{A}^1_{2}(\lambda)B_{2,\overline{\bb{w}}}
\mathcal{A}^1_{1}(\lambda)B_{\overline{\bb{w}}}
\mathcal{A}^{4}_{1}(\lambda)\\
&+\mathcal{A}^1_{2}(\lambda)B_{1,\overline{\bb{w}}}
\mathcal{A}^1_{1}(\lambda)B_{2,\overline{\bb{w}}}
\mathcal{A}^{4}_{1}(\lambda)-\mathcal{A}^1_{2}(\lambda)
B_{2,\overline{\bb{w}}}\big(\mathcal{A}^{4}_{1}(\lambda)
-\mathcal{A}^1_{1}(\lambda)\big),\\
%%%%%%%%%%%%%%%%%%%%%%%%%
\mathcal{R}_5(\lambda)\triangleq&\mathcal{A}^1_{2}(\lambda)B_{1,\overline{\bb{w}}}
\mathcal{A}^1_{1}(\lambda)B_{1,\overline{\bb{w}}}
\mathcal{A}^{4}_{1}(\lambda)-\mathcal{A}^1_{2}(\lambda)
B_{1,\overline{\bb{w}}}\mathcal{A}^{4}_{1}(\lambda)
-\mathcal{A}^1_{2}(\lambda)B_{2,\overline{\bb{w}}}\mathcal{A}^1_{1}(\lambda).
\end{align*}
In the course of the proof, we  will repeatedly use
\begin{equation}\label{est2-2-10}
\left\{\begin{aligned}
&\|B_{\overline{\bb{w}}}\bb{g}\|_{\mathbb{L}^p(\R^3)}
+\||x|B_{\overline{\bb{w}}}\bb{g}\|_{\mathbb{L}^p(\R^3)}\leq C\normmm{\bb{w}}_{\varepsilon,\Omega}(\|\bb{g}\|_{\mathbb{L}^p(\R^3)}+\|\nabla \bb{g}\|_{\mathbb{L}^p(\R^3)}),\\
&\||x|^{s}\bb{f}\|_{\mathbb{L}^p(\R^3)}\leq C_{R,s}\|\bb{f}\|_{\mathbb{L}^p(\R^3)}, \quad s>0,\,\,\bb{f}\in \mathbb{L}^p_{R+1}(\R^3).
\end{aligned}\right.
\end{equation}
These estimates together with  \eqref{est2-2-5}-\eqref{est2-2-6} and \eqref{est2-2-8}-\eqref{est2-2-9} with $\delta=\frac14$ yield
\begin{align*}
&\|\partial^{\beta}_x\mathcal{R}_1(\lambda)\bb{f}\|_{\mathbb{L}^p(\R^3)}
\lesssim_{\gamma}\sum^{\infty}_{j=0}\big( C_{\gamma}\normmm{\bb{w}}_{\varepsilon,\Omega}\big)^j
|\lambda|^{\frac{|\beta|-5}2}\|\bb{f}\|_{\LL^p(\R^3)},\quad \lambda\in C_{+\gamma},\\
&\|\partial^{\beta}_x\mathcal{R}_2(\lambda)\bb{f}\|_{\mathbb{L}^p(\R^3)}
\lesssim_{\theta}\normmm{\bb{w}}_{\varepsilon,\Omega}\sum^{\infty}_{j=0}\big( C_{\theta}\normmm{\bb{w}}_{\varepsilon,\Omega}\big)^j
|\lambda|^{\frac{|\beta|-3}2}\|\bb{f}\|_{\LL^p(\R^3)},\;\;\lambda\in \Sigma_{\theta,\ell_0},\\
&\|\partial^{\beta}_x\mathcal{R}_3(\lambda)\bb{f}\|_{\mathbb{L}^p(\R^3)}
\lesssim_{\gamma}\normmm{\bb{w}}^3_{\varepsilon,\Omega}\sum^{\infty}_{j=0}\big( C_{\gamma}\normmm{\bb{w}}_{\varepsilon,\Omega}\big)^j
|\lambda|^{\frac{|\beta|-5}2}\|\bb{f}\|_{\LL^p(\R^3)},\;\;\lambda\in \C_{+\gamma}.
\end{align*}
Hence we conclude
\begin{align*}
&\big\|\partial^{\beta}_x\mathcal{R}_1(\lambda)\bb{f}\big\|_{\mathbb{L}^p(\R^3)}+\big\|\partial^{\beta}_x\mathcal{R}_3(\lambda)\bb{f}\big\|_{\mathbb{L}^p(\R^3)}
\lesssim_{\gamma}
|\lambda|^{-(5-|\beta|)/2}\|\bb{f}\|_{L^p(\R^3)},\quad \lambda\in C_{+\gamma},\\
&\big\|\partial^{\beta}_x\mathcal{R}_2(\lambda)\bb{f}\big\|_{\mathbb{L}^p(\R^3)}
\lesssim_{\theta}|\lambda|^{-(3-|\beta|)/2}\|\bb{f}\|_{L^p(\R^3)},\qquad\qquad\qquad\qquad\qquad\,\lambda\in \Sigma_{\theta,\ell_0},
\end{align*}
 provided that
\[(C_{\gamma}+C_{\theta})\normmm{\bb{w}}_{\varepsilon,\Omega}\leq 1.\]
Meanwhile,  we have
$$\|\partial^{\beta}_x\mathcal{R}_4(\lambda)\bb{f}\|_{\LL^p(\R^3)}\lesssim_{\gamma,\delta,}|\lambda|^{-\frac{5-|\beta|}2+\delta}\|\bb{f}\|_{L^p(\R^3)},\quad  \lambda\in \C_{+\gamma},\;0<\delta<\tfrac12.$$

Now we are in position to  estimate $\mathcal{R}_5(\lambda)$.  Observing by \eqref{eq-A}
\[\mathcal{A}^1_{1}(\lambda)+\mathcal{A}^1_{2}(\lambda)=\mathcal{A}^2_{2}(\lambda)
+\mathcal{A}^2_{1}(\lambda)\]
we decompose $\mathcal{R}_5(\lambda)=\mathcal{R}_{5,1}(\lambda)+\mathcal{R}_{5,2}(\lambda)+\mathcal{R}_{5,2}(\lambda)$,
where
\begin{align*}
\mathcal{R}_{5,1}(\lambda)\triangleq&\big(\mathcal{A}^2_{1}(\lambda)-\mathcal{A}^1_{1}(\lambda)\big)B_{1,\overline{\bb{w}}}
\mathcal{A}^1_{1}(\lambda)B_{1,\overline{\bb{w}}}
\mathcal{A}^{4}_{1}(\lambda)\\
&-\big(\mathcal{A}^2_{1}(\lambda)
-\mathcal{A}^1_{1}(\lambda)\big)
B_{1,\overline{\bb{w}}}\mathcal{A}^{4}_{1}(\lambda)-\big(\mathcal{A}^2_{1}(\lambda)-\mathcal{A}^1_{1}(\lambda)\big)B_{2,\overline{\bb{w}}}\mathcal{A}^1_{1}(\lambda),\\
%%%%%%%%%%%%%%%%
\mathcal{R}_{5,2}(\lambda)\triangleq &\mathcal{A}^2_{2}(\lambda)B_{1,\overline{\bb{w}}}
\mathcal{A}^1_{1}(\lambda)B_{1,\overline{\bb{w}}}
\mathcal{A}^{4}_{1}(\lambda)-\mathcal{A}^2_{2}(\lambda)
B_{1,\overline{\bb{w}}}\mathcal{A}^{4}_{1}(\lambda),\\
%%%%%%%%%%%%%%%%%%%%
\mathcal{R}_{5,3}(\lambda)\triangleq&-\mathcal{A}^2_{2}(\lambda)B_{2,\overline{\bb{w}}}\mathcal{A}^1_{1}(\lambda).
\end{align*}
By \eqref{est2-2-5}-\eqref{est2-2-9} and \eqref{est2-2-10}, we easily get
\begin{align*}
&\|\partial^{\beta}_x\mathcal{R}_{5,1}(\lambda)\bb{f}\|_{\mathbb{L}^p(\R^3)}\lesssim_{\theta}|\lambda|^{-2+(|\beta|/2)}\|\bb{f}\|_{\mathbb{L}^p(\R^3)},\qquad\quad\lambda\in \Sigma_{\theta,\ell_0},\\
&\|\partial^{\beta}_x\mathcal{R}_{5,2}(\lambda)\bb{f}\|_{\mathbb{L}^p(\R^3)}\lesssim_{\gamma,\delta}|\lambda|^{-((5-|\beta|)/2)+\delta}\|\bb{f}\|_{\mathbb{L}^p(\R^3)},
\quad \lambda\in \C_{+\gamma},\;0<\delta<\tfrac12.
\end{align*}
For $\mathcal{R}_{5,3}(\lambda)$, we observe
\[(\lambda I-\Delta-\Rr\partial_1)\mathcal{P}_{\R^3}=\lambda^{-1}\mathcal{P}_{\R^3}
+\lambda^{-1}(\Delta+\Rr\partial_1)(\lambda I-\Delta-\Rr\partial_1)^{-1}\mathcal{P}_{\R^3}.\]
which together with \eqref{A^1_1} implies
\begin{equation}\label{est2-2-11}
\mathcal{R}_{5,3}(\lambda)=-\lambda^{-1}\mathcal{A}^2_{2}(\lambda)B_{2,\overline{\bb{w}}}\mathcal{P}_{\mathds{R}^3}-\lambda^{-1}\mathcal{A}^2_{2}(\lambda)B_{2,\overline{\bb{w}}}(\Delta+\mathfrak{R}\partial_1)\mathcal{A}^1_{1}(\lambda).
\end{equation}
By \eqref{est2-2-7} and \eqref{est2-2-9} with $\delta=\tfrac14$, we have
\begin{align}
&\|\partial^{\beta}_x\mathcal{A}^2_{2}(\lambda)B_{2,\overline{\bb{w}}}(\Delta+\Rr\partial_1)\mathcal{A}^1_{1}(\lambda)\bb{f}\|_{\LL^p(\R^3)}
\leq C_{\gamma}
|\lambda|^{-(3-|\beta|)/2}\|\bb{f}\|_{\LL^p(\R^3)},\quad \lambda\in \C_{+\gamma}.\label{est2-2-12}
\end{align}
In addition, since the kernel function of $\mathcal{P}_{\R^3}$ is bounded by $|x|^{-3}$ and $\bb{f}=0$ in $B^c_{R+2}$, we get
\begin{align*}
|[\mathcal{P}_{\R^3}\bb{f}](x)|\leq C |x|^{-3}\int_{\R^3}|\bb{f}(y)|\,\mathrm{d}y\leq C_R |x|^{-3}\|\bb{f}\|_{\LL^p(\R^3)},\quad|x|\geq 3(R+2).
\end{align*}
This inequality gives
\begin{align*}
\||x|^2B_{2,\overline{\bb{w}}}\mathcal{P}_{\R^3}\bb{f}\|_{\LL^p(\R^3)}\leq& \||x|^2B_{2,\overline{\bb{w}}}\mathcal{P}_{\R^3}\bb{f}\|_{\LL^p(B_{3(R+2)})}
+\||x|^2B_{2,\overline{\bb{w}}}\mathcal{P}_{\R^3}\bb{f}\|_{\LL^p(B^c_{3(R+2)})}\\
\leq& C_R\|\bb{f}\|_{\LL^p(\R^3)}\normmm{\bb{w}}_{\varepsilon,\Omega}\big(1+\||\cdot|^{-5/2}
(1+s_{\Rr}(\cdot))^{-(1/2)+\varepsilon}\|_{\LL^p(B^c_{3(R+2)})}\big)\\
\leq &C_{R}\normmm{\bb{w}}_{\varepsilon,\Omega}\|\bb{f}\|_{\LL^p(\R^3)},
\end{align*}
for $\varepsilon\in (0,\min(\tfrac12,\tfrac{3p-3}p))$. Hence we conclude by \eqref{est2-2-9} with $\delta=\tfrac14$
\begin{equation}\label{est2-2-13}
\|\partial^{\beta}_x\mathcal{A}^2_{2}(\lambda)B_{2,\overline{\bb{w}}}\mathcal{P}_{\R^3}\bb{f}\|_{\LL^p(\R^3)}
\lesssim_{\gamma}\normmm{\bb{w}}_{\varepsilon,\Omega}|\lambda|^{-(3-|\beta|)/2}
\|\bb{f}\|_{\LL^p(\R^3)}.
\end{equation}
This, combining with \eqref{est2-2-11}-\eqref{est2-2-12}, yields for $\varepsilon\in (0,\min(\tfrac12,\tfrac{3p-3}p))$
\begin{equation}\label{est-2-20}
\|\partial^{\beta}_x\mathcal{R}_{5,3}\bb{f}\|_{\LL^p(\R^3)}
\lesssim_{\gamma}\normmm{\bb{w}}_{\varepsilon,\Omega}|\lambda|^{-(5-|\beta|)/2}\|\bb{f}\|_{\LL^p(\R^3)},\quad\lambda\in \C_{+\gamma}.
\end{equation}

Finally, we define
\begin{align*}
&\mathcal{R}^{G,1}_{\mathfrak{R},{\omega},\overline{\bb{w}}}(\lambda)=\mathcal{R}_2(\lambda)
+\mathcal{R}_{5,1}(\lambda),\\
&\mathcal{R}^{G,2}_{\mathfrak{R},{\omega},\overline{\bb{w}}}(\lambda)=
\mathcal{R}_1(\lambda)+\mathcal{R}_3(\lambda)+\mathcal{R}_4(\lambda)
+\mathcal{R}_{5,2}(\lambda)+\mathcal{R}_{5,3}(\lambda).
\end{align*}
 Collecting \eqref{est2-2-5} and the estimates of $\mathcal{R}_1(\lambda)$-$\mathcal{R}_5(\lambda)$, we
prove \eqref{est.glo1}-\eqref{est.glo3} and so complete the proof of this theorem.
\end{proof}
\vskip0.2cm

Let $D=\R^3$, $\Omega$ or $\Omega_{R+3}$ and define
\begin{align*}
&\hat{W}^{1,p}(D)=\big\{\pi \in L^p_{{\rm loc}}(D)\,\big|\, \nabla \pi \in \LL^p(D)\big\},\;\, \mathring{W}^{1,p}(D)=\Big\{\pi \in \hat{W}^{1,p}(D)\,\Big|\, \int_{\Omega_{R+3}}\pi\,\mathrm{d}x=0\Big\}.
\end{align*}
In the light of \eqref{Pi-G}-\eqref{decom-RG}, we split $\mathring{\Pi}^{G}_{\Rr,{\omega},\overline{\bb{w}}}(\lambda)$ into two parts:
\begin{equation}\label{Press-G}
\left\{\begin{aligned}
&\mathring{\Pi}^{G,1}_{\Rr,{\omega},\overline{\bb{w}}}(\lambda)\bb{f}=\mathring{\mathcal{Q}}_{\R^3}
\big(B_{\overline{\bb{w}}}((\lambda-\Delta-\Rr\partial_1)^{-1}\mathcal{P}_{\R^3}
+\mathcal{R}^{G,1}_{\Rr,{\omega},\overline{\bb{w}}}(\lambda))\bb{f}\big),\\
&\mathring{\Pi}^{G,2}_{\Rr,{\omega},\overline{\bb{w}}}(\lambda)\bb{f}=\mathring{\mathcal{Q}}_{\R^3}
\big(B_{\overline{\bb{w}}}\mathcal{R}^{G,2}_{\Rr,{\omega},\overline{\bb{w}}}(\lambda)\bb{f}\big).
\end{aligned}\right.
\end{equation}
As a consequence of Theorem \ref{TH2-1}, we have

\begin{corollary}\label{Cor.G1}
Under the assumption of Theorem \ref{TH2-1}, there exists  a positive constant $\eta=\eta_{\theta,\gamma,R,\Rr^*,{\omega}^*}$ such that if
$\normmm{\bb{w}}_{\varepsilon,\Omega}<\eta$, then
\[\mathring{\Pi}^G_{\Rr,{\omega},\overline{\bb{w}}}(\lambda)\in \mathscr{A}(\C_+,\mathcal{L}(\LL^p_{R+2}(\R^3),\mathring{W}^{1,p}(\R^3)))\]
with the decomposition
\begin{equation}\label{decom-Pi}
\mathring{\Pi}^G_{\Rr,{\omega},\overline{\bb{w}}}(\lambda)=
\mathring{\Pi}^{G,1}_{\Rr,{\omega},\overline{\bb{w}}}(\lambda)
+\mathring{\Pi}^{G,2}_{\Rr,{\omega},\overline{\bb{w}}}(\lambda)
\end{equation}
where
\begin{equation*}
\left\{\begin{aligned}
&\mathring{\Pi}^{G,1}_{\Rr,{\omega},\overline{\bb{w}}}(\lambda)\in \mathscr{A}(\Sigma_{\theta,\ell_0},\mathcal{L}(\LL^p_{R+2}(\R^3),\mathring{W}^{1,p}(\R^3))),\\
&\mathring{\Pi}^{G,2}_{\Rr,{\omega},\overline{\bb{w}}}(\lambda)\in \mathscr{A}(\C_+,\mathcal{L}(\LL^p_{R+2}(\R^3),\mathring{W}^{1,p}(\R^3))).
\end{aligned}\right.\end{equation*}
Moreover  for  $\bb{f}\in \LL^p_{R+2}(\R^3)$ and  $0<\delta<1/2$, we have
\begin{align}
&\|\nabla\mathring{\Pi}^{G,1}_{\Rr,{\omega},\overline{\bb{w}}}(\lambda)
\bb{f}\|_{\LL^p(\R^3)}+\|\mathring{\Pi}^{G,1}_{\Rr,{\omega},\overline{\bb{w}}}(\lambda)
\bb{f}\|_{\LL^p(\Omega_{R+3})}\leq \frac{C_{\theta,R,\mathfrak{R}^*,{\omega}^*}}{|\lambda|^{1/2}}\|\bb{f}
\|_{\LL^p(\R^3)},\quad\; \lambda\in \Sigma_{\theta,\ell_0},\label{est.glo4}\\
&\|\nabla\mathring{\Pi}^{G,2}_{\Rr,{\omega},\overline{\bb{w}}}(\lambda)
\bb{f}\|_{\LL^p(\R^3)}+\|\mathring{\Pi}^{G,2}_{\Rr,{\omega},\overline{\bb{w}}}(\lambda)
\bb{f}\|_{\LL^p(\Omega_{R+3})}\leq  \frac{C_{\gamma, \delta,R,\mathfrak{R}^*,{\omega}^*}}{|\lambda|^{2-\delta}}\|\bb{f}
\|_{\LL^p(\R^3)},\;\;\, \lambda\in \C_{+\gamma}.\label{est.glo5}
\end{align}
\end{corollary}
In the rest part in this section, we will study the behavior of operators $\mathcal{R}^G_{\Rr,{\omega},\overline{\bb{w}}}(\lambda)$ and $\Pi^G_{\Rr,{\omega},\overline{\bb{w}}}(\lambda)$ acting on $\mathbb{L}^p_{R+2}(\R^3)$ near ${\rm Re}\,\lambda=0$ and $|x|=\infty$. Set $\Lambda^sf=\mathcal{F}^{-1}(|\xi|^s\hat{f}(\xi)) $.
We start from some preliminary lemmas.
\begin{lemma}\label{Lem.Oseen}
Let $\Gamma_{ij}(x,t)$ be the kernel function of $T^G_{\mathfrak{R},0,0}(t)(\mathcal{P}_{\R^3})_{ij}$. Then
\begin{equation}\label{est-J1-00}
|\Lambda^s\Gamma_{ij}(x,t)|\leq \frac{C_s}{(t+|x+\mathfrak{R}t\bb{e}_1|^2)^{\frac{3+s}2}},\quad (x,t)\in \R^3\times [0,\infty)\setminus {(0,0)},\;\;s\geq 0.
\end{equation}
\end{lemma}
\begin{proof}
Let $\Gamma^0_{ij}(x,t)$ be the kernel function of  $e^{-t\Delta}(\mathcal{P}_{\R^3})_{i,j}$.
Since $\Gamma_{ij}(x,t)=\Gamma^0_{ij}(x+t\mathfrak{R}\bb{e}_1,t)$, we conclude \eqref{est-J1-00}
from the following classical estimate
\begin{equation*}|\Lambda^s\Gamma^0_{ij}(x,t)|\leq C_s{(t+|x|^2)^{-\frac{3+s}2}},\quad(x,t)\in \R^3\times [0,\infty)\setminus {(0,0)}, \;\; s\geq 0.\end{equation*}\end{proof}
\begin{lemma}\label{Lem-J} Let $s,r\geq0$ with  $s+2-r>0$. Then
$$J_{s,r}(x)\triangleq\int^{\infty}_0\frac{t^{r}}{(t+|x+\Rr t\bb{e}_1|^2)^{(3+s)/2}}\mathrm{d}t$$
satisfies  for  $\theta\in (0,\frac12)$
\begin{equation}\label{est-J1}
J_{s,r}(x)\leq\mathscr{B}(\tfrac{r+1}2,1+\tfrac{s-r}2)\left\{\begin{aligned}
   &2^{-1}|\Rr|^{-r-1}|x|^{-2-s+r},\quad\;\, \text{if }|x|\leq \theta|\Rr|,\\
   &\tfrac{C_{\theta,s,r}|\Rr|^{(s/2)-r}}{|x|^{1+(s/2)-r}(1+2|\Rr|s_{\Rr}(x))^{1+(s/2)}},\quad \text{if }|x|> \theta|\Rr|,
\end{aligned}\right.
\end{equation}
 where $\mathscr{B}(\cdot,\cdot)$ denotes the Beta function.
 In particular, we have for $|x|\leq\theta|\mathfrak{R}|^{-1}$
\begin{equation}\label{est-J2}
J_{s,r}(x)\leq C_{\theta,r}\mathscr{B}(r+1,\tfrac{s+1-2r}2)|x|^{-1-s+2r},\quad 1+s-2r>0.
\end{equation}
\end{lemma}
To prove Lemma \ref{Lem-J}, we begin with some basic integral identities.
\begin{lemma}\label{Lem.DI}
 Assume that $a,b>0$ and $s\geq 0$, then
\begin{align}
&\int^{\infty}_0\frac{t^r}{(at^2+b)^{(3+s)/2}}\,\mathrm{d}t=2^{-1}b^{-\frac{2+s}2+
\frac{r}2}a^{-\frac{r+1}2}\mathscr{B}(\tfrac{r+1}2,\tfrac{s+2-r}2), \quad\; \,0\leq r<s+2,\label{est-DI1}\\
&\int^{\infty}_0\frac{t^r}{(at+b)^{(3+s)/2}}\,\mathrm{d}t=2^{-1}b^{-\frac{1+s}2
+r}a^{-r-1}\mathscr{B}(r+1,\tfrac{s+1-2r}2), \quad 0\leq r<\tfrac{1+s}2.\label{est-DI2}
\end{align}
\end{lemma}
\begin{proof}  Since \eqref{est-DI1}  implies \eqref{est-DI2},
it suffices to prove \eqref{est-DI1}.  A simple computation yields that if $s+2-r>0$,
\begin{align*}
\int^{\infty}_0\frac{t^r}{(at^2+b)^{(3+s)/2}}\,\mathrm{d}t=&b^{\frac{r-s-2}2}a^{-\frac{r+1}{2}}\int^{\pi/2}_0(\sin{\theta})^{r}
(\cos{\theta})^{1+s-r}\,\mathrm{d}\theta\\
=&2^{-1}b^{\frac{r-s-2}2}a^{-\frac{r+1}{2}} \mathscr{B}(\tfrac{r+1}2,\tfrac{s+2-r}2)
\end{align*}
by making use of $t =\sqrt{\frac{b}{a}}\tan{\theta}$ in the first equality.
  \end{proof}

\begin{proof}[Proof of Lemma \ref{Lem-J}]
Obviously,
$$t+|x+\Rr t\bb{e}_1|^2=\Rr^2t^2+(1+2\Rr x_1)t+|x|^2.$$
We will divide $x\in \R^3$ into three  regions to prove this lemma.
\vskip 0.1cm
\textbf{Region 1: $|x|\leq \theta|\mathfrak{R}|^{-1}$.}
In this region,
$$t+|x-\mathfrak{R}t\bb{e}_1|^2\ge \mathfrak{R}^2t^2+(1-2\theta)t+|x|^2.$$
This inequality, together with Lemma \ref{Lem.DI}, yields
\[J_{s,r}(x)\leq\left\{\begin{aligned}
& 2^{-1}\mathscr{B}(\tfrac{r+1}2,1+\tfrac{s-r}2)|\mathfrak{R}|^{-r-1}|x|^{-2-s+r},\\
&C_{\theta,r}\mathscr{B}(r+1,\tfrac{s+1-2r}2)|x|^{-1-s+2r},\quad r<s/2+1/2.
\end{aligned}\right.\]
\vskip 0.1cm
\textbf{Region 2: $|x|> \theta|\mathfrak{R}|^{-1}$ and $1+2\mathfrak{R}x_1\geq 0$.}
Observing in this region that
$$t+|x+\mathfrak{R}t\bb{e}_1|^2\geq \mathfrak{R}^2t^2+|x|^2,\quad 1+2|\mathfrak{R}|s_{\mathfrak{R}}(x)\leq(\theta^{-1}+4)|\mathfrak{R}||x|,$$
 we get by \eqref{est-DI1}
\begin{align*}
J_{s,r}(x)\leq& 2^{-1}\mathscr{B}(\tfrac{r+1}2,1+\tfrac{s-r}2)|\mathfrak{R}|^{-r-1}|x|^{-2-s+r}\\
\leq& C_{\theta,s}\mathscr{B}(\tfrac{r+1}2,1+\tfrac{s-r}2)|\mathfrak{R}|^{\frac{s}2-r}|x|^{-1-\frac{s}2+r}
(1+2|\mathfrak{R}|s_{\mathfrak{R}}(x)|)^{-1-\frac{s}2}.
\end{align*}
\vskip 0.15cm
\textbf{Region 3: $|x|> \theta|\mathfrak{R}|^{-1}$ and $1+2\mathfrak{R}x_1< 0$.} \; We have
\begin{align*}
t+|x+\mathfrak{R}t\bb{e}_1|^2=& \Big(|\mathfrak{R}|t+\frac{1+2\mathfrak{R}x_1}{2|\mathfrak{R}|}\Big)^2
+\frac{4\mathfrak{R}^2|x|^2-\big(1+2\mathfrak{R}x_1\big)^2}{4\mathfrak{R}^2}\\
\geq &\Big(|\mathfrak{R}|t+\frac{1+2\mathfrak{R}x_1}{2|\mathfrak{R}|}\Big)^2
+\frac{\big(1+2|\mathfrak{R}|s_{\mathfrak{R}}(x)\big)|x|}{2|\mathfrak{R}|}.
\end{align*}
With the help of the above relation, we obtain
\begin{align*}
J_{s,r}(x)\leq & \int^{\infty}_0\frac{t^r}{\big((|\mathfrak{R}|t+\frac{1+2\mathfrak{R}x_1}
{2|\mathfrak{R}|})^2+\frac{(1+2\mathfrak{R}s_{\mathfrak{R}}(x))|x|}{2|\mathfrak{R}|}
\big)^{(3+s)/2}}\,\mathrm{d}t\\
=&\frac 1{|\mathfrak{R}|^{r+1}}\int^{\infty}_{\frac{1+2\mathfrak{R}x_1}
{2|\mathfrak{R}|}}\frac{(\tau-\frac{1+2\mathfrak{R}x_1}{2|\mathfrak{R}|})^{r}}
{\big(\tau^2+\frac{(1+2|\mathfrak{R}|s_{\mathfrak{R}}(x))|x|}{2|\mathfrak{R}|}\big)^{(3+s)/2}}\,\mathrm{d}\tau
\\
\leq&\frac {C_r}{|\mathfrak{R}|^{r+1}}\int^{\infty}_{0}
\Big(\frac{\tau^{r}}
{\big(\tau^2+\frac{(1+2|\mathfrak{R}|s_{\mathfrak{R}}(x))|x|}{2|\mathfrak{R}|}\big)^{(3+s)/2}}+\frac{|x|^{r}}
{\big(\tau^2+\frac{(1+2|\mathfrak{R}|s_{\mathfrak{R}}(x))|x|}{2|\mathfrak{R}|}\big)^{(3+s)/2}}\Big)\mathrm{d}\tau.
\end{align*}
Thus, we deduce by \eqref{est-DI1}
\begin{align*}
J_{s,r}(x)\leq& \tfrac{C_r\mathscr{B}(r+1,(s+1-2r)/2)}{|\mathfrak{R}|^{r+1}}
\Big(\tfrac{|\mathfrak{R}|^{1+\frac{s-r}2}}
{(|x|(1+2|\mathfrak{R}|s_{\mathfrak{R}}(x)))^{1+((s-r)/2)}}
+\tfrac{|\mathfrak{R}|^{1+\frac{s}2}|x|^{r}}
{(|x|(1+2|\mathfrak{R}|s_{\mathfrak{R}}(x)))^{1+(s/2)}}\Big)\\
\leq& C_{\theta,r}|\mathfrak{R}|^{(s/2)-r}|x|^{-1-(s/2)+r}(1+2|\mathfrak{R}|s_{\mathfrak{R}}(x))^{-1-(s/2)}.
\end{align*}

Summing up, we complete the proof of Lemma \ref{Lem-J}.
\end{proof}
To briefly state our results,  we define
\[\|\bb{g}\|_{X^{s,r}_{\ell}}\triangleq\sup_{|x|\geq \ell}|x|^{1+\frac{s}2-r}(1+s_{\Rr}(x))|\bb{g}(x)|, \quad s,r\geq 0,\;\ell>0\]
and denote by $(\vartriangle_h \bb{g})(\lambda)=\bb{g}(\lambda+h)-\bb{g}(\lambda)$ the difference quotient.

\begin{theorem}\label{TH2-2}
Let $1<p<\infty$, $\varepsilon\in(0,\tfrac12)$, $\rho\in (0,\frac12)$, $0<\Rr_*\leq |\Rr|\leq \Rr^*$ and $|{\omega}|\leq {\omega}^*$. Then, there exists a constant $\eta=\eta_{R,p,\rho,\mathfrak{R}_*,\mathfrak{R^*}}>0$ such that if
$\normmm{\bb{w}}_{\varepsilon,\Omega}\leq \eta$, then
\[\mathcal{R}^G_{\mathfrak{R},{\omega},\overline{\bb{w}}}(\lambda)\in C(\overline{\C_+},\mathcal{L}(\mathbb{L}^p_{R+2}(\R^3),{\W}^{2,p}(B_{9(R+2)}))).\]
Moreover, for $\lambda,\lambda+h\in \overline{\C_+}$ and $\bb{f}\in \mathbb{L}^p_{R+2}(\R^3)$, we have for $s\in [0,2]$
\begin{align}
&\|\Lambda^{s}\partial^{k}_{\lambda}\mathcal{R}^G_{\mathfrak{R},{\omega},\overline{\bb{w}}}(\lambda)\bb{f}\|_{\mathbb{L}^p(B_{9(R+2)})}
\leq C_{s, R,\Rr_*,\Rr^*,\omega^*}\|\bb{f}\|_{\mathbb{L}^p(\R^3)},\quad k=0,1,\label{est.loc1}\\
&\|\Lambda^{s}(\vartriangle_h\partial_{\lambda}\mathcal{R}^G_{\mathfrak{R},{\omega},\overline{\bb{w}}})(\lambda)\bb{f}\|_{L^p(B_{9(R+2)})}\leq C_{s,\rho,\Rr_*,\Rr^*,\omega^*}|h|^{\rho}\|\bb{f}\|_{\mathbb{L}^p(\R^3)},\label{est.loc2}
\end{align}
 and for $s\in [0,2)$
\begin{align}
&\|\Lambda^{s}\partial^k_{\lambda}\mathcal{R}^G_{\mathfrak{R},{\omega},\overline{\bb{w}}}(\lambda)\bb{f}\|_{X^{s,k}_{9(R+2)}}
\leq C_{s,R,\Rr_*,\Rr^*}\|\bb{f}\|_{\mathbb{L}^p(\R^3)},\quad k=0,1,\label{est.decay1}\\
&\|\Lambda^{s}(\vartriangle_h \partial_{\lambda}\mathcal{R}^G_{\mathfrak{R},{\omega},\overline{\bb{w}}})(\lambda)\bb{f}\|_{X^{s,1+\rho}_{9(R+2)}}
\leq C_{s,\rho,R,\Rr_*,\Rr^*}|h|^{\rho}\|\bb{f}\|_{\mathbb{L}^p(\R^3)}.
\label{est.decay2}
\end{align}
 In particular,  we have for $j\leq 2$, $0<\varrho\ll 1/2$ and  $0<|h|\leq h_0$
\begin{align}
&\|\nabla^j\mathcal{R}^G_{\mathfrak{R},{\omega},\overline{\bb{w}}}(\lambda)\bb{f}\|_{\mathbb{L}^p(B_{9(R+2)})}
\leq C_{R,\Rr_*,\Rr^*,\omega^*}(1+|\lambda|)^{-1+\frac{j}2}\|\bb{f}\|_{\mathbb{L}^p(\R^3)},\label{est.loc3}\\
&\|\nabla^j\partial_{\lambda}\mathcal{R}^G_{\mathfrak{R},{\omega},\overline{\bb{w}}}(\lambda)\bb{f}\|_{\mathbb{L}^p(B_{9(R+2)})}
\leq C_{\varrho,R,\Rr_*,\Rr^*,\omega^*}(1+|\lambda|)^{-2+\frac{j}2+\varrho}\|\bb{f}\|_{\mathbb{L}^p(\R^3)},\label{est.loc4}\\
&\|\nabla^j(\vartriangle_h \partial_{\lambda}\mathcal{R}^G_{\mathfrak{R},{\omega},\overline{\bb{w}}})(\lambda)\bb{f}\|_{\mathbb{L}^p(B_{9(R+2)})}
\leq C_{\rho,h_0,\varrho,R,\Rr_*,\Rr^*,\omega^*}|h|^{\rho}(1+|\lambda|)^{\frac{j-4}2+\varrho}\|\bb{f}\|_{\mathbb{L}^p(\R^3)}.\label{est.loc5}
\end{align}
\end{theorem}
\begin{proof}  We always assure that $\rho\in(0,\frac12)$, $\lambda, \lambda+h\in \overline{\C_+}$ and $\bb{f}\in \mathbb{L}^p_{R+2}(\R^3)$.  Define
\begin{equation}\label{eq-Kj}
\mathcal{K}_{j}(\lambda)\triangleq(\lambda I+\mathcal{L}_{\mathfrak{R},{\omega},\bb{0},\R^3})^{-1}\mathcal{P}_{\R^3}(-B_{\overline{\bb{w}}}
(\lambda I+\mathcal{L}_{\mathfrak{R},{\omega},\bb{0},\R^3})^{-1}\mathcal{P}_{\R^3})^j,\quad j\geq 0.
\end{equation}
Obviously,
\begin{align}
&\mathcal{R}^G_{\mathfrak{R},{\omega},\overline{\bb{w}}}(\lambda)=\sum^{\infty}_{j=0}\mathcal{K}_{j}(\lambda),\quad \mathcal{K}_{j}(\lambda)=-\mathcal{K}_0(\lambda)B_{\overline{\bb{w}}}\mathcal{K}_{j-1}(\lambda),\quad j\geq 1,\label{eq.K-add}
\end{align}
In the course of the proof, we shall repeatedly  use  the fact
$$\|\bb{g}\|_{\LL^1(\R^3)}\leq C(q,R)\|\bb{g}\|_{\LL^q(\R^3)},
\quad\; \bb{g}\in \LL^q_{R+2}(\R^3),\;1\leq q\leq \infty. $$
\vskip0.25cm
	
\noindent \underline {\bf  Part 1: Proof of \eqref{est.loc1}-\eqref{est.decay2}}.\;  It suffices to prove that there exist two positive constants
\[M_1=C_{R,\mathfrak{R}_*,\mathfrak{R}^*},\;\;M_2=C_{\rho,R,\mathfrak{R}_*,\mathfrak{R}^*}\]
such that for every $s\in [0,2]$, $k=0,1$ and $j\geq 0$,
\begin{equation}\label{est2-3-1}
\left\{\begin{aligned}
&\|\Lambda^s(\partial^{k}_{\lambda}\mathcal{K}_{j})(\lambda)\bb{f}\|_{\mathbb{L}^{p}(B_{9(R+2)})}\leq C_{s,R,\mathfrak{R}_*,\mathfrak{R}^*,\omega^*}(M_1\normmm{\bb{w}}_{\varepsilon,\Omega})^j\|\bb{f}\|_{\mathbb{L}^p(\R^3)},\\
&\|\Lambda^s(\partial^{k}_{\lambda}\mathcal{K}_{j})(\lambda)\bb{f}\|_{X^{s,k}_{9(R+2)}}\leq C_{s,R,\mathfrak{R}_*,\mathfrak{R}^*}(M_1\normmm{\bb{w}}_{\varepsilon,\Omega})^j\|\bb{f}\|_{\mathbb{L}^p(\R^3)},\quad s\neq 2,\\
&\|\Lambda^s(\vartriangle_{h}\partial_{\lambda}\mathcal{K}_{j})(\lambda)\bb{f}\|_{\mathbb{L}^{p}(B_{9(R+2)})}
\leq C_{s,\rho,R,\mathfrak{R}_*,\mathfrak{R}^*,\omega^*}(M_2\normmm{\bb{w}}_{\varepsilon,\Omega})^j|h|^{\rho}\|\bb{f}\|_{\mathbb{L}^p(\R^3)},\\
&\|\Lambda^s(\vartriangle_{h}\partial_{\lambda}\mathcal{K}_{j})(\lambda)\bb{f}\|_{X^{s,1+\rho}_{9(R+2)}}\leq C_{s,\rho,R,\mathfrak{R}_*,\mathfrak{R}^*}(M_2\normmm{\bb{w}}_{\varepsilon,\Omega})^j|h|^{\rho}\|\bb{f}\|_{\mathbb{L}^p(\R^3)}, \quad s\neq 2.
\end{aligned}\right.
\end{equation}
In fact,  \eqref{est.loc1}-\eqref{est.decay2} follow  from \eqref{eq.K-add}-\eqref{est2-3-1} provided that
\begin{equation}\label{asum-w}
(M_1+M_2)\normmm{\bb{w}}_{\varepsilon,\Omega}<1.
\end{equation}
	
Now we will establish the iterative scheme on account of the idea of the scale decomposition, to prove \eqref{est2-3-1}. We agree that $\lesssim$ and $\lesssim_{a,b,\cdots}$ both depends on $R,\,\mathfrak{R}_*,\,\mathfrak{R}^*$.

\textbf{Step 1: Preliminary analysis of $\mathcal{K}_j(\lambda)$}.\;
	
\vskip0.2cm

\noindent\textbf{Case 1: $j=0$}. \quad Define
\[[\text{I}_{s,r}\bb{g}](x)\triangleq \int_{\R^3}K_{s,r}(x,y)|\bb{g}(y)|\,\mathrm{d}y,\quad
K_{s,r}(x,y)\triangleq\int^{\infty}_0\frac{t^{r}}{(t+|\mathcal{Q}({\omega} t)x-y+\mathfrak{R}t\bb{e}_1|^2)^{\frac{3+s}2}}\mathrm{d}t.\]
Thanks to the facts: $\mathcal{K}_0(\lambda)=(\lambda I+\mathcal{L}_{\mathfrak{R},{\omega},\bb{0},\R^3})^{-1}\mathcal{P}_{\R^3}$ and
\[|e^{-\lambda_1 t}-e^{-\lambda_2 t}|\leq 2^{1-\delta}t^{\delta}|\lambda_1-\lambda_2|^{\delta},\quad \lambda_1,\,\lambda_2\in \overline{\C_+},\,\,\delta\in [0,1],\]
we get by \eqref{est2-2-1} and Lemma \ref{Lem.Oseen} that, for every $s\in [0,2]$ and $k=0,1$
\begin{equation}\label{est2-3-2}
\big|[\Lambda^s(\partial^k_\lambda\mathcal{K}_0)(\lambda)\bb{g}](x)\big|\leq C_0[\text{I}_{s,1}\bb{g}](x),\;\;
\big|[\Lambda^s(\vartriangle_h\partial^k_\lambda\mathcal{K}_0)(\lambda)\bb{g}](x)\big|\leq C_0|h|^{\rho}[\text{I}_{s,k+\rho}\bb{g}](x).
\end{equation}
Hence, we deduce \eqref{est2-3-1} with  $j=0$ except the case $(s,k)=(2,0)$ by making use of the following claim which will be proved in Step 3 below.
\vskip 0.1cm
\textbf{Claim 1}:\, For every $s\in [0,2]$, $r\in [0,\frac32)$ and $\bb{g}\in \LL^q_{\ell_1}(\R^3)$ with $\ell_1\ge 1$ and $q\in (1, \infty)$,
\begin{align}
&\|\text{I}_{s,r}\bb{g}\|_{\mathbb{L}^{q}(B_{\ell_2})}\leq C_1\|\bb{g}\|_{\mathbb{L}^q(\R^3)},\quad\, \ell_2> 0,\;\;  (s,r)\neq(2,0),\label{claim1-1}\\
&\|\text{I}_{s,r}\bb{g}\|_{X_{s,r,\ell_2 }}\leq C_2\|\bb{g}\|_{\mathbb{L}^q(\R^3)},\quad\;\; \ell_2 >\ell_1,\;\; s\neq 2 \label{claim1-2}
\end{align}
with $C_1=C_{s,r,\ell_1+\ell_2,\mathfrak{R}_*,\mathfrak{R}^*}$ and $C_2=C_{s,r,\ell_2 /(\ell_2 -\ell_1),\mathfrak{R}_*,\mathfrak{R}^*}$.
\vskip 0.1cm

For case $(s,k)=(2,0)$, we adopt the idea in \cite{HS09} to get
\begin{equation}\label{est2-3-3}
\|\Lambda^2\mathcal{K}_0(\lambda)\bb{g}\|_{\LL^{p}(B_{\ell_2})}\leq C_3\|\bb{g}\|_{\LL^p(\R^3)},\quad \bb{g}\in \LL^p_{\ell_1}(\R^3),\quad \ell_1,\ell_2>0
\end{equation}
for some positive constant $C_3=C_{\ell_1+\ell_2,\Rr^*,\omega^*}$, see also \cite{Shi08,Shi10}.
\vskip0.2cm

\noindent\textbf{Case 2: $j\geq 1$}.\quad
In spite of  $\mathcal{K}_{j}(\lambda)=\mathcal{K}_{0}(\lambda)
B_{\overline{\bb{w}}}\mathcal{K}_{j-1}(\lambda)$ for $j\geq 1$, we can not iterate  $\mathcal{K}_{j-1}(\lambda)\bb{f}$  to estimate  $\mathcal{K}_{j}(\lambda)\bb{f}$ by \eqref{claim1-1}-\eqref{est2-3-3} since $\mathcal{K}_{j-1}(\lambda)\bb{f}$ does not have compact support.
 But, we can do it by the decay estimate $\mathcal{K}_{j-1}(\lambda)\bb{f}$ at large scale of $|x|$. Precisely, we decompose $\mathcal{K}_{j}(\lambda)$ with $\ell_1\geq R+2$ as
\begin{equation}\label{2-ex2}
\mathcal{K}_{j}(\lambda)=-\mathcal{K}_{0}(\lambda)\chi_{B_{\ell_1}}B_{\overline{\bb{w}}}\mathcal{K}_{j-1}(\lambda)-\mathcal{K}_{0}(\lambda)\chi_{B^c_{\ell_1}}B_{\overline{\bb{w}}}\mathcal{K}_{j-1}(\lambda)\triangleq\mathcal{K}^{{\rm in},\ell_1}_{j}(\lambda)+\mathcal{K}^{{\rm out},\ell_1}_{j}(\lambda).
\end{equation}

We first deal with $\mathcal{K}^{{\rm in},\ell_1}_{j}(\lambda)$. By the Leibniz rule and \eqref{est2-3-2},  we have
\begin{equation}\label{est2-3-4}
\left\{\begin{aligned}
&\big|\Lambda^2\mathcal{K}^{{\rm in},\ell_1}_{j}(\lambda)\bb{f}\big|\leq \big|[\Lambda^2\mathcal{K}_{0}(\lambda)(\chi_{B_{\ell_1}}B_{\overline{\bb{w}}}
\mathcal{K}_{j-1}(\lambda)\bb{f})](x)\big|,\\
%%%%%%%
&\big|\Lambda^2(\partial_{\lambda}\mathcal{K}^{{\rm in},\ell_1}_{j})(\lambda)\bb{f}\big|
\leq \big|[\Lambda^2\mathcal{K}_{0}(\lambda)(\chi_{B_{\ell_1}}
B_{\overline{\bb{w}}}(\partial_{\lambda}\mathcal{K}_{j-1})(\lambda)\bb{f})](x)\big|\\
&\qquad\qquad\qquad\qquad\quad+C_0[\text{I}_{2,1}(\chi_{B_{\ell_1}}
B_{\overline{\bb{w}}}\mathcal{K}_{j-1}(\lambda)\bb{f})](x),\\
%%%%%%%%%%
&\big|\Lambda^2(\vartriangle_{h}\partial_{\lambda}\mathcal{K}^{{\rm in},\ell_1}_{j})(\lambda)\bb{f}\big|
\leq \big|[\Lambda^2\mathcal{K}_{0}(\lambda)(\chi_{B_{\ell_1}}B_{\overline{\bb{w}}}
(\vartriangle_h\partial_{\lambda}\mathcal{K}_{j-1})(\lambda)\bb{f})](x)\big|\\
&\qquad\qquad\qquad\qquad\qquad\;\,+C_0[\text{I}_{2,1}(\chi_{B_{\ell_1}}B_{\overline{\bb{w}}}
(\vartriangle_h\mathcal{K}_{j-1})(\lambda)\bb{f})](x)\\
&\qquad\qquad\qquad\qquad\qquad\;\,+C_0|h|^{\rho}\sum_{r=0,1}[\text{I}_{2,r+\rho}(\chi_{B_{\ell_1}}
B_{\overline{\bb{w}}}(\partial^{1-r}_{\lambda}\mathcal{K}_{j-1})(\lambda)\bb{f})](x),
\end{aligned}\right.
\end{equation}
and for every $s\in [0,2)$ and $k=0,1$,
 \begin{equation}\label{est2-3-5}\left\{
\begin{aligned}
&|\Lambda^s(\partial^k_{\lambda}\mathcal{K}^{{\rm in},\ell_1}_{j})(\lambda)\bb{f}|\leq C_0\sum_{r=0}^k[\text{I}_{s,r}(\chi_{B_{\ell_1}}B_{\overline{\bb{w}}}
(\partial^{k-r}_{\lambda}\mathcal{K}_{j-1})(\lambda)\bb{f})](x),\\
&\big|\Lambda^s(\vartriangle_{h}\partial^k_{\lambda}\mathcal{K}^{{\rm in},\ell_1}_{j})(\lambda)\bb{f}\big|
\leq C_0\sum_{r=0}^k\Big[|h|^{\rho}[\text{I}_{s,r+\rho}(\chi_{B_{\ell_1}}B_{\overline{\bb{w}}}(\partial^{k-r}_{\lambda}\mathcal{K}_{j-1})(\lambda)\bb{f})](x)\\
&\qquad\qquad\qquad\qquad\qquad\;\;+[\text{I}_{s,r}(\chi_{B_{\ell_1}}B_{\overline{\bb{w}}}(\vartriangle_h\partial^{k-r}_{\lambda}\mathcal{K}_{j-1})(\lambda)\bb{f})](x)\Big].
\end{aligned}\right.
\end{equation}
These estimates help us to iterate $\mathcal{K}_{j-1}(\lambda)$ to estimate $\mathcal{K}^{{\rm in},\ell_1}_j(\lambda)$ by \eqref{claim1-1}-\eqref{est2-3-3}.

Now we consider $\mathcal{K}^{{\rm out},\ell_1}_{j}(\lambda)$.  Denote  $C_4=C_{\ell_2}$ the  uniformly constant
such that
$$\|g\|_{L^q(B_{\ell_2})}\leq C_{4}\|g\|_{L^{\infty}(B_{\ell_2})},\quad \; q\in [1,\infty].$$
For every $s\in [0,2]$, $r_1,r_2\ge0$ with $r_1+r_2<\frac32$, define
\[\text{II}^{\ell_1}_{s,r_1,r_2}(x)\triangleq\int_{|y|\geq \ell_1}\frac{ J_{s,r_1}(x-y)}{|y|^{(5/2)-r_2}(1+s_{\mathfrak{R}}(y))^{(3/2)-\varepsilon}}\,\mathrm{d}y,\quad\ell_1>0,
\;\varepsilon\in (0,\tfrac12).\]
We observe that for every radial function $g$,
\[[\text{I}_{s,r}\bb{h}](x)\leq \int_{\R^3}J_{s,r}(x-y)|{g}(y)|\,\mathrm{d}y, \;\; |\bb{h}(x)|\le |g(x)|.\]
Hence,  we have  for  $k=0,1$
\begin{equation}\label{est2-3-6}
\left\{\begin{aligned}
&\big|\Lambda^s(\partial^k_{\lambda}\mathcal{K}^{{\rm out},\ell_1}_{j})(\lambda)\bb{f}\big|
\leq C_0\normmm{\bb{w}}_{\varepsilon,\Omega}\sum_{m}^1\sum^k_{r=0}
\|\nabla^m(\partial^r_{\lambda}\mathcal{K}_{j-1})(\lambda)\bb{f}\|_{X^{m,r}_{\ell_1}}\text{II}^{\ell_1}_{s,k-r,r}(x),\\
&\big|\Lambda^s(\vartriangle_{h}\partial^k_{\lambda}\mathcal{K}^{{\rm out},\ell_1}_{j})(\lambda)\bb{f}\big|\\
\leq &C_0\normmm{\bb{w}}_{\varepsilon,\Omega}\sum_{m=0}^1\sum^k_{r=0}\Big[|h|^{\rho}\|\nabla^m(\partial^r_{\lambda}\mathcal{K}_{j-1})(\lambda)\bb{f}\|_{X^{m,r}_{\ell_1}}\text{II}^{\ell_1}_{s,k-r+\rho,r}(x)\\
&+\|\nabla^m
(\vartriangle_h\partial^r_{\lambda}\mathcal{K}_{j-1})(\lambda)\bb{f}\|_{X^{m,r+\rho}_{\ell_1}}\text{II}^{\ell_1}_{s,k-r,r+\rho}(x)\Big].
\end{aligned}\right.
\end{equation}
This implies that we  can iterate  $\mathcal{K}_{j-1}(\lambda)$ to estimate $\mathcal{K}^{{\rm out},\ell_1}_j(\lambda)$ by the following claim which will be proved later in Step 3.
\vskip 0.1cm
\textbf{Claim 2}\; For  $0\leq r_1+r_2<\frac32$,
\begin{align}
&\text{II}^{\ell_1}_{s,r_1,r_2}(x)\leq C_5(1+|x|)^{-1-\frac{s}2+r_1+r_2}(1+s_{\Rr}(x))^{-1},\;\;
s\in [0,2)\;\,x\in \R^3,\label{new-ob2}
\end{align}
with $C_5=C_{s,r_1,r_2,\Rr_*,\Rr^*}$,  and
\begin{align}
\text{II}^{\ell_1}_{2,r_1,r_2}(x)\leq C_6=C_{r_1,r_2,\ell_1/(\ell_1 -\ell_2),\Rr_*,\Rr^*}, \;\;|x|\leq \ell_2<\ell_1.\label{new-ob3}
\end{align}

{\bf{Step 2:\, Iterative scheme and the proof of \eqref{est2-3-1}}}.

For $j=0$, according to \eqref{est2-3-2},  we get  \eqref{est2-3-1} by $\eqref{claim1-1}$-$\eqref{est2-3-3}$ with $\ell_1=R+2$ and $\ell_2=9(R+2)$.

Next, we are in position to prove \eqref{est2-3-1} for $j\geq 1$. When $s\in [0,2)$, in the light of \eqref{2-ex2} and \eqref{est2-3-5}-\eqref{est2-3-6} with $\ell_1=7(R+2)$, we  deduce by \eqref{claim1-1}-\eqref{claim1-2} and \eqref{new-ob2} with  $\ell_2 =9(R+2)$,  that for $s\in [0,2)$ and $k=0,1$
\begin{equation}\label{3}
\left\{\begin{aligned}
&\|\Lambda^s(\partial^k_{\lambda}\mathcal{K}_{j})(\lambda)\bb{f}\|_{\mathbb{L}^{p}(B_{9(R+2)})}
+\|\Lambda^s(\partial^k_{\lambda}\mathcal{K}_{j})(\lambda)\bb{f}\|_{X_{s,k,9(R+2)}}\\
\lesssim&_{s} \normmm{\bb{w}}_{\varepsilon,\Omega}\sum_{r=0}^k\Big[\|(\partial^r_{\lambda}\mathcal{K}_{j-1})
(\lambda)\bb{f}\|_{\mathbb{W}^{1,p}(B_{7(R+2)})}+\sum_{m=0}^1\|\nabla^m(\partial^r_{\lambda}\mathcal{K}_{j-1})
(\lambda)\bb{f}\|_{X^{m,r}_{7(R+2)}}\Big],\\
&\|\Lambda^s(\vartriangle_{h}\partial_{\lambda}\mathcal{K}_{j})(\lambda)\bb{f}
\|_{\mathbb{L}^{p}(B_{9(R+2)})}+\|\Lambda^s(\vartriangle_{h}\partial_{\lambda}
\mathcal{K}_{j})(\lambda)\bb{f}\|_{X^{s,1+\rho}_{9(R+2)}}\\
\lesssim&_{s,\rho}\normmm{\bb{w}}_{\varepsilon,\Omega}\Big[
\sum_{r=0,1}\big\|\big((\vartriangle_h\partial^r_{\lambda}\mathcal{K}_{j-1})(\lambda)\bb{f}, |h|^{\rho}(\partial^r_{\lambda}\mathcal{K}_{j-1})
(\lambda)\bb{f}\big)\big\|_{\W^{1,p}(B_{7(R+2)})}\\
&+\sum_{m,r=0,1}\Big(|h|^{\rho}\|\nabla^m(\partial^{r}_{\lambda}\mathcal{K}_{j-1})(\lambda)
\bb{f}\|_{X^{m,r}_{7(R+2)}}+\|\nabla^m(\vartriangle_h \partial^{r}_{\lambda}\mathcal{K}_{j-1})(\lambda)\bb{f}\|_{X^{m,r+\rho}_{7(R+2)}}\Big)\Big].
\end{aligned}\right.
\end{equation}
When $s=2$, we choose $\ell_1=10(R+2)$ and $\ell_2 =9(R+2)$. Then, according to \eqref{2-ex2}-\eqref{est2-3-4} and \eqref{est2-3-6}, we get by \eqref{claim1-1}, \eqref{est2-3-3}, \eqref{new-ob3} and the fact
\[ \|\bb{g}\|_{X^{s,r}_{10(R+2)}}\leq \|\bb{g}\|_{X^{s,r}_{7(R+2)}},\;\;\|\bb{g}\|_{\LL^p(B_{10(R+2)})}\leq \|\bb{g}\|_{\LL^p(B_{7(R+2)})}+C_{R}\|\bb{g}\|_{X^{s,r}_{7(R+2)}},\]
that for $k=0,1$,
\begin{equation}\label{3-add-22}
\left\{\begin{aligned}
&\|\Lambda^2(\partial_{\lambda}^k\mathcal{K}_{j})(\lambda)\bb{f}\|_{\mathbb{L}^{p}(B_{9(R+2)})}\\
\lesssim&_{\omega^*} \normmm{\bb{w}}_{\varepsilon,\Omega}\sum_{r=0}^k\Big[\|(\partial^r_{\lambda}\mathcal{K}_{j-1})
(\lambda)\bb{f}\|_{\mathbb{W}^{1,p}(B_{7(R+2)})}+\sum_{m=0}^1\|\nabla^m(\partial^r_{\lambda}\mathcal{K}_{j-1})
(\lambda)\bb{f}\|_{X^{m,r}_{7(R+2)}}\Big],\\
&\|\Lambda^2(\vartriangle_{h}\partial_{\lambda}\mathcal{K}_{j})(\lambda)\bb{f}
\|_{\mathbb{L}^{p}(B_{9(R+2)})}\\
\lesssim&_{\rho,\omega^*}\normmm{\bb{w}}_{\varepsilon,\Omega}\Big[
\sum_{r=0,1}\big\|\big((\vartriangle_h\partial^r_{\lambda}\mathcal{K}_{j-1})(\lambda)\bb{f}, |h|^{\rho}(\partial^k_{\lambda}\mathcal{K}_{j-1})
(\lambda)\bb{f}\big)\big\|_{\W^{1,p}(B_{7(R+2)})}\\
&+\sum_{m,r=0,1}\Big(|h|^{\rho}\|\nabla^m(\partial^{r}_{\lambda}\mathcal{K}_{j-1})(\lambda)
\bb{f}\|_{X^{m,r}_{7(R+2)}}+\|\nabla^m(\vartriangle_h \partial^{r}_{\lambda}\mathcal{K}_{j-1})(\lambda)\bb{f}\|_{X^{m,r+\rho}_{7(R+2)}}\Big)\Big].
\end{aligned}\right.
\end{equation}
In view of iterative inequalities \eqref{3}-\eqref{3-add-22}, the proof of \eqref{est2-3-1} for $j\ge1$
 can be reduced to proving that, there exist
${\widetilde M}_{1}=C_{R,\Rr_*,\Rr^*}$ and ${\widetilde M}_{2}=C_{\rho,R,\Rr_*,\Rr^*}$ such that for all $k,m=0,1$,
\begin{equation}\label{est2-3-9}
\left\{\begin{aligned}
&\|(\partial^{k}_{\lambda}\mathcal{K}_{j})(\lambda)\bb{f}\|_{\mathbb{W}^{1,p}(B_{7(R+2)})}\leq C_R\widetilde{M}_1(\widetilde{M}_1\normmm{\bb{w}}_{\varepsilon,\Omega})^j\|\bb{f}\|_{\mathbb{L}^p(\R^3)},\\
&\|\nabla^m(\partial^{k}_{\lambda}\mathcal{K}_{j})(\lambda)\bb{f}\|_{X^{m,k}_{7(R+2)}}\leq C_R\widetilde{M}_1(\widetilde{M}_1\normmm{\bb{w}}_{\varepsilon,\Omega})^j\|\bb{f}\|_{\mathbb{L}^p(\R^3)},\\
&\|(\triangle_h \partial^k_{\lambda}\mathcal{K}_{j})(\lambda)\bb{f}\|_{\mathbb{W}^{1,p}(B_{7(R+2)})}\leq C_R\widetilde{M}_2((\widetilde{M}_1+\widetilde{M}_2)\normmm{\bb{w}}_{\varepsilon,\Omega})^j|h|^{\rho}
\|\bb{f}\|_{\mathbb{L}^p(\R^3)},\\
&\|\nabla^m(\triangle_h\partial^k_{\lambda}\mathcal{K}_{j})(\lambda)\bb{f}\|_{X^{m,k+\rho}_{7(R+2)}}\leq C_R\widetilde{M}_2((\widetilde{M}_1+\widetilde{M}_2)\normmm{\bb{w}}_{\varepsilon,\Omega})^j|h|^{\rho}
\|\bb{f}\|_{\mathbb{L}^p(\R^3)}.
\end{aligned}\right.\end{equation}
which inserted into \eqref{3}-\eqref{3-add-22}, implies \eqref{est2-3-1} with
$ M_1=2{\widetilde M}_{1}$ and $M_2=2{\widetilde M}_{1}+{\widetilde M}_{2}$.

To prove \eqref{est2-3-9}, we give out the following claim which will be proved in Step 3.
\vskip 0.15cm
\textbf{Claim 3}:\; There exists a constant $ C_7=C_{r,\ell_1+\ell_2,\Rr_*,\Rr^*}>0$ such that
for $\bb{g}\in \LL^{q}_{\ell_1}(\R^3)$
\begin{equation}\label{new.ob}\left\{\begin{aligned}
&\|\text{I}_{s,r}\bb{g}\|_{\mathbb{L}^{\frac{4q}{4-q}}(B_{\ell_2})}\leq C_7\|\bb{g}\|_{\mathbb{L}^{q}(\R^3)},\quad (s,r)\in [0,1]\times[0,3/2), \;\; \;1<q<4,\\
& \|\text{I}_{s,r}\bb{g}\|_{\mathbb{L}^{\infty}(B_{\ell_2})}\leq C_7\|\bb{g}\|_{\mathbb{L}^{q}(\R^3)},\;\;\quad\,(s,r)\in [0,1]\times[0,3/2), \;\;\; 4\leq q\leq \infty.
\end{aligned}\right.\end{equation}
Then we choose a sequence $\{(\ell_{1,j},\ell_{2,j})\}^{\infty}_{j=0}$ such that  $R+2=\ell_{1,0}\le\ell_{1,j}< \ell_{2,j}$ for $j\geq 0$.  Let $p_{-1}\triangleq p$  and construct a sequence $\{p_j\}_{j=0}^{\infty}$ with
\begin{equation}\label{new.ob'}
p_j=\tfrac{4p_{j-1}}{4-p_{j-1}} \;\;\;\text{ if }\;\;  1< p_{j-1}<4; \;\;\;p_j=\infty, \;\;\text{ if }\;\;   4\leq p_{j-1}\leq \infty.
\end{equation}
Set
\begin{align*}
C_{2,j}\triangleq C_2|_{\ell_1=\ell_{1,j},\ell_2 =\ell_{2,j}},\quad C_{4,j}\triangleq C_4|_{\ell_2=\ell_{2,j}}, \quad C_{7,j}\triangleq C_7|_{\ell_1=\ell_{1,j},\ell_2=\ell_{2,j}}.
\end{align*}
Inserting \eqref{claim1-2}, \eqref{new-ob2} and \eqref{new.ob} into \eqref{est2-3-2} and \eqref{est2-3-5}-\eqref{est2-3-6}, we obtain  for  $k=0,1$
\begin{equation}\label{est2-3-10}
\left\{\begin{aligned}
&\sum^1_{k=0}\Big[\|(\partial^k_{\lambda}\mathcal{K}_0)(\lambda)\bb{f}\|_{\W^{1,p_0}(B_{\ell_{2,0}})}
+\sum^1_{m=0}\|\nabla^m(\partial^k_{\lambda}\mathcal{K}_{0})(\lambda)\bb{f}\|_{X^{m,k}_{\ell_{2,0}}}\Big]\leq {\widetilde M}_{1,j}\|\bb{f}\|_{\LL^p(\R^3)},\\
%%%%%%%%%%%%%%%%%%%%%%%%%%%%%%%%%%%%%%%%%%%%%%%
&\sum_{k=0,1}\Big[\|(\partial^k_{\lambda}\mathcal{K}_{j})(\lambda)\bb{f}\|_{\mathbb{W}^{1,p_j}(B_{\ell_{2,j}})}
+\sum_{m=0,1}\Big[\|\nabla^m(\partial^k_{\lambda}\mathcal{K}_{j})(\lambda)\bb{f}\|_{X^{1,k}_{\ell_{2,j}}}\Big]\\
\leq&{\widetilde M}_{1,j}\normmm{\bb{w}}_{\varepsilon,\Omega}\sum^1_{k=0}\Big[\|(\partial^k_{\lambda}\mathcal{K}_{j-1})(\lambda)
\bb{f}\|_{\mathbb{W}^{1,{p_{j-1}}}(B_{\ell_{1,j}})}+\sum^1_{m=0}\|\nabla^m\partial^k_{\lambda}\mathcal{K}_{j-1}(\lambda)\bb{f}\|_{X^{m,k}_{\ell_{1,j}}}\Big],
\end{aligned}\right.
\end{equation}
and
\begin{equation}\label{est2-3-11}
\left\{\begin{aligned}
&\sum_{k=0,1}\Big[\|(\vartriangle_h\partial^k_{\lambda}\mathcal{K}_0)(\lambda)
\bb{f}\|_{\W^{1,p_0}(B_{\ell_{2,0}})}+\sum_{m=0,1}\|\nabla^m(\vartriangle_{h}
\partial^k_{\lambda}\mathcal{K}_{0})(\lambda)\bb{f}\|_{X^{m,k+\rho}_{\ell_{2,0}}}\big]\\
\leq& {\widetilde M}_{2,j}|h|^{\rho}\|\bb{f}\|_{\LL^p(\R^3)},\\
%%%%%%%%%%%%%%%%%%%%%%%%%%%%%%%%%%%%%%%%%%%%%%%%%
&\sum_{k=0,1}\Big[\|(\vartriangle_{h}\partial^k_{\lambda}\mathcal{K}_{j})(\lambda)
\bb{f}\|_{\mathbb{W}^{1,p_{j}}(B_{\ell_{2,j}})}
+\sum_{m=0,1}\|\nabla^m(\vartriangle_{h}\partial^k_{\lambda}\mathcal{K}_{j})(\lambda)\bb{f}\|_{X^{m,k+\rho}_{\ell_{2,j}}}\Big]\\
\leq& {\widetilde M}_{2,j}\normmm{\bb{w}}_{\varepsilon,\Omega}
\sum_{k=0,1}\Big[\big\|\big((\vartriangle_h\partial^k_{\lambda}\mathcal{K}_{j-1})(\lambda)\bb{f}, |h|^{\rho}(\partial^k_{\lambda}
\mathcal{K}_{j-1})(\lambda)\bb{f}\big)\big\|_{\W^{1,p_{j-1}}(B_{\ell_{1,j}})}\\
&+
\sum_{m=0,1}\Big(|h|^{\rho}\|\nabla^m(\partial^{k}_{\lambda}\mathcal{K}_{j-1})(\lambda)\bb{f}\|_{X^{m,k}_{\ell_{1,j}}}+\|\nabla^m(\vartriangle_h\partial^{k}_{\lambda}
\mathcal{K}_{j-1})(\lambda)\bb{f}\|_{X^{m,k+\rho}_{\ell_{1,j}}}\Big)\Big],
\end{aligned}\right.
\end{equation}
where\begin{align*}
&{\widetilde M}_{1,j}=4C_0\sup_{s\in[0,1]}\sum_{k=0,1}\Big((C_{7,j}+C_{2,j})|_{r=1-k}+(1+C_{4,j})
\big(C_{5}|_{r_1=k \atop r_2=0}+C_{5}|_{r_1=0 \atop r_2=k}\big)\Big),\\
&{\widetilde M}_{2,j}(\rho)=4C_0\sup_{s\in[0,1]}\sum_{k=0,1}\Big((C_{7,j}+C_{2,j})|_{r=1-k}
+(C_{7,j}+C_{2,j})|_{r=1-k+\rho}\\
&\qquad\qquad\,+(1+C_{4,j})\big(C_{5}|_{r_1=k \atop r_2=\rho}+C_{5}|_{r_1=0 \atop r_2=k+\rho}+ C_{5}\big|_{r_1=k+\rho \atop r_2=0}+
C_{5}\big|_{r_1=\rho \atop r_2=k}\big)\Big).
\end{align*}
Since $p_j=\infty$ if $j\geq 3$, invoking that
\[\|\bb{g}\|_{X_{s,r,\ell_{1,j}}}\leq \|\bb{g}\|_{X_{s,r,\ell_{2,j}}}+C_{8,j}\|\bb{g}\|_{\LL^{\infty}(B_{\ell_{1,j},\ell_{2,j}})},\quad C_{8,j}\triangleq C_{\ell_{2,j}},\]
we  improve \eqref{est2-3-10}-\eqref{est2-3-11} for  $j\geq 4$ to be
\begin{align}
&\sum_{k=0,1}\Big[\|(\partial^k_{\lambda}\mathcal{K}_{j})(\lambda)\bb{f}\|_{\mathbb{W}^{1,p_j}(B_{\ell_{2,j}})}
+\sum_{m=0,1}\|\nabla^m(\partial^k_{\lambda}\mathcal{K}_{j})(\lambda)\bb{f}\|_{X^{m,k}_{\ell_{2,j}}}\Big]\nonumber\\
\leq& 2(1+C_{8,j}){\widetilde M}_{1,j}\normmm{\bb{w}}_{\varepsilon,\Omega}\sum_{r=0,1}\Big[\|(\partial^k_{\lambda}\mathcal{K}_{j-1})(\lambda)
\bb{f}\|_{\mathbb{W}^{1,{p_{j-1}}}(B_{\ell_{2,j}})}\nonumber\\
&+\sum_{m=0,1}\|\nabla^m(\partial^k_{\lambda}\mathcal{K}_{j-1})(\lambda)\bb{f}\|_{X^{m,k}_{\ell_{2,j}}}\Big],
\label{est2-3-12}
\end{align}
and
\begin{align}
&\sum_{k=0,1}\Big[\|(\vartriangle_{h}\partial^k_{\lambda}\mathcal{K}_{j})(\lambda)
\bb{f}\|_{\mathbb{W}^{1,p_{j}}(B_{\ell_{2,j}})}
+\sum_{m=0,1}\|\nabla^m(\vartriangle_{h}\partial^k_{\lambda}\mathcal{K}_{j})(\lambda)\bb{f}\|_{X^{m,k+\rho}_{\ell_{2,j}}}\Big]\nonumber\\
\leq&2(1+C_{8,j}){\widetilde M}_{2,j}\normmm{\bb{w}}_{\varepsilon,\Omega}
\sum_{r=0,1}\Big[\big\|\big((\vartriangle_h\partial^r_{\lambda}\mathcal{K}_{j-1})(\lambda)\bb{f}, |h|^{\rho}(\partial^r_{\lambda}
\mathcal{K}_{j-1})(\lambda)\bb{f}\big)\big\|_{\W^{1,p_{j-1}}(B_{\ell_{2,j}})}\nonumber\\
&+\sum_{m=0,1}\Big(|h|^{\rho}\|\nabla^m(\partial^{r}_{\lambda}\mathcal{K}_{j-1})(\lambda)\bb{f}\|_{X^{m,r}_{\ell_{2,j}}}+\|\nabla^m(\vartriangle_h\partial^{r}_{\lambda}
\mathcal{K}_{j-1})(\lambda)\bb{f}\|_{X^{m,r+\rho}_{\ell_{2,j}}}\Big)\Big].\label{est2-3-13}
\end{align}
To close  the iterative scheme \eqref{est2-3-10}-\eqref{est2-3-13}, we  choose
\begin{align*}
&\ell_{1,j}=(2j+1)(R+2)\;\;\text{if}\;\; j\leq 3,\;\;
\ell_{1,j}=7(R+2)\;\;\text{if}\;\; j\geq 4,\\
& \ell_{2,j}=\ell_{1,j}+2(R+2),
\end{align*}
such that  ${\widetilde M}_{1,j}$ and ${\widetilde M}_{2,j}$ are uniformly bounded in $j$. So we denote
\begin{align*}
{\widetilde M}_{1}=2\max_{j\geq 0}(1+C_{8,j}){\widetilde M}_{1,j},\quad {\widetilde M}_{2}(\rho)=2\max_{j\geq 0}(1+C_{8,j}){\widetilde M}_{2,j}.
\end{align*}
From \eqref{est2-3-10} and \eqref{est2-3-12}, we deduce by induction
\begin{align}\nonumber
&\sum_{k=0,1}\Big[\|(\partial^{k}_{\lambda}\mathcal{K}_{j})(\lambda)
\bb{f}\|_{\mathbb{W}^{1,p_j}(B_{\ell_{2,j}})}
+\sum_{m=0,1}\|\nabla^m(\partial^{k}_{\lambda}\mathcal{K}_{j})(\lambda)\bb{f}\|_{X^{m,k}_{\ell_{2,j}}}\Big]\\
\leq &{\widetilde M}_{1}({\widetilde M}_{1}\normmm{\bb{w}}_{\varepsilon,\Omega})^j\|\bb{f}\|_{\mathbb{L}^p(\R^3)}\label{est2-3-14}
\end{align}
with
\begin{equation*}
\left\{\begin{aligned}
&p_j=\infty, \; \,\,\forall j\geq 0,\;\;\; \text{if }4\leq p<\infty,\\
&p_0=\tfrac{4p}{4-p},\,\,p_j=\infty, \,\,\forall j\geq 1,\;\;\; \text{if }2\leq p<4,\\
&p_0=\tfrac{4p}{4-p},\,\,p_1=\tfrac{4p}{4-2p},\,\,p_j=\infty,\;\,\,\forall j\geq 2,\;\;\; \text{if }\tfrac43\leq p<2,\\
&p_0=\tfrac{4p}{4-p},\,\,p_1=\tfrac{4p}{4-2p},\,\,p_2=\tfrac{4p}{4-3p},\,\,p_j=\infty,\,\,\forall j\geq 3,\;\;\; \text{if }1< p<\tfrac43.
\end{aligned}\right.
\end{equation*}
Inserting  \eqref{est2-3-14} into  \eqref{est2-3-11} and \eqref{est2-3-13}, we get for $ j\geq 1$
\begin{align*}
&\sum_{k=0,1}\Big[\|(\vartriangle_{h}\partial^k_{\lambda}\mathcal{K}_{j})(\lambda)
\bb{f}\|_{\mathbb{W}^{1,p_{j}}(B_{\ell_{2,j}})}
+\sum_{m=0,1}\|\nabla^m(\vartriangle_{h}\partial^k_{\lambda}\mathcal{K}_{j})(\lambda)\bb{f}\|_{X^{m,k+\rho}_{\ell_{2,j}}}\Big]\\
\leq&  {\widetilde M}_{2}|h|^{\rho}({\widetilde M}_{1}\normmm{\bb{w}}_{\varepsilon,\Omega})^j \|\bb{f}\|_{\LL^p(\R^3)}+{\widetilde M}_{2}\normmm{\bb{w}}_{\varepsilon,\Omega}
\sum_{r=0,1}\|(\vartriangle_h\partial^r_{\lambda}\mathcal{K}_{j-1})(\lambda)\bb{f}\|_{\W^{1,p_{j-1}}(B_{\ell_{3,j}})}\\
&+(1/2){\widetilde M}_{2}\normmm{\bb{w}}_{\varepsilon,\Omega}\sum_{k,m=0,1}\|\nabla^m(\vartriangle_h\partial^{r}_{\lambda}
\mathcal{K}_{j-1})(\lambda)\bb{f}\|_{X^{m,r+\rho}_{\ell_{3,j}}},
\end{align*}
with $\ell_{3,j}=\ell_{1,j}$ if $j\leq 3$ and  $\ell_{3,j}=\ell_{1,j}$ if $j\geq 4$. This iteration scheme, together with the first inequality of \eqref{est2-3-11}, yields that for $k=0,1$ and $j\geq 0$
\begin{align*}
&\sum_{k=0,1}\Big[\|(\vartriangle_{h}\partial^k_{\lambda}\mathcal{K}_{j})(\lambda)
\bb{f}\|_{\mathbb{W}^{1,p_{j}}(B_{\ell_{2,j}})}
+\sum_{m=0,1}\|\nabla^m(\vartriangle_{h}\partial^k_{\lambda}\mathcal{K}_{j})(\lambda)\bb{f}\|_{X^{m,k+\rho}_{\ell_{2,j}}}\Big]\\
\leq &
2{\widetilde M}_{2}(({\widetilde M}_{1}+{\widetilde M}_{2})\normmm{\bb{w}}_{\varepsilon,\Omega})^{j}\|\bb{f}\|_{\mathbb{L}^p(\R^3)},
\end{align*}
which together with \eqref{est2-3-14} implies \eqref{est2-3-9}.

\vskip0.2cm

\textbf{Step 3. Proof of  Claim 1-Claim 3}.

\vskip 0.15cm

\noindent\textbf{Proof of \eqref{claim1-1} and \eqref{new.ob}}.\quad Set $\bb{g}\in \mathbb{L}^p_{\ell_1}$ and $(s,r)\in [0,2]\times[0,\frac32)$ with $(s,r)\not=(2,0)$. Since $\bb{g}=\bb{0}$ in $B^c_{\ell_1}$,  we can rewrite
\begin{align*}
K_{s,r}(x,y)=&K^1_{s,r}(x,y)+K^2_{s,r}(x,y)\\
\triangleq&\Big[\int^{\infty}_{2(\ell_1+\ell_2)/|\Rr|}
+\int^{2(\ell_1+\ell_2)/|\Rr|}_0\Big]
\frac{t^{r}\chi_{|y|\leq \ell_1}}{(t+|\mathcal{Q}({\omega} t)x-y+\mathfrak{R}t\bb{e}_1|^2)^{(3+s)/2}}\mathrm{d}t.
\end{align*}
Obviously, when $t>{2(\ell_1+\ell_2)}/{|\Rr|}$,
\[|\mathcal{Q}({\omega} t)x-y+\mathfrak{R}t\bb{e}_1|^2\geq (x_1-y_1+\Rr t)^2\geq\Rr^2t^2/4,\quad x\in B_{\ell_2},\;y\in B_{\ell_1}.\]
Hence, we have for $x\in B_{\ell_2}$ and $y\in B_{\ell_1}$
$$K^1_{s,r}(x,y)\leq C\int^{\infty}_{2(\ell_1+\ell_2)/|\Rr|}
\frac{t^{r}}{(t+\mathfrak{R}^2t^2)^{(3+s)/2}}\mathrm{d}t\leq \frac{C|\Rr|^{-r-1}}{(\ell_1+\ell_2)^{2+s-r}}.$$
This inequality implies
\[\sup_{x\in B_{\ell_2}}\|K^1_{s,r}(x,\cdot)\|_{L^q(B_{\ell_1})}\leq C\tfrac{|\Rr|^{-r-1}\ell_1}{(\ell_1+\ell_2)^{s-r}}.\]
For $K^2_{s,r}(x,y)$, by the Minkowski inequality, we have
\begin{align*}
\|K^2_{s,r}(x,\cdot)\|_{L^q(B_{\ell_1})}\leq &
\int^{2(\ell_1+\ell_2)/|\Rr|}_0t^r\Big(\int_{|y|\leq \ell_1}\big(t+|\mathcal{Q}({\omega} t)x-y+\mathfrak{R}t\bb{e}_1|^2\big)^{-\frac{(3+s)q}2}\,\mathrm{d}y\Big)^{\frac1q}\mathrm{d}t.
\end{align*}
Observing  for $x\in B_{\ell_2}$ and $t\leq 2(\ell_1+\ell_2)/|\Rr|$
\begin{align*}
&\Big(\int_{|y|\leq \ell_1}\big(t+|\mathcal{Q}({\omega} t)x-y+\mathfrak{R}t\bb{e}_1|^2\big)^{-\frac{(3+s)q}2}\,\mathrm{d}y\Big)^{1/q}\\
\leq &\big(\int_{|z|\leq 3(\ell_1+\ell_2)}\big(t+|z|^2\big)^{-\frac{(3+s)q}2}\,\mathrm{d}y\Big)^{1/q}\leq C_{s,q}t^{\frac3{2q}-\frac{3+s}2},
\end{align*}
we have  for  $q\in [1,\infty)$ with $s<2r+\frac3{q}-1$
\begin{align*}
\sup_{x\in B_{\ell_2}}\|K^2_{s,r}(x,\cdot)\|_{L^q(B_{\ell_1})}\leq &C_{s,q}
\int^{\frac{2(\ell_1+\ell_2)}{|\Rr|}}_0t^{r+\frac3{2q}-\frac{3+s}2}\mathrm{d}t\leq C_{s,q.r}|\Rr|^{\frac{s+1}2-r-\frac3{2q}}(\ell_1+\ell_2)^{\frac3{2q}-\frac{1+s}2}.
\end{align*}

Hence, in the light of $K_{s,r}(x,y)=K_{s,r}(y,x)$, we obtain that \eqref{claim1-1} and \eqref{new.ob}  by Lemma \ref{Lem.ST} via choosing $q=1$ or $q=\tfrac43$, respectively.
\vskip 0.2cm
\noindent\textbf{Proof of  \eqref{claim1-2}}.\quad	Let $\bb{g}\in \LL^p_{\ell_1}(\R^3)$.
We observe that for every $x,y\in\R^3$
\[ |\mathcal{O}({\omega}t)x-y+\mathfrak{R}t\bb{e}_1|=|x-\mathcal{O}^T({\omega}t)y+\mathfrak{R}t\bb{e}_1|\geq |x-y^*+\mathfrak{R}t\bb{e}_1|\quad\text{ for all }t\geq 0\]
where $y^*=(y^*_1,y^*_2,y^*_3)$ satisfies
$$|y^*|=|y|,\;\;\; y^*_1=y_1,\;\;\;\; x_2y^*_3=x_3y^*_2.$$
Hence, we have
\[[\text{I}_{s,r}\bb{g}](x)\leq \int_{\R^3}J_{s,r}(x-y^*)|\bb{g}(y)|\,\mathrm{d}y,\quad(s,r)\in [0,2)\times [0,\tfrac32).\]
Since
$$|x-y^*|\geq \tfrac{\ell_2 -\ell_1}{\ell_2 }|x|\geq \ell_2 -\ell_1,\quad x\in B^c_{\ell_2 },\;y\in B_{\ell_1},\;\;\ell_2 >\ell_1,$$
we get by \eqref{est-J2}
\[[\text{I}_{s,r}\bb{g}](x)\leq \int_{B_{\ell_1}}\frac{C_{s,r,\ell_2 ,(\ell_2 -\ell_1)^{-1}}|\mathfrak{R}|^{(s/2)-r}}
{|x-y^*|^{1+\frac{s}2-r}(1+2|\mathfrak{R}|s_{\mathfrak{R}}(x-y^*))^{1+\frac{s}2}}
|\bb{g}(y)|\,\mathrm{d}y,\;\; x\in B^c_{\ell_2 }.\]
 This inequality, combining with
$$s_{\mathfrak{R}}(x)\leq s_{\mathfrak{R}}(x-y^*)+s_{\mathfrak{R}}(y^*)\leq s_{\mathfrak{R}}(x-y^*)+2|y^*|,$$
 yields for all $x\in B^c_{\ell_2 }$
\begin{align*}
&(1+s_{\Rr}(x))[\text{I}_{s,r}\bb{g}](x)\\
\leq &C_{s,r,\Rr_*,\Rr^*,\ell_2 ,(\ell_2 -\ell_1)^{-1}}\int_{B_{\ell_1}}\frac{(1+s_{\mathfrak{R}}(x-y^*))+2\ell_1}
{|x-y^*|^{1+\frac{s}2-r}(1+2|\mathfrak{R}|s_{\mathfrak{R}}(x-y^*))^{1+\frac{s}2}}|\bb{g}(y)|\,\mathrm{d}y\\
\leq &C_{s,r,\Rr_*,\Rr^*,\ell_2 ,(\ell_2 -\ell_1)^{-1}}|x|^{-1-\frac{s}2+r}\|\bb{g}\|_{\mathbb{L}^1(\R^3)}
\leq C_{s,r,\Rr_*,\Rr^*,\ell_2 ,(\ell_2 -\ell_1)^{-1}}|x|^{-1-\frac{s}2+r}\|\bb{g}\|_{\mathbb{L}^p(\R^3)}.
\end{align*}
This  ends the proof of \eqref{claim1-2}.

\vskip 0.2cm
\noindent\textbf{Proof of \eqref{new-ob2}}.\quad
Let $s\in [0,2)$, $r_1+r_2\in [0,3/2)$ and $\varepsilon\in (0,\tfrac12)$. We decompose
\begin{align*}
\text{II}^{\ell_1}_{s,r_1,r_2}(x)=&\text{II}^{\ell_1,1}_{s,r_1,r_2}(x)+\text{II}^{\ell_1, 2}_{s,r_1,r_2}(x)\\
\triangleq&\Big[\int_{B^c_{\ell_1}\cap B_{(2|\Rr|)^{-1}}(x)}+\int_{B^c_{\ell_1}\cap B^c_{(2|\Rr|)^{-1}}(x)}\Big]\frac{ J_{s,r_1}(x-y)}{|y|^{\frac52-r_2}(1+s_{\mathfrak{R}}(y))^{\frac32-\varepsilon}}\,\mathrm{d}y.\end{align*}

Let us begin with $\text{II}^{\ell_1,1}_{s,r_1,r_2}(x)$. When $2+s-r_1<3$,  by \eqref{est-J1} and the fact that
\[1+s_{\Rr}(x)\leq 1+s_{\Rr}(x-y)+s_{\Rr}(y)\leq (1+1/|\Rr|)(1+s_{\Rr}(y)),\quad |x-y|\leq \tfrac1{2|\mathfrak{R}|},\]
we have
\begin{align}\nonumber
\text{II}^{\ell_1,1}_{s,r_1,r_2}(x)&
\leq  \int_{B^c_{\ell_1}\cap B_{(2|\Rr|)^{-1}}(x)}\frac{C_{s,r_1,\Rr_*,\Rr^*}}{|x-y|^{2+s-r_1}}\frac1{|y|^{\frac52-r_2}(1+s_{\mathfrak{R}}(y))^{\frac32-\varepsilon}}\,\mathrm{d}y\\
&\leq \frac{C_{s,r_1,\Rr_*,\Rr^*}}{(1+s_{\mathfrak{R}}(x))^{\frac32-\varepsilon}}\int_{B^c_{\ell_1}\cap B_{(2|\Rr|)^{-1}}(x)}\frac{1}{|x-y|^{2+s-r_1}}\frac1{|y|^{\frac52-r_2}}\,\mathrm{d}y\label{est-2-3-15}
\end{align}
which implies that
\[\|\text{II}^{\ell_1,1}_{s,r_1,r_2}\|_{L^{\infty}(\R^3)}
\leq \int_{|z|\leq \theta|\mathfrak{R}|^{-1}}\frac{C_{s,r_1,r_2,1/\ell_1,\Rr_*,\Rr^*}}{|z|^{2+s-r_1}}\,\mathrm{d}y\leq C_{s,r_1,r_2,1/\ell_1,\Rr_*,\Rr^*}.\]
Further, according to that $\frac{|x|}2\leq |y|\leq \frac{3|x|}2$ if $|y-x|\leq \frac1{2|\Rr|}\leq \frac{|x|}2$, we get from \eqref{est-2-3-15}
\begin{align*}
\text{II}^{\ell_1,1}_{s,r_1,r_2}(x)
&\leq \frac{C_{s,r_1,r_2,\Rr_*,\Rr^*}}{|x|^{\frac52-r_2}(1+s_{\mathfrak{R}}(x))^{\frac32-\varepsilon}}\int_{|y-x|\leq (2|\mathfrak{R}|)^{-1}}\frac{1}{|x-y|^{2+s-r_1}}
\,\mathrm{d}y\\
&\leq C_{s,r_1,r_2,\mathfrak{R}_*,\Rr^*}|x|^{-(5/2)+r_2}(1+s_{\mathfrak{R}}(x))^{-(3/2)+\varepsilon}.
\end{align*}
Collecting the above two estimates, we obtain for $x\in\R^3$
\begin{equation}\label{est2-3-16}
\text{II}^{\ell_1,1}_{s,r_1,r_2}(x)
\leq C_{s,r_1,r_2,\mathfrak{R}_*,\Rr^*}(1+|x|)^{-\frac52+r_2}(1+s_{\mathfrak{R}}(x))^{-\frac32+\varepsilon},\; \;2+s-r_1<3.
\end{equation}
When $2+s-r_1\geq 3$, since $s\in [0,2)$,  we get by \eqref{est-J2}
\[\text{II}^{\ell_1,1}_{s,r_1,r_2}(x)\leq  \int_{B^c_{\ell_1}\cap B_{(2|\Rr|)^{-1}}(x)}\frac{C_{s,r_1}}{|x-y|^{1+s-2r_1}}\frac1{|y|^{\frac52-r_2}(1+s_{\mathfrak{R}}(y))^{\frac32-\varepsilon}}\,\mathrm{d}y.\]
Hence, in same way as deducing \eqref{est2-3-16}, we get for $x\in\R^3$
\begin{equation}\label{est2-3-17}
\text{II}^{\ell_1,1}_{s,r_1,r_2}(x)\leq C_ {s,r_1,r_2,\mathfrak{R}_*,\Rr^*}(1+|x|)^{-(5/2)+r_2}(1+s_{\mathfrak{R}}(x))^{-(3/2)+\varepsilon},\;\;2+s-r_1\geq 3.
\end{equation}
	\vskip0.2cm
\
Next, we deal with $\text{II}^{\ell_1,2}_{s,r_1,r_2}(x)$. Noting that
\begin{align*}
&\tfrac{1+|\Rr|s_{\Rr}(y)}{1+s_{\Rr}(y)}\leq \min\{1,|\Rr|\}\quad\text{\rm for }\; y\in \R^3;\quad\tfrac{1+|x-y|}{|x-y|}\leq 1+2|\Rr|\quad\text{\rm for}\; |x-y|\geq \tfrac1{2|\Rr|},
\end{align*}
we have by \eqref{est-J1}
\begin{align*}
&\text{II}^{\ell_1,2}_{s,r_1,r_2}(x)\\
\leq &\int_{ B^c_{\ell_1}\cap B^c_{1/(2|\Rr|)}(x)}\frac{C_{s,r_1,\Rr_*,\Rr^*}}{(1+|x-y|)^{\frac{2+s}2-r_1}
(1+s_{\mathfrak{R}}(x-y))^{\frac{2+s}2}}\frac{1}{|y|^{\frac52-r_2}
(1+s_{\mathfrak{R}}(y))^{\frac32-\varepsilon}}\,\mathrm{d}y.
\end{align*}
This, combining with the fact that
 \begin{equation}\label{est2-3-18}
 \tfrac{|x-y|}2\leq |y|\leq \tfrac{3|x-y|}2\;\;\;\text{for}\;\;\;|x-y|\geq 2|x|,
 \end{equation}
and Lemma \ref{Lem.In}, yields
\[\|\text{II}^{\ell_1,2}_{s,r_1,r_2}\|_{L^{\infty}(B_{1/(4|\Rr|)})}\leq\int_{B^c_{\ell_1}}
\frac{C_{_s,r_1,\Rr_*,\Rr^*}}{|y|^{\frac72+\frac{s}2-r_1-r_2}(1+s_{\mathfrak{R}}(y))^{\frac32-\varepsilon}}\,\mathrm{d}y\leq C_{s,r_1,r_2,\Rr_*,\Rr^*}.\]
On the other hand, for $|x|\geq \frac 1{4|\Rr|}$, we make the following decomposition
\begin{align*}
\text{II}^{\ell_1,2}_{s,r_1,r_2}(x)\leq&\Big[\int_{B^c_{\ell_1}\cap B^c_{2|x|}(x)}+ \int_{B^c_{\ell_1}\cap B_{1/(2|\Rr|),2|x|}(x)}\Big]\\
&\;\;\frac{C_{s,r_1,\Rr_*,\Rr^*} }{(1+|x-y|)^{1+\frac{s}2-r_1}(1+s_{\mathfrak{R}}(x-y))^{1+\frac{s}2}}
\frac{1}{|y|^{\frac52-r_2}(1+s_{\mathfrak{R}}(y))^{\frac32-\varepsilon}}\,\mathrm{d}y\\
\triangleq& J_1(x)+J_2(x).
\end{align*}
Since $s_{\Rr}(x)\leq s_{\Rr}(x-y)+s_{\Rr}(y)$, we obtain by \eqref{est2-3-18} and Lemma \ref{Lem.In}
\begin{align*}
(1+s_{\Rr}(x))J_1(x)&\leq \int_{|y|\geq |x|}\frac{C_{s,r_1,r_2,\Rr_*,\Rr^*}}{|y|^{\frac72+\frac{s}2-r_1-r_2}(1+s_{\mathfrak{R}}(y))^{\frac32-\varepsilon}}\,\mathrm{d}y\\
&\quad+\int_{|x-y|\geq 2|x|}\frac{C_{s,r_1,r_2,\Rr_*,\Rr^*}}
{|x-y|^{\frac72+\frac{s}2-r_1-r_2}(1+s_{\mathfrak{R}}(y))^{1+\frac{s}2}}\,\mathrm{d}y\\
&\leq C_{s,r_1,r_2,\Rr_*,\Rr^*}|x|^{-1-\frac{s}2+r_1+r_2}.
\end{align*}
For $J_2(x)$, with the help of Lemma $3.1$ in \cite{Far92}, we have
\begin{align*}
J_2(x)&\leq \int_{B^c_{\ell_1}\cap B_{\theta/|\Rr|,2|x|}(x)}\frac{C_{r_1,r_2,\Rr_*,\Rr^*}|x|^{r_1+r_2}}{(1+|x-y|)^{1+\frac{s}2}(1+s_{\mathfrak{R}}(x-y))^{1+\frac{s}2}}\frac{1}{|y|^{\frac52}(1+s_{\mathfrak{R}}(y))^{\frac32-\varepsilon}}\,\mathrm{d}y\\
&\leq C_{s,r_1,r_2,\Rr_*,\Rr^*}|x|^{-1-\frac{s}2+r_1+r_2}(1+s_{\mathfrak{R}}(x))^{-1}.
\end{align*}
Summing up, we get
\begin{equation}\label{est2-3-19}
\text{II}^{\ell_1,2}_{s,r_1,r_2}(x)\leq C_{s,r_1,r_2,\Rr_*,\Rr^*}(1+|x|)^{-1-\frac{s}2+r_1+r_2}(1+s_{\Rr}(x))^{-1},\quad x\in \R^3.
\end{equation}

Collecting \eqref{est2-3-16}-\eqref{est2-3-17} and \eqref{est2-3-19}, we prove \eqref{new-ob2}.
\vskip 0.2cm

\noindent\textbf{Proof of \eqref{new-ob3}}.\quad	
Noting that
$$\tfrac{\ell_1-\ell_2}{\ell_1}|y|\leq |y-x|\leq \tfrac{\ell_1+\ell_2}{\ell_1}|y|,\quad |x|\leq \ell_2,\;|y|\geq \ell_1,\;0<\ell_2<\ell_1.$$
we get by Lemma \ref{Lem-J} and Lemma \ref{Lem.In} that for $r_1+r_2\in [0,3/2)$
\begin{align*}
\text{II}^{\ell_1}_{2,r_1,r_2}(x)&\leq\int_{|y|\geq \ell_1}\frac{C_{r_1,\Rr_*,\Rr^*}}{|x-y|^{2-r_1}(1+s_{\Rr}(x-y))^{\frac32}}\frac{1}{|y|^{\frac52-r_2}(1+s(y))^{\frac32-\varepsilon}}\,\mathrm{d}y\\
&\leq\int_{|y|\geq \ell_1}\frac{C_{r_1,r_2,\ell_1/(\ell_1-\ell_2),\mathfrak{R}_*,\mathfrak{R}^*}}
{|y|^{\frac92-r}(1+s_{\Rr}(y))^{\frac32-\varepsilon}}\,\mathrm{d}y
\leq C_{r_1,r_2,\ell_1/(\ell_1-\ell_2),\mathfrak{R}_*,\mathfrak{R}^*}.
\end{align*}

Hence, we conclude \eqref{new-ob3},
and so complete the proof of {\bf Claim 1-Claim 3}.	
\vskip0.25cm

\noindent \underline{\bf{Part 2: Proof of \eqref{est.loc3}-\eqref{est.loc5}.}}\quad
In view of \eqref{est.loc1}-\eqref{est.loc2}, we only consider the case $|\lambda|\geq 2h_0$.
For the sake of statement, we agree that $\lesssim$ and $\lesssim_{a,b,\cdots}$  both still depend on $p,R,\Rr_*,\Rr^*,{\omega}^*$ and set
\[E_j\triangleq (M_1\normmm{\bb{w}}_{\varepsilon,\Omega})^j,\quad F_j\triangleq ((M_1+M_2)\normmm{\bb{w}}_{\varepsilon,\Omega})^j\]
where $M_1$ and $M_2$ are the constants in \eqref{est2-3-1}.
Thanks to \eqref{eq.K-add} and \eqref{asum-w}, we only need to prove that for all $|\beta|\leq 2$ and $0<\varrho\ll 1/2$
\begin{equation}\label{est-K_j}
\left\{\begin{aligned}
&\|\partial^{\beta}_x\mathcal{K}_{j}(\lambda)\bb{f}\|_{\mathbb{L}^p(B_{9(R+2)})}\leq C_{R,\Rr_*,\Rr^*,\omega^*}|\lambda|^{-1+(|\beta|/2)}E_j\|\bb{f}\|_{\mathbb{L}^p(\R^3)},\\
&\|\partial^{\beta}_x(\partial_{\lambda}\mathcal{K}_{j})(\lambda)\bb{f}\|_{\mathbb{L}^p(B_{9(R+2)})}\leq C_{\varrho,R,\Rr_*,\Rr^*,\omega^*}|\lambda|^{-2+(|\beta|/2)+\varrho}E_j\|\bb{f}\|_{\mathbb{L}^p(\R^3)},\\
&\|\partial^{\beta}_x(\vartriangle_h\partial_{\lambda}\mathcal{K}_{j})(\lambda)\bb{f}\|_{\mathbb{L}^p(B_{9(R+2)})}\leq C_{\rho,h_0,\varrho,R,\Rr_*,\Rr^*}|h|^{\rho}|\lambda|^{-2+(|\beta|/2)
+\varrho}F_j\|\bb{f}\|_{\mathbb{L}^p(\R^3)}.
\end{aligned}\right.
\end{equation}

To prove \eqref{est-K_j}, we will adopt the ``tree self-semilar iterative" idea. \vskip0.2cm

\textbf{Step 1: Preliminary analysis of $\partial^{\beta}_x\mathcal{K}_j(\lambda)$}.\;
\vskip0.15cm

\noindent\textbf{Case 1: $|\beta|=0$}.\quad
Let $\mathcal{M}\triangleq\mathcal{L}_{\Rr,\omega,\bb{0},\R^3}-\omega[\bb{e}_1\times]=-\Delta -\mathfrak{R}\partial_1-{\omega}(\bb{e}_1\times {x})\cdot\nabla$, where
 $$[\bb{e}_1\times]^k\triangleq \underbrace{\bb{e}_1\times\bb{e}_1\times\cdots\bb{e}_1\times}_k, \;\;\; \text{for
every integer}\;\  k>0.$$
In the light of \eqref{eq-Kj}  and the fact that \begin{align*}
&(\lambda I+\mathcal{L}_{\mathfrak{R},{\omega},\bb{0},\R^3})^{-1}\mathcal{P}_{\R^3}
=\lambda^{-1}\mathcal{P}_{\R^3}-\lambda^{-1}\mathcal{L}_{\Rr,{\omega},\bb{0},\R^3}
(\lambda+\mathcal{L}_{\Rr,{\omega},\bb{0},\R^3})^{-1}\mathcal{P}_{\R^3},
\end{align*}
we can rewrite $\mathcal{K}_j(\lambda)$ as
\begin{equation}\label{K_j'}
\mathcal{K}_j(\lambda)=\left\{\begin{aligned}
&\tfrac1{\lambda}
\big(\mathcal{P}_{\R^3}-\mathcal{M}\mathcal{K}_0(\lambda)\big)
-\tfrac{{\omega}}{\lambda}\bb{e}_1\times
\mathcal{K}_0(\lambda),\qquad\qquad\qquad \quad  j=0,\\
&-\tfrac1{\lambda}\big(\mathcal{P}_{\R^3}B_{\overline{\bb{w}}}\mathcal{K}_{j-1}(\lambda)
+\mathcal{M}\mathcal{K}_j(\lambda)\big)-\tfrac{{\omega}}{\lambda}\bb{e}_1\times\mathcal{K}_j(\lambda),\quad j\geq 1.
\end{aligned}\right.
\end{equation}
Iterating \eqref{K_j'}, we obtain
\begin{equation*}
\mathcal{K}_j(\lambda)=\left\{\begin{aligned}
&\big(\tfrac1{\lambda}-\tfrac{{\omega}\bb{e}_1\times}{\lambda^2}\big)
\big(\mathcal{P}_{\R^3}-\mathcal{M}\mathcal{K}_0(\lambda)\big)
+\tfrac{{\omega}^2}{\lambda^2}[\bb{e}_1\times]^2\mathcal{K}_0(\lambda),\qquad\qquad\,\quad j=0,\\
&\tfrac{{\omega}^2[\bb{e}_1\times]^2}{\lambda^2}\mathcal{K}_j(\lambda)-\big(\tfrac1{\lambda}-\tfrac{{\omega}\bb{e}_1\times}{\lambda^2}\big)
\big(\mathcal{P}_{\R^3}B_{\overline{\bb{w}}}\mathcal{K}_{j-1}(\lambda)
+\mathcal{M}\mathcal{K}_j(\lambda)\big),\quad j\geq 1.
\end{aligned}\right.
\end{equation*}
Hence, by making use of \eqref{est2-3-1} and the mean value theorem, we deduce
\begin{equation}\label{K-0-add-1-0}\left\{\begin{aligned}
&\|\mathcal{K}_0(\lambda)\bb{f}\|_{\mathbb{L}^p(B_{9(R+2)})}\lesssim |\lambda|^{-1}\|\bb{f}\|_{\mathbb{L}^p(\R^3)},\\
&\|(\partial_{\lambda}\mathcal{K}_0)(\lambda)\bb{f}\|_{\mathbb{L}^p(B_{9(R+2)})}\lesssim |\lambda|^{-2}\|\bb{f}\|_{\mathbb{L}^p(\R^3)}+|\lambda|^{-1}\|\mathcal{M}(\partial_{\lambda}
\mathcal{K}_0)(\lambda)\bb{f}\|_{\mathbb{L}^p(B_{9(R+2)})},\\
&\|(\vartriangle_h\partial_{\lambda}\mathcal{K}_0)(\lambda)\bb{f}\|_{\mathbb{L}^p(B_{9(R+2)})}
\lesssim_{\rho,h_0}|h|^{\rho}|\lambda|^{-2}\|\bb{f}\|_{\mathbb{L}^p(\R^3)}\\
&\qquad\qquad\qquad\qquad\qquad\qquad\;\;+|\lambda|^{-1}
\|\mathcal{M}(\vartriangle_h\partial_{\lambda}\mathcal{K}_0)
(\lambda)\bb{f}\|_{\mathbb{L}^p(B_{9(R+2)})},\end{aligned}\right.\end{equation}
and for all $j\geq 1$
\begin{equation}\label{K-0-add-j-1}
\left\{\begin{aligned}
&\|\mathcal{K}_j(\lambda)\bb{f}\|_{\mathbb{L}^p(B_{9(R+2)})}\\
\lesssim&|\lambda|^{-1}E_j\|\bb{f}\|_{\mathbb{L}^p(\R^3)}+|\lambda|^{-1}\|\mathcal{P}_{\R^3}B_{\overline{\bb{w}}}\mathcal{K}_{j-1}(\lambda)\bb{f}\|_{\mathbb{L}^p(B_{9(R+2)})},\\
%%%%%%%%%%%%%
&\|(\partial_{\lambda}\mathcal{K}_j)(\lambda)\bb{f}\|_{\mathbb{L}^p(B_{9(R+2)})}\\
\lesssim& |\lambda|^{-2}E_j\|\bb{f}\|_{\mathbb{L}^p(\R^3)}+|\lambda|^{-2}\|\mathcal{P}_{\R^3}B_{\overline{\bb{w}}}\mathcal{K}_{j-1}(\lambda)\bb{f}\|_{\mathbb{L}^p(B_{9(R+2)})}\\
&+|\lambda|^{-1}\big\|\big(\mathcal{M}(\partial_{\lambda}\mathcal{K}_j)(\lambda)\bb{f},\,\mathcal{P}_{\R^3}B_{\overline{\bb{w}}}(\partial_{\lambda}\mathcal{K}_{j-1})(\lambda)\bb{f}\big)\big\|_{\mathbb{L}^p(B_{9(R+2)})},\\
%%%%%%%%%%%%%%%%%
&\|(\vartriangle_h\partial_{\lambda}\mathcal{K}_j)(\lambda)\bb{f}\|_{\mathbb{L}^p(B_{9(R+2)})}\\
\lesssim&_{\rho,h_0} |h|^{\rho}|\lambda|^{-2}F_j\|\bb{f}\|_{\mathbb{L}^p(\R^3)}+|\lambda|^{-2}\|\mathcal{P}_{\R^3}B_{\overline{\bb{w}}}(\vartriangle_h\mathcal{K}_{j-1})(\lambda)\bb{f}\|_{\mathbb{L}^p(B_{9(R+2)})}\\
&+|h|^{\rho}|\lambda|^{-2}\big\|\big(\mathcal{P}_{\R^3}B_{\overline{\bb{w}}}\mathcal{K}_{j-1}(\lambda)\bb{f}, \,\mathcal{P}_{\R^3}B_{\overline{\bb{w}}}(\partial_{\lambda}\mathcal{K}_{j-1})(\lambda)\bb{f}\big)\big\|_{\mathbb{L}^p(B_{9(R+2)})}\\
&+|\lambda|^{-1}\big\|\big(\mathcal{M}(\vartriangle_h\partial_{\lambda}\mathcal{K}_j)(\lambda)\bb{f},\,\mathcal{P}_{\R^3}B_{\overline{\bb{w}}}
(\vartriangle_h\partial_{\lambda}\mathcal{K}_{j-1})(\lambda)\bb{f}
\big)\big\|_{\mathbb{L}^p(B_{9(R+2)})}.
\end{aligned}\right.
\end{equation}

Next, we deal with the nonlocal terms invoking $B_{\overline{\bb{w}}}$ in \eqref{K-0-add-j-1}. Since $$B_{\overline{\bb{w}}}(\partial^k_{\lambda}\mathcal{K}_{j-1})(\lambda)\bb{f}, \; B_{\overline{\bb{w}}}(\vartriangle_h \partial^k_{\lambda}\mathcal{K}_{j-1})(\lambda)\bb{f}\not\in \LL^p(\R^3),$$
 here we will adopt the ``divide and conquer" argument, such as dividing $\bb{g}$ into
$$\bb{g}=\chi_{B_{9(R+2)}}\bb{g}+\chi^c_{B_{9(R+2)}}\bb{g},$$
to obtain  for every $T\in \mathcal{L}(\LL^p(\R^3))\cap \mathcal{L}(\LL^q(\R^3))$
\begin{equation}\label{claim}
\|T\bb{g}\|_{\LL^p(B_{9(R+2)})}\leq C_{p,q,R}\|T\|_{\mathcal{L}(\LL^p(\R^3))\cap \mathcal{L}(\LL^q(\R^3))}\|\bb{g}\|_{\LL^{p,q}_{R}(\R^3)},
\end{equation}
where $\LL^{p,q}_{R}(\R^3)\triangleq\LL^p(B_{9(R+2)})\cap \LL^q(B^c_{9(R+2)})$.

From \eqref{S-1} and \eqref{est2-3-1}, we have for $q_1\triangleq  \max\{p,3\}$
\begin{equation}\label{K-0-add-j-2'}
\left\{\begin{aligned}
&\|B_{\overline{\bb{w}}}\mathcal{K}_{j-1}(\lambda)\bb{f}\|_{\mathbb{L}^{p,q_1}_R(\R^3)}
+\|B_{\overline{\bb{w}}}(\partial_{\lambda}\mathcal{K}_{j-1})(\lambda)\bb{f}\|_{\mathbb{L}^{p,q_1}_R(\R^3)}\lesssim E_j\|\bb{f}\|_{\mathbb{L}^p(\R^3)},\\
&\|B_{\overline{\bb{w}}}(\vartriangle_h\partial_{\lambda}\mathcal{K}_{j-1})(\lambda)\bb{f}
\|_{\mathbb{L}^{p,q_1}_R(\R^3)}\lesssim_{\rho}|h|^{\rho} F_j\|\bb{f}\|_{\mathbb{L}^p(\R^3)}.
\end{aligned}\right.
\end{equation}
This, together with \eqref{claim}, implies
\begin{equation}\label{K-0-add-j-33}
\left\{\begin{aligned}
&\|\mathcal{P}_{\R^3}(\partial^k_{\lambda}B_{\overline{\bb{w}}}\mathcal{K}_{j-1})(\lambda)\bb{f}\|_{\LL^p(B_{9(R+2)})}
\lesssim E_j\|\bb{f}\|_{\mathbb{L}^p(\R^3)},\quad k=0,1,\\
&\|\mathcal{P}_{\R^3}B_{\overline{\bb{w}}}(\vartriangle_h\partial_{\lambda}\mathcal{K}_{j-1})(\lambda)
\bb{f}\|_{\LL^p(B_{9(R+2)})}\lesssim_{\rho}|h|^{\rho} F_j\|\bb{f}\|_{\mathbb{L}^p(\R^3)}.
\end{aligned}\right.
\end{equation}
On the other hand, in the light of \eqref{K_j'}, we know that
\begin{equation}\label{eq-B_2}
B_{2,\overline{\bb{w}}}\mathcal{K}_{j-1}(\lambda)=\left\{\begin{aligned}
&\tfrac1{\lambda}B_{2,\overline{\bb{w}}}\mathcal{P}_{\R^3}-\tfrac1{\lambda}B_{2,\overline{\bb{w}}}
\mathcal{L}_{\mathfrak{R},{\omega},\bb{0},\R^3}\mathcal{K}_0(\lambda)
+\tfrac1{\lambda}B_{2,\overline{\bb{w}}}\Delta\mathcal{K}_0(\lambda),\;\,j=1,\\
&-\tfrac1{\lambda}B_{2,\overline{\bb{w}}}\mathcal{P}_{\R^3}B_{\overline{\bb{w}}}\mathcal{K}_{j-2}(\lambda)-\tfrac1{\lambda}B_{2,\overline{\bb{w}}}\mathcal{L}_{\mathfrak{R},{\omega},\bb{0},\R^3}\mathcal{K}_{j-1}(\lambda),\;\;\,j\geq 2.
\end{aligned}\right.
\end{equation}
Set $q_2\triangleq \max\{6,p\}$. We conclude by \eqref{S-1} and \eqref{est2-3-1} that
\begin{equation} \label{K-0-add-j-5}
\left\{\begin{aligned}
&\|B_{2,\overline{\bb{w}}}(\mathcal{L}_{\mathfrak{R},{\omega},\bb{0},\R^3}+\Delta)
(\partial^k_{\lambda}\mathcal{K}_{j-1})(\lambda)\bb{f}\|_{\mathbb{L}^{p,q_2}_R(\R^3)}\lesssim E_j\|\bb{f}\|_{\mathbb{L}^p(\R^3)},\,k=0,1,\\
&\|B_{2,\overline{\bb{w}}}(\mathcal{L}_{\mathfrak{R},{\omega},\bb{0},\R^3}
+\Delta)(\vartriangle_h\partial_{\lambda}\mathcal{K}_{j-1})(\lambda)\bb{f}\|_{\mathbb{L}^{p,q_2}_R(\R^3)}
\lesssim_{\rho}|h|^{\rho}F_j\|\bb{f}\|_{\mathbb{L}^p(\R^3)}.
\end{aligned}\right.
\end{equation}
This, combining with \eqref{K-0-add-j-2'}, gives from \eqref{claim}
\begin{equation*}
\left\{\begin{aligned}
&\big\|\big(\mathcal{P}_{\R^3}B_{2,\overline{\bb{w}}}\mathcal{K}_{j-1}(\lambda)\bb{f},\,\mathcal{P}_{\R^3}B_{2,\overline{\bb{w}}}
(\partial_{\lambda}\mathcal{K}_{j-1})(\lambda)\bb{f}\big)\big\|_{\mathbb{L}^p(B_{9(R+2)})}\\
\lesssim & |\lambda|^{-1}\Big[E_j\|\bb{f}\|_{\mathbb{L}^p(\R^3)}
+\big\|\big(\mathcal{P}_{\R^3}(B_{2,\overline{\bb{w}}}\Delta\mathcal{K}_{j-1}(\lambda)\bb{f},\,
\mathcal{P}_{\R^3}B_{2,\overline{\bb{w}}}\Delta(\partial_{\lambda}\mathcal{K}_{j-1})
(\lambda)\bb{f}\big)\big\|_{\LL^p(B_{9(R+2)})}\Big],\\
%%%%%%%%%%%%%%%%%%%%%%%
&\|\mathcal{P}_{\R^3}B_{2,\overline{\bb{w}}}(\vartriangle_h\partial_{\lambda}\mathcal{K}_{j-1})(\lambda)
\bb{f}\|_{\mathbb{L}^p(B_{9(R+2)})}\\
\lesssim &_ {\rho,h_0}|h|^{\rho}|\lambda|^{-1}F_j\|\bb{f}\|_{\mathbb{L}^p(\R^3)}\\
&+|h|^{\rho}|\lambda|^{-1}\big\|\big(\mathcal{P}_{\R^3}B_{2,\overline{\bb{w}}}\Delta\mathcal{K}_{j-1}(\lambda)\bb{f},\,
\mathcal{P}_{\R^3}B_{2,\overline{\bb{w}}}\Delta(\partial_{\lambda}\mathcal{K}_{j-1})(\lambda)\bb{f}\big)\big\|_{\LL^p(B_{9(R+2)})}\\
&+|\lambda|^{-1}\big\|\big(\mathcal{P}_{\R^3}B_{2,\overline{\bb{w}}}\Delta(\vartriangle_h\mathcal{K}_{j-1})(\lambda)\bb{f},\,
\mathcal{P}_{\R^3}B_{2,\overline{\bb{w}}}\Delta(\vartriangle_h\partial_{\lambda}\mathcal{K}_{j-1})(\lambda)\bb{f}\big)
\big\|_{\LL^p(B_{9(R+2)})}.
\end{aligned}\right.
\end{equation*}

Plugging the above set of inequalities and \eqref{K-0-add-j-33} into \eqref{K-0-add-j-1}, we conclude for $j\geq 1$
 \begin{equation}\label{K-0-add-j-4}
 \|\mathcal{K}_j(\lambda)\bb{f}\|_{\mathbb{L}^p(B_{9(R+2)})}\lesssim |\lambda|^{-1}E_j\|\bb{f}\|_{\LL^p(\R^3)},
 \end{equation}
 and
\begin{equation}\label{K-0-add-j-6}
\left\{\begin{aligned}
&\|(\partial_{\lambda}\mathcal{K}_j)(\lambda)\bb{f}\|_{\mathbb{L}^p(B_{9(R+2)})}\\
\lesssim&|\lambda|^{-2}E_j\|\bb{f}\|_{\LL^p(\R^3)}+|\lambda|^{-2}\sum_{k=0,1}\|\mathcal{P}_{\R^3}B_{2,\overline{\bb{w}}}\Delta(\partial^k_{\lambda}\mathcal{K}_{j-1})(\lambda)\bb{f}\|_{\mathbb{L}^p(B_{9(R+2)})}\\
&+|\lambda|^{-1}\big\|\big(\mathcal{M}(\partial_{\lambda}\mathcal{K}_j)(\lambda)\bb{f},\,\mathcal{P}_{\R^3}B_{1,\overline{\bb{w}}}(\partial_{\lambda}\mathcal{K}_{j-1})(\lambda)\bb{f}\big)\big\|_{\mathbb{L}^p(B_{9(R+2)})},\\
&\|(\vartriangle_h\partial_{\lambda}\mathcal{K}_j)(\lambda)\bb{f}\|_{\mathbb{L}^p(B_{9(R+2)})}\\
\lesssim&_{\rho,h_0}|\lambda|^{-2}\Big[|h|^{\rho}F_j\|\bb{f}\|_{\LL^p(\R^3)}+\sum_{k=0,1}\Big(|h|^{\rho}\|\mathcal{P}_{\R^3}B_{2,\overline{\bb{w}}}\Delta(\partial^k_{\lambda}\mathcal{K}_{j-1})(\lambda)\bb{f}\|_{\mathbb{L}^p(B_{9(R+2)})}\\
&+\|\mathcal{P}_{\R^3}B_{2,\overline{\bb{w}}}\Delta(\vartriangle_h\partial^k_{\lambda}\mathcal{K}_{j-1})(\lambda)\bb{f}\|_{\mathbb{L}^p(B_{9(R+2)})}\Big)\Big]\\
&+|\lambda|^{-1}\big\|\big(\mathcal{M}(\vartriangle_h\partial_{\lambda}\mathcal{K}_j)(\lambda)\bb{f},\,\mathcal{P}_{\R^3}B_{1,\overline{\bb{w}}}(\vartriangle_h\partial_{\lambda}\mathcal{K}_{j-1})(\lambda)\bb{f})\big)\big\|_{\mathbb{L}^p(B_{9(R+2)})}.
\end{aligned}\right.
\end{equation}
In view of \eqref{K-0-add-1-0} and \eqref{K-0-add-j-4}-\eqref{K-0-add-j-6}, the proof of \eqref{est-K_j} for $|\beta|=0$ is reduced to the case $|\beta|\neq 0$. For details, see Step 3 below.
\vskip 0.15cm

\noindent\textbf{Case 2: $|\beta|=1,2$}.\quad
We know
\[(\lambda I+\mathcal{L}_{\mathfrak{R},{\omega},\bb{0},\R^3})^{-1}\mathcal{P}_{\R^3}
=\mathcal{A}^N_{1}(\lambda)+\mathcal{A}^N_{2}(\lambda),\quad N\geq 1,\]
where $\mathcal{A}^N_{i}(\lambda)\,(i=1,2)$ is defined in \eqref{eq-A}. By a simple calculation, we have
\begin{align*}
&\mathcal{A}^N_{1}(\lambda)\bb{g}=\frac{1}{(2\pi)^3}\sum^{N-1}_{j=0}\int_{\R^3}\frac1{(\lambda+|\xi|^2-i\mathfrak{R}\xi_1)^{j+1}}
\partial^j_t\Big(\mathcal{O}^{T}({\omega}t)\mathbb{P}(\xi)\hat{\bb{g}}(\xi)e^{i\mathcal{O}({\omega}t)x\cdot\xi}\Big)\Big|_{t=0}\,\mathrm{d}\xi,\\
&\mathcal{A}^N_{2}(\lambda)\bb{g}=\frac{1}{(2\pi)^3}\int^{\infty}_0\int_{\R^3}\frac{e^{-(\lambda+|\xi|^2-i\mathfrak{R}\xi_1)t}}{(\lambda+|\xi|^2-i\mathfrak{R}\xi_1)^N}\partial^N_t\Big(\mathcal{O}^{T}({\omega}t)\mathbb{P}(\xi)\hat{\bb{g}}(\xi)e^{i\mathcal{O}({\omega}t)x\cdot\xi}\Big)\,\mathrm{d}\xi\mathrm{d}t.
\end{align*}
By Leibniz rule, we have for $\ell\geq 1$,
\begin{align*}
\partial^{\ell}_t\Big(\mathcal{O}^{T}({\omega}t)\mathbb{P}(\xi)\hat{\bb{g}}
(\xi)e^{i\mathcal{O}({\omega}t)x\cdot\xi}\Big)= \sum^{\ell}_{k=0}\left(\begin{aligned} k\\ \ell \end{aligned}\right)(-{\omega})^{k}\mathcal{O}^{T}({\omega}t)[\bb{e}_1\times]^{k}\mathbb{P}(\xi)\hat{\bb{g}}(\xi)\partial^{\ell-k}_t(e^{i\mathcal{O}({\omega}t)x\cdot\xi}).
\end{align*}
In addition, we get for $k\geq 1$,
\begin{align*}
\partial^{k}_t(e^{i\mathcal{O}({\omega}t)x\cdot\xi})=&{\omega}^k e^{i\mathcal{O}({\omega}t)x\cdot\xi}\Big(\big(i(\bb{e}_1\times\mathcal{O}({\omega}t)
 {x})\cdot\xi\big)^{k}\\
&+\sum\limits_{\scriptstyle k_1+\cdots +k_n=k,n\geq 2 \hfill\atop\scriptstyle k_{n-1}\geq \ldots\geq k_1\geq 2\hfill} d_{k_1,\ldots,k_{n}}\big(i([\bb{e}_1\times]^{k_1}\mathcal{O}({\omega}t) {x})\cdot\xi\big)\\
&\quad\quad\quad\quad\quad\quad\quad\quad\cdots\big(i([\bb{e}_1\times]^{k_{n-1}}\mathcal{O}({\omega}t) {x})\cdot\xi\big)\big(i([\bb{e}_1\times]
\mathcal{O}({\omega}t) {x})\cdot\xi\big)^{k_n}.
\end{align*}
Hence, setting
\begin{align*}
&G_j(\lambda,\xi)=(\lambda+|\xi|^2-i\mathfrak{R}\xi_1)^{-j}\mathbb{P}(\xi),\quad j\geq 1,\\
&H(\lambda,\xi)=e^{-\lambda t}\mathcal{O}^{T}(at)\mathcal{F}\big(T^G_{\Rr,0,0}(t)\mathcal{P}_{\R^3}\bb{g})(\mathcal{O}(at)x)\big)(\xi),
\end{align*}
we deduce that
\begin{align}
\mathcal{A}^N_{1}(\lambda)\bb{g}=&\sum^{N-1}_{j=0}\sum^j_{\ell=0}\sum_{|\alpha|=|\gamma|=\ell}\frac {i^{\ell}|{\omega}|^{j}}{(2\pi)^3}c^j_{\alpha,\gamma}\int_{\R^3}x^{\alpha}\xi^{\gamma}G_{j+1}(\lambda,\xi)\hat{\bb{g}}(\xi)e^{i {x}\cdot\bb{\xi}}\,\mathrm{d}\xi \nonumber\\
=&\sum^{N-1}_{j=0}\sum^j_{\ell=0}\sum_{|\alpha|=|\gamma|=\ell}|{\omega}|^{j}c^j_{\alpha,\gamma}x^{\alpha}\partial^{\gamma}_{x}\frac{1}{(\lambda-\Delta-\mathfrak{R}\partial_1)^{\ell+1}}\mathcal{P}_{\R^3}\bb{g}\label{A^N_1}
\end{align}
and
\begin{align*}
\mathcal{A}^N_{2}(\lambda)\bb{g}=&\sum^N_{\ell=0}\sum_{|\alpha|=|\gamma|=\ell}\frac {i^{m}|{\omega}|^{N}}{(2\pi)^3}c^N_{\alpha,\gamma}\int^{\infty}_0\int_{\R^3}x^{\alpha}\xi^{\gamma}G_N(\lambda,\xi)H(\lambda,\xi)e^{ix\cdot\xi}\,\mathrm{d}\xi\mathrm{d}t\\
=&|{\omega}|^N\sum^N_{\ell=0}\sum_{|\alpha|=|\gamma|=\ell}c^N_{\alpha,\gamma}x^{\alpha}\partial^{\gamma}_x\frac{1}{(\lambda-\Delta-\mathfrak{R}\partial_1)^N}(\lambda I+\mathcal{L}_{\mathfrak{R},{\omega},\bb{0},\R^3})^{-1}\mathcal{P}_{\R^3}\bb{g}.
\end{align*}
So, we decompose
\begin{align}
\mathcal{K}_{j}(\lambda)=\mathcal{K}^N_{1,j}(\lambda)+\mathcal{K}^N_{2,j}(\lambda),\quad  j\geq 0,\;N\geq 1,\label{K_j2}
\end{align}
with
\begin{align*}
&\mathcal{K}^N_{1,j}(\lambda)\triangleq \mathcal{A}^N_{1}(\lambda)(-B_{\overline{\bb{w}}}(\lambda+\mathcal{L}_{\mathfrak{R},{\omega},\bb{0},\R^3})^{-1}\mathcal{P}_{\R^3})^j ,\\
&\mathcal{K}^N_{2,j}(\lambda)\triangleq \mathcal{A}^N_{2}(\lambda)(-B_{\overline{\bb{w}}}(\lambda +\mathcal{L}_{\mathfrak{R},{\omega},\bb{0},\R^3})^{-1}\mathcal{P}_{\R^3})^j \\
&\qquad\quad=|{\omega}|^N\sum_{\ell\leq N}\sum_{|\alpha|=|\gamma|=\ell}c^N_{\alpha,\gamma}x^{\alpha}
\partial^{\gamma}_x
\frac{1}{(\lambda-\Delta-\mathfrak{R}\partial_1)^N}\mathcal{K}_{j}(\lambda).
\end{align*}	

Let us begin with $\mathcal{K}^N_{2,j}(\lambda)$. Since $|\beta|=1,2$, we observe that for all $j\geq 0$
\begin{align*}
\partial^{\beta}_x\mathcal{K}^N_{2,j}(\lambda)\bb{f}=&|{\omega}|^N\sum^N_{\ell= 0}
\sum_{|\alpha|=\ell,|\gamma|=\ell+|\beta|}c^N_{\alpha,\gamma} x^{\alpha}\partial^{\gamma}_x(\lambda-\Delta-\mathfrak{R}\partial_1)^{-N}\mathcal{K}_{j}(\lambda)\bb{f}\\
=&|{\omega}|^N\sum_{|\alpha|=0,|\gamma|=|\beta|}c^N_{\alpha,\gamma} \partial^{\gamma}_x\Lambda^{-\max\{1,|\beta|-2\varrho\}}
(\lambda-\Delta-\mathfrak{R}\partial_1)^{-N}\Lambda^{\max\{1,|\beta|-2\varrho\}}\mathcal{K}_{j}(\lambda)\bb{f} \\
&+|{\omega}|^N\sum^N_{\ell=1}\sum_{|\alpha|=\ell,|\gamma|=\ell+|\beta|}c^N_{\alpha,\gamma} x^{\alpha}\partial^{\gamma}_x\Lambda^{-2+2\varrho}
(\lambda-\Delta-\mathfrak{R}\partial_1)^{-N}\Lambda^{2-2\varrho}\mathcal{K}_{j}(\lambda)\bb{f}     \nonumber\\
\triangleq & \mathcal{T}^{N}_{1,|\beta|,\varrho}(\lambda)\Lambda^{\max\{1,|\beta|-2\varrho\}}\mathcal{K}_{j}(\lambda)\bb{f}+\mathcal{T}^N_{2,|\beta|,\varrho}(\lambda)\Lambda^{2-2\varrho}\mathcal{K}_{j}(\lambda)\bb{f}.
\end{align*}
It is well known that
\begin{equation}\label{K-j-add-2-2}
\|\partial^{\alpha}_{x}\Lambda^{-\gamma}\partial^k_{\lambda}(\lambda-\Delta-\mathfrak{R}\partial_1)^{-\ell}\|_{\mathcal{L}(\mathbb{L}^p(\R^3))}\leq C_{\mathfrak{R}^*}|\lambda|^{-k-\ell+\frac{|\alpha|-\gamma}2},\quad 0\leq |\alpha|-\gamma\leq 2\ell.
\end{equation}
This estimate gives  for  $k\leq 2$
\begin{equation*}
\left\{\begin{aligned}
&\|(\partial^k_{\lambda}\mathcal{T}^{N}_{1,|\beta|,\varrho})(\lambda)\|_{\mathcal{L}(\LL^p(\R^3))}
\leq C_{N,\Rr^*,{\omega}^*}|\lambda|^{-k-N+\varrho},\\
&\|(1+|x|)^{-N}(\partial^k_{\lambda}\mathcal{T}^N_{2,|\beta|,\varrho})(\lambda)\|_{\mathcal{L}(\LL^p(\R^3))}\leq C_{N,\Rr^*,{\omega}^*}|\lambda|^{-k-\frac{N+2-|\beta|}2+\varrho}.
\end{aligned}\right.
\end{equation*}
Moreover,  form \eqref{est2-3-1}, we observe for  $s\in [1,2)$
\begin{equation}\label{K-j-add-2-4}
\left\{\begin{aligned}
&\|\Lambda^s\mathcal{K}_{j}(\lambda)\bb{f}\|_{\mathbb{L}^{p,q_3}_R(\R^3)}
+\|\Lambda^s(\partial^k_{\lambda}\mathcal{K}_{j})(\lambda)\bb{f}\|_{\mathbb{L}^{p,q_3}_R(\R^3)}\lesssim_{s} E_j\|\bb{f}\|_{\LL^p(\R^3)},\\
&\|\Lambda^s(\vartriangle_h\partial_{\lambda}\mathcal{K}_{j})(\lambda)\bb{f}\|_{\mathbb{L}^{p,q_4}_R(\R^3)}
\lesssim_{s,\rho}|h|^{\rho}F_j\|\bb{f}\|_{\LL^p(\R^3)}
\end{aligned}\right.
\end{equation}
with $q_3\triangleq \max\{p,6\}$ and $q_4\triangleq \max(p, \tfrac3{(1/2)-\rho})$. Thus by \eqref{claim} we get   for $j\geq 0$,
\begin{equation}\label{K-j-add-2-6}
\left\{\begin{aligned}
&\|T((1+|x|)^{-N}\partial^{\beta}_x(\partial^k_{\lambda}\mathcal{K}^N_{2,j})(\lambda)\bb{f})\|_{\mathbb{L}^p(B_{9(R+2)})}\\
\lesssim&_{N,\varrho}|\lambda|^{-(N+2-|\beta|)/2+\varrho}\|T\|_{\mathcal{L}(\mathbb{L}^p(\R^3))\cap \mathcal{L}(\mathbb{L}^{q_3}(\R^3))}E_j\|\bb{f}\|_{\LL^p(\R^3)},\quad k=0,1,\\
&\|T((1+|x|)^{-N}\partial^{\beta}_x(\vartriangle_h\partial_{\lambda}\mathcal{K}^N_{2,j})(\lambda)\bb{f})\|_{\mathbb{L}^p(B_{9(R+2)})}\\
\lesssim&_{N,\rho,\varrho}|h|^{\rho}|\lambda|^{-(N+2-\beta)/2+\varrho}
\|T\|_{\mathcal{L}(\mathbb{L}^p(\R^3))\cap\mathcal{L}(\mathbb{L}^{q_4}(\R^3))}F_j\|\bb{f}\|_{\LL^p(\R^3)}.
\end{aligned}\right.
\end{equation}

Now we consider $\mathcal{K}^N_{1,j}(\lambda)$. In view of \eqref{A^N_1} and \eqref{K-j-add-2-2}, we deduce for $k\in \mathds{N}$
\begin{equation}\label{K-j-add-2-8'}
\left\{\begin{aligned}
&\|T(1+|x|)^{-N+1}\partial^{\beta}_{x}(\partial^k_{\lambda}\mathcal{A}^N_{1})(\lambda)\|_{\mathcal L(\mathbb{L}^p(\R^3))}
\leq C_{N,\mathfrak{R}^*,{\omega}^*}|\lambda|^{-k-1+(|\beta|/2)}\|T\|_{\mathcal{L}(\mathbb{L}^p(\R^3))}, \\
&\|T(1+|x|)^{-N+1}\partial^{\beta}_{x}(\vartriangle_h\partial^k_{\lambda}
\mathcal{A}^N_{1})(\lambda)\|_{\mathcal L(\mathbb{L}^p(\R^3))}
\leq \tfrac{C_{N,h_0,\mathfrak{R}^*,{\omega}^*}|h|^{\rho}}{|\lambda|^{k+2-(|\beta|/2)}}\|T\|_{\mathcal{L}(\mathbb{L}^p(\R^3))}.
\end{aligned}\right.
\end{equation}
When $j\geq 1$, since $\mathcal{K}^N_{1,j}(\lambda)= -\mathcal{A}^N_{1}(\lambda)B_{\overline{\bb{w}}}\mathcal{K}_{j-1}(\lambda)$,  we have
by \eqref{claim}-\eqref{K-0-add-j-2'} and \eqref{K-j-add-2-8'}
\begin{equation}\label{K-j-add-2-9-1}
\left\{\begin{aligned}
&\|T((1+|x|)^{-N+1}\partial^{\beta}_{x}\mathcal{K}^N_{1,j}(\lambda)\bb{f})
\|_{\mathbb{L}^p(B_{9(R+2)})}\\
\lesssim&_{N}|\lambda|^{-1+(|\beta|/2)}
\|T\|_{\mathcal{L}(\mathbb{L}^p(\R^3))\cap \mathcal{L}(\mathbb{L}^{q_1}(\R^3))}
E_j\|\bb{f}\|_{\mathbb{L}^p(\R^3)},\\
&\|T((1+|x|)^{-N+1}\partial^{\beta}_{x}(\partial_{\lambda}\mathcal{K}^N_{1,j})(\lambda)\bb{f})\|_{\mathbb{L}^p(B_{9(R+2)})}\\
\lesssim&_{N}|\lambda|^{-2+(|\beta|/2)}\|T\|_{\mathcal{L}(\mathbb{L}^p(\R^3))\cap \mathcal{L}(\mathbb{L}^{q_1}(\R^3))}E_j\|\bb{f}\|_{\mathbb{L}^p(\R^3)} \\
&+\|T((1+|x|)^{-N+1}\partial^{\beta}_x\mathcal{A}^N_{1}(\lambda)B_{\overline{\bb{w}}}
(\partial_{\lambda}\mathcal{K}_{j-1})(\lambda)\bb{f})\|_{\mathbb{L}^p(B_{9(R+2)})},\\
&\|T((1+|x|)^{-N+1}\partial^{\beta}_x(\vartriangle_h\partial_{\lambda}\mathcal{K}^N_{1,j})
(\lambda)\bb{f})\|_{\mathbb{L}^p(B_{9(R+2)})}\\
\lesssim&_{N,\rho,h_0}|h|^{\rho}|\lambda|^{-2+(|\beta|/2)}\|T\|_{\mathcal{L}
(\mathbb{L}^p(\R^3))\cap\mathcal{L}(\mathbb{L}^{q_1}(\R^3))}F_j\|\bb{f}\|_{\mathbb{L}^p(\R^3)}\\
&+\|T((1+|x|)^{-N+1}\partial^{\beta}_x\mathcal{A}^N_{1}(\lambda)B_{\overline{\bb{w}}}(\vartriangle_h\partial_{\lambda}
\mathcal{K}_{j-1})(\lambda)\bb{f})\|_{\mathbb{L}^p(B_{9(R+2)})}.
\end{aligned}\right.
\end{equation}
Meanwhile, we have
\begin{equation}\label{K-j-add-2-10}
\left\{\begin{aligned}
&\|T((1+|x|)^{-N+1}\partial^{\beta}_x\mathcal{A}^N_{1}(\lambda)
B_{\overline{\bb{w}}}(\partial_{\lambda}\mathcal{K}_{j-1})(\lambda)\bb{f})\|_{\mathbb{L}^p(B_{9(R+2)})}\\
\lesssim&_{N}|\lambda|^{-1+(|\beta|/2)}\|T\|_{\mathcal{L}
(\mathbb{L}^p(\R^3))\cap\mathcal{L}(\mathbb{L}^{q_1}(\R^3))}E_j\|\bb{f}\|_{\mathbb{L}^p(\R^3)}, \\
&\|T((1+|x|)^{-N+1}\partial^{\beta}_x\mathcal{A}^N_{1}(\lambda)B_{\overline{\bb{w}}}(\vartriangle_h\partial_{\lambda}
\mathcal{K}_{j-1})(\lambda)\bb{f})\|_{\mathbb{L}^p(B_{9(R+2)})}\\
\lesssim&_{N,\rho}|h|^{\rho}|\lambda|^{-1+(|\beta|/2)}\|T\|_{\mathcal{L}(\mathbb{L}^p(\R^3))\cap \mathcal{L}(\mathbb{L}^{q_1}(\R^3))}F_j\|\bb{f}\|_{\mathbb{L}^p(\R^3)}.
\end{aligned}\right.
\end{equation}
Further,  in the light of \eqref{eq-B_2}, we have by \eqref{claim}-\eqref{K-0-add-j-2'}, \eqref{K-0-add-j-5} and \eqref{K-j-add-2-8'}
\begin{align}\nonumber
&\|T((1+|x|)^{-N+1}\partial^{\beta}_x\mathcal{A}^N_{1}(\lambda)
B_{2,\overline{\bb{w}}}(\partial_{\lambda}\mathcal{K}_{j-1})(\lambda)\bb{f})\|_{\mathbb{L}^p(B_{9(R+2)})}\\
\lesssim&_{N}|\lambda|^{-2+(|\beta|/2)}\|T\|_{\mathcal{L}(\mathbb{L}^p(\R^3))\cap
\mathcal{L}(\mathbb{L}^{q_1}(\R^3))\cap \mathcal{L}(\mathbb{L}^{q_2}(\R^3))}E_j\|\bb{f}\|_{\mathbb{L}^p(\R^3)}\nonumber\\
&+|\lambda|^{-1}\sum_{k=0,1}\|T((1+|x|)^{-N+1}\partial^{\beta}_{x}\mathcal{A}^{N}_{1}(\lambda)
B_{2,\overline{\bb{w}}}\Delta(\partial^k_{\lambda}\mathcal{K}_{j-1})(\lambda)\bb{f})\|_{\mathbb{L}^p(B_{9(R+2)})},\label{K-j-add-2-11-0}\\
&\|T((1+|x|)^{-N+1}\partial^{\beta}_x\mathcal{A}^N_{1}(\lambda)B_{2,\overline{\bb{w}}}(\vartriangle_h\partial_{\lambda}
\mathcal{K}_{j-1})(\lambda)\bb{f})\|_{\mathbb{L}^p(B_{9(R+2)})}\nonumber\\
\lesssim&_{N,\rho,h_0}|h|^{\rho}|\lambda|^{-2+(|\beta|/2)}\|T\|_{\mathcal{L}(\mathbb{L}^p(\R^3))
\cap\mathcal{L}(\mathbb{L}^{q_1}(\R^3))\cap \mathcal{L}(\mathbb{L}^{q_2}(\R^3))}F_j\|\bb{f}\|_{\mathbb{L}^p(\R^3)}\nonumber\\
&+\sum_{k=0,1}\Big[|h|^{\rho}|\lambda|^{-2}\|T((1+|x|)^{-N+1}\partial^{\beta}_{x}\mathcal{A}^{N}_{1}(\lambda)
B_{2,\overline{\bb{w}}}\Delta(\partial^k_{\lambda}\mathcal{K}_{j-1})(\lambda)\bb{f})\|_{\mathbb{L}^p(B_{9(R+2)})}\nonumber\\
&+|\lambda|^{-1}\|T((1+|x|)^{-N+1}\partial^{\beta}_{x}\mathcal{A}^{N}_{1}(\lambda)
B_{2,\overline{\bb{w}}}\Delta(\vartriangle_h\partial^k_{\lambda}
\mathcal{K}_{j-1})(\lambda)\bb{f})\|_{\mathbb{L}^p(B_{9(R+2)})}\Big].\label{K-j-add-2-11}
\end{align}

Based on the above analysis in case $|\beta|=1,2$, we know that \eqref{K-j-add-2-6}-\eqref{K-j-add-2-8'} give   the required estimates of $\partial^{\beta}_x\mathcal{K}_0(\lambda)$ and  $\partial^{\beta}_x\mathcal{K}_{2, j}(\lambda)$, and \eqref{K-j-add-2-11} helps us to  reduce the estimate of $\partial^{\beta}_x\mathcal{K}_{j}(\lambda)$ with $j\geq 1$ to those of $\mathcal{K}_{j-1}(\lambda)$.
Hence, we explore a ``tree self-similar" iterative mechanism on account of \eqref{K_j2}. Roughly speaking,
 for $N\ge3$,
 \begin{align*}
\mathcal{K}_{j}(\lambda)\; =&\;\;\mathcal{K}^N_{1,j}(\lambda)+\mathcal{K}^N_{2,j}(\lambda)\;\quad\qquad \qquad \qquad \qquad
(\text{stop if}\;  j=0)\\
\Downarrow\; &\;\;\mathcal{K}_{j-1}(\lambda)+\mathcal{K}^N_{2,j}(\lambda) \\
\; =&\;\;\mathcal{K}^1_{1,j-1}(\lambda)+\mathcal{K}^1_{2,j-1}(\lambda)+\mathcal{K}^N_{2,j}(\lambda) \;\quad\qquad  (\text{stop if}\;  j=1)\\
\Downarrow\; &\;\; \mathcal{K}_{j-2}(\lambda)+\mathcal{K}^1_{2,j-1}(\lambda)+\mathcal{K}^N_{2,j}(\lambda) \\
=&\mathcal{K}^1_{1,j-2}(\lambda)
+\mathcal{K}^1_{2,j-2}(\lambda)+\mathcal{K}^1_{2,j-1}(\lambda)+\mathcal{K}^N_{2,j}(\lambda).
\end{align*}
This iterative progress will help us to obtain the required estimates of  $\partial^{\beta}_x\mathcal{K}_{j}(\lambda)$.

\vskip 0.2cm
\textbf{Step 2. Proof of \eqref{est-K_j} for $1\leq |\beta|\leq 2$.}
	
First, we adopt the decomposition \eqref{K_j2} with $N=3$. When $j=0$, as a consequence of  \eqref{K-j-add-2-6} with $T=\chi_{B_9(R+2)}(1+|x|)^{3}$ and \eqref{K-j-add-2-8'} with $T=\chi_{B_9(R+2)}(1+|x|)^{2}$,  we deduce \eqref{est-K_j} for $|\beta|=1,2$.
When $j\geq 1$, by  \eqref{K-j-add-2-6} with $T=\chi_{B_9(R+2)}(1+|x|)^{3}$, we obtain
\begin{equation}\label{K-j-add-2-13}
\left\{\begin{aligned}
&\|\partial^{\beta}_x\mathcal{K}^{3}_{2,j}(\lambda)\bb{f}\|_{\mathbb{L}^p(B_{9(R+2)})}
\lesssim|\lambda|^{-2+(|\beta|/2)}E_j\|\bb{f}\|_{\mathbb{L}^p(\R^3)},\\
&\|\partial^{\beta}_x(\partial_{\lambda}\mathcal{K}^{3}_{2,j})(\lambda)\bb{f}\|_{\mathbb{L}^p(B_{9(R+2)})}
\lesssim|\lambda|^{-2+(|\beta|/2)}E_j\|\bb{f}\|_{\mathbb{L}^p(\R^3)},\\
&\|\partial^{\beta}_x(\vartriangle_h\partial_{\lambda}\mathcal{K}^{3}_{2,j})(\lambda)\bb{f})\|_{\mathbb{L}^p(B_{9(R+2)})}
\lesssim_{\rho}|h|^{\rho}|\lambda|^{-2+(|\beta|/2)}F_j\|\bb{f}\|_{\mathbb{L}^p(\R^3)}.
\end{aligned}\right.
\end{equation}
 In addition,  utilizing \eqref{K-j-add-2-9-1} and \eqref{K-j-add-2-11} with $T=\chi_{B_9(R+2)}(1+|x|)^{2}$, we  deduce
\begin{equation}\label{K-j-add-2-14-0}
\|\partial^{\beta}_{x}\mathcal{K}^{3}_{1,j}(\lambda)\bb{f}\|_{\mathbb{L}^p(B_{9(R+2)})}
\lesssim|\lambda|^{-1+(|\beta|/2)}E_j\|\bb{f}\|_{\mathbb{L}^p(\R^3)},
\end{equation}
and
\begin{equation}\label{K-j-add-2-14}
\left\{\begin{aligned}
&\|\partial^{\beta}_{x}(\partial_{\lambda}\mathcal{K}^{3}_{1,j})
(\lambda)\bb{f}\|_{\mathbb{L}^p(B_{9(R+2)})}\\
\lesssim&|\lambda|^{-2+(|\beta|/2)}E_j\|\bb{f}\|_{\mathbb{L}^p(\R^3)} +\|\partial^{\beta}_x\mathcal{A}^{3}_{1}(\lambda)B_{1,\overline{\bb{w}}}
(\partial_{\lambda}\mathcal{K}_{j-1})(\lambda)\bb{f}\|_{\mathbb{L}^p(B_{9(R+2)})}\\
&+|\lambda|^{-1}\sum_{k=0,1}\|\partial^{\beta}_{x}\mathcal{A}^{3}_{1}(\lambda)B_{2,\overline{\bb{w}}}
\Delta(\partial^k_{\lambda}\mathcal{K}_{j-1})(\lambda)\bb{f}\|_{\mathbb{L}^p(B_{9(R+2)})},\\
%%%%%%%%%%%%%%%%%%%
&\|\partial^{\beta}_x(\vartriangle_h\partial_{\lambda}
\mathcal{K}^{3}_{1,j})(\lambda)\bb{f}\|_{\mathbb{L}^p(B_{9(R+2)})}\\
\lesssim&_{\rho}|h|^{\rho}|\lambda|^{-2+(|\beta|/2)}F_j\|\bb{f}\|_{\mathbb{L}^p(\R^3)}
 +\|\partial^{\beta}_x\mathcal{A}^{3}_{1}(\lambda)B_{1,\overline{\bb{w}}}(\vartriangle_h\partial_{\lambda}
\mathcal{K}_{j-1})(\lambda)\bb{f}\|_{\mathbb{L}^p(B_{9(R+2)})}\\
&+|\lambda|^{-1}\sum_{k=0,1}\Big[|h|^{\rho}\|\partial^{\beta}_{x}\mathcal{A}^{3}_{1}(\lambda)
B_{2,\overline{\bb{w}}}\Delta(\partial^k_{\lambda}\mathcal{K}_{j-1})(\lambda)\bb{f}\|_{\mathbb{L}^p(B_{9(R+2)})}\\
&+\|\partial^{\beta}_{x}\mathcal{A}^{3}_{1}(\lambda)
B_{2,\overline{\bb{w}}}\Delta(\vartriangle_h\partial^k_{\lambda}
\mathcal{K}_{j-1})(\lambda)\bb{f}\|_{\mathbb{L}^p(B_{9(R+2)})}\Big].
\end{aligned}\right.
\end{equation}

Next, we use the decomposition \eqref{K_j2} with $N=1$ for $\mathcal{K}_{j-1}(\lambda)$
to proceed the estimate \eqref{K-j-add-2-14}. In the light of \eqref{S-1} and \eqref{K-j-add-2-8'}  with $T=\chi_{B_9(R+2)}(1+|x|)^{2}$, we know
\begin{align*}
&\big\|\big(\chi_{B_9(R+2)}\partial^{\beta}_x\mathcal{A}^{3}_{1}(\lambda)\nabla\overline{\bb{w}}(1+|x|),\,\chi_{B_9(R+2)}\partial^{\beta}_x\mathcal{A}^{3}_{1}(\lambda)\overline{\bb{w}}(1+|x|)\big)\big\|_{\mathcal{L}(\mathbb{L}^p(\R^3))}\\
\lesssim&C_{R,\mathfrak{R}^*,\omega^*}\normmm{\bb{w}}_{\varepsilon,\Omega}|\lambda|^{-1+(|\beta|/2)}.
\end{align*}
Hence, by making use of
$$\left\{\begin{aligned}
&\eqref{K-j-add-2-6}, \;\;\text{with  } T=\chi_{B_9(R+2)}\partial^{\beta}_x\mathcal{A}^{3}_{1}(\lambda)\nabla\overline{\bb{w}}(1+|x|),\\
&\eqref{K-j-add-2-8'}{\rm -}\eqref{K-j-add-2-10},\;\; \text{with } T=\chi_{B_{9(R+2)}}\partial^{\beta}_{x}\mathcal{A}^{3}_{1}(\lambda)\nabla\overline{\bb{w}},
\end{aligned}\right.$$
 we deducefor $j\geq 1$
\begin{equation*}
\left\{\begin{aligned}
&\|\partial^{\beta}_{x}\mathcal{A}^{3}_{1}(\lambda)B_{2,\overline{\bb{w}}}
\Delta\mathcal{K}_{j-1}(\lambda)\bb{f}\|_{\mathbb{L}^p(B_{9(R+2)})}
\lesssim |\lambda|^{-1+(|\beta|/2)}E_j\|\bb{f}\|_{\mathbb{L}^p(\R^3)},\\
&\|\partial^{\beta}_{x}\mathcal{A}^{3}_{1}(\lambda)B_{2,\overline{\bb{w}}}
\Delta(\partial_{\lambda}\mathcal{K}_{j-1})(\lambda)\bb{f}\|_{\mathbb{L}^p(B_{9(R+2)})}
\lesssim|\lambda|^{-1+(|\beta|/2)}E_j\|\bb{f}\|_{\mathbb{L}^p(\R^3)},\\
&\|\partial^{\beta}_{x}\mathcal{A}^{3}_{1}(\lambda)B_{2,\overline{\bb{w}}}
\Delta(\vartriangle_h\partial_{\lambda}\mathcal{K}_{j-1})(\lambda)\bb{f}\|_{\mathbb{L}^p(B_{9(R+2)})}
\lesssim_{\rho}|h|^{\rho}|\lambda|^{-1+(|\beta|/2)}F_j\|\bb{f}\|_{\mathbb{L}^p(\R^3)}.
\end{aligned}\right.
\end{equation*}
Moreover, with the help of
$$\left\{\begin{aligned}
&\eqref{K-j-add-2-6}, \;\;\text{with } T=\chi_{B_9(R+2)}\partial^{\beta}_x\mathcal{A}^{3}_{1}(\lambda)\overline{\bb{w}}(1+|x|),\\
&\eqref{K-j-add-2-8'}{\rm -}\eqref{K-j-add-2-9-1},\;\eqref{K-j-add-2-11-0}{\rm -}\eqref{K-j-add-2-11}, \;\;\text{with } T=\chi_{B_9(R+2)}\partial^{\beta}_x\mathcal{A}^{3}_{1}(\lambda)\overline{\bb{w}},
\end{aligned}\right.$$
we calculate
\begin{equation*}
\left\{\begin{aligned}
&\|\partial^{\beta}_x\mathcal{A}^{3}_{1}(\lambda)B_{1,\overline{\bb{w}}}
(\partial_{\lambda}\mathcal{K}_0)(\lambda)
\bb{f}\|_{\mathbb{L}^p(B_{9(R+2)})}\lesssim_{\varrho}
|\lambda|^{-2+(|\beta|/2)+\varrho}E_1\|\bb{f}\|_{\mathbb{L}^p(\R^3)},\\
&\|\partial^{\beta}_x\mathcal{A}^{3}_{1}(\lambda) B_{1,\overline{\bb{w}}}(\vartriangle_h\partial_{\lambda}\mathcal{K}_{0})(\lambda)
\bb{f}\|_{\mathbb{L}^p(B_{9(R+2)})}\lesssim_{\rho,h_0,\varrho}|h|^{\rho}|\lambda|^{-2+(|\beta|/2)+\varrho}
F_1\|\bb{f}\|_{\mathbb{L}^p(\R^3)},
\end{aligned}\right.
\end{equation*}
and for all $j\geq 2$,
\begin{equation*}
\left\{\begin{aligned}
&\|\partial^{\beta}_x\mathcal{A}^{3}_{1}(\lambda)B_{1,\overline{\bb{w}}}
(\partial_{\lambda}\mathcal{K}^1_{2,j-1})(\lambda)
\bb{f}\|_{\mathbb{L}^p(B_{9(R+2)})}\lesssim_{\varrho}
|\lambda|^{-2+(|\beta|/2)+\varrho}E_j\|\bb{f}\|_{\mathbb{L}^p(\R^3)},\\
&\|\partial^{\beta}_x\mathcal{A}^{3}_{1}(\lambda) B_{1,\overline{\bb{w}}}(\vartriangle_h\partial_{\lambda}\mathcal{K}^1_{2,j-1})(\lambda)
\bb{f}\|_{\mathbb{L}^p(B_{9(R+2)})}\lesssim_{\rho,h_0,\varrho}|h|^{\rho}|\lambda|^{-2+(|\beta|/2)+\varrho}
F_j\|\bb{f}\|_{\mathbb{L}^p(\R^3)}
\end{aligned}\right.
\end{equation*}
and
\begin{equation*}
\left\{\begin{aligned}
&\|\partial^{\beta}_x\mathcal{A}^{3}_{1}(\lambda)B_{1,\overline{\bb{w}}}
(\partial_{\lambda}\mathcal{K}^1_{1,j-1})(\lambda)\bb{f}\|_{\mathbb{L}^p(B_{9(R+2)})}\\
\lesssim&|\lambda|^{-2+\frac{|\beta|}2}E_j\|\bb{f}\|_{\mathbb{L}^p(\R^3)}+\|\partial^{\beta}_x
\mathcal{A}^{3}_{1}(\lambda)B_{1,\overline{\bb{w}}}\mathcal{A}^{1}_{1}(\lambda)B_{1,\overline{\bb{w}}}
(\partial_{\lambda}\mathcal{K}_{j-2})(\lambda)\bb{f}\|_{\mathbb{L}^p(B_{9(R+2)})}\\
&+|\lambda|^{-1}\sum_{k=0,1}\|\partial^{\beta}_x\mathcal{A}^{3}_{1}(\lambda)B_{1,\overline{\bb{w}}}
\mathcal{A}^{1}_{1}(\lambda)B_{2,\overline{\bb{w}}}\Delta(\partial^k_{\lambda}
\mathcal{K}_{j-2})(\lambda)\bb{f}\|_{\mathbb{L}^p(B_{9(R+2)})},\\
%%%%%%%%%%%%%%%%%%%%
&\|\partial^{\beta}_x\mathcal{A}^{3}_{1}(\lambda)B_{1,\overline{\bb{w}}}(\vartriangle_h
\partial_{\lambda}\mathcal{K}^1_{1,j-1})(\lambda)
\bb{f}\|_{\mathbb{L}^p(B_{9(R+2)})}\\
\lesssim&_{\rho,h_0}|h|^{\rho}|\lambda|^{-2+\frac{|\beta|}2}F_j\|\bb{f}\|_{\mathbb{L}^p(\R^3)}
+\|\partial^{\beta}_x\mathcal{A}^{3}_{1}(\lambda)B_{1,\overline{\bb{w}}}\mathcal{A}^{1}_{1}(\lambda)
B_{1,\overline{\bb{w}}}(\vartriangle_h\partial_{\lambda}\mathcal{K}_{j-2})(\lambda)
\bb{f}\|_{\mathbb{L}^p(B_{9(R+2)})}\\
&+|h|^{\rho}|\lambda|^{-2}\sum_{k=0,1}\|\partial^{\beta}_x\mathcal{A}^{3}_{1}(\lambda)B_{1,\overline{\bb{w}}}
\mathcal{A}^{1}_{1}(\lambda)
B_{2,\overline{\bb{w}}}\Delta(\partial^k_{\lambda}\mathcal{K}_{j-2})(\lambda)\bb{f}\|_{\mathbb{L}^p(B_{9(R+2)})}\\
&+|\lambda|^{-1}\sum_{k=0,1}\|\partial^{\beta}_x\mathcal{A}^{3}_{1}(\lambda)B_{1,\overline{\bb{w}}}
\mathcal{A}^{1}_{1}(\lambda)
B_{2,\overline{\bb{w}}}\Delta(\vartriangle_h\partial^k_{\lambda}\mathcal{K}_{j-2})(\lambda)
\bb{f}\|_{\mathbb{L}^p(B_{9(R+2)})}.
\end{aligned}\right.
\end{equation*}
Then, plugging the above four sets of estimates into \eqref{K-j-add-2-14},  we deduce
\begin{equation}\label{K-j-add-2-15}
\left\{\begin{aligned}
&\|\partial^{\beta}_x(\partial_{\lambda}\mathcal{K}^{3}_{1,1})(\lambda)
\bb{f}\|_{\mathbb{L}^p(B_{9(R+2)})}\lesssim_{\varrho}|\lambda|^{-2+(|\beta|/2)+\varrho}E_j\|\bb{f}\|_{\mathbb{L}^p(\R^3)},\\
&|\partial^{\beta}_x(\vartriangle_h\partial_{\lambda}\mathcal{K}^{3}_{1,1})(\lambda)
\bb{f}\|_{\mathbb{L}^p(B_{9(R+2)})}\lesssim_{\rho,h_0,\varrho}|h|^{\rho}|\lambda|^{-2+(|\beta|/2)+\varrho}F_j\|\bb{f}\|_{\mathbb{L}^p(\R^3)}
\end{aligned}\right.
\end{equation}
and for $j\geq 2$,
\begin{equation}\label{K-j-add-2-16}
\left\{\begin{aligned}
&\|\partial^{\beta}_x(\partial_{\lambda}\mathcal{K}^{3}_{1,j})(\lambda)\bb{f}\|_{\mathbb{L}^p(B_{9(R+2)})}\\
\lesssim&_{\varrho} |\lambda|^{-2+\frac{|\beta|}2+\varrho}E_j\|\bb{f}\|_{\mathbb{L}^p(\R^3)}+\|\partial^{\beta}_x
\mathcal{A}^{3}_{1}(\lambda)B_{1,\overline{\bb{w}}}
\mathcal{A}^1_{1}(\lambda)B_{1,\overline{\bb{w}}}(\partial_{\lambda}\mathcal{K}_{j-2})(\lambda)
\bb{f}\|_{\mathbb{L}^p(B_{9(R+2)})}\\
&+\sum_{k=0,1}|\lambda|^{-1}\|\partial^{\beta}_x\mathcal{A}^{3}_{1}(\lambda)
B_{1,\overline{\bb{w}}}\mathcal{A}^1_{1}(\lambda)
B_{2,\overline{\bb{w}}}\Delta(\partial^k_{\lambda}\mathcal{K}_{j-2})(\lambda)
\bb{f}\|_{\mathbb{L}^p(B_{9(R+2)})},\\
%%%%%%%%%%%%%%%%%%%%%%%%
&\|\partial^{\beta}_x(\vartriangle_h\partial_{\lambda}\mathcal{K}^{3}_{1,j})(\lambda)
\bb{f}\|_{\mathbb{L}^p(B_{9(R+2)})}\\
\lesssim&_{\rho,h_0,\varrho}|h|^{\rho}|\lambda|^{-2+\frac{|\beta|}2+\varrho}F_j\|\bb{f}\|_{\mathbb{L}^p(\R^3)}\\
&+\|\partial^{\beta}_x\mathcal{A}^{3}_{1}(\lambda)B_{1,\overline{\bb{w}}}
\mathcal{A}^1_{1}(\lambda)B_{1,\overline{\bb{w}}}(\vartriangle_h\partial_{\lambda}
\mathcal{K}_{j-2})(\lambda)\bb{f}\|_{\mathbb{L}^p(B_{9(R+2)})}\\
&+|\lambda|^{-1}\sum_{k=0,1}\Big[|h|^{\rho}\|\partial^{\beta}_x
\mathcal{A}^{3}_{1}(\lambda)B_{1,\overline{\bb{w}}}\mathcal{A}^1_{1}(\lambda)
B_{2,\overline{\bb{w}}}\Delta(\partial^k_{\lambda}
\mathcal{K}_{j-2})(\lambda)\bb{f}\|_{\mathbb{L}^p(B_{9(R+2)})}\\
&+\|\partial^{\beta}_x\mathcal{A}^{3}_{1}(\lambda)
B_{1,\overline{\bb{w}}}\mathcal{A}^1_{1}(\lambda)
B_{2,\overline{\bb{w}}}\Delta(\vartriangle_h\partial^k_{\lambda}\mathcal{K}_{j-2})(\lambda)
\bb{f}\|_{\mathbb{L}^p(B_{9(R+2)})}\Big].
\end{aligned}\right.
\end{equation}
Thus \eqref{est-K_j} for $j=1$ and $|\beta|=1,2$ is obtained by \eqref{K-j-add-2-13} and \eqref{K-j-add-2-15}.

Finally, we use \eqref{K_j2} with $N=1$ again for $\mathcal{K}_{j-2}(\lambda)$ in \eqref{K-j-add-2-16}. We observe from \eqref{S-1} and \eqref{K-j-add-2-8'} that
\begin{align*}
&\|\chi_{B_9(R+2)}\partial^{\beta}_x\mathcal{A}^{3}_{1}(\lambda)B_{1,\overline{\bb{w}}}\mathcal{A}^{1}_{1}(\lambda)\nabla\overline{\bb{w}}(1+|x|)\|_{\mathcal{L}(\mathbb{L}^p(\R^3))}\lesssim C_{R,\mathfrak{R}^*,\omega^*}\normmm{\bb{w}}_{\varepsilon,\Omega}|\lambda|^{-(3-|\beta|/2)}\\
&\|\chi_{B_9(R+2)}\partial^{\beta}_x\mathcal{A}^{3}_{1}(\lambda)B_{1,\overline{\bb{w}}}\mathcal{A}^{1}_{1}(\lambda)\overline{\bb{w}}(1+|x|)\|_{\mathcal{L}(\mathbb{L}^p(\R^3))}\lesssim C_{R,\mathfrak{R}^*,\omega^*}\normmm{\bb{w}}_{\varepsilon,\Omega}|\lambda|^{-(3-|\beta|/2)}.
\end{align*}
Hence, we get from  \eqref{K-j-add-2-16} that for all $j\geq 2$
\begin{equation}\label{K-j-add-2-17}
\left\{\begin{aligned}
&\|\partial^{\beta}_x(\partial_{\lambda}\mathcal{K}^{3}_{1,j})(\lambda)
\bb{f}\|_{\mathbb{L}^p(B_{9(R+2)})}\lesssim_{\varrho}|\lambda|^{-2+(|\beta|/2)+\varrho}E_j\|\bb{f}\|_{\mathbb{L}^p(\R^3)},\\
&\|\partial^{\beta}_x(\vartriangle_h\partial_{\lambda}\mathcal{K}^{3}_{1,j})
(\lambda)\bb{f}\|_{\mathbb{L}^p(B_{9(R+2)})}\lesssim_{\rho,h_0,\varrho}|h|^{\rho}|
\lambda|^{-2+(|\beta|/2)-\varrho}F_j\|\bb{f}\|_{\mathbb{L}^p(\R^3)}.
\end{aligned}\right.
\end{equation}
by making use of
$$\left\{\begin{aligned}
&\eqref{K-j-add-2-6},\;\text{with }T=\chi_{B_{9(R+2)}}\partial^{\beta}_x\mathcal{A}^{3}_{1}(\lambda)B_{1,\overline{\bb{w}}}\mathcal{A}^1_{1}(\lambda)\overline{\bb{w}}(1+|x|) ,\\
&\eqref{K-j-add-2-6},\;\text{with }T=\chi_{B_{9(R+2)}}\partial^{\beta}_x\mathcal{A}^{3}_{1}(\lambda)B_{1,\overline{\bb{w}}}\mathcal{A}^1_{1}(\lambda)\nabla {\bb{w}}(1+|x|),\\
&\eqref{K-j-add-2-8'}{\rm -}\eqref{K-j-add-2-11},\;\text{with } T=\chi_{B_{9(R+2)}}\partial^{\beta}_x\mathcal{A}^{3}_{1}(\lambda)B_{1,\overline{\bb{w}}}
\mathcal{A}^1_{1}(\lambda)\overline{\bb{w}},\\
&\eqref{K-j-add-2-8'}{\rm -}\eqref{K-j-add-2-11},\;\text{with } T=\chi_{B_{9(R+2)}}\partial^{\beta}_x\mathcal{A}^{3}_{1}(\lambda)B_{1,\overline{\bb{w}}}
\mathcal{A}^1_{1}(\lambda)\nabla\overline{\bb{w}}\end{aligned}\right.$$
Collecting \eqref{K-j-add-2-13} and \eqref{K-j-add-2-17}, we prove \eqref{est-K_j} for $j\geq 2$ and $|\beta|=1,2$, and so end the proof of \eqref{est-K_j} for all $j\geq 0$ and $|\beta|=1,2$
\vskip 0.15cm

\textbf{Step 3: Proof of \eqref{est-K_j} for $|\beta|=0$.}

Using \eqref{K-0-add-1-0} and the estimate \eqref{est-K_j} for $|\beta|=1,2$,  we easily verify \eqref{est-K_j} for $j=0$ and $|\beta|=0$. Meanwhile, for $j\geq 1$, we obtain from \eqref{K-0-add-j-4}-\eqref{K-0-add-j-6}
\begin{equation}\label{K-j-add-2-18}
\|\mathcal{K}_j(\lambda)\bb{f}\|_{\mathbb{L}^p(B_{9(R+2)})}\lesssim|\lambda|^{-1}E_j\|\bb{f}\|_{\LL^p(\R^3)}
\end{equation}
and
\begin{equation}\label{K-j-add-2-19}
\left\{\begin{aligned}
&\|(\partial_{\lambda}\mathcal{K}_j)(\lambda)\bb{f}\|_{\mathbb{L}^p(B_{9(R+2)})}\\
\lesssim&_{\varrho}|\lambda|^{-2+\varrho}E_j\|\bb{f}\|_{\mathbb{L}^p(\R^3)}+|\lambda|^{-1}\|\mathcal{P}_{\R^3}B_{1,\overline{\bb{w}}}(\partial_{\lambda}\mathcal{K}_{j-1})(\lambda)
\bb{f}\|_{\mathbb{L}^p(B_{9(R+2)})}\\
&+|\lambda|^{-2}\big\|\big(\mathcal{P}_{\R^3}B_{2,\overline{\bb{w}}}\Delta\mathcal{K}_{j-1}(\lambda)\bb{f},\,
\mathcal{P}_{\R^3}B_{2,\overline{\bb{w}}}
\Delta(\partial_{\lambda}\mathcal{K}_{j-1})(\lambda)\bb{f}\big)\big\|_{\mathbb{L}^p(B_{9(R+2)})},\\
%%%%%%%%%%%%%%
&\|(\vartriangle_h\partial_{\lambda}\mathcal{K}_j)(\lambda)\bb{f}\|_{\mathbb{L}^p(B_{9(R+2)})}\\
\lesssim&_{\rho,h_0,\varrho} |h|^{\rho}|\lambda|^{-2+\varrho}F_j\|\bb{f}\|_{\mathbb{L}^p(\R^3)}+|\lambda|^{-1}\|\mathcal{P}_{\R^3}B_{1,\overline{\bb{w}}}(\vartriangle_h\partial_{\lambda}
\mathcal{K}_{j-1})(\lambda)\bb{f})\|_{\mathbb{L}^p(B_{9(R+2)})}\\
&+|\lambda|^{-2}\Big[|h|^{\rho}\big\|\big(\mathcal{P}_{\R^3}B_{2,\overline{\bb{w}}}\Delta\mathcal{K}_{j-1}(\lambda)
\bb{f},\,\mathcal{P}_{\R^3}B_{2,\overline{\bb{w}}}
\Delta(\partial_{\lambda}\mathcal{K}_{j-1})(\lambda)\bb{f}\big)\big\|_{\mathbb{L}^p(B_{9(R+2)})}\\
&+\big\|\big(\mathcal{P}_{\R^3}B_{2,\overline{\bb{w}}}\Delta(\vartriangle_h\mathcal{K}_{j-1})(\lambda)\bb{f},\,
\mathcal{P}_{\R^3}B_{2,\overline{\bb{w}}}\Delta(\vartriangle_h\partial_{\lambda}\mathcal{K}_{j-1})(\lambda)
\bb{f}\big)\big\|_{\mathbb{L}^p(B_{9(R+2)})}\Big].
\end{aligned}\right.
\end{equation}

Now, we are going to  further estimate \eqref{K-j-add-2-19} by the ``tree self-similar" iterative mechanism used in Step 2. We first use \eqref{K_j2} with $N=1$ for $\mathcal{K}_{j-1}(\lambda)$. Hence we obtain
\begin{equation}\label{K-j-add-2-20}
\left\{\begin{aligned}
&\|(\partial_{\lambda}\mathcal{K}_1)(\lambda)\bb{f}\|_{\mathbb{L}^p(B_{9(R+2)})}
\lesssim_{\varrho}|\lambda|^{-2+\varrho}E_1\|\bb{f}\|_{\mathbb{L}^p(\R^3)},\\
&\|(\vartriangle_h\partial_{\lambda}\mathcal{K}_1)(\lambda)\bb{f}
\|_{\mathbb{L}^p(B_{9(R+2)})}\lesssim_{\rho,h_0,\varrho}|h|^{\rho}|\lambda|^{-2+\varrho}F_1
\|\bb{f}\|_{\mathbb{L}^p(\R^3)}
\end{aligned}\right.
\end{equation}
and for $j\geq 2$
\begin{equation*}
\left\{\begin{aligned}
&\|(\partial_{\lambda}\mathcal{K}_j)(\lambda)\bb{f}\|_{\mathbb{L}^p(B_{9(R+2)})}\\
\lesssim&_{\varrho}|\lambda|^{-2+\varrho}E_j\|\bb{f}\|_{\mathbb{L}^p(\R^3)}+|\lambda|^{-1}\|\mathcal{P}_{\R^3}B_{1,\overline{\bb{w}}}\mathcal{A}^{1}_{1}(\lambda)
B_{1,\overline{\bb{w}}}(\partial_{\lambda}\mathcal{K}_{j-2})
(\lambda)\bb{f}\|_{\mathbb{L}^p(B_{9(R+2)})}\\
&+|\lambda|^{-2}\sum_{k=0,1}\|\mathcal{P}_{\R^3}B_{1,\overline{\bb{w}}}\mathcal{A}^{1}_{1}(\lambda)
B_{2,\overline{\bb{w}}}\Delta(\partial^k_{\lambda}\mathcal{K}_{j-2})(\lambda)\bb{f}\|_{\mathbb{L}^p(B_{9(R+2)})}\\
%%%%%%%%%%%%%%%%
&\|(\vartriangle_h\partial_{\lambda}\mathcal{K}_j)(\lambda)\bb{f}\|_{\mathbb{L}^p(B_{9(R+2)})}\\
\lesssim&_{\rho,h_0,\varrho}|h|^{\rho}|\lambda|^{-2+\varrho}F_j\|\bb{f}\|_{\mathbb{L}^p(\R^3)}+|\lambda|^{-1}\|\mathcal{P}_{\R^3}B_{1,\overline{\bb{w}}}\mathcal{A}^{1}_{1}(\lambda)
B_{1,\overline{\bb{w}}}(\vartriangle_h\partial_{\lambda}\mathcal{K}_{j-2})
(\lambda)\bb{f}\|_{\mathbb{L}^p(B_{9(R+2)})}\\
&+|\lambda|^{-2}\sum_{k=0,1}\Big[|h|^{\rho}\|\mathcal{P}_{\R^3}B_{1,\overline{\bb{w}}}
\mathcal{A}^{1}_{1}(\lambda)B_{2,\overline{\bb{w}}}\Delta
(\partial^k_{\lambda}\mathcal{K}_{j-2})(\lambda)\bb{f}\|_{\mathbb{L}^p(B_{9(R+2)})}\\
&+\|\mathcal{P}_{\R^3}B_{1,\overline{\bb{w}}}\mathcal{A}^{1}_{1}(\lambda)
B_{2,\overline{\bb{w}}}\Delta(\vartriangle_h\partial^k_{\lambda}\mathcal{K}_{1,j-2})(\lambda)
\bb{f}\|_{\mathbb{L}^p(B_{9(R+2)})}\Big].
\end{aligned}\right.
\end{equation*}	
with the help of
\begin{equation*}\left\{\begin{aligned}
&\eqref{K-j-add-2-6},\;\;\text{with}\;\; T=\mathcal{P}_{\R^3}\overline{\bb{w}}(1+|x|) ,\;\mathcal{P}_{\R^3}\nabla\overline{\bb{w}}(1+|x|),\\
& \eqref{K-j-add-2-8'}{\rm -}\eqref{K-j-add-2-9-1},\;\eqref{K-j-add-2-11-0}-\eqref{K-j-add-2-11}, \;\;\text{with }\;T=\mathcal{P}_{\R^3}\overline{\bb{w}},\\
&\eqref{K-j-add-2-8'}{\rm -}\eqref{K-j-add-2-10},\;\;\text{with }\; T=\mathcal{P}_{\R^3}\nabla \overline{\bb{w}}\end{aligned}\right.\end{equation*}

Finally, using \eqref{K_j2} with $N=1$ again for $\mathcal{K}_{j-2}(\lambda)$ and
putting
\begin{equation*}\left\{\begin{aligned}
&\eqref{K-j-add-2-6}\;\;\text{with}\;\; T=\mathcal{P}_{\R^3}B_{1,\overline{\bb{w}}}\mathcal{A}^1_{1}(\lambda)\overline{\bb{w}}(1+|x|),\;\;\mathcal{P}_{\R^3}B_{1,\overline{\bb{w}}}\mathcal{A}^1_{1}(\lambda)\nabla\overline{\bb{w}}(1+|x|),\\
& \eqref{K-j-add-2-8'}{\rm -}\eqref{K-j-add-2-10},\;\text{with}\;\; T=\mathcal{P}_{\R^3}B_{1,\overline{\bb{w}}}\mathcal{A}^1_{1}(\lambda)\overline{\bb{w}},\;\;\mathcal{P}_{\R^3}B_{1,\overline{\bb{w}}}\mathcal{A}^1_{1}(\lambda)\nabla\overline{\bb{w}}
\end{aligned}\right.\end{equation*}
into the above set of inequalities, we deduce for every $j\geq 2$
\begin{equation}\label{K-j-add-2-22}
\left\{\begin{aligned}
&\|(\partial_{\lambda}\mathcal{K}_j)(\lambda)\bb{f}\|_{\mathbb{L}^p(B_{9(R+2)})}\lesssim_{\varrho}
|\lambda|^{-2+\varrho}E_j\|\bb{f}\|_{\mathbb{L}^p(\R^3)}\\
&\|(\vartriangle_h\partial_{\lambda}\mathcal{K}_j)(\lambda)\bb{f}\|_{\mathbb{L}^p(B_{9(R+2)})}\lesssim_{\rho,h_0,\varrho} |h|^{\rho}|\lambda|^{-2+\varrho}F_j\|\bb{f}\|_{\mathbb{L}^p(\R^3)}
\end{aligned}\right.
\end{equation}

Collecting \eqref{K-j-add-2-18} and \eqref{K-j-add-2-20}-\eqref{K-j-add-2-22},
we prove \eqref{est-K_j} for $j\geq 1$ and $|\beta|=0$.
This completes the proof of Theorem \ref{TH2-2}.
\end{proof}
For the pressure, we  have the following result.
\begin{corollary}\label{Cor.G2}
Under the assumption of Theorem \ref{TH2-2},  there exists a positive constant $\eta=\eta_{R,p,\rho,\mathfrak{R}_*,\mathfrak{R^*},{\omega}^*}>0$ such that if
$\normmm{\bb{w}}_{\varepsilon,\Omega}\leq \eta$, then
\[\Pi^G_{\mathfrak{R},{\omega},\overline{\bb{w}}}(\lambda)\in C(\overline{C}_+,\mathcal{L}(\mathbb{L}^p_{R+2}(\R^3),{W}^{1,p}(B_9(R+2))))\]
Moreover,  for $\lambda,\lambda+h\in \overline{\C_+}$ and $\bb{f}\in \LL_{R+2}^p(\R^3)$, we have
\begin{align}
&\|\nabla(\partial^k_{\lambda}\Pi^G_{\mathfrak{R},{\omega},\overline{\bb{w}}})(\lambda)
\bb{f}\|_{\LL^p(B_{9(R+2)})}
\leq C_{R,\mathfrak{R}_*,\mathfrak{R}^*,{\omega}^*}\|\bb{f}\|_{\mathbb{L}^p(\R^3)},\quad k=0,1,\label{est.loc1'}\\
&\|(\nabla(\vartriangle_h \partial_{\lambda}\Pi^G_{\mathfrak{R},{\omega},\overline{\bb{w}}})(\lambda)\bb{f}\|_{
\LL^p(B_{9(R+2)})}\leq C_{R,\rho,\mathfrak{R}_*,\mathfrak{R}^*,{\omega}^*}|h|^{\rho}\|\bb{f}\|_{\mathbb{L}^p(\R^3)},\label{est.loc2'}\\
&\big|[\Pi^G_{\mathfrak{R},{\omega},\overline{\bb{w}}}(\lambda)
\bb{f}](x)\big|\leq C_{R,\mathfrak{R}_*,\mathfrak{R}^*,{\omega}^*}|x|^{-3/2}\|\bb{f}\|_{\mathbb{L}^p(\R^3)},\quad x\in B^c_{10(R+2)}.\label{est.decay1'}
\end{align}
In particular, for  $0<|h|\leq h_0$, we have for $0<\varrho\ll 1/2$
\begin{align}
&\|\nabla(\partial^k_{\lambda}\Pi^G_{\mathfrak{R},{\omega},\overline{\bb{w}}})(\lambda)
\bb{f}\|_{\LL^p(B_{9(R+2)}}
\leq C_{R,\varrho,\mathfrak{R}_*,\mathfrak{R}^*,{\omega}^*}|\lambda|^{-1+\varrho}\|\bb{f}\|_{\mathbb{L}^p(\R^3)},\quad k=0,1,\label{est.loc3'}\\
&\|\nabla(\vartriangle_h \partial_{\lambda}\Pi^G_{\mathfrak{R},{\omega},\overline{\bb{w}}})(\lambda)
\bb{f}|_{\LL^p(B_{9(R+2)}}\leq C_{R,\rho,\varrho,\mathfrak{R}_*,\mathfrak{R}^*,{\omega}^*}|h|^{\rho}|\lambda|^{-1+\varrho}\|\bb{f}\|_{\mathbb{L}^p(\R^3)}.\label{est.loc4'}
\end{align}
\end{corollary}
\begin{proof}
  From \eqref{S-1} and \eqref{est.loc1}-\eqref{est.decay2}, we have for $q=\max(p,3)$
\begin{align*}
&\|B_{\overline{\bb{w}}}(\partial^k_{\lambda}\mathcal{R}^G_{\Rr,{\omega},\overline{\bb{w}}})(\lambda)\bb{f}
\|_{\LL^{p,q}_R(\R^3)}\lesssim_{R,\mathfrak{R}_*,\mathfrak{R}^*,{\omega}^*}\normmm{\bb{w}}_{\varepsilon,\Omega}
\|\bb{f}\|_{\mathbb{L}^p(\R^3)},\quad k=0,1,\\
&\|B_{\overline{\bb{w}}}(\vartriangle_k\partial_{\lambda}\mathcal{R}^G_{\Rr,{\omega},\overline{\bb{w}}})
(\lambda)\bb{f}\|_{\LL^{p,q}_R(\R^3)}\lesssim_{R,\rho,\mathfrak{R}_*,\mathfrak{R}^*,{\omega}^*}\normmm{\bb{w}}_{\varepsilon,\Omega}\|\bb{f}\|_{\mathbb{L}^p(\R^3)}.
\end{align*}
Hence we easily prove \eqref{est.loc1'}-\eqref{est.loc2'} by \eqref{P-est} and \eqref{claim}.

Next, we prove \eqref{est.decay1'}. Rewrite
$$\Pi^G_{\mathfrak{R},{\omega},\overline{\bb{w}}}(\lambda)\bb{f}=\mathcal{Q}_{\R^3}(\chi_{B_{9(R+2)}}
B_{\overline{\bb{w}}}\mathcal{R}^G_{\mathfrak{R},{\omega},\overline{\bb{w}}}(\lambda)\bb{f})
+\mathcal{Q}_{\R^3}(\chi_{B^c_{9(R+2)}}B_{\overline{\bb{w}}}\mathcal{R}^G_{\mathfrak{R},{\omega},
\overline{\bb{w}}}(\lambda)\bb{f}).$$
Since the kernel function of $\mathcal{Q}_{\R^3}$ is bounded by $|x|^{-2}$, we have  by \eqref{est.loc1}
and \eqref{est.decay1}
\begin{align*}
\big|[\Pi^G_{\mathfrak{R},{\omega},\overline{\bb{w}}}(\lambda)\bb{f}](x)\big|
\lesssim&_{R,\mathfrak{R}_*,\mathfrak{R}^*,{\omega}^*}\normmm{\bb{w}}_{\varepsilon,
 \Omega}\|\bb{f}\|_{\mathbb{L}^p(\R^3)}\\
&\times\Big(\frac1{|x|^2}+\int_{B^c_{9(R+2)}}\frac1{|x-y|^2}
\frac1{|y|^{\frac52}(1+s_{\Rr}(y))}\,\mathrm{d}y\Big),\;\; x\in B^c_{10(R+2)}.
\end{align*}
We  observe for every $x\in B^c_{10(R+2)}$
\begin{align*}
\int_{B^c_{9(R+2)}}\frac1{|x-y|^2}
\frac1{|y|^{\frac52}(1+s_{\mathfrak{R}}(y))}\,\mathrm{d}y\lesssim& \frac1{|x|^{2}}\Big[\int_{9(R+2)\leq |y|<\frac{9|x|}{10}}+\int_{|y|\geq \frac{3|x|}2}\Big]\frac1{|y|^{\frac52}(1+s_{\mathfrak{R}}(y))}\,\mathrm{d}y\\
&+\frac1{|x|^{\frac52}}\int_{9|x|/10\leq |y|<3|x|/2}\frac1{|x-y|^2}\,\mathrm{d}y\lesssim |x|^{-\frac32}.\end{align*}
where we used Lemma \ref{Lem.In} in the last inequality. Hence, we prove \eqref{est.decay1'}.

\vskip0.15cm

Finally, we prove \eqref{est.loc3'}-\eqref{est.loc4'}. In view of \eqref{eq.K-add}, we write
\[\nabla\Pi^G_{\mathfrak{R},{\omega},\overline{\bb{w}}}(\lambda)\bb{f}
=\sum_{j=0}^{\infty}\Big(\nabla \mathcal{Q}_{\R^3}(B_{1,\overline{\bb{w}}}\mathcal{K}_{j}
(\lambda)\bb{f})+\nabla \mathcal{Q}_{\R^3}(B_{2,\overline{\bb{w}}}\mathcal{K}_{j}(\lambda)\bb{f})\Big).\]
We first use \eqref{claim} with $T=\nabla \mathcal{Q}_{\R^3}$, \eqref{K-0-add-j-2'} and \eqref{eq-B_2}-\eqref{K-0-add-j-5} to
reduce the estimate of  $\nabla \mathcal{Q}_{\R^3}(B_{2,\overline{\bb{w}}}
\mathcal{K}_{j}(\lambda)\bb{f})$ to the estimate
of $\nabla \mathcal{Q}_{\R^3}(B_{2,\overline{\bb{w}}}\Delta\mathcal{K}_{j}(\lambda)\bb{f})$.
Then, utilizing the self-similar iterative mechanism \eqref{K_j2} with $N=1$ two times, we finally deduce \eqref{est.loc3'}-\eqref{est.loc4'}  by  \eqref{K-j-add-2-6}-\eqref{K-j-add-2-11}.
\end{proof}
\vskip 0.2cm
Set
\[\Pi^{G^*}_{\mathfrak{R},{\omega},\overline{\bb{w}}}(\lambda)=\mathcal{Q}_{\R^3}
(B^*_{\overline{\bb{w}}}\mathcal{R}^{G^*}_{\mathfrak{R},{\omega},\overline{\bb{w}}}(\lambda)),
\quad \mathring{\Pi}^{G^*}_{\mathfrak{R},{\omega},\overline{\bb{w}}}(\lambda)
=\mathring{\mathcal{Q}}_{\R^3}(B^*_{\overline{\bb{w}}}
\mathcal{R}^{G^*}_{\mathfrak{R},{\omega},\overline{\bb{w}}}(\lambda)).\]
Then,we have for $\bb{f}\in \mathbb{L}^p(\R^3)$
\begin{equation}\label{eq.GP*}
(\lambda+L^*_{\mathfrak{R},{\omega},\overline{\bb{w}}})
\mathcal{R}^{G^*}_{\mathcal{R},{\omega},\overline{\bb{w}}}(\lambda)\bb{f}
+\nabla(\mathring{\mathcal{Q}}_{\R^3}\bb{f}
+\nabla \mathring{\Pi}^{G^*}_{\mathcal{R},{\omega},\overline{\bb{w}}}(\lambda)\bb{f})=\bb{f},
 \quad \Div \mathcal{R}^{G^*}_{\mathcal{R},{\omega},\overline{\bb{w}}}(\lambda)\bb{f}=0.
\end{equation}
In the same way as treating $(\mathcal{R}^{G}_{\Rr,{\omega},\overline{\bb{w}}}(\lambda),
 \mathring{\Pi}^{G}_{\mathfrak{R},{\omega},\overline{\bb{w}}}(\lambda))$,
 we obtain
\begin{corollary}\label{TH2-3}
Let $\bb{f}\in \mathbb{L}^p_{R+2}(\R^3)$, $p\in (1,\infty)$. Then,  Theorem \ref{TH2-1},
and Corollary \ref{Cor.G1}, Theorem \ref{TH2-2} and Corollary \ref{Cor.G2} hold for
$\mathcal{R}^{G^*}_{\Rr,{\omega},\overline{\bb{w}}}(\lambda)\bb{f}$ and
 $\mathring{\Pi}^{G^*}_{\mathfrak{R},{\omega},\overline{\bb{w}}}(\lambda)\bb{f}$ (or $\Pi^{G^*}_{\mathfrak{R},{\omega},\overline{\bb{w}}}(\lambda)\bb{f}$).
\end{corollary}

%%%%%%%%%%%%%%%%%%%%%%%%%%%%%%%%%%%%%%%%%%%%%%%%%%%%%%%%%%%%%%%%%%%%%%%%%%%%
\section{Interior resolvent problem}
%%%%%%%%%%%%%%%%%%%%%%%%%%%%%%%%%%%%%%%%%%%%%%%%%%%%%%%%%%%%%%%%
\setcounter{section}{3}\setcounter{equation}{0}

In this section we will discuss the resolvent problem associated to $L_{\mathfrak{R},{\omega},\bb{w}}$ in $\Omega_{R+3}$
\begin{equation}\label{BR}
(\lambda I+L_{\mathfrak{R},{\omega},\bb{w}})\bb{u}+\nabla P=\bb{f}, \quad \Div \bb{u}=0\quad \text{in }\Omega_{R+3}, \quad \bb{u}|_{\partial\Omega_{R+3}}=0.
\end{equation}
with the addition condition $\int_{\Omega_{R+3}} P\,\mathrm{d}x=0$.
From the Helmholtz decomposition  in \cite{FS94,FM77}, given $\bb{f}\in\LL^p(\Omega_{R+3})$, there exist unique $\bb{h}\in \mathbb{J}^p(\Omega_{R+3})$ and $g\in \mathring{W}^{1,p}(\Omega_{R+3})$ such that
$$\bb{f}=\bb{h}+\nabla g.$$
This leads us to define operators $\mathcal{P}_{\Omega_{R+3}}\bb{f}=\bb{h}$ and $\mathring{\mathcal{Q}}_{\Omega_{R+3}}\bb{f}=g$ satisfying
\[\mathcal{P}_{\Omega_{R+3}}\in \mathcal{L}(\mathbb{L}^p(\Omega_{R+3}),\mathbb{J}^p (\Omega_{R+3})),\quad
\mathring{\mathcal{Q}}_{\Omega_{R+3}}\in\mathcal{L}(\mathbb{L}^p(\Omega_{R+3}),
\mathring{W}^{1,p}(\Omega_{R+3})).\]
Define
\begin{equation}\label{Op-I}
\left\{\begin{aligned}
&\mathcal{L}_{\Rr,{\omega},\bb{w},\Omega_{R+3}}=\mathcal{P}_{\Omega_{R+3}}L_{\Rr,{\omega},\bb{w}},\\
&D_p(\mathcal{L}_{\Rr,{\omega},\bb{w},\Omega_{R+3}})
=\big\{\bb{u}\in {\W}^{2,p}(\Omega_{R+3})\cap \mathbb{J}^p(\Omega_{R+3})\,\big|\,\bb{u}|_{\partial\Omega_{R+3}}=0\big\}.
\end{aligned}\right.
\end{equation}
Then, system \eqref{BR} is equivalent to
\begin{equation}\label{BR'}
(\lambda I+\mathcal{L}_{\mathfrak{R},{\omega},\bb{w},\Omega})\bb{u}=\mathcal{P}_{\Omega_{R+3}}\bb{f}
\end{equation}
with $P=\mathring{\mathcal{Q}}_{\Omega_{R+3}}\bb{f}+\mathring{\mathcal{Q}}_{\Omega_{R+3}}(L_{\Rr,{\omega},\bb{w}}\bb{u})$. Denote $\mathcal{P}_{\Omega_{R+3}}\Delta$ by $\Delta_{\Omega_{R+3}}$. Since
$D_p(-\Delta_{\Omega_{R+3}})\subset D_p(\mathcal{L}_{\Rr,{\omega},\bb{w},\Omega_{R+3}}+\Delta_{\Omega_{R+3}})$, we can
 view \eqref{BR'} as a perturbation  of the classical Stokes resolvent problem
\begin{equation}\label{CSB}
(\lambda I-\Delta_{\Omega_{R+3}})\bb{v}=\mathcal{P}_{\R^3}\bb{f}.
\end{equation}
It is well known from \cite{FS94,Gi81} that, for every $\bb{f}\in \mathbb{J}^p(\Omega_{R+3})$ and $\lambda\in \Sigma_{\theta}\cup\{0\}$ with $0<\theta <\pi/2$,
\begin{equation}\label{est3-1-1}
\|\nabla^j(\lambda I-\Delta_{\Omega_{R+3}})^{-1}\bb{f}\|_{\mathbb{L}^p(\Omega_{R+3})}\leq C_{\theta} (1+|\lambda|)^{-1+\frac{j}2}\|\bb{f}\|_{\mathbb{L}^p(\Omega_{R+3})},\quad j\leq 2.
\end{equation}
Hence, we have the following theorem for problem \eqref{BR}.
\begin{theorem}\label{TH3-1}
Let $p\in (1,\infty)$, $\theta\in (0,\frac{\pi}2)$, $\varepsilon\in (0,\frac12)$, $0<|\mathfrak{R}|\leq \mathfrak{R}^*$, and $|{\omega}|\leq {\omega}^*$.
Then there exist two constants $\eta, \ell_1>0$, depending only on
$R,p,\theta,\mathfrak{R}^*,{\omega}^*$, such that if
\[\normmm{\bb{w}}_{\varepsilon,\Omega}\leq \eta,\]
then for every $\bb{f}\in \mathbb{L}^p(\Omega_{R+3})$,
 problem \eqref{BR} admits a  unique solution
 $$(\bb{u},P) =(\mathcal{R}^I_{\mathfrak{R},{\omega},\bb{w}}(\lambda)\bb{f}, \mathring{\mathcal{Q}}_{\Omega_{R+3}}\bb{f}+
 \mathring{\Pi}^I_{\mathfrak{R},{\omega},\bb{w}}(\lambda)\bb{f})$$
satisfying
\begin{align*}
&\mathcal{R}^I_{\mathfrak{R},{\omega},\bb{w}}(\lambda)\in \mathscr{A}(\Sigma_{\theta,\ell_1}
\cup\overline{\C_+},
\mathcal{L}(\mathbb{L}^p(\Omega_{R+3}),{\W}^{2,p}(\Omega_{R+3})\cap \mathbb{J}^p(\Omega_{R+3}))),\\
&\mathring{\Pi}^I_{\mathfrak{R},{\omega},\bb{w}}(\lambda)\in \mathscr{A}(\Sigma_{\theta,\ell_1}\cup
\overline{\C_+},\mathcal{L}(\mathbb{L}^p(\Omega_{R+3}),\mathring{{W}}^{1,p}(\Omega_{R+3}))).
\end{align*}
Moreover, for every $k\geq 0$ and $\lambda\in \Sigma_{\theta,\ell_1}\cup\overline{\C_+}$,
one has
\begin{align}
&\|\nabla^j(\partial^{k}_{\lambda}\mathcal{R}^{I}_{\mathfrak{R},{\omega},\bb{w}})(\lambda)
\|_{\mathcal{L}(\mathbb{L}^p(\Omega_{R+3}))}\leq C_{k,\theta,R}(1+\lambda)^{-1-k+(j/2)},\quad j=0,1,2,\label{est.bdd3}\\
&\|\nabla(\partial^{k}_{\lambda} \mathring{\Pi}^{I}_{\mathfrak{R},{\omega},\bb{w}})(\lambda)\|_{\mathcal{L}(\mathbb{L}^p(\Omega_{R+3}))}\leq C_{k,\theta,R}(1+\lambda)^{-k},\label{est.bdd4}\\
&\|(\partial^{k}_{\lambda}\mathring{\Pi}^{I}_{\mathfrak{R},{\omega},\bb{w}})(\lambda)
\|_{\mathcal{L}(\mathbb{L}^p(\Omega_{R+3}),L^p(\Omega_{R+3}))}\leq C_{k,\theta,R}(1+\lambda)^{-k-((1/2)-1/(2p))}.\label{est.bdd5}
\end{align}
\end{theorem}
\begin{proof}
It suffices to prove that for every $\lambda\in \Sigma_{\theta,\ell_1}\cup\overline{\C_+}$,
\begin{align}
 &\big(I+(\mathcal{L}_{\mathfrak{R},{\omega},\bb{w},\Omega_{R+3}}+\Delta_{\Omega_{R+3}})(\lambda I-\Delta_{\Omega_{R+3}})^{-1}\big)^{-1}\in {\mathcal{L}(\mathbb{J}^p(\Omega_{R+3}))},\label{est3-1-2}\\
&\big\|\big(I+(\mathcal{L}_{\mathfrak{R},{\omega},\bb{w},\Omega_{R+3}}+\Delta_{\Omega_{R+3}})(\lambda I-\Delta_{\Omega_{R+3}})^{-1}\big)^{-1}\big\|_{\mathcal{L}
(\mathbb{J}^p(\Omega_{R+3}))}\leq C_{\theta,R,\mathfrak{R}^*,{\omega}^*}.\label{est3-1-3}
\end{align}
In fact,
we have from \eqref{est3-1-1} and \eqref{est3-1-2}-\eqref{est3-1-3} that for all $\lambda\in \Sigma_{\theta,\ell_1}\cup\overline{\C_+}$
\begin{align*}
&(\lambda I+\mathcal{L}_{\mathfrak{R},{\omega},\bb{w},\Omega_{R+3}})^{-1}=(\lambda I-\Delta_{\Omega_{R+3}})^{-1}
\big(I+(\mathcal{L}_{\mathfrak{R},{\omega},\bb{w},\Omega_{R+3}}+\Delta_{\Omega_{R+3}})(\lambda I-\Delta_{\Omega_{R+3}})^{-1}\big)^{-1},\\
&\|\nabla^j(\lambda I+\mathcal{L}_{\mathfrak{R},{\omega},\bb{w},\Omega_{R+3}})^{-1}\|_{\mathcal{L}(\mathbb{J}^p(\Omega_{R+3}))}\leq C_{\theta,R,\Rr^*,{\omega}^*}(1+|\lambda|)^{-1+(j/2)},\;\;j\leq 2.
\end{align*}
Hence, $(\bb{u},P)\triangleq (\mathcal{R}^I_{\mathfrak{R},{\omega},\bb{w}}(\lambda)\bb{f},\mathring{\mathcal{Q}}_{\Omega_{R+3}}\bb{f}+\mathring{\Pi}^I_{\mathfrak{R},{\omega},\bb{w}}(\lambda)\bb{f})$ with
\begin{equation*}
\mathcal{R}^I_{\mathfrak{R},{\omega},\bb{w}}(\lambda)\triangleq(\lambda I+\mathcal{L}_{\mathfrak{R},{\omega},\bb{w},\Omega_{R+3}})^{-1}\mathcal{P}_{\R^3},\quad\mathring{\Pi}^I_{\mathfrak{R},{\omega},\bb{w}}(\lambda)\triangleq
\mathring{\mathcal{Q}}_{\Omega_{R+3}}L_{\mathfrak{R},{\omega},\bb{w}}\mathcal{R}^I_{\mathfrak{R},{\omega},\bb{w}}(\lambda).
\end{equation*}
solve problem \eqref{BR} uniquely, and satisfies  \eqref{est.bdd3}-\eqref{est.bdd4} since
\begin{align*}
&\partial^k_{\lambda}\mathcal{R}^I_{\mathfrak{R},{\omega},\bb{w}}(\lambda)=(-1)^kk!(\lambda I+\mathcal{L}_{\mathfrak{R},{\omega},\bb{w},\Omega_{R+3}})^{-1-k}\mathcal{P}_{\R^3},\\
&\partial^k_{\lambda}\mathring{\Pi}^I_{\mathfrak{R},{\omega},\bb{w}}(\lambda)=\mathring{\mathcal{Q}}_{\Omega_{R+3}}L_{\mathfrak{R},{\omega},\bb{w}}\partial^k_{\lambda}\mathcal{R}^I_{\mathfrak{R},{\omega},\bb{w}}(\lambda).
\end{align*}

To prove \eqref{est.bdd5}, we adopt the known result in \cite{FS94}, that is, for every $\phi\in C^{\infty}_0(\Omega_{R+3})$, there exists a $\Phi\in W^{2,p'}(\Omega_{R+3})$  solves
$$\Delta\Phi=\bar{\phi}\quad \text{in}\; \Omega_{R+3},\quad \partial_{\bb{\nu}}\Phi|_{\partial\Omega_{R+3}}=0$$
where $\bb{\nu}$ is the unit outer normal to $\partial\Omega_{R+3}$ and $\bar{\phi}=\phi-|\Omega_{R+3}|^{-1}\int_{\Omega_{R+3}}\phi\,\mathrm{d}x$, such that
\begin{equation}\label{est-3-add-2}
\|\Phi\|_{W^{2,p'}(\Omega_{R+3})}\leq C\|\bar{\phi}\|_{L^{p'}(\Omega_{R+3})}\leq C\|\phi\|_{L^{p'}(\Omega_{R+3})}.
\end{equation}
This yields \begin{align*}
&\langle\mathring{\mathcal{Q}}_{\Omega_{R+3}}\Delta\partial^k_{\lambda}\mathcal{R}^I_{\mathfrak{R},{\omega},
\bb{w}}(\lambda)\bb{f},\phi\rangle_{\Omega_{R+3}}=\langle\Delta\partial^k_{\lambda}\mathcal{R}^I_{\mathfrak{R},
{\omega},\bb{w}}(\lambda)\bb{f},\nabla\Phi\rangle_{\Omega_{R+3}}\\
=&-\sum^3_{j=1}\langle\partial_{\bb{\nu}}(\partial^k_{\lambda}\mathcal{R}^I_{\mathfrak{R},{\omega},\bb{w}}(\lambda)
\bb{f})_j,\partial_j\Phi\rangle_{\partial\Omega_{R+3}}
+\sum^3_{j,m=1}\langle\partial_m(\partial^k_{\lambda}\mathcal{R}^I_{\mathfrak{R},{\omega},\bb{w}}(\lambda)
\bb{f})_j,\partial_m\partial_j\Phi\rangle_{\Omega_{R+3}}
\end{align*}
where $\langle\cdot,\cdot\rangle_{\partial\Omega_{R+3}}$ and $\langle\cdot,\cdot\rangle_{\Omega_{R+3}}$ denotes the inner-product in $\partial\Omega_{R+3}$ and $\Omega_{R+3}$.
Hence, by \eqref{est.bdd3}, \eqref{est-3-add-2} and the interpolation inequality
$$\|g\|_{L^p(\partial\Omega_{R+3})}\leq \|\nabla g\|^{1-(1/p)}_{L^p(\Omega_{R+3})}\|g\|^{1/p}_{L^p(\Omega_{R+3})}+\|g\|_{L^p(\Omega_{R+3})},$$
we have
\begin{equation}\label{est-3-add-1}
\|\mathring{\mathcal{Q}}_{\Omega_{R+3}}\Delta\partial^k_{\lambda}\mathcal{R}^I_{\mathfrak{R},
{\omega},\bb{w}}(\lambda)\bb{f}\|_{L^p(\Omega_{R+3})}
\leq C_{\theta,\Rr^*,{\omega}^*}(1+|\lambda|)^{-k-((1/2)-1/(2p))}.
\end{equation}
This inequality together with
\begin{align*}
&\|\mathring{\mathcal{Q}}_{\Omega_{R+3}}(L_{\mathfrak{R},{\omega},\bb{w}}+\Delta)\partial^k_{\lambda}\mathcal{R}^I_{\mathfrak{R},{\omega},\bb{w}}(\lambda)\bb{f}\|_{L^p(\Omega_{R+3})}\\
\leq &\|\nabla \mathring{\mathcal{Q}}_{\Omega_{R+3}}(L_{\mathfrak{R},{\omega},\bb{w}}+\Delta)\partial^k_{\lambda}\mathcal{R}^I_{\mathfrak{R},{\omega},\bb{w}}(\lambda)\bb{f}\|_{\mathbb{L}^p(\Omega_{R+3})}\\
\leq &\|\partial^k_{\lambda}\mathcal{R}^I_{\mathfrak{R},{\omega},\bb{w}}(\lambda)\bb{f}\|_{\mathbb{W}^{1,p}(\Omega_{R+3})}\leq C_{\theta,\Rr^*,{\omega}^*}(1+|\lambda|)^{-(1/2)-k}\|\bb{f}\|_{\LL^p(\Omega_{R+3})},
\end{align*}
 yields \eqref{est.bdd5}.

Now, we are going to prove \eqref{est3-1-2}-\eqref{est3-1-3}.
\vskip 0.1cm
\noindent\textbf{Case 1: $\lambda\in \Sigma_{\theta,\ell_1}$}. \quad Observe from \eqref{est3-1-1} that
\begin{align}\nonumber
&\|(\mathcal{L}_{\mathfrak{R},{\omega},\bb{w},\Omega_{R+3}}+\Delta_{\Omega_{R+3}})(\lambda I-\Delta_{\Omega_{R+3}})^{-1}
\|_{\mathcal{L}(\mathbb{J}^p(\Omega_{R+3}))}\\
\leq& C_{\theta,p,R}(1+|\lambda|)^{-\frac12}(\mathfrak{R}^*+{\omega}^*+\normmm{\bb{w}}_{\varepsilon,\Omega})
\|\bb{g}\|_{\mathbb{L}^p(\Omega_{R+3})},\quad \lambda \in \Sigma_{\theta}\cup\{0\}.\label{est-3-2}
\end{align}
Hence, Choosing $\ell_1,\eta>0$ such that $C_{\theta,R,p}(\mathfrak{R}^*+{\omega}^*)(1+\ell_1)^{-\frac12}<\frac12$
 and $C_{\theta,R,p}(1+\ell_1)^{-\frac12}\eta<\frac12$, we deduce that
 \[\|(\mathcal{L}_{\mathfrak{R},{\omega},\bb{w},\Omega_{R+3}}+\Delta_{\Omega_{R+3}})(\lambda I-\Delta_{\Omega_{R+3}})^{-1}
\|_{\mathcal{L}(\mathbb{J}^p(\Omega_{R+3}))}< 1,\quad \lambda\in \Sigma_{\theta,\ell_1}\]
only if  $\normmm{\bb{w}}_{\varepsilon,\Omega}\leq \eta$. This, together with the Neumann series proves \eqref{est3-1-2}-\eqref{est3-1-3} for  every $\lambda\in \Sigma_{\varepsilon,\ell_1}$.
\vskip 0.1cm
\noindent\textbf{Case 2: $\lambda\in \overline{\C_+}$ and $|\lambda|\leq \ell_1$}.\quad
Assume that \eqref{est3-1-2} holds, then we have from \eqref{est3-1-1}
$$\big(I+(\mathcal{L}_{\mathfrak{R},{\omega},\bb{w},\Omega_{R+3}}+\Delta_{\Omega_{R+3}})(\lambda I-\Delta_{\Omega_{R+3}})^{-1}\big)^{-1}
   \in C(\{\lambda\in\overline{\C_+}; \mathcal{L}(\mathbb{J}^p(\Omega_{R+3}))).$$
This deduces \eqref{est3-1-3} for $\lambda\in \overline{\C_+}$ and $|\lambda|\leq \ell_1$.

In what follows, we are going to prove \eqref{est3-1-2}. Thanks to \eqref{est3-1-1}, we easily verify that $(\lambda I-\Delta_{\Omega_{R+3}})^{-1}$
is compact from $\mathbb{J}^p(\Omega_{R+3})$ to $\W^{1,p}(\Omega_{R+3})\cap \mathbb{J}^p(\Omega_{R+3})$.
This implies that $(\mathcal{L}_{\mathfrak{R},{\omega},\bb{w},\Omega_{R+3}}+\Delta_{\Omega_{R+3}})(\lambda I-\Delta_{\Omega_{R+3}})^{-1}$ is compact from  $\mathbb{J}^p(\Omega_{R+3})$ to itself
 since $\mathcal{L}_{\mathfrak{R},{\omega},\bb{w},\Omega_{R+3}}+\Delta_{\Omega_{R+3}}\in
\mathcal L(\W^{1,p}(\Omega_{R+3})\cap \mathbb{J}^p(\Omega_{R+3}), \mathbb{J}^p(\Omega_{R+3}))$. Hence, to prove \eqref{est3-1-2},  it suffices to prove the injectivity of
$I+(\mathcal{L}_{\mathfrak{R},{\omega},\bb{w},\Omega_{R+3}}+\Delta_{\Omega_{R+3}})(\lambda I-\Delta_{\Omega_{R+3}})^{-1}$
 by the Fredholm alternative theorem.

Let $\bb{g}\in \mathbb{J}^p(\Omega_{R+3})$  satisfy
$(I+(\mathcal{L}_{\mathfrak{R},{\omega},\bb{w},
\Omega_{R+3}}+\Delta_{\Omega_{R+3}})(\lambda-\Delta_{\Omega_{R+3}})^{-1})\bb{g}=\bb{0}$. Obviously,  $\bb{v}=(\lambda I-\Delta_{\Omega_{R+3}})^{-1}\bb{g}$ satisfies
\begin{equation}\label{est-3-4}
(\lambda I-\Delta_{\Omega_{R+3}})\bb{v}=-(\mathcal{L}_{\mathfrak{R},{\omega},
\bb{w},\Omega_{R+3}}+\Delta_{\Omega_{R+3}})\bb{v}.
\end{equation}
If $p\geq 2$, then $\bb{v}\in {\W}^{2,2}(\Omega_{R+3})$ and
 $\theta\in \mathring{W}^{1,2}(\Omega_{R+3})$.
If $p<2$, by the classical theory  for the Stokes system in bounded domain with Dirichlet boundary and the bootstrap argument, we deduce  $\bb{v}\in {\W}^{2,2}(\Omega_{R+3})$. Thus, thanks to that $\Div \bb{v}=0$, we deduce
\begin{align*}
0=\lambda\|\bb{v}\|^2_{\LL^2(\Omega_{R+3})}+\|\nabla \bb{v}\|^2_{\LL^2(\Omega_{R+3})}+\langle(L_{\mathfrak{R},{\omega},\bb{w}}+\Delta)\bb{v},\bb{v}
\rangle_{\Omega_{R+3}}.
\end{align*}
By a simple calculation, we obtain
\begin{align*}
\mathop{\rm Re}\,\langle (L_{\mathfrak{R},{\omega},\bb{w}}+\Delta)\bb{v},\bb{v}\rangle_{\Omega_{R+3}}=&-\langle\mathop{\rm Re}\bb{v}\cdot\nabla
\mathop{\rm Re}\bb{v}+\mathop{\rm Im}\bb{v}\cdot\nabla \mathop{\rm Im}\bb{v}, \bb{w}\rangle.
\end{align*}
Since $\bb{v}|_{\partial\Omega_{R+3}}=\bb{0}$, we have by Poincare's inequality
\[|\mathop{\rm Re}\,\langle(L_{\mathfrak{R},{\omega},\bb{w}}+\Delta_{\Omega_{R+3}})\bb{v},\bb{v}\rangle_{\Omega_{R+3}}|
\leq C_R\normmm{\bb{w}}_{\varepsilon,\Omega}\|\nabla \bb{v}\|^2_{\LL^2(\Omega_{R+3})}.\]
This implies $\|\nabla \bb{v}\|_{\mathbb{L}^2(\Omega_{R+3})}=0$ only if $C_R\normmm{\bb{w}}_{\varepsilon,\Omega}<1$. Since $\bb{v}|_{\partial\Omega_{R+3}}=\bb{0}$,  we have $\bb{v}=\bb 0$, and so $\bb{g}=0$ in $\Omega_{R+3}$.
 Thus, $I+(\mathcal{L}_{\mathfrak{R},{\omega},\bb{w},\Omega_{R+3}}+\Delta_{\Omega_{R+3}})(\lambda I-\Delta_{\Omega_{R+3}})^{-1}$ is a injection from $\mathbb{J}^p(\Omega_{R+3})$ to itself.  This completes the proof of Theorem \ref{TH3-1}.
\end{proof}

Consider the resolvent problem associated
 with the dual operator $L^*_{\mathfrak{R},{\omega},\bb{w}}$ of $L_{\mathfrak{R},{\omega},\bb{w}}$,
\begin{equation}\label{RB*}
(\lambda I+L^*_{\mathfrak{R},{\omega},\bb{w}})\bb{v}+\nabla(\mathring{\Theta}+\mathring{
\mathcal{Q}}_{\Omega_{R+3}}\bb{f})=\bb{f},\quad \Div \bb{v}=0 \text{ in }\Omega_{R+3},\quad \bb{v}|_{\partial\Omega_{R+3}}=\bb 0,
\end{equation}
In the same way as proving  Theorem \ref{TH3-1}, we have
\begin{theorem}\label{TH3-2}
Under the assumption of Theorem \ref{TH3-1}, there exist two constant $\eta,\ell_1>0$, depending on
 $R,p,\theta,\mathfrak{R}^*,{\omega}^*$, such that   if $\normmm{\bb{w}}_{\varepsilon,\Omega}\leq \eta$,
then for every $\bb{f}\in \mathbb{L}^p(\Omega_{R+3})$, problem \eqref{RB*} admits a  unique solution
 $$(\bb{v},\mathring{\Theta}) =(\mathcal{R}^{I^*}_{\mathfrak{R},{\omega},\bb{w}}(\lambda)\bb{f},
 \mathring{\Pi}^{I^*}_{\mathfrak{R},{\omega},\bb{w}}(\lambda)\bb{f})$$
 ssatisfying
\begin{align*}
&\mathcal{R}^{I^*}_{\mathfrak{R},{\omega},\bb{w}}(\lambda)\in \mathscr{A}(\Sigma_{\theta,\ell_1}\cup\overline{\C_+},
\mathcal{L}(\mathbb{L}^p(\Omega_{R+3}),{\W}^{2,p}(\Omega_{R+3})\cap \mathbb{J}^{p}(\Omega_{R+3}))),\\
&\mathring{\Pi}^{I^*}_{\mathfrak{R},{\omega},\bb{w}}(\lambda)\in \mathscr{A}(\Sigma_{\theta,\ell_1}\cup\overline{\C_+},
\mathcal{L}(\mathbb{L}^p(\Omega_{R+3}),\mathring{W}^{1,p}(\Omega_{R+3}))),
\end{align*}
and for every $k\geq 0$ and $\lambda\in \Sigma_{\theta,\ell_1}\cup\overline{\C_+}$.
\begin{align}
&\|\nabla^j(\partial^{k}_{\lambda}\mathcal{R}^{I^*}_{\mathfrak{R},{\omega},\bb{w}})(\lambda)
\|_{\mathcal{L}(\mathbb{L}^p(\Omega_{R+3}))}\leq C_{k,\theta,R,\mathfrak{R}^*,{\omega}^*}(1+|\lambda|)^{-1-k+(j/2)},\quad j=0,1,2,\label{est.bdd3'}\\
&\|\nabla(\partial^{k}_{\lambda} \mathring{\Pi}^{I^*}_{\mathfrak{R},{\omega},\bb{w}})(\lambda)\|_{\mathcal{L}(\mathbb{L}^p
(\Omega_{R+3}))}\leq C_{k,\theta,R,\mathfrak{R}^*,{\omega}^*}(1+|\lambda|)^{-k},\label{est.bdd4'}\\
&\|(\partial^{k}_{\lambda}\mathring{\Pi}^{I^*}_{\mathfrak{R},{\omega},\bb{w}})(\lambda)\|_{\mathcal{L}
(\mathbb{L}^p(\Omega_{R+3}),L^p(\Omega_{R+3}))}
\leq  C_{k,\theta,R,\mathfrak{R}^*,{\omega}^*}(1+|\lambda|)^{-k-(1/2)-(1/(2p))}.\label{est.bdd5'}
\end{align}
\end{theorem}
%%%%%%%%%%%%%%%%%%%%%%%%%%%%%%%%%%%%%%%%%%%%%%%%%%%%%%%%%%%%%%%%%%%%%%%%%%%%%%%%
\section{Resolvent problem in an exterior domain}\label{sec.4}
\setcounter{section}{4}\setcounter{equation}{0}

In this section, we will  construct the solution operators of  the resolvent problem in
the exterior domain $\Omega$:
\begin{equation}\label{ER}
(\lambda I+L_{\mathfrak{R},{\omega},\bb{w}})\bb{u}+\nabla p=\bb{f}\in \LL^p_{R+2}(\Omega),\quad \Div \bb{u}=0\text{ in }\Omega,\quad \bb{u}|_{\partial\Omega}=\bb{0}
\end{equation}
by means of the cut-off technique. To recover the divergence free on $\bb{u}$
destroyed by cut-off technique, we will invoke the Bogovski\v{i} operator
 given in \cite{Bo79,Bo80}. To state it, we define
\[ W^{0,p}_0(E)=L^p(E),\quad W^{m,p}_{0}(E)=\{g\in W^{m,p}(E)\,|\, \partial^{\alpha}_x g|_{\partial E}=0,
\;\;|\alpha|\leq m-1\}.\]

\begin{lemma}[Bogovski\v{i} operators]\label{Lem.Bogo}
Let $p\in (1,\infty)$ and  $E\subset \R^3$.
\begin{enumerate}
	\item[{\rm (1)}] Assume that $E$ is a bounded Lipschitz domain, then there exists a bounded linear
	operator $\mathbb{B}$ from $W^{m,p}_0(E)$ to ${\W}^{m+1,p}(\R^3)$
	such   that $\mathop{\rm supp}\mathbb{B}[f]\subset E$ and
	\begin{equation}\label{est.Bogo1}
	\|\mathbb{B}[g]\|_{{\W}^{m+1,p}(\R^3)}\leq C_{m,p}\|g\|_{W^{m,p}(E)}
	\end{equation}
	for every integer $m\geq 0$. In addition, if $g\in W^{m,p}_0(E)$ satisfies $\int_{E}g\,\mathrm{d}x=0$,
	then
$\Div \mathbb{B}[g]=g$ in $E$ and $\Div \mathbb{B}[g]=0$ in $\R^3\setminus E$.
	\item[{\rm (2)}] Let $m$ be a positive integer. Let $\phi\in C^{\infty}(\R^3)$ such that $\nabla \phi$ has a compact support and $\mathop{\rm supp}\nabla \phi\subset E$. If $\bb{u}\in {\W}^{m,p}(E)$ satisfies $\Div \bb{u}=0$ in $E$ and $\bb{\nu}\cdot \bb{u}|_{\partial E}=0$ where $\bb{\nu}$ is the unit outer normal vector of $E$, then $\nabla \phi\cdot \bb{u}\in W^{m,p}_0(\mathop{\rm supp}\nabla \phi)$ and $\int_{\mathop{\rm supp}\nabla \phi} \nabla\phi\cdot \bb{u}\,\mathrm{d}x=0$.
\item[{\rm (3)}] Let $\phi$ be the same function given in {\rm (2)}, then
	\begin{equation}\label{est.Bogo2}
	\begin{split}
	&\|\mathbb{B}[\nabla\phi\cdot \bb{u}]\|_{\mathbb{W}^{j,p}(\R^3)}\leq C_{p,R}\|\bb{u}\|_{\mathbb{\W}^{j-1}(\mathop{\rm supp \nabla \phi})},\quad j=1,2,\\
	&\|\mathbb{B}[\nabla\phi\cdot \nabla g]\|_{\mathbb{W}^{j,p}(\R^3)}\leq C_{p,R}\|g\|_{W^{j,p}(\mathop{\rm supp \nabla \phi})},\quad j=0,1,2.
	\end{split}
	\end{equation}
\end{enumerate}
\end{lemma}

Let $\varphi\in C^{\infty}_0(\R^3)$ satisfy $0\leq \varphi\leq 1$, $\varphi=1$ in $B_{R+1}$ and $\varphi=0$ in $B^c_{R+2}$. Denote $\bb{f}_{\Omega_{R+3}}$ and $\bb{f}_0$ by the restriction of $\bb{f}$ on $\Omega_{R+3}$ and the zero extension of $\bb{f}$ to $\R^3$, respectively. Then, we can construct a parametrix $(\Phi_{\mathfrak{R},{\omega},\bb{w}}(\lambda)\bb{f},\Psi_{\mathfrak{R},{\omega},\bb{w}}(\lambda)\bb{f})$ of \eqref{ER} as follows:
\begin{equation}\label{Op-Pa}
\left\{\begin{aligned}
&\Phi_{\mathfrak{R},{\omega},\bb{w}}(\lambda)\bb{f}=
(1-\varphi)\mathcal{R}^G_{\mathfrak{R},{\omega},\overline{\bb{w}}}(\lambda)\bb{f}_0
+\varphi \mathcal{R}^I_{\mathfrak{R},{\omega},\bb{w}}(\lambda)\bb{f}_{\Omega_{R+3}}
+\B[\nabla \varphi\cdot \mathcal{D}(\lambda)\bb{f}],\\
&\Psi_{\mathfrak{R},{\omega},\bb{w}}(\lambda)\bb{f}=(1-\varphi)(\mathring{\mathcal{Q}}_{\R^3}
+\mathring{\Pi}^G_{\mathfrak{R},{\omega},\overline{\bb{w}}}(\lambda))\bb{f}_0+\varphi(\mathring{
\mathcal{Q}}_{\Omega_{R+3}}+\mathring{\Pi}^I_{\mathfrak{R},
{\omega},\bb{w}}(\lambda))\bb{f}_{\Omega_{R+3}},
\end{aligned}\right.
\end{equation}
with
$$\mathcal{D}(\lambda)\bb{f}
=\mathcal{R}^G_{\mathfrak{R},{\omega},
\overline{\bb{w}}}(\lambda)\bb{f}_0-\mathcal{R}^I_{\mathfrak{R},{\omega},\bb{w}}(\lambda)\bb{f}_{\Omega_{R+3}}.$$
One easily verifies that
\begin{equation}\label{eq.pa}
\begin{cases}
(\lambda I+L_{\mathfrak{R},{\omega},\bb{w}})\Phi_{\mathfrak{R},{\omega},\bb{w}}(\lambda)\bb{f}+\nabla \Psi_{\mathfrak{R},{\omega},\bb{w}}(\lambda)\bb{f}=(I+T+K_{\mathfrak{R},{\omega},\bb{w}}(\lambda))\bb{f}
\quad \text{in }\Omega,\\
\Div \Phi_{\mathfrak{R},{\omega},\bb{w}}(\lambda)\bb{f}=0\quad\text{in }\Omega,\quad \Phi_{\mathfrak{R},{\omega},\bb{w}}(\lambda)\bb{f}|_{\partial\Omega}=\bb{0},
\end{cases}
\end{equation}
where
\begin{align}
&T\bb{f}=-\nabla \varphi\cdot\big(\mathring{\mathcal{Q}}_{\R^3}\bb{f}_0
-\mathring{\mathcal{Q}}_{\Omega_{R+3}}\bb{f}_{\Omega_{R+3}}\big)-
\B[\nabla\varphi\cdot\nabla(\mathring{\mathcal{Q}}_{\R^3}\bb{f}_0-\mathring{\mathcal{Q}}_{\Omega_{R+3}}\bb{f}_{\Omega_{R+3}}) ],\label{est-4-1}\\
&K_{\mathfrak{R},{\omega},\bb{w}}(\lambda)\bb{f}=(\Delta\varphi)\mathcal{D}(\lambda)\bb{f}+2\nabla\varphi\cdot\nabla \mathcal{D}(\lambda)\bb{f}+\mathfrak{R}(\partial_1\varphi)
\mathcal{D}(\lambda)\bb{f}\notag\\
&\quad\quad\quad\quad\quad\quad+{\omega}\big((\bb{e}_1\times  {x})\cdot\nabla \varphi\big)\mathcal{D}(\lambda)\bb{f}-(\bb{w}\cdot\nabla \varphi)\mathcal{D}(\lambda)\bb{f}+L_{\mathfrak{R},{\omega},\bb{w}}\B[\nabla \varphi\cdot\mathcal{D}(\lambda)\bb{f}]\notag\\
&\quad\quad\quad\quad\quad\quad
-\B[\nabla \varphi\cdot(L_{\mathfrak{R},{\omega},\bb{w}}\mathcal{D}(\lambda)
\bb{f})]-\B[\nabla \varphi\cdot\nabla {\Xi}(\lambda)\bb{f}]-\nabla\varphi\cdot {\Xi}(\lambda)\bb{f},\label{est-4-2}\\
&{\Xi}(\lambda)\bb{f}
=\mathring{\Pi}^G_{\mathfrak{R},{\omega},\overline{\bb{w}}}(\lambda)\bb{f}_0-\mathring{\Pi}^I_{\mathfrak{R},{\omega},
\bb{w}}(\lambda)\bb{f}_{\Omega_{R+3}}\notag
\end{align}
We see that $T+K_{\mathfrak{R},{\omega},\bb{w}}(\lambda)$ is a compact operator from $\mathbb{L}^p_{R+2}(\Omega)$ to itself, and
\begin{equation}\label{est-4-3}
\left\{\begin{aligned}
&K_{\mathfrak{R},{\omega},\bb{w}}(\lambda)\in \mathscr{A}(\C_+,\mathcal{L}(\LL^p_{R+2}))\cap C(\overline{\C_+},\mathcal{L}(\LL^p_{R+2}(\Omega))),\\
&\|K_{\mathfrak{R},{\omega},\bb{w}}(\lambda)\|_{\mathcal{L}(\mathbb{L}^p_{R+2}(\R^3))}
\leq C_{R,\mathfrak{R}_*,\mathfrak{R}^*,{\omega}^*}(1+|\lambda|)^{-1/2+1/(2p)},\quad\lambda\in\overline{\C_+}.
\end{aligned}\right.
\end{equation}
In fact, by Poincare's inequality and Lemma \ref{Lem.Bogo}, we know
$T$ is a compact operator from $\mathbb{L}^p_{R+2}(\Omega)$ to itself. In addition, from  Lemma \ref{Lem.Bogo}, we have
\begin{align}\nonumber
&\|K_{\mathfrak{R},{\omega},\bb{w}}(\lambda)\|_{\mathcal{L}(\mathbb{L}^p_{R+2}(\R^3))}\\
\lesssim&_{R,\Rr^*,{\omega}^*}(\|\mathcal{D}(\lambda)\|_{\mathcal{L}(\LL^p_{R+2}(\Omega),\W^{1,p}(\Omega_{R+2}))}
+\|{\Xi}(\lambda)\|_{\mathcal{L}(\LL^p_{R+2}(\Omega),
L^p(\Omega_{R+2}))}).\label{est-4-3-add}
\end{align}
which together with Theorem \ref{TH2-2}, Corollary \ref{Cor.G2} and Theorem \ref{TH3-1} implies that $K_{\mathfrak{R},{\omega},\bb{w}}(\lambda)$ is a compact operator from $\mathbb{L}^p_{R+2}(\Omega)$ to itself and satisfies \eqref{est-4-3}.

If $(I+T+K_{\mathfrak{R},{\omega},\bb{w}}(\lambda))^{-1}\in
\mathcal{L}(\mathbb{L}^p_{R+2}(\Omega))$,
 we  construct the solution operators
\begin{equation}\label{op-S}\left\{
\begin{aligned}
&\mathcal{R}_{\mathfrak{R},{\omega},\bb{w}}(\lambda)=\Phi_{\mathfrak{R},{\omega},\bb{w}}
(\lambda)(I+T+K_{\mathfrak{R},{\omega},\bb{w}}(\lambda))^{-1},\\
&\Pi_{\mathfrak{R},{\omega},\bb{w}}(\lambda)=\Psi_{\mathfrak{R},{\omega},\bb{w}}(\lambda)
(I+T+K_{\mathfrak{R},{\omega},\bb{w}}(\lambda))^{-1},
\end{aligned}\right.
\end{equation}
such that
$\bb{u}=\mathcal{R}_{\mathfrak{R},{\omega},\bb{w}}(\lambda)\bb{f}$
and $P=\Pi_{\mathfrak{R},{\omega},\bb{w}}(\lambda)\bb{f}$
satisfy \eqref{ER} provided $\bb{f}\in \mathbb{L}^p_{R+2}(\Omega)$.
\vskip 0.3cm

Now we are in position to prove the invertibility of $I+T+K_{\mathfrak{R},{\omega},\bb{w}}(\lambda)$ for all $\lambda\in \overline{\C_+}$.
\begin{proposition}\label{Pro-4-1}
Let $p\in (1,\infty)$, $\varepsilon\in(0,\frac12)$, $0<\mathfrak{R}_*\leq |\mathfrak{R}|\leq \mathfrak{R}^*$ and $|{\omega}|\leq {\omega}^*$.  Then, there exists a constant $\eta=\eta_{p,R,\mathfrak{R}_*,\mathfrak{R}^*,{\omega}^*}>0$ such that if
$\normmm{\bb{w}}_{\varepsilon,\Omega}\leq \eta$,
then, $(I+T+K_{\mathfrak{R},{\omega},\bb{w}}(\lambda))^{-1}\in\mathscr{A}(\C_+,\mathcal{L}(\LL^p_{R+2}))\cap C(\overline{\C_+},\mathcal{L}(\LL^p_{R+2}(\Omega)))$ satisfies
\begin{equation}\label{est.pa}
\|(I+T+K_{\mathfrak{R},{\omega},\bb{w}}(\lambda))^{-1}\|_{\mathcal{L}(\mathbb{L}^p_{R+2}(\Omega))}\leq C_{R,\mathfrak{R}_*,\mathfrak{R}^*,{\omega}^*},\quad \lambda\in \overline{\C_+}.
\end{equation}
\end{proposition}
We begin the proof with an known lemma.
\begin{lemma}[\cite{HS09}]\label{Lem-4-1} Let $p\in (1,\infty)$, and
 $T$ be the operator defined in \eqref{est-4-1}. Then,
 $(I+T)^{-1}\in \mathcal{L}(\mathbb{L}^p_{R+2}(\Omega))$ satisfies
$\|(I+T)^{-1}\|_{\mathcal{L}(\mathbb{L}^p_{R+2}(\Omega))} \leq C_{R}$.
\end{lemma}
\begin{proof}[Proof of Proposition \ref{Pro-4-1}.]
Define $\C_{\ell}\triangleq\{\lambda\in \C_+\,|\,|\lambda|\geq\ell\}$.
We will divide into two cases to prove this proposition.
\vskip0.15cm
\noindent\textbf{Case 1: $\lambda\in \overline{\C_{\ell_0}}$}.\quad
From \eqref{est-4-3}  there exists a $\ell_0=C_{p,R,\mathfrak{R}_*,\mathfrak{R}^*,{\omega}^*}$, such that
\begin{equation}\label{est-4-4}
\|(I+T)^{-1}K_{\mathfrak{R},{\omega},\bb{w}}(\lambda)\|_{\mathcal{L}(\mathbb{L}^p_{R+2}(\Omega))}
\leq 1/2,\quad \lambda\in \overline{\C_{\ell_0}}.
\end{equation}
Hence, we have from the Neumann series expansion that
$$
(I+T+K_{\mathfrak{R},{\omega},\bb{w}}(\lambda))^{-1}=\sum\limits^{\infty}_{j=0}(-(I+T)^{-1}
K_{\mathfrak{R},{\omega},\bb{w}}(\lambda))^j(I+T)^{-1},$$
satisfies
$$\left\{\begin{aligned}
&(I+T+K_{\mathfrak{R},{\omega},\bb{w}}(\lambda))^{-1}\in
 \mathscr{A}(\C_{\ell_0},\mathcal{L}(\mathbb{L}^p_{R+2}(\Omega)))\cap
  C(\overline{\C_{\ell_0}},\mathcal{L}(\mathbb{L}^p_{R+2}(\Omega))),\\
&\|(I+T+K_{\mathfrak{R},{\omega},\bb{w}}(\lambda))^{-1}\|_{\mathcal{L}(\mathbb{L}^p_{R+2}(\Omega))}
\leq C_R,\quad \lambda\in \overline{\C_{\ell_0}}.
\end{aligned}\right.$$

\noindent\textbf{Case 2: $\lambda\in \overline{\C_{+}}$ and $|\lambda|\leq \ell_0$}.\quad
If \begin{equation}\label{est-4-555}
(I+T+K_{\mathfrak{R},{\omega},\bb{w}}(\lambda))^{-1}\in \mathcal{L}(\mathbb{L}^p_{R+2}(\Omega)),
\end{equation}
we have  $(I+T+K_{\mathfrak{R},{\omega},\bb{w}}(\lambda))^{-1} \in C(\overline{\C_+},\mathcal{L}(\mathbb{L}^p_{R+2}(\Omega)))$ by  \eqref{est-4-3}. This implies that \eqref{est.pa} holds for all $\lambda\in \overline{\C_{+}}$ with $|\lambda|\leq \ell_0$.

Thus, in what follows, we will focus on the proof of \eqref{est-4-555}. Since $T+K_{\mathfrak{R},{\omega},\bb{w}}(\lambda)\in \mathcal{L}(\LL^p_{\R+2}(\Omega))$ is a compact operator from $\mathbb{L}^p_{R+2}(\Omega)$ to itself, It suffices to prove that $I+T+K_{\mathfrak{R},{\omega},\bb{w}}(\lambda)$ is a injection by the Fredholm alternative theorem. For this purpose,  we will prove that, given $\lambda\in \overline{C_+}$, if
$\bb{f}\in \mathbb{L}^p_{R+2}(\Omega)$ satisfies
$$(I+T+K_{\mathfrak{R},{\omega},\bb{w}}(\lambda))\bb{f}=\bb{0},$$
then $\bb{f}=\bb 0$.  For such $\bb{f}$, we construct by \eqref{eq.pa}
 $$(\bb{u}, P)=\big(\Phi_{\mathfrak{R},{\omega},\bb{w}}(\lambda)\bb{f},
 \Psi_{\mathfrak{R},{\omega},\bb{w}}(\lambda)\bb{f}\big)$$
 such that
\begin{equation}\label{eq.pa1}
(\lambda-\Delta)\bb{u}+\nabla P=-(L_{\mathfrak{R},{\omega},\bb{w}}+\Delta)\bb{u},\quad\Div \bb{u}=0 \text{ in}\;\Omega,\quad \bb{u}|_{\partial\Omega}=\bb{0}.
\end{equation}

We claim that
\begin{equation}\label{est-4-6}
\mathop{\rm Re}\lambda\|\bb{u}\|^2_{\mathbb{L}^2(\Omega)}+\|\nabla \bb{u}\|_{\mathbb{L}^2(\Omega)}=-\mathop{\rm Re}\,\langle \bb{u}\cdot\nabla \bb{w},\bb{u}\rangle_{\Omega}.
\end{equation}
In fact, from the local solvability theory of the classical Stokes equations
  with non-slip boundary condition and
 the bootstrap argument used in the proof of Theorem \ref{TH3-1},
 we obtain that $\bb{u}\in {\W}^{2,2}_{loc}(\overline{\Omega})$
 and $P\in W^{1,2}_{loc}(\overline{\Omega})$.
 Choose a bump function $\phi\in C^{\infty}_0(\R)$ satisfying
 $\phi=1$ in $B_1$ and $\phi=0$ outside $B_2$.
 Multiplying \eqref{eq.pa1} by $\phi_{\ell}\bar{\bb{u}}$ with
$\phi_{\ell}(x)\triangleq\phi(\frac{|x|}{\ell})$, we get
\begin{align}
&\lambda\langle \bb{u},\phi_{\ell}\bb{u}\rangle_{\Omega}+\langle\nabla \bb{u},\phi_{\ell}\nabla \bb{u}\rangle_{\Omega}\nonumber\\
=&-\langle\nabla \bb{u},\nabla\phi_{\ell}\otimes \bb{u}\rangle_{\Omega}-\langle (L_{\mathfrak{R},{\omega},\bb{w}}+\Delta)\bb{u},\phi_{\ell}\bb{u}\rangle_{\Omega}
+\langle\widetilde{P},\nabla\phi_{\ell}\cdot \bb{u}\rangle_{\Omega}\label{est-con}
\end{align}
where
$$\widetilde{P}=P+\frac1{|\Omega_{R+3}|}\int_{\Omega_{R+3}}\mathcal{Q}_{\R^3}\bb{f}_0\,\mathrm{d}x+\frac1{|\Omega_{R+3}|}\int_{\Omega_{R+3}}\Pi^G_{\mathfrak{R},{\omega},\bb{w}}(\lambda)\bb{f}_0\,\mathrm{d}x.$$
Invoking that $\Div \bb{u}=\Div \bb{w}=0$, we easily calculate
\begin{align*}
\mathop{\rm Re}\,\langle (L_{\mathfrak{R},{\omega},\bb{w}}+\Delta)\bb{u},
\phi_{\ell}\bb{u}\rangle_{\Omega}=&
-\frac12\int_{\Omega}|\bb{u}|^2(\mathfrak{R}\partial_1\phi_{\ell}+\bb{w}\cdot\nabla\phi_{\ell})\,\mathrm{d}x
+\mathop{\rm Re}\,\langle \bb{u}\cdot\nabla \bb{w},\phi_{\ell} \bb{u}\rangle_{\Omega}\\
\triangleq& J_1+J_2.
\end{align*}
Since $\bb{u}=\mathcal{R}^G_{\Rr,{\omega},\overline{\bb{w}}}(\lambda)\bb{f}_0$ in $B^c_{2(R+2)}$, we have by \eqref{est.decay1}
\[|\nabla^k\bb{u}(x)|=O(|x|^{-1-k/2}(1+s_{\Rr}(x))^{-1}),\quad k=0,1,\;\;x\in B^c_{9(R+2)}.\]
This, combining with Lemma \ref{Lem.In}, yields for $\ell\geq 10(R+2)$,
\begin{align*}
&|\langle\nabla \bb{u},\nabla\phi_{\ell}\otimes \bb{u}\rangle_{\Omega}|
\le\ell^{-1}\int_{\ell\leq|x|\leq 2\ell}
|x|^{-\frac52}\,\mathrm{d}x=O(\ell^{-\frac12}),\\
&|J_1|\le \ell^{-1}\int_{\ell\leq|x|\leq 2\ell}
\Big(\frac1{|x|^{2}(1+s_{\Rr}(x))}+\frac{\normmm{\bb{w}}_{\varepsilon,\Omega}}{|x|^3}\Big)\,
\mathrm{d}x=O(\ell^{-\frac12}),\\
&|J_2|\le \normmm{\bb{w}}_{\varepsilon,\Omega}\Big(\int_{B_{9(R+2)}\cap \Omega}|\bb{u}|^2\,\mathrm{d}x+\int_{B^c_{9(R+2)}}|x|^{-\frac72}\,\mathrm{d}x\Big)<\infty.
\end{align*}
Moreover, since $\bb{f}_0=0$ in $B^c_{R+2}$ and the kernel
function of $\mathcal{Q}_{\R^3}$ is bounded by $|x|^{-2}$, we obtain for $ x\in B^c_{9(R+2)}$
\[\big|[\mathcal{Q}_{\R^3}\bb{f}_0](x)\big|\leq
\int_{|y|\leq R+2}\frac{|\bb{f}_0(y)|}{|x-y|^2}\,\mathrm{d}y
\leq\int_{|x-y|\geq \frac{8}{9}|x|}\frac{|\bb{f}_0(y)|}{|x-y|^2}\,
\mathrm{d}y=O(|x|^{-2}).\]
This, together with \eqref{est.decay1'} and
$\widetilde{P}=\mathcal{Q}_{\R^3}\bb{f}_0+\Pi^G_{\mathfrak{R},{\omega},\bb{w}}(\lambda)\bb{f}_0$
in $B^c_{R+2}$, gives
\begin{equation}\label{est-con-1}
|\widetilde{P}(x)|=O(|x|^{-3/2}),\quad x\in B^c_{10(R+2)},
\end{equation}
which implies
\[\big|\langle \widetilde{P},\bb{u}\cdot\nabla \phi_{\ell}\rangle_{\Omega}\big|
\le \ell^{-1}\int_{\ell\leq|x|\leq 2\ell}|x|^{-5/2}\,\mathrm{d}x=O(\ell^{-1/2}).\]
Hence, letting $\ell\to \infty$ in  \eqref{est-con},
 we prove \eqref{est-4-6} by the Lebesgue dominated convergence theorem.

Owing to $\Div \bb{u}=\Div\bb{w}=0$ and $\bb{u}|_{\partial\Omega}=\bb{0}$, one gets
\[\mathop{\rm Re}\,\langle \bb{u}\cdot\nabla \bb{w}, \bb{u}\rangle_{\Omega}
=-\int_{\Omega}(\mathop{\rm Re}\bb{u}\cdot\nabla)\mathop{\rm Re}\bb{u}\cdot
\bb{w}+(\mathop{\rm Im}\bb{u}\cdot\nabla)\mathop{\rm Im}\bb{u}\cdot \bb{w}\,\mathrm{d}x.\]
This identity, together  with \eqref{est-4-6} and  Lemma \ref{Lem.Hardy}, implies
\[\mathop{\rm Re}\lambda\|\bb{u}\|^2_{\mathbb{L}^2(\Omega)}+\|\nabla \bb{u}\|^2_{\mathbb{L}^2(\Omega)}
\leq  C\normmm{\bb{w}}_{\varepsilon,\Omega}\|\nabla\bb{u}\|^{2}_{\mathbb{L}^2(\Omega)}.\]
Hence, we deduce $\nabla \bb{u}=\bb{0}$ in $\Omega$ only if $C\normmm{\bb{w}}_{\varepsilon,\Omega}<1$. This fact, combining with \eqref{eq.pa1} and  \eqref{est-con-1},
 we get $\widetilde{P}=0$ in $\Omega$. Thus, we have from \eqref{Op-Pa}
\begin{equation}\label{est-4-8}
\left\{\begin{aligned}
&(1-\varphi)\mathcal{R}^G_{\mathfrak{R},{\omega},\overline{\bb{w}}}(\lambda)\bb{f}_0+\varphi \mathcal{R}^I_{\mathfrak{R},{\omega},\bb{w}}(\lambda)\bb{f}_{\Omega_{R+3}}+\B[\nabla \varphi\cdot \mathcal{D}(\lambda)\bb{f}]=\bb{0},\\
&(1-\varphi)(\mathring{\mathcal{Q}}_{\R^3}+\mathring{\Pi}^G_{\mathfrak{R},{\omega},\overline{\bb{w}}}(\lambda)
)\bb{f}_0
+\varphi(\mathring{\mathcal{Q}}_{\Omega_{R+3}}+\mathring{\Pi}^I_{\mathfrak{R},{\omega},\bb{w}}(\lambda))\bb{f}_{\Omega_{R+3}}\\
&\qquad\qquad\qquad\qquad\qquad=-\frac1{|\Omega_{R+3}|}\int_{\Omega_{R+3}}(\mathcal{Q}_{\R^3}+\Pi^G_{\mathfrak{R},{\omega},\overline{\bb{w}}}(\lambda))\bb{f}_0\,\mathrm{d}x,
\end{aligned}\right.
\end{equation}
which implies
\begin{equation}\left\{
\begin{aligned}\label{est-4-10}
&\mathcal{R}^{G}_{\mathfrak{R},{\omega},\overline{\bb{w}}}(\lambda)\bb{f}_0=\bb{0},\quad \mathcal{Q}_{\R^3}\bb{f}_0+\Pi^G_{\mathfrak{R},{\omega},\overline{\bb{w}}}\bb{f}_0=0\quad \text{in }\overline{B^c_{R+2}},\\
&\mathcal{R}^I_{\mathfrak{R},{\omega},\bb{w}}(\lambda)\bb{f}_{\Omega_{R+3}}=\bb{0}\quad \text{in }B_{R+1},\\
&(\mathring{\mathcal{Q}}_{\Omega_{R+3}}+\mathring{\Pi}^I_{\mathfrak{R},{\omega},\bb{w}}(\lambda))
\bb{f}_{\Omega_{R+3}}=\frac{-1}{|\Omega_{R+3}|}\int_{\Omega_{R+3}}(\mathcal{Q}_{\R^3}
+\Pi^G_{\mathfrak{R},{\omega},\overline{\bb{w}}}(\lambda))\bb{f}_0\,\mathrm{d}x \;\;\text{in }B_{R+1}.
\end{aligned}\right.\end{equation}
On the other hand, observing $\mathcal{R}^I_{\mathfrak{R},{\omega},\bb{w}}(\lambda)
\bb{f}_{\Omega_{R+3}}|_{\partial\Omega_{R+3}}=\bb{0}$, $\chi_{\Omega_{R+3}}\bb{f}_{\Omega_{R+3}}=\bb{f}_0$ and the first inequality in \eqref{est-4-10},
 we easily verify that
\begin{align*}
&\bb{v}_1=\mathcal{R}^G_{\mathfrak{R},{\omega},\overline{\bb{w}}}(\lambda)\bb{f}_0,\quad \theta_1=\mathcal{Q}_{\R^3}\bb{f}_0+\Pi^G_{\mathfrak{R},{\omega},\overline{\bb{w}}}(\lambda)\bb{f}_0,\\
&\bb{v}_2=\chi_{\Omega_{R+3}}\mathcal{R}^I_{\mathfrak{R},{\omega},\bb{w}}(\lambda)\bb{f}_{\Omega_{R+3}},\\
& \theta_2=\chi_{\Omega_{R+3}}(\mathring{\mathcal{Q}}_{\Omega_{R+3}}
+\mathring{\Pi}^I_{\mathfrak{R},{\omega},\bb{w}}(\lambda))\bb{f}_{\Omega_{R+3}}+\frac1{|\Omega_{R+3}|}\int_{\Omega_{R+3}}(\mathcal{Q}_{\R^3}+\Pi^G_{\mathfrak{R},{\omega},\overline{\bb{w}}}(\lambda))\bb{f}_0\,\mathrm{d}x,
\end{align*}
all solve
\begin{equation*}
(\lambda +L_{\mathfrak{R},{\omega},\overline{\bb{w}}})\bb{v}+\nabla \theta
=\bb{f}_0,\quad \Div \bb{v}=0 \quad\text{in }B_{R+3},\quad \bb{v}\big|_{\partial B_{R+3}}=\bb{0}.
\end{equation*}
By uniqueness, we have  $\bb{v}_1=\bb{v}_2$ and $\theta_1-\theta_2={\rm c}$ for some constant {\rm c}
in $B_{R+3}$, which implies in $\Omega_{R+3}$
\begin{align*}
&\mathcal{R}^I_{\mathfrak{R},{\omega},\bb{w}}(\lambda)\bb{f}_{\Omega_{R+3}}
=\mathcal{R}^G_{\mathfrak{R},{\omega},\overline{\bb{w}}}(\lambda)\bb{f}_0,\\
&\mathring{\mathcal{Q}}_{\R^3}\bb{f}_0
+\mathring{\Pi}^G_{\mathfrak{R},{\omega},\overline{\bb{w}}}(\lambda)\bb{f}_0-\mathring{\mathcal{Q}}_{\Omega_{R+3}}\bb{f}_{\Omega_{R+3}}
-\mathring{\Pi}^I_{\mathfrak{R},{\omega},\bb{w}}(\lambda)\bb{f}_{\Omega_{R+3}}={\rm c}.
\end{align*}
Plugging these equalities into \eqref{est-4-8}, we get
\begin{equation}\label{est-4-11}
\mathcal{R}^G_{\mathfrak{R},{\omega},\overline{\bb{w}}}(\lambda)\bb{f}_0=\bb{0},\quad \mathcal{Q}_{\R^3}\bb{f}_0+\Pi^G_{\mathfrak{R},{\omega},\overline{\bb{w}}}(\lambda)\bb{f}_0={\rm c}\quad \text{in }\Omega_{R+3}.\end{equation}
Since
\[{\rm c}|\Omega_{R+3}|=\int_{\Omega_{R+3}}\Big[\mathring{\mathcal{Q}}_{\R^3}\bb{f}_0
+\mathring{\Pi}^G_{\mathfrak{R},{\omega},\overline{\bb{w}}}\bb{f}_0-\mathring{\mathcal{Q}}_{\Omega_{R+3}}\bb{f}_{\Omega_{R+3}}
-\mathring{\Pi}^I_{\mathfrak{R},{\omega},\bb{w}}\bb{f}_{\Omega_{R+3}}\Big]\mathrm{d}x=0,\]
we deduce by \eqref{est-4-10}-\eqref{est-4-11}
\[\bb{f}=(\lambda+L_{\mathfrak{R},{\omega},\bb{w}})\mathcal{R}^G_{\mathfrak{R},{\omega},\overline{\bb{w}}}
(\lambda)\bb{f}_0
+\nabla(\mathcal{Q}_{\R^3}\bb{f}_0+\Pi^G_{\mathfrak{R},{\omega},\overline{\bb{w}}}(\lambda)\bb{f}_0)=\bb{0}\quad \text{in }\Omega,\]
and so complete the proof of Proposition \ref{Pro-4-1}.
\end{proof}
Next, we will discuss  the decay  of $(I+T+K_{\mathfrak{R},{\omega},\bb{w}}(\lambda))^{-1}$
with respect to $\lambda$.
\begin{proposition}\label{Pro-4-2}
Assume that $p\in (1,\infty)$, $\theta\in(0,\frac\pi2)$, $0< |\mathfrak{R}|\leq \mathfrak{R}^*$ and $|{\omega}|\leq {\omega}^*$. Let $\varepsilon\in (0,\frac12)$ if $p\geq \frac65$ otherwise $\varepsilon \in (0,\frac{3p-3}p)$.Then, there exist constants $\eta=\eta_{p,R,\mathfrak{R}^*,{\omega}^*}>0$
and
 $$\ell_3 = \ell_{p,R,\mathfrak{R}^*,{\omega}^*}>\ell_2\triangleq\max(\ell_0,\ell_1), \;\;
 \ell_0,\ell_1 \;\text{\rm same as in  Th.\ref{TH2-1} and  Th.\ref{TH3-1}} $$
  such that if $\normmm{\bb{w}}_{\varepsilon,\Omega}\leq \eta$,
then
\begin{equation}\label{est.pa1}
\big(I+T+K_{\mathfrak{R},{\omega},\bb{w}}(\lambda)\big)^{-1}=(I+T)^{-1}
+S^1_{\mathfrak{R},{\omega},\bb{w}}(\lambda)+S^2_{\mathfrak{R},{\omega},\bb{w}}(\lambda),\quad \lambda \in \overline{\C_{+\ell_3 }}
\end{equation}
where
\[S^1_{\mathfrak{R},{\omega},\bb{w}}(\lambda)\in \mathscr{A}\big(\Sigma_{\varepsilon,\ell_3 },
\mathcal{L}(\mathbb{L}^p_{R+2}(\Omega))\big),\quad S^2_{\mathfrak{R},{\omega},\bb{w}}(\lambda)\in \mathscr{A}\big(\C_{+\ell_3 },
\mathcal{L}(\mathbb{L}^p_{R+2}(\Omega))\big).\]
satisfying
\begin{align}
&\|S^1_{\mathfrak{R},
{\omega},\bb{w}}(\lambda)\|_{\mathcal{L}(\mathbb{L}^p_{R+2}(\Omega))}
\leq C_{\theta,R,\mathfrak{R}^*,{\omega}^*}|\lambda|^{-(1/2)+(1/(2p))},\quad \lambda\in \Sigma_{\theta,\ell_3 },\label{est.pa2}\\
&\|S^2_{\mathfrak{R},
{\omega},\bb{w}}(\lambda)\|_{\mathcal{L}(L^p_{R+2}(\Omega))}
\leq C_{\theta,R,\delta,\mathfrak{R}^*,{\omega}^*}|\lambda|^{-2+\delta},\qquad\quad\;\;\;\lambda\in \C_{+\ell_3 }, \,\, 0<\delta\ll 1/2.\label{est.pa3}
\end{align}
\end{proposition}
\begin{proof}
Let $\bb{f}\in \mathbb{L}^p_{R+2}(\Omega)$, We  split $K_{\mathfrak{R},{\omega},\bb{w}}(\lambda)\bb{f}=K^1_{\mathfrak{R},{\omega},\bb{w}}(\lambda)\bb{f}
+K^2_{\mathfrak{R},{\omega},\bb{w}}(\lambda)\bb{f}$ from \eqref{decom-RG} and \eqref{decom-Pi},
 where
\begin{align*}
K^1_{\mathfrak{R},{\omega},\bb{w}}(\lambda)\bb{f}\triangleq&(\Delta\varphi)
\mathcal{D}^{1}(\lambda)\bb{f}
+2\nabla\varphi\cdot\nabla\mathcal{D}^{1}(\lambda)
\bb{f}+\mathfrak{R}(\partial_1\varphi)
\mathcal{D}^{1}(\lambda)\bb{f}\\
&+{\omega}\big((\bb{e}_1\times  {x})\cdot\nabla \varphi\big)\mathcal{D}^{1}(\lambda)\bb{f}
-(\bb{w}\cdot\nabla \varphi)\mathcal{D}^{1}(\lambda)\bb{f}+L_{\mathfrak{R},{\omega},\bb{w}}\B[\nabla \varphi\cdot\mathcal{D}^{1}(\lambda)\bb{f}]\\
&
-\B[\nabla\varphi\cdot(L_{\mathfrak{R},{\omega},\bb{w}}\mathcal{D}^{1}(\lambda)\bb{f})]-\B[\nabla \varphi\cdot\nabla\Xi^{1} (\lambda)\bb{f}]-\nabla\varphi\cdot\Xi^{1} (\lambda)\bb{f},\\
K^2_{\mathfrak{R},{\omega},\bb{w}}(\lambda)\bb{f}\triangleq&K_{\mathfrak{R},{\omega},\bb{w}}(\lambda)\bb{f}
-K^1_{\mathfrak{R},{\omega},\bb{w}}(\lambda)\bb{f}
\end{align*}
with
\begin{align*}
&\mathcal{D}^{1}(\lambda)\bb{f}=(\lambda I-\Delta-\Rr\partial_1)^{-1}\mathcal{P}_{\R^3}+\mathcal{R}^{G,1}_{\mathfrak{R},{\omega},\overline{\bb{w}}}(\lambda)\bb{f}_0-\mathcal{R}^I_{\mathfrak{R},{\omega},\bb{w}}(\lambda)\bb{f}_{\Omega_{R+3}},\\
&\Xi^{1}(\lambda)\bb{f}=\mathring{\Pi}^{G,1}_{\mathfrak{R},{\omega},\overline{\bb{w}}}(\lambda)\bb{f}_0
-\mathring{\Pi}^I_{\mathfrak{R},{\omega},\bb{w}}(\lambda)\bb{f}_{\Omega_{R+3}}.
\end{align*}
By Theorem \ref{TH2-1}, Corollary \ref{Cor.G1}, Theorem \ref{TH3-1} and Lemma \ref{Lem.Bogo}, we deduce
\begin{align}
&K^1_{\mathfrak{R},
{\omega},\bb{w}}(\lambda)\in \mathscr{A}(\Sigma_{\theta,\ell_2},\mathcal{L}(\mathbb{L}^p_{R+2}(\Omega))),\quad K^2_{\mathfrak{R},
{\omega},\bb{w}}(\lambda)\in \mathscr{A}(\C_{+\ell_2},\mathcal{L}(\mathbb{L}^p_{R+2}(\Omega))),\nonumber\\
&\|K^1_{\mathfrak{R},
{\omega},\bb{w}}(\lambda)\|_{\mathcal{L}(\mathbb{L}^p_{R+2}(\Omega))}\leq C_{\theta,R,\mathfrak{R}^*,{\omega}^*}|\lambda|^{-(1/2)+(1/(2p))},\quad \lambda\in \Sigma_{\theta,\ell_2},\label{est-4-15}\\
&\|K^2_{\mathfrak{R},
{\omega},\bb{w}}(\lambda)\|_{\mathcal{L}(\mathbb{L}^p_{R+2}(\Omega))}\leq C_{\theta,R,\delta,\mathfrak{R}^*,{\omega}^*}|\lambda|^{-2+\delta},\qquad\quad\;\;\;\, \lambda\in \C_{+\ell_2},\;0<\delta\ll 1/2.\label{est-4-16}
\end{align}
Hence, we choose a integer $N>0$ such that $(\tfrac12-\tfrac1{2p})N\geq 2$, and set
\begin{align*}
&S^1_{\mathfrak{R},
{\omega},\bb{w}}(\lambda)\triangleq\sum^{N}_{j=1}\big(-(I+T)^{-1}K^1_{\mathfrak{R},
{\omega},\bb{w}}(\lambda)\big)^j(I+T)^{-1},\\
&S^2_{\mathfrak{R},{\omega},\bb{w}}(\lambda)\triangleq \sum^{\infty}_{j=1}\big(-(I+T)^{-1}K_{\mathfrak{R},
{\omega},\bb{w}}(\lambda)\big)^j(I+T)^{-1}-S^1_{\mathfrak{R},{\omega},\bb{w}}(\lambda).
\end{align*}
By \eqref{est-4-15}-\eqref{est-4-16},  we deduce that there exists a $\ell_3 =C_{\theta, p,R,\mathfrak{R}^*,{\omega}^*}>\ell_2$ such that \eqref{est.pa1}-\eqref{est.pa3} hold. So we complete the proof of this proposition.
\end{proof}
\begin{proposition}\label{Pro-4-3}
Let $\rho\in (0,\frac12)$. Under the assumption of Proposition \ref{Pro-4-1}, there exists a constant $\eta=\eta_{p,\rho,R,\Rr_*,\mathfrak{R}^*,{\omega}^*}>0$ such that if $\normmm{\bb{w}}_{\varepsilon,\Omega}\leq \eta$, then for $\lambda,\lambda+h\in \overline{\C_+}$
\begin{align*}
&\|\partial_{\lambda}(I+T+K_{\mathfrak{R},{\omega},\bb{w}}(\lambda))^{-1}\|_{\mathcal{L}(
\mathbb{L}^p_{R+2}(\Omega))}\leq C_{\varrho, R,\Rr_*,\mathfrak{R}^*,{\omega}^*}(1+|\lambda|)^{-1+\varrho},\;\;0<\varrho\ll\tfrac12\\
&\|\vartriangle_h \partial_{\lambda}(I+T+K_{\mathfrak{R},{\omega},\bb{w}}(\lambda))^{-1}\|_{\mathcal{L}(
\mathbb{L}^p_{R+2}(\Omega))}\leq C_{\rho,\varrho, R,\Rr_*,\mathfrak{R}^*,{\omega}^*}|h|^{\rho}.
\end{align*}
In particular, for $0<|h|\leq h_0$
\begin{align*}
\|\vartriangle_h \partial_{\lambda}(I+T+K_{\mathfrak{R},{\omega},\bb{w}}(\lambda))^{-1}\|_{\mathcal{L}(
\mathbb{L}^p_{R+2}(\Omega))}
\leq C_{\rho,h_0,\varrho, R,\Rr_*,\mathfrak{R}^*,{\omega}^*}|h|^{\rho}|\lambda|^{-1+\varrho},\;\;0<\varrho\ll\tfrac12.
\end{align*}
\end{proposition}
\begin{proof}
By Theorem \ref{TH2-2}, Corollary \ref{Cor.G2}, Theorem \ref{TH3-1} and Lemma \ref{Lem.Bogo}, we deduce for $\lambda, \lambda+h\in \overline{\C_+}$, $\rho\in (0,1/2)$ and $0<\varrho\ll 1/2$
\begin{align*}
&\|(\partial_{\lambda}K_{\mathfrak{R},{\omega},\bb{w}})(\lambda)
\|_{\mathcal{L}(\mathbb{L}^p_{R+2}(\Omega))}
\leq C_{\varrho, R,\Rr_*,\mathfrak{R}^*,{\omega}^*}(1+|\lambda|)^{-1+\varrho},\\
&\|(\vartriangle_h \partial_{\lambda}K_{\mathfrak{R},{\omega},\bb{w}})(\lambda)\|_{\mathcal{L}(\mathbb{L}^p_{R+2}
(\Omega))}\leq C_{\rho, R,\Rr_*,\mathfrak{R}^*,{\omega}^*}|h|^{\rho},\\
&\|(\vartriangle_h \partial_{\lambda}K_{\mathfrak{R},{\omega},
\bb{w}})(\lambda)\|_{\mathcal{L}(\mathbb{L}^p_{R+2}(\Omega))}
\leq C_{\rho,h_0,\varrho, R,\Rr_*,\mathfrak{R}^*,{\omega}^*}|h|^{\rho}(1+|\lambda|)^{-1+\varrho},\quad 0<|h|<h_0.
\end{align*}
These estimates, combining with \eqref{est.pa} and the fact that
\begin{equation*}
\partial_{\lambda}(I+T+K_{\mathfrak{R},{\omega},\bb{w}}(\lambda))^{-1}=
(I+T+K_{\mathfrak{R},{\omega},\bb{w}}(\lambda))^{-1}K_{\mathfrak{R},
{\omega},\bb{w}}(\lambda)(I+T+K_{\mathfrak{R},{\omega},\bb{w}}(\lambda))^{-1},
\end{equation*}
 yield Proposition \ref{Pro-4-3}.
\end{proof}
Finally, relying on the above analysis of $(I+T+K_{\mathfrak{R},{\omega},\bb{w}}(\lambda))^{-1}$, we have the solution operators defined in \eqref{op-S} process the following properties.

\begin{theorem}\label{TH4-1} Under the assumption of  Proposition \ref{Pro-4-2},
 there exists a positive constant $\eta=\eta_{p,R,\mathfrak{R}^*,{\omega}^*}$ such that if
$\normmm{\bb{w}}_{\varepsilon,\Omega}\leq \eta$,
 then
\begin{align*}\left\{\begin{aligned}
&\mathcal{R}_{\mathfrak{R},{\omega},\bb{w}}(\lambda)\in
\mathscr{A}(\C_{+\ell_3 },\mathcal{L}(\mathbb{L}^p_{R+2}(\Omega),\mathbb{W}^{2,p}(\Omega)\cap \mathbb{J}^p(\Omega))),\\
 &\Pi_{\mathfrak{R},{\omega},\bb{w}}(\lambda)\in
\mathscr{A}(\C_{+\ell_3 },\mathcal{L}(\mathbb{L}^p_{R+2}(\Omega),\hat{{W}}^{1,p}(\Omega))),
\end{aligned}\right.\end{align*}
satisfying  for $\bb{f}\in \mathbb{L}^p_{R+2}(\Omega)$ and $\lambda\in \C_{+\ell_2 }$
\begin{equation}\label{decom-Pa}
\left\{\begin{aligned}
&\mathcal{R}_{\mathfrak{R},{\omega},\bb{w}}(\lambda)\bb{f}=\mathcal{R}^{1}_{\mathfrak{R},{\omega},\bb{w}}(\lambda)\bb{f}
+\mathcal{R}^{2}_{\mathfrak{R},{\omega},\bb{w}}(\lambda)\bb{f},\\
&\Pi_{\mathfrak{R},{\omega},\bb{w}}(\lambda)\bb{f}=\Pi \bb{f}+\Pi^{1}_{\mathfrak{R},{\omega},\bb{w}}(\lambda)\bb{f}
+\Pi^{2}_{\mathfrak{R},{\omega},\bb{w}}(\lambda)\bb{f},\\
&\Pi \bb{f}=(1-\varphi)\mathring{\mathcal{Q}}_{\R^3}[(1+T)^{-1}\bb{f}]_0+\varphi
\mathring{\mathcal{Q}}_{\Omega_{R+3}}[(1+T)^{-1}\bb{f}]_{\Omega_{R+3}},
\end{aligned}\right.
\end{equation}
where $\varphi$ is the same function as in \eqref{Op-Pa} such that
\begin{equation}\label{est.re}
\left\{\begin{aligned}
&\mathcal{R}^{1}_{\mathfrak{R},{\omega},\bb{w}}(\lambda)\in \mathscr{A}(\Sigma_{\theta,\ell_3 },\mathcal{L}(\mathbb{L}^p_{R+2}(\Omega),
{\W}^{2,p}(\Omega))),\\
&\mathcal{R}^{2}_{\mathfrak{R},{\omega},\bb{w}}(\lambda)\in \mathscr{A}(\C_{+\ell_3 },\mathcal{L}(\mathbb{L}^p_{R+2}(\Omega),{\W}^{2,p}(\Omega))),\\
&\Pi^{1}_{\mathfrak{R},{\omega},\bb{w}}(\lambda)\in \mathscr{A}(\Sigma_{\theta,\ell_3 },\mathcal{L}(\mathbb{L}^p_{R+2}(\Omega),
\hat{{W}}^{1,p}(\Omega))),\\
&\Pi^{2}_{\mathfrak{R},{\omega},\bb{w}}(\lambda)\in \mathscr{A}(\C_{+\ell_3 },\mathcal{L}(\mathbb{L}^p_{R+2}(\Omega),\hat{{W}}^{1,p}(\Omega))),
\end{aligned}\right.
\end{equation}
satisfying for every $|\beta|\leq 2$
\begin{align}
&\|\partial^{\beta}_x\mathcal{R}^{1}_{\mathfrak{R},{\omega},\bb{w}}(\lambda)
\|_{\mathcal{L}(\mathbb{L}^p_{R+2}(\Omega),\mathbb{L}^p(\Omega))}
\leq C_{\theta,R,\mathfrak{R}^*,{\omega}^*}|\lambda|^{-1+\frac{|\beta|}2},\qquad\lambda\in \Sigma_{\theta,\ell_2 },\label{re.1}\\
&\|\partial^{\beta}_{x}\mathcal{R}^{2}_{\mathfrak{R},{\omega},\bb{w}}(\lambda)\|_{\mathcal{L}(\mathbb{L}^p_{R+2}(\Omega),\mathbb{L}^p(\Omega))}
\leq C_{\theta,R,\mathfrak{R}^*,{\omega}^*}|\lambda|^{-\frac52+\frac{|\beta|}2+\delta},\quad \lambda\in \C_{+\ell_2 },\,\,0<\delta\ll \tfrac12,\label{re.2}\\
&\|\nabla \Pi^{1}_{\mathfrak{R},{\omega},\bb{w}}(\lambda)\|_{\mathcal{L}(\mathbb{L}^p_{R+2}(\Omega),
\LL^p(\Omega))}
\leq C_{\theta,R,\mathfrak{R}^*,{\omega}^*},\qquad\qquad\qquad\lambda\in \Sigma_{\theta,\ell_2 },\label{re.3}\\
&\|\Pi^{1}_{\mathfrak{R},{\omega},\bb{w}}(\lambda)\|_{\mathcal{L}(\mathbb{L}^p_{R+2}(\Omega),L^p(\Omega_{R+3}))}
\leq C_{\theta,R,\mathfrak{R}^*,{\omega}^*}|\lambda|^{-\frac12+\frac1{2p}},\quad\quad\lambda\in \Sigma_{\theta,\ell_2 }\label{re.4},\\
&\|\nabla \Pi^{2}_{\mathfrak{R},{\omega},\bb{w}}(\lambda)\|_{\mathcal{L}(\mathbb{L}^p_{R+2}(\Omega),
\LL^p(\Omega))}
\leq C_{\theta,R,\mathfrak{R}^*,{\omega}^*} |\lambda|^{-2+\delta},\qquad \quad\lambda\in \C_{+\ell_2} ,\,\,0<\delta\ll \tfrac12,\label{re.5}\\
&\|\Pi^{2}_{\mathfrak{R},{\omega},\bb{w}}(\lambda)\|_{\mathcal{L}(\mathbb{L}^p_{R+2}(\Omega),
L^p(\Omega_{R+3}))}
\leq C_{\theta,R,\mathfrak{R}^*,{\omega}^*}|\lambda|^{-2+\delta},\qquad \;\;\lambda\in \C_{+\ell_2},\,\,0<\delta\ll \tfrac12.\label{re.6}
\end{align}
\end{theorem}
\begin{proof}
In view of \eqref{decom-RG}, \eqref{decom-Pi} and \eqref{est.pa1}, we have the decomposition \eqref{decom-Pa} with
\begin{align}
&\mathcal{R}^{1}_{\mathfrak{R},{\omega},\bb{w}}(\lambda)\bb{f}
=(1-\varphi)(\lambda-\Delta-\Rr\partial_1)^{-1}\mathcal{P}_{\R^3}[(I+T)^{-1}\bb{f}+S^1_{\mathfrak{R},{\omega},\bb{w}}(\lambda)\bb{f}]_0\nonumber\\
&\qquad\qquad\qquad\,+(1-\varphi)\mathcal{R}^{G,1}_{\mathfrak{R},{\omega},\overline{\bb{w}}}(\lambda)[(I+T)^{-1}\bb{f}+S^1_{\mathfrak{R},{\omega},\bb{w}}(\lambda)\bb{f}]_0\nonumber\\
&\qquad\qquad\qquad\,+\varphi\mathcal{R}^I_{\mathfrak{R},{\omega},\bb{w}}(\lambda)[((I+T)^{-1}\bb{f}+S^1_{\mathfrak{R},{\omega},\bb{w}}(\lambda))\bb{f}]_{\Omega_{R+3}}\nonumber\\
&\qquad\qquad\qquad\,+\B\big[\nabla \varphi\cdot \mathcal{D}^{1}(\lambda)
[((I+T)^{-1}+S^1_{\mathfrak{R},{\omega},\bb{w}}(\lambda))\bb{f}]\big],\label{RE}\\
&\mathcal{R}^{2}_{\mathfrak{R},{\omega},\bb{w}}(\lambda)\bb{f}
=\mathcal{R}_{\mathfrak{R},{\omega},\bb{w}}(\lambda)\bb{f}
-\mathcal{R}^{1}_{\mathfrak{R},{\omega},\bb{w}}(\lambda)\bb{f},\nonumber\\
&\Pi^{1}_{\mathfrak{R},{\omega},\bb{w}}(\lambda)\bb{f}=(1-\varphi)
\mathring{\mathcal{Q}}_{\R^3}
[S^1_{\mathfrak{R},{\omega},\bb{w}}(\lambda)\bb{f}]_0+\varphi
\mathring{\mathcal{Q}}_{\Omega_{R+3}}[S^1_{\mathfrak{R},{\omega},\bb{w}}
(\lambda)\bb{f}]_{\Omega_{R+3}}\nonumber\\
&\qquad\qquad\qquad+(1-\varphi)
\mathring{\Pi}^{G,1}_{\mathfrak{R},{\omega},\overline{\bb{w}}}(\lambda)
[((I+T)^{-1}\bb{f}+S^1_{\mathfrak{R},{\omega},\bb{w}}(\lambda))\bb{f}]_0\nonumber\\
&\qquad\qquad\qquad+\varphi\mathring{\Pi}^I_{\mathfrak{R},{\omega},\bb{w}}(\lambda)[((I+T)^{-1}\bb{f}+S^1_{\mathfrak{R},{\omega},
\bb{w}}(\lambda))\bb{f}]_{\Omega_{R+3}},\nonumber\\
&\Pi^{2}_{\mathfrak{R},{\omega},\bb{w}}(\lambda)f=\Pi_{\mathfrak{R},{\omega},
\bb{w}}(\lambda)\bb{f}-\Pi \bb{f}-\Pi^{E,1}_{\mathfrak{R},{\omega},\bb{w}}(\lambda)\bb{f}.\nonumber
\end{align}
By Theorem \ref{TH2-1}, Corollary \ref{Cor.G1}, Theorem \ref{TH3-1}, Lemma \ref{Lem-4-1} and Proposition \ref{Pro-4-2}, we deduce \eqref{est.re}-\eqref{re.6}, and so finish the proof of Theorem \ref{TH4-1}.
\end{proof}
\begin{theorem}\label{TH4-2}
Let $\rho\in (0,\frac12)$. Under the assumption of Proposition \ref{Pro-4-1}, there exists a constant $\eta=\eta_{p,R,\mathfrak{R}_*,\mathfrak{R}^*,{\omega}^*}>0$ such that if
$\normmm{\bb{w}}_{\varepsilon,\Omega}\leq \eta$, then
\[\mathcal{R}_{\mathfrak{R},{\omega},\bb{w}}(\lambda)\in C(\overline{\C_+},\mathcal{L}(\mathbb{L}^p_{R+2}(\R^3),{\W}^{2,p}(B_{9(R+2)})))\]
satisfying for every $\lambda,\lambda+h\in \overline{\C_+}$, $0<\varrho\ll \frac12$ and $j\leq 2$,
\begin{align}
&\|\nabla^j\mathcal{R}_{\mathfrak{R},{\omega},\bb{w}}(\lambda)\|_{\mathcal{L}(\mathbb{L}^p_{R+2}(\Omega),\mathbb{L}^p(B_{9(R+2)}))}
\leq C_{R,\mathfrak{R}_*,\mathfrak{R}^*,{\omega}} (1+|\lambda|)^{-1+(j/2)},\label{est.re6'}\\
&\|\nabla^j(\partial_{\lambda}\mathcal{R}_{\mathfrak{R},{\omega},\bb{w}})(\lambda)
\|_{\mathcal{L}(\mathbb{L}^p_{R+2}(\Omega),\mathbb{L}^p(B_{9(R+2)}))}
\leq C_{\varrho,R,\mathfrak{R}_*,\mathfrak{R}^*,{\omega}}(1+|\lambda|)^{-2+(j/2)+\varrho},\label{est.re6}\\
&\|\partial^{\beta}(\vartriangle_h\partial_{\lambda}\mathcal{R}_{\mathfrak{R},{\omega},\bb{w}})(\lambda)\|_{\mathcal{L}(\mathbb{L}^p_{R+2}(\Omega),\mathbb{L}^p(B_{9(R+2)}))}\leq C_{\rho,R,\mathfrak{R}_*,\mathfrak{R}^*}|h|^{\rho},\label{est.re7}
\end{align}
In particular, for $0<|h|\leq h_0$
\begin{equation}\label{est.re8}
\|\nabla^j(\vartriangle_h \partial_{\lambda}\mathcal{R}_{\mathfrak{R},{\omega},\bb{w}})(\lambda)\|_{\mathcal{L}(\mathbb{L}^p_{R+2}(\Omega),\mathbb{L}^p(B_{9(R+2)}))}\leq C_{\rho,\varrho,h_0,R,\mathfrak{R}_*,\mathfrak{R}^*,{\omega}^*}|h|^{\rho}(1+|\lambda|)^{\frac{j-4}2+\varrho}.
\end{equation}
\end{theorem}
Since Theorem \ref{TH4-2} is a direct consequence of  Theorem \ref{TH2-2}, Theorem \ref{TH3-1}, Lemma \ref{Lem.Bogo} and Proposition \ref{Pro-4-3}, here we omit its proof.
\begin{remark}\rm\label{Rem-4-1}
Following the proof of Theorem \ref{TH4-1} and Theorem \ref{TH4-2}, we can deduce that there exist operators $\mathcal{R}^*_{\mathfrak{R},
{\omega},\bb{w}}(\lambda)$ and $
\Pi^*_{\mathfrak{R},{\omega},\bb{w}}(\lambda))$ satisfying all estimates in Theorem \ref{TH4-1} and Theorem \ref{TH4-2}, such that $(\bb{v},\Theta)=(\mathcal{R}^*_{\mathfrak{R},
{\omega},\bb{w}}(\lambda)\bb{f},\Pi^*_{\mathfrak{R},
{\omega},\bb{w}}(\lambda)\bb{f})$ solves
\begin{equation}\label{ER*}
(\lambda I+L^*_{\mathfrak{R},
{\omega},\bb{w}})\bb{v}+\nabla \Theta=\bb{f}\in\mathbb{L}^p_{R+2}(\Omega),\quad \Div \bb{v}=\bb{0} \quad {\rm in}\;\Omega,\quad \bb{v}|_{\partial\Omega}=\bb 0.
\end{equation}
\end{remark}

%%%%%%%%%%%%%%%%%%%%%%%%%%%%%%%%%%%%%%%%%%%%%%%%%%%%%%%%%%%%%%%%%%%%%%%%%%%%%%%
\section{Behavior of  $T_{\Rr,{\omega},\bb{w}}(t)\mathcal{P}_{\Omega}$ and $T^*_{\Rr,{\omega},\bb{w}}(t)\mathcal{P}_{\Omega}$ acting on $\LL^p_{R+2}(\Omega)$.}
\setcounter{section}{5}\setcounter{equation}{0}
In this section, we consider the behavior with respect to $t$ of  the solution to the linear problem:
\begin{equation}\label{eq-5-1}
\left\{\begin{aligned}
&\partial_t \bb{u}+L_{\mathfrak{R},{\omega},\bb{w}} \bb{u}+\nabla P=\bb 0,\quad \Div \bb{u}=0 \quad \text{in }\Omega\times (0,\infty),\\
&\bb{u}|_{\partial\Omega}=\bb{0},\quad \bb{u}(x,0)=\mathcal{P}_{\Omega}\bb{f},\quad \bb{f}\in\mathbb{L}^p_{R+2}(\Omega).
\end{aligned}\right.
\end{equation}
\subsection{Behavior in a short time }
\begin{theorem}\label{TH5-1}
Let $p\in (1,\infty)$, $0<\Rr_*\leq |\mathfrak{R}|\leq \mathfrak{R}^*$ and $|{\omega}|\leq {\omega}^*$. Assume that $\varepsilon\in (0,\frac12)$ if $p\geq \frac65$ otherwise $\varepsilon \in (0,\frac{3p-3}p)$. Then there exists a constant $\eta=\eta_{p,R,\mathfrak{R}^*,{\omega}^*}>0$ such that if
$\normmm{\bb{w}}_{\varepsilon,\Omega}\leq \eta$,
then problem \eqref{eq-5-1} admits a solution $(\bb{u},P)$ with $\bb u$ represented by
\begin{equation}\label{u}
\bb{u}(t)=\lim_{\ell\to\infty}\frac{1}{2\pi i}\int^{\gamma+i\ell}_{\gamma-i\ell}e^{\lambda t}\mathcal{R}_{\mathfrak{R},{\omega},\bb{w}}(\lambda)\bb{f}\,\mathrm{d}\lambda,\quad\;\;\; \gamma>\ell_3,
\end{equation}
satisfying
\begin{equation}\label{semE1}
\bb{u}\in C(\overline{\R_+};\mathbb{J}^p (\Omega))\cap C(\R_+;{\W}^{2,p}(\Omega))\cap C^1(\R_+;\mathbb{L}^p(\Omega)),\quad P\in C(\R_+,\hat{W}^{1,p}(\Omega)),
\end{equation}
with $\R_+=(0,\infty)$ and $\overline{\R_+}=[0,\infty)$
such that
\begin{align}
&\big\|\big(\bb{u}(t), t^{1/2}\nabla \bb{u}(t),t\nabla^2 \bb{u}(t),t\partial_t\bb{u}(t),t\nabla P(t)\big)\big\|_{\mathbb{L}^p(\Omega)}\leq _{R,\gamma, \mathfrak{R}^*,{\omega}^*}e^{\gamma t}\|\bb{f}\|_{\mathbb{L}^p(\Omega)},\label{semE2}\\
&t^{1/2+1/(2p)}(\|\partial_t \bb{u}(t)\|_{{\W}^{-1,p}(\Omega_{R+3})}+ \| P(t)\|_{L^p(\Omega_{R+3})})\leq _{R,\gamma, \mathfrak{R}^*,{\omega}^*}e^{\gamma t}\|\bb{f}\|_{\mathbb{L}^p(\Omega)}.\label{semE3}
\end{align}
where $\ell_3$ is the same constant as in Theorem \ref{TH4-1}.
\end{theorem}
\begin{proof}
For $\gamma>\ell_3>0$, there exists $\theta_0\in (\pi/2,\pi)$ such that
\[\Gamma_{\theta_0,\gamma}\triangleq\{\gamma+re^{\pm i\theta_0}\,|\,r\geq 0\}\subset\Sigma_{\theta,\ell_3}.\]
We set from Theorem \ref{TH4-1}
\begin{align*}
&\bb{u}^{(k)}_{\ell}(t)=\frac1{2\pi i}\int^{\gamma+i\ell}_{\gamma-i\ell}e^{\lambda t}\mathcal{R}^{k}_{\mathfrak{R},{\omega},\bb{w}}(\lambda)\bb{f}\,\mathrm{d}
\lambda,\,\,\, P^{(k)}_{\ell}(t)=\frac1{2\pi i}\int^{\gamma+i\ell}_{\gamma-i\ell}e^{\lambda t}\Pi^{k}_{\mathfrak{R},{\omega},\bb{w}}(\lambda)\bb{f}\,\mathrm{d}\lambda,\\
&\bb{u}^{(1)}(t)=\frac1{2\pi i}\int_{\Gamma_{\theta_0,\gamma}}e^{\lambda t}\mathcal{R}^{1}_{\mathfrak{R},{\omega},\bb{w}}(\lambda)\bb{f}\,\mathrm{d}\lambda,\,\,\,\bb{u}^{(2)}(t)=\frac1{2\pi i}\int^{\gamma+i\infty}_{\gamma-i\infty}e^{\gamma t}\mathcal{R}^{2}_{\mathfrak{R},{\omega},\bb{w}}(\lambda)\bb{f}\,\mathrm{d}\lambda,\\
&P^{(1)}(t)=\frac1{2\pi i}\int_{\Gamma_{\theta_0,\gamma}}e^{\lambda t}\Pi^{1}_{\mathfrak{R},{\omega},\bb{w}}(\lambda)\bb{f}\,\mathrm{d}\lambda,
\,\,\,P^{(2)}(t)=\frac1{2\pi i}\int^{\gamma+i\infty}_{\gamma-i\infty}e^{\gamma t}\Pi^{2}_{\mathfrak{R},{\omega},\bb{w}}(\lambda)\bb{f}\,\mathrm{d}\lambda,\\
&P^0_{\ell}(t)=\frac1{2\pi i}\int^{\gamma+i\ell}_{\gamma-i\ell}e^{\lambda t}\Pi \bb{f}\,\mathrm{d}\lambda.
\end{align*}
with $k=0,1$. By Theorem \ref{TH4-1} and Lemma 5.1 in \cite{HS09}, we deduce that
\begin{align*}
&\bb{u}^{(1)},\bb{u}^{(2)}\in  C^1(\overline{\R_+};\LL^p(\Omega))\cap C^1(\R_+;\LL^p(\Omega))\cap  C(\R_+;{\W}^{2,p}(\Omega)),\\
&P^{(1)},P^{(2)}\in C(\R_+;\hat{W}^{1,p}(\Omega)),
\end{align*}
satisfying
\begin{equation}\label{est-5-1}\left\{
\begin{aligned}
&\lim_{\ell\to\infty}\sup_{0<T_1\leq t\leq T_2}\|\bb{u}^{(1)}_{\ell}(t)-\bb{u}^{(1)}(t)\|_{\mathbb{W}^{1,p}(\Omega)}=0,\\
&\lim_{\ell\to\infty}\sup_{0\leq t\leq T}(\|\bb{u}_{\ell}^{(2)}(t)-\bb{u}^{(2)}(t)\|_{{\W}^{2,p}(\Omega)}+\|\partial_t\bb{u}^{(2)}_{\ell}(t)-\partial_t\bb{u}^{(2)}(t)\|_{\mathbb{L}^p(\Omega)})=0,\\
&\lim_{\ell\to\infty}\sup_{0<T_1\leq t\leq T_2}(\|\nabla P^{(1)}_{\ell}(t)-\nabla P^{(1)}(t)\|_{L^p(\Omega)}+\|P^{(1)}_{\ell}(t)- P^{(1)}(t)\|_{L^p(\Omega_{R+3})})=0,\\
&\lim_{\ell\to\infty}\sup_{0\leq t\leq T}(\|\nabla P^{(2)}_{\ell}(t)-\nabla P^{(2)}(t)\|_{L^p(\Omega)}+\|P^{(2)}_{\ell}(t)- P^{(2)}(t)\|_{L^p(\Omega_{R+3})})=0,
\end{aligned}\right.
\end{equation}
 and
\begin{equation}\label{est-5-2}
\left\{\begin{aligned}
&t^{1/2}\|\nabla^j \bb{u}^{(1)}(t)\|_{\mathbb{L}^p(\Omega)}+t\|\partial_t\bb{u}^{(1)}(t)\|_{\mathbb{L}^p(\Omega)}
\leq C_{R,\gamma,\mathfrak{R}^*,\omega^*}e^{\gamma t}\|\bb{f}\|_{\mathbb{L}^p(\Omega)},\quad j\le 2,\\
&\|\bb{u}^{(2)}(t)\|_{{\W}^{2,p}(\Omega)}+\|\partial_t\bb{u}^{(2)}(t)\|_{\mathbb{L}^p(\Omega)}\leq C_{R,\gamma,\mathfrak{R}^*,\omega^*}e^{\gamma t}\|\bb{f}\|_{\mathbb{L}^p(\Omega)},\\
&t \|\nabla P^{(1)}(t)\|_{L^p(\Omega)}+t^{(1/2)+(1/(2p))}\| P^{(1)}(t)\|_{L^p(\Omega_{R+3})}\leq C_{R,\gamma,\mathfrak{R}^*,\omega^*}e^{\gamma t}\|\bb{f}\|_{\mathbb{L}^p(\Omega)},\\
&\|\nabla P^{(2)}(t)\|_{L^p(\Omega)}+\|P^{(2)}(t)\|_{L^p(\Omega_{R+3})}\leq C_{R,\gamma,\mathfrak{R}^*,\omega^*}e^{\gamma t}\|\bb{f}\|_{\mathbb{L}^p(\Omega)}.
\end{aligned}\right.
\end{equation}
 For $P^0_{\ell}$,  we get by Lemma 5.2  and its remark in \cite{HS09}
\begin{equation}\label{est-5-0}
P^0_{\ell}\to 0 \quad\text{in }\mathcal{D}'(\Omega\times \R_+).
\end{equation}

Set
$$\bb{u}_{\ell}=\bb{u}^{(1)}_{\ell}+\bb{u}^{(2)}_{\ell},\quad P_{\ell}=P^{(0)}_{\ell}+P^{(1)}_{\ell}+P^{(2)}_{\ell},\quad \bb{u}=\bb{u}^{(1)}+\bb{u}^{(2)},\quad P=P^{(1)}+P^{(2)}.$$
Obviously,
\[\partial_{t}\bb{u}_{\ell}+L_{\mathfrak{R},{\omega},\bb{w}}\bb{u}_{\ell}
+\nabla P_{\ell}=\frac1{2\pi i}\int^{\gamma+i\ell}_{\gamma-i\ell}e^{\lambda t}\,\mathrm{d}\lambda \bb{f},\quad \Div \bb{u}_{\ell}=0 \text{ in }\Omega\times\R_+,\quad \bb{u}_{\ell}|_{\partial\Omega}=\bb{0}.\]
Letting $\ell\to\infty$, we have by \eqref{est-5-1} and \eqref{est-5-0} that $\Div \bb{u}=0$ in $\Omega\times \R_+$, $\bb{u}|_{\partial\Omega}=\bb{0}$ and
\begin{equation}\label{est-5-6}
\partial_{t}\bb{u}+L_{\mathfrak{R},{\omega},\bb{w}}\bb{u}+\nabla P=\bb 0\quad \text{in }\mathcal{D}'(\Omega\times \R_+).
\end{equation}

Next, we prove
\begin{equation}\label{est-Ida}
\lim_{t\to 0+}\|\bb{u}(t)-\mathcal{P}_{\Omega}\bb{f}\|_{\mathbb{L}^p(\Omega)}=0.
\end{equation}
From \eqref{re.2} and Lemma 2.1 in \cite{HS09}, we know
$$\lim_{t\to 0+}\|\bb{u}^2(t)\|_{\mathbb{L}^p(\Omega)}=0.$$
So it suffices to prove
\begin{equation}\label{est-5-10}
\lim_{t\to 0+}\|\bb{u}^{(1)}(t)-\mathcal{P}_{\Omega}\bb{f}\|_{\mathbb{L}^p(\Omega)}=0.
\end{equation}
 For this propose, we decompose $\mathcal{R}^{1,k}_{\mathfrak{R},{\omega},\bb{w}}(\lambda)$ defined in \eqref{RE} as
 \begin{align*}
&\mathcal{R}^{1,1}_{\mathfrak{R},{\omega},\bb{w}}(\lambda)\bb{f}\triangleq(1-\varphi)
(\lambda-\Delta-\mathfrak{R}\partial_1)^{-1}\mathcal{P}_{\R^3}[(I+T)^{-1}\bb{f}]_0+\varphi
\mathcal{R}^I_{\mathfrak{R},{\omega},\bb{w}}(\lambda)
[(I+T)^{-1}\bb{f}]_{\Omega_{R+3}}\\
&\qquad\qquad\qquad+\B\big[\nabla \varphi\cdot \big((\lambda-\Delta-\mathfrak{R}\partial_1)^{-1}\mathcal{P}_{\R^3}[(I+T)^{-1}\bb{f}]_0\\
&\qquad\qquad\qquad-\mathcal{R}^I_{\mathfrak{R},{\omega},\bb{w}}(\lambda)
[(I+T)^{-1}\bb{f}]_{\Omega_{R+3}}\big)\big],\\
&\mathcal{R}^{1,2}_{\mathfrak{R},{\omega},\bb{w}}(\lambda)\bb{f}\triangleq
\mathcal{R}^{1}_{\mathfrak{R},{\omega},\bb{w}}(\lambda)\bb{f}-
\mathcal{R}^{1,1}_{\mathfrak{R},{\omega},\bb{w}}(\lambda)\bb{f}.
\end{align*}
which induces us to rewrite $\bb{u}^{(1)}(t)=\bb{u}^{(1,1)}(t)+\bb{u}^{(1,2)}(t)$ with
\[\bb{u}^{(1,k)}(t)\triangleq\frac{1}{2\pi i}\int_{\Gamma_{\theta_0,\gamma}}e^{\lambda t}\mathcal{R}^{1,k}_{\mathfrak{R},{\omega},\bb{w}}(\lambda)\bb{f}\,\mathrm{d}
\lambda,\quad k=1,2.\]
 By \eqref{est.glo1}, \eqref{est.pa2}, Lemma \ref{Lem.Bogo} and Lemma \ref{Lem-4-1}, we have
\begin{align*}
&\mathcal{R}^{1,2}_{\mathfrak{R},{\omega},\bb{w}}(\lambda)\in \mathscr{A}(\Sigma_{\theta,\ell_3},\mathcal{L}(\mathbb{L}^p_{R+2}(\Omega),{\W}^{2,p}(\Omega))),\\
&\|\mathcal{R}^{1,2}_{\mathfrak{R},{\omega},\bb{w}}(\lambda)\bb{f}\|_{\mathbb{L}^p(\Omega)}
\leq C_{R,\theta,\mathfrak{R}^*,{\omega}^*}|\lambda|^{-1-\frac12+\frac1{2p}}\|\bb{f}
\|_{\mathbb{L}^p(\Omega)},\quad \lambda\in \Sigma_{\theta,\ell_3}.
\end{align*}
This, combining with Lemma 2.1 in \cite{HS09}, yields
\[\lim_{t\to 0+}\|\bb{u}^{1,2}(t)\|_{\mathbb{L}^p(\Omega)}=0\]
For $\bb{u}^{(1,1)}(t)$, we know that
\begin{align*}
&\lim_{t\to 0+}\Big\|\frac{1}{2\pi i}\int_{\Gamma_{\theta_0,\gamma}}e^{\lambda t}(\lambda-\Delta-\mathfrak{R}\partial_1)^{-1}\mathcal{P}_{\R^3}\bb{g}\,\mathrm{d}
\lambda-\mathcal{P}_{\R^3}\bb{g}\Big\|_{\LL^p(\R^3)}=0,\quad \bb{g}\in \LL^p(\R^3),\\
&\lim_{t\to 0+}\Big\|\frac{1}{2\pi i}\int_{\Gamma_{\theta_0,\gamma}}e^{\lambda t}\mathcal{R}^I_{\Rr,{\omega},\bb{w}}(\lambda)\bb{g}\,\mathrm{d}
\lambda-\mathcal{P}_{\Omega_{R+3}}\bb{g}\Big\|_{\LL^p(\Omega_{R+3})}=0,\quad \bb{g}\in \LL^p(\Omega_{R+3}),
\end{align*}
which imply
\[\lim_{t\to 0+}\|\bb{u}^{(1,1)}(t)-W\bb{f}\|_{\mathbb{L}^p(\Omega)}=0\]
with
\begin{align*}
W\bb{f}\triangleq &(1-\varphi)\mathcal{P}_{\R^3}[(I+T)^{-1}\bb{f}]_0+\varphi \mathcal{P}_{\Omega_{R+3}}[
(I+T)^{-1}\bb{f}]_{\Omega_{R+3}}\\
&+\B\big[\nabla \varphi\cdot \big(\mathcal{P}_{\R^3}[(I+T)^{-1}\bb{f}]_0-\mathcal{P}_{\Omega_{R+3}}[
(I+T)^{-1}\bb{f}]_{\Omega_{R+3}}\big)\big].
\end{align*}
This, together with  $W\bb{f}=\mathcal{P}_{\Omega}\bb{f}$ in  Lemma 5.3 of \cite{HS09},
 yields
\[\lim_{t\to 0+}\|\bb{u}^{(1,1)}(t)-\mathcal{P}_{\Omega}\bb{f}\|_{\mathbb{L}^p(\Omega)}=0.\]
This finishes the proof of \eqref{est-5-10}.

Finally, we prove
\begin{equation}\label{est-5-7}
\|\partial_t \bb{u}(t)\|_{{\W}^{-1,p}(\Omega_{R+3})}\leq C_{\gamma,R,\mathfrak{R}^*,{\omega}^*}e^{\gamma t}t^{-(1/2)-(1/(2p))}\|\bb{f}\|_{\mathbb{L}^p(\Omega)}.
\end{equation}
Since $$\partial_t\bb{u}_{\ell}(t)=\frac1{2\pi i}\int^{\gamma+i\ell}_{\gamma-i\ell}e^{\lambda t}\lambda\mathcal{R}_{\mathfrak{R},{\omega},\bb{w}}(\lambda)\bb{f}\,\mathrm{d}\lambda,$$
 and
\begin{equation}\label{est.R}
\left\{\begin{aligned}
&(\lambda I+L_{\mathfrak{R},{\omega},\bb{w}})\mathcal{R}_{\mathfrak{R},{\omega},
\bb{w}}(\lambda)\bb{f}+\nabla \Pi_{\mathfrak{R},{\omega},
\bb{w}}(\lambda)\bb{f}=\bb{f}\quad \text{ in }\Omega,\\
&\Div \mathcal{R}_{\mathfrak{R},{\omega},
\bb{w}}(\lambda)\bb{f}=0 \;\; \text{ in }\Omega,\quad \mathcal{R}_{\mathfrak{R},{\omega},
\bb{w}}(\lambda)\bb{f}|_{\partial\Omega}=\bb{0},
\end{aligned}\right.
\end{equation}
we decompose $\partial_t\bb{u}_{\ell}(t)=\bb{v}^{(0)}_{\ell}(t)+\bb{v}^{(1)}_{\ell}(t)+\bb{v}^{(2)}_{\ell}(t)+\bb{v}^{(3)}_{\ell}(t)$ with
\begin{align*}
&\bb{v}^{(0)}_{\ell}(t)\triangleq \frac1{2\pi i}\int^{\gamma+i\ell}_{\gamma-i\ell}e^{\lambda t}(\bb{f}-\nabla \Pi\bb{f})\,\mathrm{d}\lambda,\\
& \bb{v}^{(1)}_{\ell}(t)\triangleq -\frac1{2\pi i}\int^{\gamma+i\ell}_{\gamma-i\ell}e^{\lambda t}\big(\Delta\mathcal{R}^{1}_{\mathfrak{R},{\omega},
\bb{w}}(\lambda)\bb{f}+\nabla \Pi^{1}_{\mathfrak{R},{\omega},\bb{w}}(\lambda)\bb{f}\big)\,\mathrm{d}\lambda,\\
& \bb{v}^{(2)}_{\ell}(t)\triangleq -\frac1{2\pi i}\int^{\gamma+i\ell}_{\gamma-i\ell}e^{\lambda t}(L_{{\mathfrak{R},{\omega},
\bb{w}}}-\Delta)\mathcal{R}^{1}_{\mathfrak{R},{\omega},\bb{w}}(\lambda)\bb{f}\,\mathrm{d}\lambda,\\
& \bb{v}^{(3)}_{\ell}(t)\triangleq -\frac1{2\pi i}\int^{\gamma+i\ell}_{\gamma-i\ell}e^{\lambda t}\big(L_{{\mathfrak{R},{\omega},
\bb{w}}}\mathcal{R}^{2}_{\mathfrak{R},{\omega},\bb{w}}(\lambda)\bb{f}+\nabla \Pi^{2}_{\mathfrak{R},{\omega},\bb{w}}(\lambda)\bb{f}\big)\,\mathrm{d}\lambda.
\end{align*}
Set
\begin{align*}
& \bb{v}^{(1)}(t)\triangleq -\frac1{2\pi i}\int_{\Gamma_{\theta_0,\gamma}}e^{\lambda t}\big(\Delta\mathcal{R}^{1}_{\mathfrak{R},{\omega},
\bb{w}}(\lambda)\bb{f}+\nabla \Pi^{1}_{\mathfrak{R},{\omega},\bb{w}}(\lambda)\bb{f}\big)\,\mathrm{d}\lambda,\\
& \bb{v}^{(2)}(t)\triangleq -\frac1{2\pi i}\int_{\Gamma_{\theta_0,\gamma}}e^{\lambda t}(L_{{\mathfrak{R},{\omega},
\bb{w}}}-\Delta)\mathcal{R}^{1}_{\mathfrak{R},{\omega},\bb{w}}(\lambda)\bb{f}\,\mathrm{d}\lambda,\\
& \bb{v}^{(3)}(t)\triangleq -\frac1{2\pi i}\int^{\gamma+i\infty}_{\gamma-i\infty}e^{\lambda t}\big(L_{{\mathfrak{R},{\omega},
\bb{w}}}\mathcal{R}^{2}_{\mathfrak{R},{\omega},\bb{w}}(\lambda)\bb{f}+\nabla \Pi^{2}_{\mathfrak{R},{\omega},\bb{w}}(\lambda)\bb{f}\big)\,\mathrm{d}\lambda.
\end{align*}
Then, by \eqref{S.1}, Theorem \ref{TH4-1} and Lemma 5.1-Lemma 5.2 in \cite{HS09},
  we obtain
\begin{align*}
&\|\bb{v}^{(1)}\|_{\W^{-1,p}(\Omega_{R+3})}\leq C_{\gamma,R,\Rr^*,{\omega}^*,}e^{\gamma t}t^{-(1/2)-(1/(2p))}\|\bb{f}\|_{\LL^p(\Omega)},\\
&\|\bb{v}^{(2)}\|_{\LL^{p}(\Omega_{R+3})}\leq C_{\gamma,R,\Rr^*,{\omega}^*,}e^{\gamma t}t^{-1/2}\|\bb{f}\|_{\LL^p(\Omega)},\\
&\|\bb{v}^{(3)}\|_{\LL^{p}(\Omega_{R+3})}\leq C_{\gamma,R,\Rr^*,{\omega}^*,}e^{\gamma t}\|\bb{f}\|_{\LL^p(\Omega)},\\
&\lim_{\ell\to \infty}\sup_{0<T_1\leq t\leq T_2}\|\bb{v}^{(1)}_{\ell}(t)-\bb{v}^{(1)}(t)\|_{{\W}^{-1,p}(\Omega_{R+3})}=0,\\
&\lim_{\ell\to \infty}\sup_{0<T_1\leq t\leq T_2}\|\bb{v}^{(k)}_{\ell}(t)-\bb{v}^{(k)}(t)\|_{\LL^{p}(\Omega_{R+3})}=0,\; k=1,2,\\
&\lim_{\ell\to\infty}\bb{v}^{(0)}_{\ell}= \bb{0} \text{ in }\mathcal{D}'(\Omega\times \R_+).
\end{align*}
Since $\bb{u}_{\ell}(t)\to \bb{u}(t)$ in $\mathcal{D}'(\Omega\times \R_+)$, we conclude $\partial_t \bb{u}(t)=\bb{v}^{(1)}(t)+\bb{v}^{(2)}(t)+\bb{v}^{(3)}(t)$ by  uniqueness.
This proves \eqref{est-5-7} and so completes the proof of Theorem\ref{TH5-1}.
\end{proof}

To show the uniqueness of the solution $\bb{u}$ obtained in Theorem \ref{TH5-1}, we consider the linear nonstationary problem associated to  $L^*_{\mathfrak{R},{\omega},\bb{w}}$
\begin{equation}\label{eq-5-2}
\begin{cases}
\partial_t \bb{v}+L^*_{\mathfrak{R},{\omega},\bb{w}} \bb{v}+\nabla \Theta=\bb 0,\quad \Div \bb{v}=0 \quad \text{in }\Omega\times\R_+,\\
\bb{v}|_{\partial\Omega}=\bb{0},\quad \bb{v}|_{t=0}=\mathcal{P}_{\Omega}\bb{f},\quad \bb{f}\in \mathbb{L}^p_{R+2}(\Omega).
\end{cases}
\end{equation}
Following the same argument as in the proof of Theorem \ref{TH5-1}, we have
\begin{corollary}\label{Cor-5-1}
Under the assumption of Theorem \ref{TH5-1},  there exists a positive constant $\eta=\eta_{p,R,\mathfrak{R}^*,{\omega}^*}$ such that if
$\normmm{\bb{w}}_{\varepsilon,\Omega}\leq \eta$,
 then problem \eqref{eq-5-2} admits a solution $(\bb{v},\Theta)$ with
\begin{equation}\label{eq.v}
\bb{v}(t)=\lim_{\ell\to\infty}\frac1{2\pi i}\int^{\gamma+i\ell}_{\gamma-i\ell}e^{\lambda t}\mathcal{R}^{*}_{\mathfrak{R},{\omega},\bb{w}}(\lambda)\bb{f}\,\mathrm{d}\lambda,\quad \gamma>\ell_3 ,
\end{equation}
satisfying
\begin{equation}\label{semE1'}
\bb{v}\in C(\overline{\R_+};\mathbb{J}^p (\Omega))\cap C(\R_+;{\W}^{2,p}(\Omega))\cap C^1(\R_+;\mathbb{L}^{p}(\Omega)), \quad\Theta\in C(\R_+,\hat{{\W}}^{1,p}(\Omega)).
\end{equation}
Moreover, $\bb{v}$ and $\Theta$ possess
\begin{align}
&\big\|\big(\bb{v}(t), t\partial_t \bb{v}(t), t^{1/2}\nabla \bb{v}(t),t\nabla^2 \bb{v}(t),t\nabla \Theta(t)\big)\big\|_{\mathbb{L}^p(\Omega)}\leq C_{\gamma,R,\mathfrak{R}^*,{\omega}^*}e^{\gamma t}\|\bb{f}\|_{\mathbb{L}^p(\Omega)},\label{semE2'}\\
&t^{(1/2)+(1/(2p))}\big(\|\partial_t \bb{v}(t)\|_{{\W}^{-1,p}(\Omega_{R+3})}+\| \Theta(t)\|_{L^p(\Omega_{R+3})}\big)\leq C_{\gamma,R,\mathfrak{R}^*,{\omega}^*}e^{\gamma t}\|\bb{f}\|_{\mathbb{L}^p(\Omega)}.\label{semE3'}
\end{align}
\end{corollary}
Invoking Corollary \ref{Cor-5-1}, we obtain the following proposition, which gives the uniqueness of $\bb{u}$ in Theorem \ref{TH5-1}.
\begin{proposition}\label{Pro-5-1}
Let $p\in (1,\infty)$. Assume that $\bb{u}\in C(\overline{\R_+};\mathbb{J}^p (\Omega))\cap C(\R_+;{\W}^{2,p}(\Omega))\cap C^1(\R_+;\mathbb{L}^p(\Omega))$ and $P\in C(\R_+,\hat{W}^{1,p}(\Omega))$ such that
\begin{equation}\label{eq-5-3}
\begin{cases}
\partial_t \bb{u}+L_{\mathfrak{R},{\omega},\bb{w}} \bb{u}+\nabla P=\bb 0,\quad \Div \bb{u}=0 \quad \text{in }\Omega\times\R_+,\\
\bb{u}|_{\partial\Omega}=\bb 0,\quad \bb{u}|_{t=0}=\bb 0.
\end{cases}
\end{equation}
Then $\bb{u}=\bb 0$ in $\Omega\times\R_+$.
\end{proposition}
\begin{proof}
 For any $\bb{\phi}\in \mathbb{C}^{\infty}_{0,\sigma}$,
problem \eqref{eq-5-2} with $\bb{f}=\bb{\phi}$ admits a solution $(\bb{v},\Theta)$  satisfying \eqref{semE1'}-\eqref{semE3'} by Corollary \ref{Cor-5-1}. For every $T>0$,
 one easily verifies that $\widetilde{\bb{v}}(t)\triangleq\bb{v}(T-t)$ and $\widetilde{\Theta}(t)\triangleq\Theta(T-t)$ solve
 \begin{equation}\label{eq-5-2'}
\left\{\begin{aligned}
&-\partial_t \widetilde{\bb{v}}+L^*_{\mathfrak{R},{\omega},\bb{w}} \widetilde{\bb{v}}+\nabla \widetilde{\Theta}=\bb 0,\quad \Div \widetilde{\bb{v}}=0 \quad \text{in }\Omega\times(-\infty,T),\\
&\widetilde{\bb{v}}|_{\partial\Omega}=\bb 0,\quad \widetilde{\bb{v}}|_{t=T}=\bb{\phi}.
\end{aligned}\right.
\end{equation}
  We have by integration by parts
\begin{align*}
0=&\int^T_0\int_{\Omega}\big(\partial_t\bb{u}(t)+L_{\mathfrak{R},{\omega},
\bb{w}}\bb{u}(t)
+\nabla P\big)\cdot \widetilde{\bb{v}}(t)\,\mathrm{d}x\mathrm{d}t\\
=&\int_{\Omega}\bb{u}(T)\cdot \bb{\phi}\mathrm{d}x+\int^T_0\int_{\Omega}\bb{u}(t)\cdot \big(-\partial_t \widetilde{\bb{v}}(t)+L^*_{\mathfrak{R},{\omega},\bb{w}}\widetilde{\bb{v}}(t)\big)\,\mathrm{d}x\mathrm{d}t\\
=&\int_{\Omega}\bb{u}(T)\cdot \bb{\phi}\mathrm{d}x-\int^T_0\int_{\Omega}\bb{u}(t)\cdot \nabla\widetilde{\Theta}\,\mathrm{d}x\mathrm{d}t=\int_{\Omega}\bb{u}(T)\cdot \bb{\phi}\mathrm{d}x.
\end{align*}
Hence, in the light of the arbitrariness of $\bb{\phi}$ and $T$,  we conclude that $\bb{u}=\bb 0$ in $\mathbb{J}^p(\Omega)$ for every $t>0$ and so complete the proof of Proposition \ref{Pro-5-1}.
\end{proof}

\begin{remark}\rm\label{Rem-5-1}
By duality, it is obvious that  $\bb{v}$ defined in \eqref{eq.v} is unique.
\end{remark}
%%%%%%%%%%%%%%%%%%%%%%%%%%%%%%%%%%%%%%%%%%%%%%%%%%%%%%%%%%%%%%%%%%%%%%%%%%%%%%%%%%%%%%%%%%%%
\subsection{Behavior in Large time }
From Proposition \ref{TH1}, Theorem \ref{TH5-1} and Corollary \ref{Pro-5-1}, we have  the following representation of $T_{\mathfrak{R},{\omega},\bb{w}}(t)\mathcal{P}_{\Omega}$ acting on $\mathbb{L}^p_{R+2}(\Omega)$:
\begin{equation}\label{est-6-1}
T_{\mathfrak{R},{\omega},\bb{w}}(t)\mathcal{P}_{\Omega}=\lim_{\ell\to \infty}\frac1{2\pi i}\int^{\gamma+i\ell}_{\gamma-i\ell} e^{\lambda t}\mathcal{R}_{\mathfrak{R},{\omega},\bb{w}}(\lambda)\,\mathrm{d}\lambda,\quad \gamma>\ell_3.
\end{equation}
With it, we will follow the idea in \cite{Iw89,KS98} to study the behavior of $T_{\mathfrak{R},{\omega},\bb{w}}(t)\bb{f}$ and $T^*_{\mathfrak{R},{\omega},\bb{w}}(t)\bb{f}$ as $t\to\infty$ in localization space  when $\bb{f}\in \mathbb{L}^p_{R+2}(\Omega)$, that is so-called local energy decay properties.
\begin{theorem}\label{TH5-2}
Let $\rho\in (0,\frac12)$. Under the assumption of
Theorem \ref{TH5-1}, there exists a constant $\eta=\eta_{p,\rho,R,\mathfrak{R}_*,\mathfrak{R}^*,{\omega}^*}>0$, such that if $\normmm{\bb{w}}_{\varepsilon,\Omega}\leq \eta$, then  for $t\geq 1$
\begin{align}
&\|T_{\mathfrak{R},{\omega},\bb{w}}(t)\mathcal{P}_{\Omega}\|_{\mathcal{L}(\LL^p_{R+2}(\Omega),{\W}^{1,p}(\Omega_{R+3}))}
+\|\partial_tT_{\mathfrak{R},{\omega},\bb{w}}(t)\mathcal{P}_{\Omega}\|_{\mathcal{L}(\LL^p_{R+2}(\Omega),\W^{-1,p}(\Omega_{R+3}))}\lesssim  t^{-1-\rho},\label{semE4}\\
&\|T^{*}_{\mathfrak{R},{\omega},\bb{w}}(t)\mathcal{P}_{\Omega}\|_{\mathcal{L}(\LL^p_{R+2}(\Omega),{\W}^{1,p}(\Omega_{R+3}))}
+\|\partial_tT^{*}_{\mathfrak{R},{\omega},\bb{w}}(t)\mathcal{P}_{\Omega}\|_{\mathcal{L}(\LL^p_{R+2}(\Omega),{\W}^{-1,p}(\Omega_{R+3}))}\lesssim t^{-1-\rho},\label{semE4'}
\end{align}
where $\lesssim=\lesssim_{\rho,R,\mathfrak{R}_*,\mathfrak{R}^*,{\omega}^*}$.
\end{theorem}
\begin{proof}
By integration by part, we have
\begin{align}
\frac1{2\pi i}\int^{\gamma+i\ell}_{\gamma-i\ell} e^{\lambda t}\mathcal{R}_{\mathfrak{R},{\omega},\bb{w}}(\lambda)\bb{f}\,\mathrm{d}
\lambda=&\frac{e^{(\gamma+i\ell) t}}{2\pi t i} \mathcal{R}_{\mathfrak{R},{\omega},\bb{w}}(\gamma+i\ell)\bb{f}
-\frac{e^{(\gamma-i\ell)t}}{2\pi t i} \mathcal{R}_{\mathfrak{R},{\omega},\bb{w}}(\gamma-i\ell)\bb{f}\nonumber\\
&-\frac1{2\pi t i}\int^{\gamma+i\ell}_{\gamma-i\ell} e^{\lambda t}(\partial_{\lambda}\mathcal{R}_{\mathfrak{R},{\omega},\bb{w}})(\lambda)\bb{f}
\,\mathrm{d}\lambda.
\label{est-6-1'}
\end{align}
In addition, by \eqref{est.re6'}-\eqref{est.re6}, we obtain
\begin{align*}
&\lim_{\ell\to\infty}\frac1{2\pi t i} e^{(\gamma+i\ell) t}\mathcal{R}_{\mathfrak{R},{\omega},\bb{w}}(\gamma+i\ell)\bb{f}-\frac1{2\pi t i} e^{(\gamma-i\ell) t}\mathcal{R}_{\mathfrak{R},{\omega},\bb{w}}(\gamma-i\ell)\bb{f}=0,\quad \text{in }\mathbb{L}^p(\Omega_{R+3}),\\
&\frac1{2\pi t i}\int^{\gamma+\infty}_{\gamma-i\infty} \|e^{\lambda t}\partial_{\lambda}\mathcal{R}_{\mathfrak{R},{\omega},\bb{w}}(\lambda)\bb{f}
\|_{\mathbb{L}^p(\Omega_{R+3})}\,\mathrm{d}\lambda<\infty.
\end{align*}
Hence, taking $\ell\to \infty$ in \eqref{est-6-1'}, we get
\begin{equation}\label{est-6-2}
T_{\mathfrak{R},{\omega},\bb{w}}(t)\mathcal{P}_{\Omega}\bb{f}=-\frac1{2\pi t}\int^{\infty}_{-\infty}e^{\gamma t+ist}(\partial_{\lambda}\mathcal{R}_{\mathfrak{R},{\omega},\bb{w}})(\gamma+is)
\bb{f}\,\mathrm{d}s\quad \text{in}\,\,\mathbb{L}^p(\Omega_{R+3}).
\end{equation}
Letting $\gamma\to 0+$ in \eqref{est-6-2},  we deduce by \eqref{est.re8}
\begin{equation}\label{eq.u-0}
T_{\mathfrak{R},{\omega},\bb{w}}(t)\mathcal{P}_{\Omega}\bb{f}=-\frac1{2\pi t}\int^{\infty}_{-\infty}e^{ist}(\partial_{\lambda}\mathcal{R}_{\mathfrak{R},
{\omega},\bb{w}})(is)\bb{f}\,\mathrm{d}s\quad \text{in }\mathbb{L}^p(\Omega_{R+3}).
\end{equation}

From \eqref{est.re6} and \eqref{est.re8}, we obtain
\begin{align*}
&(\partial_{\lambda}\mathcal{R}_{\mathfrak{R},
{\omega},\bb{w}})(is)\bb{f}\in L^1(\R;\W^{1,p}(\Omega_{R+3})) ,\\
&\sup_{0<|h|\leq 1}|h|^{-\rho}\int_{\R}\|(\partial_{\lambda}\mathcal{R}_{\mathfrak{R},
{\omega},\bb{w}})(is+ih)-(\partial_{\lambda}\mathcal{R}_{\mathfrak{R},
{\omega},\bb{w}})(is)\|_{\W^{1,p}(\Omega_{R+3})}\,\mathrm ds\\
 \leq& C_{R,\rho,\Rr_*,\Rr^*,{\omega}^*}\|\bb{f}\|_{\LL^p(\Omega)},\quad \rho\in (0,\tfrac12).
\end{align*}
 Hence, by Lemma 6.1 in \cite{HS09}, we conclude from \eqref{eq.u-0}
\begin{equation}\label{est-6-3}
\|T_{\mathfrak{R},{\omega},\bb{w}}(t)\mathcal{P}_{\Omega}\bb{f}\|_{\mathbb{W}^{1,p}(\Omega_{R+3})}\leq _{R,\rho,\Rr_*,\Rr^*,{\omega}^*} t^{-1-\rho}\|\bb{f}\|_{\mathbb{L}^p(\Omega)},\quad \rho\in (0,1/2).
\end{equation}

For $\partial_tT_{\mathfrak{R},{\omega},\bb{w}}(t)\mathcal{P}_{\Omega}\bb{f}$, we calculate
\begin{align*}\partial_tT_{\mathfrak{R},{\omega},\bb{w}}(t)\mathcal{P}_{\Omega}\bb{f}
=&\frac1{2\pi t^2}\int^{\infty}_{-\infty}e^{ist}(\partial_{\lambda}\mathcal{R}_{\mathfrak{R},
{\omega},\bb{w}})(is)\bb{f}\,\mathrm{d}s-\frac1{2\pi t}\int^{\infty}_{-\infty}is\, e^{ist}(\partial_{\lambda}\mathcal{R}_{\mathfrak{R},{\omega},
\bb{w}})(is)f\,\mathrm{d}s\\
=&\frac{1+i}{2\pi t^2}\int^{\infty}_{-\infty}e^{ist}(\partial_{\lambda}\mathcal{R}_{\mathfrak{R},
{\omega},\bb{w}})(is)\bb{f}\,\mathrm{d}s-\frac{1}{2\pi t}\int^{\infty}_{-\infty}e^{ist}(\partial_{\lambda}(\lambda\mathcal{R}_{\mathfrak{R},{\omega},
\bb{w}}))(is)\bb{f}\,\mathrm{d}s\\
\triangleq&\bb{v}_1(t)+\bb{v}_1(t).
\end{align*}
Since $\bb{v}_1=\frac{1+i}{t}T_{\Rr,\omega,\bb{w}}$(t), we have from \eqref{est-6-3}
\begin{equation}\label{est-6-4}
\|\bb{v}_1(t)\|_{{\W}^{1,p}(\Omega_{R+3})}\leq C_{R,\rho,\Rr_*,\Rr^*,{\omega}^*} t^{-2-\rho}\|\bb{f}\|_{{\mathbb{L}}^p(\Omega)},\quad \rho\in (0,1/2).
\end{equation}
For $\bb{v}_2(t)$, we have by  \eqref{est.R}
\begin{align*}
\partial_{\lambda}(\lambda\mathcal{R}_{\mathfrak{R},{\omega},\bb{w}})(\lambda)
\bb{f}
=&\Delta(\partial_{\lambda}\mathcal{R}_{\mathfrak{R},{\omega},\bb{w}})(\lambda)
\bb{f}+\mathfrak{R}\partial_{1}(\partial_{\lambda}\mathcal{R}_{\mathfrak{R},{\omega},
\bb{w}})(\lambda)\bb{f}\\
&+{\omega}\big((\bb{e}_1\times  {x})\cdot\nabla (\partial_{\lambda}\mathcal{R}_{\mathfrak{R},{\omega},\bb{w}})(\lambda)
\bb{f}-\bb{e}_1\times (\partial_{\lambda}\mathcal{R}_{\mathfrak{R},{\omega},\bb{w}})(\lambda)\bb{f}
\big)\\
&-\bb{w}\cdot\nabla(\partial_{\lambda}\mathcal{R}_{\mathfrak{R},{\omega},\bb{w}})
(\lambda)\bb{f}-(\partial_{\lambda}\mathcal{R}_{\mathfrak{R},{\omega},\bb{w}})
(\lambda)\bb{f}\cdot\nabla \bb{w}
-\nabla (\partial_{\lambda}\Pi_{\mathfrak{R},{\omega},\bb{w}})(\lambda)\bb{f}.
\end{align*}
Thus, we deduce by \eqref{est.re6} and \eqref{est.re8} that $(\partial_{\lambda}(\lambda\mathcal{R}_{\mathfrak{R},
{\omega},\bb{w}}))(is)\bb{f}\in L^1(\R;\W^{-1,p}(\Omega_{R+3}))$ and
\begin{align*}
&\sup_{0<|h|\leq 1}|h|^{-\rho}\int_{\R}\|(\partial_{\lambda}(\lambda \mathcal{R}_{\mathfrak{R},
{\omega},\bb{w}}))(is+ih)\bb{f}-(\partial_{\lambda}(\lambda\mathcal{R}_{\mathfrak{R},
{\omega},\bb{w}}))(is)\bb{f}\|_{\W^{-1,p}(\Omega_{R+3})}\,\mathrm ds\\
\leq & C\|\bb{f}\|_{\LL^p(\Omega)},\quad \rho\in (0,\tfrac12),
\end{align*}
which implies
\begin{equation}\label{est-6-5}
\|\bb{v}_2(t)\|_{\mathbb{W}^{-1,p}(\Omega_{R+3})}\leq t^{-1-\rho}\|\bb{f}\|_{\mathbb{L}^p(\Omega)}\quad \rho\in (0,1/2).
\end{equation}
Collecting \eqref{est-6-3}-\eqref{est-6-5}, we prove \eqref{semE4}. In the same way as deriving  \eqref{semE4}, we obtain \eqref{semE4'},  and so complete the proof of Theorem \ref{TH5-2}.
\end{proof}

%%%%%%%%%%%%%%%%%%%%%%%%%%%%%%%%%%%%%%%%%%%%%%%%%%%%%%%%%%%%%%%%%%%%%%
\section{$L^p$-$L^q$ estimates of $T_{\mathfrak{R},{\omega},\bb{w}}(t)$ and $T^{*}_{\mathfrak{R},{\omega},\bb{w}}(t)$}
We first study the $L^p$-$L^q$ estimates of  $T_{\mathfrak{R},{\omega},\bb{w}}(t)$, that is,  the $L^p$-$L^q$ estimates of the solution map $\bb{f}\mapsto \bb{u}(t)$ of the following Cauchy problem:
\begin{equation}\label{eq-7-1}
\begin{cases}
\partial_t \bb{u} +L_{\mathfrak{R},{\omega},\bb{w}}\bb{u}+\nabla P=\bb 0,\quad \Div \bb{u}=\bb{0}\quad \text{in }\Omega\times (0,\infty),\\
\bb{u}|_{\partial\Omega}=\bb{0},\quad \bb{u}|_{t=0}=\bb{f}\in \mathbb{J}^p (\Omega).
\end{cases}
\end{equation}
Without of generality, we assume the pressure $P(x,t)$ such that
\begin{equation}\label{cond-pre}
\int_{\Omega_{R+3}}P(x,t)\,\mathrm{d}x=0\quad \text{for every } t>0.
\end{equation}
We start from the regularity estimates  of the solution of \eqref{eq-7-1}
\begin{proposition}\label{Pro-7-1}
Let $p\in (1,3]$, $0<\mathfrak{R}_*\leq |\mathfrak{R}|\leq\mathfrak{R}^*$ and $|{\omega}|\leq {\omega}^*$. Assume that $\varepsilon\in (0,\frac12)$ if $p\geq \frac65$ otherwise $\varepsilon \in (0,\frac{3p-3}p)$. Then, there exists a constant $\eta=\eta_{p,R,\mathfrak{R}_*,\mathfrak{R}^*,{\omega}^*}>0$ such that if
$\normmm{\bb{w}}_{\varepsilon,\Omega}\leq \eta$, then for $\bb{f}\in \mathbb{J}^p(\Omega)$
\begin{align}
&\|\bb{u}(t)\|_{\mathbb{L}^p(\Omega)}+t^{1/2}\|\nabla\bb{u}(t)\|_{\mathbb{L}^p(\Omega)}+t^{\frac12+\frac1{2p}}
\big(\|\partial_t\bb{u}(t)\|_{{\W}^{-1,p}(\Omega_{R+3})}
+\|P(t)\|_{\mathbb{L}^p(\Omega_{R+3})}\big)\nonumber\\
\leq & C_{{R,\mathfrak{R}_*,\mathfrak{R}^*,{\omega}^*}} \|\bb{f}\|_{\mathbb{L}^p(\Omega)},\qquad 0<t\leq 2,\label{semi-E1}\\
%%%%%%%%%%%%%%%%%%%%
&\|\bb{u}(t)\|_{{\W}^{1,p}(\Omega_{R+3})}+\|\partial_t\bb{u}(t)\|_{{\W}^{-1,p}
(\Omega_{R+3})}+\|P(t)\|_{L^p(\Omega_{R+3})}\nonumber\\
\leq& C_{{\varrho,R,\mathfrak{R}_*,\mathfrak{R}^*,{\omega}^*}}t^{-\frac3{2p}}\|\bb{f}\|_{\mathbb{L}^p(\Omega)},\qquad t>2.\label{semi-E2}
\end{align}
\end{proposition}
\begin{proof}
We first estimate $\bb{u}(t)$. Let $\varphi\in C^{\infty}_0(\R^3)$ be the bump function  in \eqref{Op-Pa}. Set
$$\bb{v}_0=(1-\varphi)\bb{f}+\B[\nabla \varphi\cdot\bb{f}],$$
where $\B$ is a Bogovski\v{i}'s operator. By Lemma \ref{Lem.Bogo}, we have
\begin{equation}\label{est-v_0}
\|\bb{v}_0\|_{\mathbb{L}^p(\R^3)}\lesssim \|\bb{f}\|_{\mathbb{L}^p(\Omega)},\quad \Div \bb{v}_0=0.
\end{equation}
Set $(\bb{v}(t),\Theta(t))=(T^G_{\mathfrak{R},{\omega},\bb{w}}(t)\bb{v}_0,
\mathring{\mathcal{Q}}_{\R^3}(\overline{\bb{w}}\cdot\nabla\bb{v}(t)+\bb{v}(t)\cdot\nabla\overline{\bb{w}}))$. Then $(\bb{v}(t),\Theta(t))$ solves
\begin{equation}\label{eq-7-2'}
\left\{\begin{aligned}
&\partial_t \bb{v}+L_{\mathfrak{R},{\omega},\overline{\bb{w}}}\bb{v}+\nabla \Theta=0,\quad \Div \bb{v}=\bb{0} \text{ in }\R^3\times (0,\infty),\\
&\bb{v}|_{t=0}=\bb{v}_0,\quad \int_{\Omega_{R+3}}\Theta\,\mathrm{d}x=0.
\end{aligned}\right.
\end{equation}
and satisfies
\begin{equation}\label{est-7-1}
\left\{\begin{aligned}
&\|\bb{v}(t)\|_{\mathbb{L}^q(\R^3)}\leq C t^{-\frac{3}2(\frac1p-\frac1 q)}\|\bb{f}\|_{\mathbb{L}^p(\Omega)},\quad p\leq q\leq \infty,\\
&\|\nabla \bb{v}(t)\|_{\mathbb{L}^q(\R^3)}\leq C t^{-\frac{1}2-\frac{3}2(\frac1p-\frac1 q)}\|\bb{f}\|_{\mathbb{L}^p(\Omega)},\quad p\leq q\leq 3,\\
&\|\nabla\Theta(t)\|_{L^q(\R^3)}+\|\Theta(t)\|_{L^q(\Omega_{R+3})}\leq Ct^{-\frac{1}2-\frac{3}2(\frac1p-\frac1 q)}\|\bb{f}\|_{\mathbb{L}^p(\Omega)},\,\,\, p\leq q\leq 3,
\end{aligned}\right.
\end{equation}
by making use of \eqref{P-est}, \eqref{sem-G} and Lemma \ref{Lem.Hardy}.

To compensate the zero boundary condition of $\bb{v}$ on $\partial\Omega$, we set  $$\widetilde{\bb{v}}(t)=(1-\varphi)\bb{v}(t)+\B[\nabla \varphi\cdot \bb{v}(t)], \; \widetilde{\Theta}(t)=(1-\varphi)\Theta(t).$$
Obviously, $(\widetilde{\bb{v}}(t),\widetilde{\Theta}(t))$ solves
\begin{equation}\label{eq-7-2}\left\{
\begin{aligned}
&\partial_t \tilde{\bb{v}}+L_{\mathfrak{R},{\omega},\bb{w}}\tilde{\bb{v}}+\nabla \widetilde{\Theta}=\bb{F}(t),\quad \Div \tilde{\bb{v}}=0 \quad \text{in }\Omega\times (0,\infty),\\
&\tilde{\bb{v}}|_{\partial\Omega}=\bb{0},\quad \tilde{\bb{v}}|_{t=0}=\tilde{\bb{v}}_0\triangleq (1-\varphi)\bb{v}_0+\B[\nabla \varphi\cdot \bb{v}_0]
\end{aligned}\right.
\end{equation}
with
\begin{align}\nonumber
\bb{F}(t)=&-\Theta\nabla \varphi+(\Delta \varphi)\bb{v}+2(\nabla \varphi)\cdot\nabla \bb{v}+\mathfrak{R}(\partial_1\varphi)\bb{v}+{\omega}\big((\bb{e}_1
\times  {x})\cdot\nabla\varphi\big)\bb{v}\\
&+(\bb{w}\cdot\nabla \varphi)\bb{v}-\B[\nabla\varphi\cdot (L_{\mathfrak{R},{\omega},\bb{w}}\bb{v}+\nabla \Theta)]+L_{\mathfrak{R},{\omega},\bb{w}}\B[\nabla\varphi\cdot \bb{v}].\label{eq-F}
\end{align}
When $t\in (0,2]$, we easily get from Lemma \ref{Lem.Bogo} and \eqref{est-7-1}
\begin{equation}\label{est-7-3}
\|(\widetilde{\bb{v}}(t), t^{\frac12}\nabla \widetilde{\bb{v}}, t^{\frac12}\bb{F}(t))\|_{\mathbb{L}^{p}(\R^3)}+t^{\frac12}\|\widetilde{\Theta}\|_{L^p(\Omega_{R+3})}\leq C_{R}\|\bb{f}\|_{\mathbb{L}^{p}(\Omega)}.
\end{equation}
Meanwhile, we can deduce for $t>2$
\begin{align}\nonumber
&\|\widetilde{\bb{v}}\|_{{\W}^{1,p}(\Omega_{R+3})}+\|\widetilde{\Theta}\|_{L^{p}(\Omega_{R+3})}+\|\bb{F}(t)\|_{\mathbb{L}^{p}(\R^3)}\\
\lesssim&\|\widetilde{\bb{v}}\|_{{\LL}^{\infty}(\Omega_{R+3})}+\|\nabla \widetilde{\bb{v}}\|_{{\LL}^{3}(\Omega_{R+3})}+\|\widetilde{\Theta}\|_{L^{3}(\Omega_{R+3})}\lesssim t^{-\frac3{2p}}\|\bb{f}\|_{\mathbb{L}^p(\Omega)}.\label{est-7-3-add-1}
\end{align}
Hence, we have from \eqref{eq-7-2}
\begin{equation}\label{est-7-4}
\|\partial_t\widetilde{\bb{v}}\|_{{\W}^{-1,p}(\Omega_{R+3})}\leq \left\{\begin{aligned}
&C_{R} t^{-1/2}\|\bb{f}\|_{\mathbb{L}^p(\Omega)},\qquad\; 0<t<2, \\
&C_{R} t^{-\frac3{2p}} \|\bb{f}\|_{\mathbb{L}^p(\Omega)},\qquad\;\; t>2.
\end{aligned}\right.
\end{equation}

Now, we set $\widetilde{\bb{u}}(t)=\bb{u}(t)-\widetilde{\bb{v}}(t)$ and $\widetilde{P}=P-\widetilde{\Theta}$. Obviously,
\begin{equation}\label{eq-7-3}
\left\{\begin{aligned}
&\partial_t \widetilde{\bb{u}}+L_{\mathfrak{R},{\omega},\bb{w}} \widetilde{\bb{u}}+\nabla \widetilde{P}=-\bb{F}(t),\quad \Div \widetilde{\bb{u}}=0\quad \text{in }\Omega\times (0,\infty),\\
&\widetilde{\bb{u}}|_{\partial\Omega}=\bb{0},\quad \widetilde{\bb{u}}|_{t=0}=\widetilde{\bb{u}}_0\triangleq \bb{f}-\widetilde{\bb{v}}_0.
\end{aligned}\right.
\end{equation}
By Lemma \ref{Lem.Bogo}, we have $\widetilde{\bb{u}}_0\in \mathbb{L}^p_{R+2}(\Omega)$ satisfying
\begin{equation}\label{est-7-5}
\Div \widetilde{\bb{u}}_0=0 \text{ in }\Omega\quad \text{and}\quad \|\widetilde{\bb{u}}_0\|_{\mathbb{L}^p(\Omega)}\leq C\|\bb{f}\|_{\mathbb{L}^p(\Omega)}.
\end{equation}
Noting that $\bb{F}(t)\in \LL^p_{R+2}(\Omega)$, we write
\[\widetilde{\bb{u}}(t)=T_{\mathfrak{R},{\omega},\bb{w}}(t)\widetilde{\bb{u}}_0
-\int^t_0T_{\mathfrak{R},{\omega},\bb{w}}(t-\tau)
\mathcal{P}_{\Omega}\bb{F}(\tau)\,\mathrm{d}\tau\triangleq\widetilde{\bb{u}}_1+\widetilde{\bb{u}}_2.\]
By Theorem \ref{TH5-1} and Theorem \ref{TH5-2} with $\rho=\max(\frac3{2p}-1,\frac14)$, we obtain
\begin{equation}\label{est-7-6}
\begin{split}
&\|\widetilde{\bb{u}}_1(t),t^{\frac12}\nabla\widetilde{\bb{u}}_1(t)\|_{\mathbb{L}^p(\Omega)}+t^{\frac12+\frac{1}{2p}}
\|\partial_t \widetilde{\bb{u}}_1\|_{{\W}^{-1,p}(\Omega_{R+3})}\leq C\|\bb{f}\|_{\mathbb{L}^p(\Omega)},\;0<t\leq 2,\\
&\|\widetilde{\bb{u}}_1(t)\|_{{\W}^{1,p}(\Omega_{R+3})}+\|\partial_t \widetilde{\bb{u}}_1\|_{{\W}^{-1,p}(\Omega_{R+3})}\leq C t^{-\frac3{2p}}\|\bb{f}\|_{\mathbb{L}^p(\Omega)},\qquad\quad\; t>2.
\end{split}
\end{equation}
For $\widetilde{\bb{u}}_2(t)$, when $t\in (0,2]$ we deduce by Theorem \ref{TH5-1} that
\begin{align*}
\|\widetilde{\bb{u}}_2(t)\|_{{\W}^{1,p}(\Omega)}\lesssim \|\bb{f}\|_{\mathbb{L}^p(\Omega)}\int^t_0\big(1+(t-\tau)^{-\frac12}\big)\tau^{-\frac12}\,\mathrm{d}s
\lesssim\|\bb{f}\|_{\mathbb{L}^p(\Omega)},
\end{align*}
which implies that
\begin{equation}\label{est-7-7}
\|\big(\widetilde{\bb{u}}_2(t),t^{\frac12}\nabla\widetilde{\bb{u}}_2(t)\big)\|_{\mathbb{L}^p(\Omega)}\leq C\|\bb{f}\|_{\mathbb{L}^p(\Omega)},\,\,\, 0<t\leq 2.
\end{equation}
When $t>2$, we decompose
$$\widetilde{\bb{u}}_2(t)=-\bigg[\int^1_0+\int^{t-1}_1+\int^t_{t-1}\bigg]T_{\mathfrak{R},
{\omega},\bb{w}}(t-\tau)\mathcal{P}_{\Omega}\bb{F}(\tau)\,\mathrm{d}\tau.$$
Then,  by Theorem \ref{TH5-1} and  Theorem \ref{TH5-2}  with $\rho=(\max(\frac3{2p}-1,\frac14)+\frac12)/2$, we get
\begin{align*}
\|\widetilde{\bb{u}}_2(t)\|_{\mathbb{W}^{1,p}(\Omega_{R+3})}
\lesssim&_{\varrho}\|\bb{f}\|_{\mathbb{L}^p(\Omega)}\Big(\int^1_0(t-\tau)^{-1-\rho}
\tau^{-\frac12}\,\mathrm{d}\tau\\
&+\int^{t-1}_1(t-\tau)^{-1-\rho}
\tau^{-\frac3{2p}}\,\mathrm{d}\tau+\int^t_{t-1}(t-\tau)^{-\frac12}\tau^{-\frac3{2p}}\,\mathrm{d}\tau\Big)\\
\lesssim&_{\varrho}\|\bb{f}\|_{\mathbb{L}^p(\Omega)}\Big(t^{-\frac3{2p}}+\int^{t-1}_1(t-\tau)^{-1-\rho}
\tau^{-\frac3{2p}}\,\mathrm{d}\tau\Big).
\end{align*}
We observe that \begin{align*}
\int^{t-1}_1(t-\tau)^{-1-\rho}
\tau^{-\frac3{2p}}\,\mathrm{d}\tau\leq &\Big[\int^{t/2}_0+\int^t_{t/2}\Big](t-\tau+1)^{-1-\rho}
(\tau+1)^{-\frac3{2p}}\,\mathrm{d}\tau\\
\leq& t^{-\frac3{2p}}+t^{-1-\rho}\int^{t/2}_0(\tau+1)^{-\frac3{2p}}\,\mathrm{d}\tau\leq t^{-\frac3{2p}}
\end{align*}
since for any fixed $\theta>0$, $\ln{s}\leq C_{\theta}s^{\theta}$ for all $ s\geq 1$. Thus, we have
\begin{equation}\label{est-7-8-8}
\|\widetilde{\bb{u}}_2(t)\|_{\mathbb{W}^{1,p}(\Omega_{R+3})}\lesssim t^{-\frac3{2p}}\|\bb{u}_0\|_{\mathbb{L}^p(\Omega)},\quad t>2.
\end{equation}
On the other hand, since
\[\partial_t\widetilde{\bb{u}}_2(t)=-\mathcal{P}_{\Omega}\bb{F}(t)-\int^t_0\partial_t T_{\mathfrak{R},{\omega},\bb{w}}(t-\tau)\mathcal{P}_{\Omega}\bb{F}(\tau)\,\mathrm{d}\tau,\]
in the same way as deducing \eqref{est-7-7}-\eqref{est-7-8-8},  we have
\begin{equation}\label{est-7-9}
\|\partial_t \widetilde{\bb{u}}_2(t)\|_{{\W}^{-1,p}(\Omega_{R+3})}\lesssim \|\bb{f}\|_{\mathbb{L}^p(\Omega)}\left\{\begin{aligned}
&t^{-\frac12-\frac1{2p}}\quad 0<t\leq 2\\
& t^{-\frac3{2p}},\quad t>2.\end{aligned}\right.
\end{equation}
Summing up, we prove that the estimates of $\bb{u}$ in \eqref{semi-E1}-\eqref{semi-E2} hold.

Now we estimate $P(t)$. Set $\bar{\phi}=\phi-\frac1{|\Omega_{R+3}|}\int_{\Omega_{R+3}}\phi\,\mathrm{d}x$, $\phi\in C^{\infty}_0(\Omega_{R+3})$. Notting that
\[\|\B[\bar{\phi}]\|_{{\W}^{1,p'}(\R^3)}\leq C\|\phi\|_{L^{p'}(\Omega_{R+3})},\quad \Div \B[\bar{\phi}]=\bar{\phi}\text{ in }\Omega_{R+3},\quad \B[\bar{\phi}]=0\text{ in }\Omega^c_{R+3},\]
we have from \eqref{cond-pre}
\begin{align*}
\langle P(t),\phi\rangle_{\Omega_{R+3}}=&\langle P(t),\bar{\phi}\rangle_{\Omega_{R+3}}
=\langle P(t),\Div \B[\bar{\phi}]\rangle_{\Omega_{R+3}}
=-\langle\nabla P(t),\B[\bar{\phi}]\rangle_{\Omega_{R+3}}\\
=&\langle \partial_t \bb{u},\B[\bar{\phi}]\rangle_{\Omega_{R+3}}+\langle \nabla \bb{u},\nabla\B[\bar{\phi}]\rangle_{\Omega_{R+3}}+\langle (L_{\mathfrak{R},{\omega},\bb{w}}-\Delta)\bb{u},\B[\bar{\phi}]\rangle_{\Omega_{R+3}}.
\end{align*}
This equality implies
\begin{align*}
|\langle P(t),\phi\rangle_{\Omega_{R+3}}|\lesssim \big( \|\partial_t \bb{u}\|_{{\W}^{-1,p}(\Omega_{R+3})}+\|\bb{u}\|_{{\W}^{1,p}(\Omega_{R+3})}\big)\|\phi\|_{\mathbb{L}^{p'}(\Omega_{R+3})}.
\end{align*}
Hence, we obtain from \eqref{est-7-3}-\eqref{est-7-4} and \eqref{est-7-6}-\eqref{est-7-9} that
\begin{equation}\label{est-7-66}
\|P(t)\|_{\mathbb{L}^p(\Omega_{R+2})}\leq \left\{\begin{aligned}
&C t^{-\frac12-\frac1{2p}}\|\bb{f}\|_{\mathbb{L}^p(\Omega)},\quad\, 0<t\leq 2,\\
&C_{\varrho} t^{-\frac3{2p}+\varrho}\|\bb{f}\|_{\mathbb{L}^p(\Omega)},\quad t>2.
\end{aligned}\right.
\end{equation}
This, together with \eqref{est-7-3}   finishes the proof of Proposition \ref{Pro-7-1}.
\end{proof}

Now  we study the $L^p$-$L^q$ estimates of $T_{\mathfrak{R},{\omega},\bb{w}}(t)$.
 \begin{proposition}\label{TH7-1}
Let $q\in [p,\infty)$. Under the assumption in Proposition \ref{Pro-7-1},  there exists a constant $\eta=\eta_{\mathfrak{R}_*,\mathfrak{R}^*,{\omega}^*}$ such that if
$\normmm{\bb{w}}_{\varepsilon,\Omega}\leq \eta$, then for $\bb{f}\in \mathbb{J}^p (\Omega)$
\begin{align}
&\|T_{\mathfrak{R},{\omega},\bb{w}}(t)\bb{f}\|_{\mathbb{L}^q(\Omega)}\leq C_{\mathfrak{R}_*,\mathfrak{R}^*,{\omega}^*}t^{-\frac32(\frac1p-\frac1q)}
\|\bb{f}\|_{\mathbb{L}^p(\Omega)},\quad \tfrac1p-\tfrac1q<\tfrac13,\label{est.7-1}\\
&\|\nabla T_{\mathfrak{R},{\omega},\bb{w}}(t)\bb{f}\|_{\mathbb{L}^p(\Omega)}\leq C_{\mathfrak{R}_*,\mathfrak{R}^*,{\omega}^*}t^{-\frac12}\|\bb{f}
\|_{\mathbb{L}^p(\Omega)}.\label{est.7-2}
\end{align}
\end{proposition}
\begin{proof}
Set $\bb{u}(t)=T_{\mathfrak{R},{\omega},\bb{w}}(t)\bb{f}$ and $P(t)=\mathring{\mathcal{Q}}_{\R^3}\bb{u}(t)$. Obviously, $(\bb{u},P(t))$ satisfies \eqref{eq-7-1}-\eqref{cond-pre}.  By Proposition \ref{Pro-7-1} and the Gagliardo-Nirengerg inequality, we have
\begin{equation}\label{est.7-3-add-1}
\left\{\begin{aligned}
&\|\bb{u}(t)\|_{\mathbb{L}^q(\Omega_{R+3})}\leq Ct^{-\frac32(\frac1p-\frac1q)}\|\bb{f}\|_{\mathbb{L}^p(\Omega)},\quad\; 0\leq \tfrac1p-\tfrac1q<\tfrac13,\,\,t>0,\\
&\|\nabla \bb{u}(t)\|_{\mathbb{L}^p(\Omega_{R+3})}\leq C t^{-\frac12}\|\bb{f}\|_{\mathbb{L}^p(\Omega)},\quad\quad\;\; t>0.
\end{aligned}\right.
\end{equation}

Now we are in position to estimate $\bb{u}(t)$ in $B^c_{R+3}$. Let $\phi\in C^{\infty}_0(\R^3)$  satisfy $\phi=0$ in $B_{R+2}$ and $\phi=1$ in $B^c_{R+3}$, and set $\bb{z}(t)=\phi \bb{u}(t)-\B[\nabla\phi\cdot \bb{u}]$. Obviously,
 \begin{equation}\label{eq-7-4}
\left\{\begin{aligned}
&\partial_t \bb{z}(t)+L_{\mathfrak{R},{\omega},\overline{\bb{w}}}\bb{z}(t)+\nabla (\phi P(t))=\bb{H}(t),\quad \Div \bb{z}(t)=0 \text{ in }\R^3\times (0,\infty),\\
&\bb{z}|_{t=0}=\bb{z}_0=\phi \bb{f}-\B[\nabla\phi\cdot \bb{f}]
\end{aligned}\right.
\end{equation}
with
\begin{align*}
\bb{H}(t)=&P(t)\nabla \phi+\B[\nabla \phi\cdot(L_{\mathfrak{R},{\omega},\overline{\bb{w}}}\bb{u}(t)+\nabla P(t))]-L_{\mathfrak{R},{\omega},\bb{w}}\B[\nabla \phi\cdot \bb{u}(t)]-(\Delta\varphi)\bb{u}(t)\\
&-2(\nabla\phi\cdot\nabla)\bb{u}(t)-\mathfrak{R}(\partial_1\varphi)\bb{u}(t)-{\omega}\big((\bb{e}_1\times {x})\cdot\nabla \phi\big)\bb{u}(t)+(\bb{w}\cdot\nabla \varphi)\bb{u}(t).
\end{align*}
Since $\mathop{\rm supp} \bb{H}(t)\subset B_{R+1,R+2}$, by Lemma \ref{Lem.Bogo} and Proposition \ref{Pro-7-1}, we rewrite
\begin{equation}\label{est.7-5}
\|\bb{H}(t)\|_{\mathbb{L}^r(\R^3)}\leq \begin{cases}
Ct^{-\frac12-\frac1{2p}}\|\bb{f}\|_{\mathbb{L}^p(\Omega)},\quad 0<t\leq 2,\\
Ct^{-\frac3{4p}-\frac14}\|\bb{f}\|_{
\mathbb{L}^p(\Omega)},\quad t>2.
\end{cases},\quad \forall r\in [1,p].
\end{equation}
Hence, we have
\begin{equation}\label{eq-7-6}
\bb{z}(t)=T^G_{\mathfrak{R},{\omega},\bb{w}}(t)\bb{z}_0
+\int^t_0T^G_{\mathfrak{R},{\omega},\bb{w}}(t-s)
\mathcal{P}_{\R^3}\bb{H}(s)\,\mathrm{d}s\triangleq \bb{z}_1(t)+\bb{z}_2(t).
\end{equation}
For $\bb{z}_1(t)$, we easily get from \eqref{sem-G} and Lemma \ref{Lem.Bogo} that
\begin{equation}\label{est.7-7}
\left\{\begin{aligned}
&\|\bb{z}_1(t)\|_{\mathbb{L}^q(\R^3)}\leq Ct^{-\frac32(\frac1p-\frac1q)}\|\bb{f}\|_{\mathbb{L}^p(\Omega)},\qquad\quad  \forall q\in [p,\infty],\\
&\|\nabla \bb{z}_1(t)\|_{\mathbb{L}^q(\R^3)}\leq Ct^{-\frac12-\frac32(\frac1p-\frac1q)}\|\bb{f}\|_{\mathbb{L}^p(\Omega)},\quad\, \forall  q\in [p,3].\
\end{aligned}\right.
\end{equation}

For $\bb{z}_2(t)$, when $t\in (0,2]$,  using \eqref{sem-G} and \eqref{est.7-5},  we obtain  for $0\leq \tfrac1p-\tfrac1q<\tfrac13$
\begin{align}
\|\bb{z}_2(t)\|_{\mathbb{L}^q(\R^3)}\leq  C\|\bb{f}\|_{\mathbb{L}^p(\Omega)}\int^t_0(t-s)^{-\frac32(\frac1p-\frac1q)}s^{-\frac12-\frac1{2p}}
\,\mathrm{d}s\leq Ct^{-\frac32(\frac1p-\frac1q)}\|\bb{f}\|_{\mathbb{L}^p(\Omega)}. \label{est-7-8}
\end{align}
When $t\geq 2$, we decompose  $\bb{z}_2(t)$
\[\bb{z}_2(t)=\Big[\int^1_{0}+\int^{t-1}_{1}+\int^t_{t-1}\Big]T^G_{\mathfrak{R},{\omega},\bb{w}}(t-s)
\mathcal{P}_{\R^3}\bb{H}(s)\,\mathrm{d}s\triangleq \text{J}_1(t)+\text{J}_2(t)+\text{J}_3(t).\]
Then we calculate by \eqref{P-est}, \eqref{sem-G} and \eqref{est.7-5} that
\begin{align*}
&\|\bb{z}_2(t)\|_{\mathbb{L}^{q}(\mathds{R}^3)}\\
\leq & C\|\bb{f}\|_{\mathbb{L}^p(\Omega)}\Big(\int^1_0(t-s)^{-\frac32(\frac1p-\frac1q)}
s^{-\frac12-\frac1{2p}}
\,\mathrm{d}s+\int^{t-1}_1(t-s)^{-\frac32(\frac1r-\frac1q)}
s^{-\frac3{2p}}\,\mathrm{d}s\\
&+\int^t_{t-1}(t-s)^{-\frac32(\frac1p-\frac1q)}
s^{-\frac3{2p}}\,\mathrm{d}s\Big)\\
\leq &C\|\bb{f}\|_{\mathbb{L}^p(\Omega)}\Big((t^{-\frac32(\frac1p-\frac1q)}+t^{-\frac3{2p}}+\int^{t}_0(t-s+1)^{-\frac32(\frac1r-\frac1q)}
(s+1)^{-\frac3{2p}}\,\mathrm{d}s\Big)\\
\leq & t^{-\frac32(\frac1p-\frac1q)}\|\bb{f}\|_{\mathbb{L}^p(\Omega)},\quad 0\leq \tfrac1p-\tfrac1q<\tfrac13.
\end{align*}
where we choose $1<r<\max(\frac32,p)$.
This, together with \eqref{est-7-8}, gives that
\begin{equation}\label{semi-7-1}
\|\bb{z}_2(t)\|_{\mathbb{L}^q(\R^3)}\leq Ct^{-\frac32(\frac1p-\frac1q)}\|\bb{f}\|_{\mathbb{L}^p(\Omega)},\quad0\leq \tfrac1p-\tfrac1q<\tfrac13,\; t>0.
\end{equation}
In same way as deducing \eqref{semi-7-1}, we obtain by \eqref{sem-G'}
\begin{equation}\label{semi-7-2}
\|\nabla \bb{z}_2(t)\|_{\mathbb{L}^p(\R^3)}\leq
Ct^{-\frac12}\|\bb{f}\|_{\mathbb{L}^p(\Omega)},\quad t>0.
\end{equation}

Since  $\bb{u}(t)=\bb{z}(t)$ in $B^c_{R+3}$, collecting \eqref{est.7-3-add-1}, \eqref{est.7-7} and \eqref{semi-7-1}-\eqref{semi-7-2}, we deduce
\eqref{est.7-1}-\eqref{est.7-2}, and  so finish the proof of Proposition \ref{TH7-1}.
\end{proof}

Following the argument used in the proof of  Theorem \ref{TH7-1}, we obtain the same results  for the operator semigroup $\{T^{*}_{\mathfrak{R},{\omega},\bb{w}}(t)\}_{t\geq 0}$.
\begin{corollary}\label{Cor-7-1}
Under the  assumption in Proposition \ref{TH7-1}, there exists a constant $\eta=\eta_{\mathfrak{R}_*,\mathfrak{R}^*,{\omega}^*}>0$ such that if
$\normmm{\bb{w}}_{\varepsilon,\Omega}\leq \eta$, then for $\bb{f}\in \mathbb{J}^p (\Omega)$
\begin{align}
&\|T^{*}_{\mathfrak{R},{\omega},\bb{w}}(t)\bb{f}\|_{\mathbb{L}^q(\Omega)}\leq
C_{\mathfrak{R}_*,\mathfrak{R}^*,{\omega}^*}t^{-\frac32(\frac1p-\frac1q)}
\|\bb{f}\|_{\mathbb{L}^p(\Omega)},\quad \tfrac1p-\tfrac1q<\tfrac13,\label{est.7-1'}\\
&\|\nabla T^{*}_{\mathfrak{R},{\omega},\bb{w}}(t)\bb{f}\|_{\mathbb{L}^p(\Omega)}\leq C_{\mathfrak{R}_*,\mathfrak{R}^*,{\omega}^*}t^{-\frac12}\|\bb{f}
\|_{\mathbb{L}^p(\Omega)}.\label{est.7-2'}
\end{align}
\end{corollary}
By Proposition \ref{TH7-1} and Corollary \ref{Cor-7-1}, we can complete the proof of Theorem \ref{TH2}.
\begin{proof}[Proof of Theorem \ref{TH2}]
By duality, we have from Theorem \ref{TH7-1} and Corollary \ref{Cor-7-1} that for every $1<p\leq q<\infty$ satisfying $0\leq \tfrac1p-\tfrac1q<\tfrac13$,
\begin{equation}\label{est.7-9}
\|T_{\mathfrak{R},{\omega},\bb{w}}(t)\bb{f}\|_{\mathbb{L}^q(\Omega)}+\|T^{*}_{\mathfrak{R},{\omega},\bb{w}}(t)\bb{f}\|_{\mathbb{L}^q(\Omega)}\leq Ct^{-\frac32(\frac1p-\frac1q)}\|\bb{f}\|_{\mathbb{L}^p(\Omega)}.
\end{equation}
To prove the valid of \eqref{est.7-9} for all $1<p\leq q<\infty$ satisfying $\tfrac1p-\tfrac1q\geq \tfrac13$, we  choose $q_1,q_2,q_3\in [p,\infty)$ satisfying
$$q_0=q\quad q_4=p,\quad 0\leq \tfrac1{q_{j}}-\tfrac1{q_{j-1}}<\tfrac13,\quad j=1,2,3,4.$$
and then deduce
\begin{align}
\|T_{\mathfrak{R},\omega,\bb{w}}(t)\bb{f}\|_{\mathbb{L}^q(\Omega)}\leq & Ct^{-\frac12(\frac1{q_1}-\frac1q)}
\|T_{\mathfrak{R},\omega,\bb{w}}(3t/4 )\bb{f}\|_{\mathbb{L}^{q_1}(\Omega)}\nonumber\\
\leq& Ct^{-\frac12(\frac1{q_1}-\frac1q)}
\cdots t^{-\frac12(\frac1{p}-\frac1{q_3})}\|\bb{f}\|_{\mathbb{L}^p(\Omega)}\leq  Ct^{-\frac12(\frac1p-\frac1q)}\|\bb{f}\|_{\mathbb{L}^p(\Omega)}.
\label{est.7-12}
\end{align}
 Similarly,  for every $1<p\leq q<\infty$ satisfying $\tfrac1p-\tfrac1q\geq \tfrac13$, we have
\begin{equation}\label{est.7-13}
\|T^*_{\mathfrak{R},\omega,\bb{w}}(t)\bb{f}\|_{\mathbb{L}^q(\Omega)}\leq Ct^{-\frac32(\frac1p-\frac1q)}
\|\bb{f}\|_{\mathbb{L}^p(\Omega)}
\end{equation}
Summing up \eqref{est.7-9}-\eqref{est.7-13}, we prove \eqref{sem-1}, which together with \eqref{est.7-2} and \eqref{est.7-2'} yields \eqref{sem-2}. Thus we complete the proof of Theorem \ref{TH2}.
\end{proof}

%%%%%%%%%%%%%%%%%%%%%%%%%%%%%%%%%%%%%%%%%%%%%%%%%%%%%%%%%%%%%%%%%%%%%%%%%%%%%%%%%%%%%%
\section{Stability in $\mathbb{L}^3(\Omega)$}
\setcounter{section}{7}\setcounter{equation}{0}
In this section, we will use the operator semigroup theory, given in Proposition \ref{TH1} and Theorem \ref{TH2}, to show the stability of stationary flows satisfying \eqref{S} and \eqref{S.1}.
\begin{proof}[Proof of Theorem \ref{TH3}]
Let $(X,\|\cdot\|_{X})$ be a Banach space defined as follows:
\begin{align*}
&X=\big\{\bb{u}\in C_b([0,\infty);\mathbb{J}^3(\Omega))\,|\, t^{\frac12}\nabla \bb{u}(t)\in C_b([0,\infty);\mathbb{L}^3(\Omega))\big\},\\
&\|\bb{u}\|_{X}=\|\bb{u}\|_{\mathbb{L}^{\infty}([0,\infty);\mathbb{L}^3(\Omega))}+\sup_{t\geq 0}t^{\frac12}\|\nabla \bb{u}(t)\|_{\mathbb{L}^3(\Omega)}.
\end{align*}
We write
\begin{equation*}
\bb{u}=T_{\mathfrak{R},{\omega},\bb{w}}(t)\bb{u}_0+B(\bb{u},\bb{u})\triangleq T_{\mathfrak{R},{\omega},\bb{w}}(t)\bb{u}_0+\int^t_0 T_{\mathfrak{R},{\omega},\bb{w}}(t-\tau)\mathcal{P}_{\Omega}(\bb{u}(\tau)\cdot\nabla) \bb{u}(\tau)\,\mathrm{d}\tau.
\end{equation*}
By Theorem \ref{TH2}, we have
\begin{align}\nonumber
t^{\frac12}\|\nabla B(\bb{u},\bb{v})(t)\|_{\mathbb{L}^3(\Omega)}\lesssim& t^{\frac12}\int^t_0(t-\tau)^{-\frac34}
\|\bb{u}(\tau)\|_{\mathbb{L}^6(\Omega)}\|\nabla \bb{v}(\tau)\|_{\mathbb{L}^3(\Omega)}
\,\mathrm{d}\tau\\
\lesssim& \sup_{t\geq 0}t^{\frac14}
\|\bb{u}(t)\|_{\mathbb{L}^6(\Omega)}\sup_{t\geq 0}t^{\frac12}\|\nabla \bb{u}\|_{\mathbb{L}^{3}(\Omega)}\nonumber\\
\lesssim&(\|\bb{u}\|_{\mathbb{L}^{\infty}([0,\infty);\mathbb{L}^3(\Omega)}+\sup_{t\geq 0}t^{\frac12}\|\nabla \bb{u}\|_{\mathbb{L}^{3}(\Omega)})\sup_{t\geq 0}t^{\frac12}\|\nabla \bb{v}(t)\|_{\mathbb{L}^3(\Omega)}\label{est8-2}
\end{align}
and for every $1<p\leq q<\infty$ satisfying $\frac1p-\frac1q<\frac13$
\begin{align}\nonumber
t^{\frac32(\frac1p-\frac1q)}\|B(\bb{u},\bb{v})(t)\|_{\mathbb{L}^q(\Omega)}
\lesssim&\int^t_0(t-\tau)^{-\frac12}
\|\bb{u}(\tau)\|_{\mathbb{L}^q(\Omega)}\|\nabla \bb{v}(\tau)\|_{\mathbb{L}^3(\Omega)}
\,\mathrm{d}\tau\\
\lesssim&\sup_{t\geq 0}t^{\frac32(\frac1p-\frac1q)}\|\bb{u}(t)\|_{\mathbb{L}^q(\Omega)}\sup_{t\geq 0}t^{\frac12}\|\nabla \bb{v}(t)\|_{\mathbb{L}^3(\Omega)}.\label{est8-1}
\end{align}

Collecting \eqref{est8-2} and \eqref{est8-1} with $p=q=3$, we have
\begin{equation}\label{est-8-4}
\|B(\bb{u},\bb{v})(t)\|_{X}\leq C_1\|\bb{u}\|_{X}\|\bb{v}\|_{X},\quad  \bb{u},\bb{v}\in X.
\end{equation}
This, combining with the fact that $T_{\mathfrak{R},{\omega},\bb{w}}(t)\bb{u}_0\in X_q$ satisfies
\begin{equation}\label{est-8-3}
\|T_{\mathfrak{R},{\omega},\bb{w}}(t)\bb{u}_0\|_{X_q}\leq C_2\|\bb{u}_0\|_{\mathbb{L}^3(\Omega)}.
\end{equation}
as follows from Theorem \ref{TH1} and Theorem \ref{TH2}, gives by Banach fixed point theorem that problem \eqref{eq.Inte} admits a unique solution in $X$ satisfying $\|\bb{u}\|_{X}\leq 2C_2\eta$ if $\|\bb{u}_0\|_{\mathbb{L}^3(\Omega)}<\eta$ with $\eta$ satisfying $4C_1C_2\eta<1.$ This proves \eqref{sta-1} by the interpolation inequality.

Next, we show \eqref{sta-2}. For every $0<\delta<1$, we choose a $\bb{u}_{0,\delta}\in \mathbb{J}^3(\Omega)\cap \mathbb{J}^{2}(\Omega)$ such that
$$\|\bb{u}_0-\bb{u}_{0,\delta}\|_{\LL^3(\Omega)}<\delta\|\bb{u}_0\|_{\LL^3(\Omega)}\leq \delta\eta.$$
According to the above analysis, we know that problem \eqref{eq.Inte} admits a unique solution $\bb{u}_{\delta}$ in $X$ with $\bb{u}_{\delta}|_{t=0}=\bb{u}_{0,\delta}$, satisfying $\|\bb{u}_{\delta}\|_{X}\leq 2C_2\eta$. So, $\bb{u}-\bb{u}_{\delta}\in X$ satisfies
\begin{equation*}
\bb{u}-\bb{u}_{\delta}=T_{\mathfrak{R},{\omega},\bb{w}}(t)(\bb{u}_{0}-\bb{v}_{0})
+\int^t_0 T_{\mathfrak{R},{\omega},\bb{w}}(t-\tau)\mathcal{P}_{\Omega}\big((\bb{u}-\bb{v})
\cdot\nabla \bb{u}+\bb{v}\cdot\nabla(\bb{u}-\bb{v})\big)(\tau)\,\mathrm{d}\tau.
\end{equation*}
By Theorem \ref{TH2}, \eqref{est8-2} and  \eqref{est8-1} with $p=2,\,q=3$ or $p=q=3$, we have
\begin{align*}
&\sup_{0\leq t<\infty}t^{\frac14}\|\bb{u}_{\delta}(t)\|_{\mathbb{L}^3(\Omega)}\leq C_3(\|\bb{u}_{0,\delta}\|_{\mathbb{L}^2(\Omega)}+\sup_{0\leq t<\infty}t^{\frac14}\|\bb{u}_{\delta}(t)\|_{\mathbb{L}^3(\Omega)}\sup_{0\leq t<\infty}t^{\frac12}\|\nabla\bb{u}_{\delta}(t)\|_{\mathbb{L}^3(\Omega)}\\
&\|\bb{u}-\bb{v}\|_{X}\leq C_4\|\bb{u}_0-\bb{u}_{\delta}\|_{\LL^3(\Omega)}+C_4(\|\bb{u}\|_{X}+\|\bb{u}_{\delta}\|_{X})\|\bb{u}-\bb{u}_{\delta}\|_{X}.
\end{align*}
These implies that if $2C_2(C_3+2C_4)\|\bb{u}_0\|_{\LL^3(\Omega)}<1/2$, then for every $t>0$
\[\|\bb{u}_{\delta}(t)\|_{\mathbb{L}^3(\Omega)}\leq 4C_3\eta t^{-1/4}, \quad \|\bb{u}(t)-\bb{u}_{\delta}(t)\|_{\LL^3(\Omega)}\leq 2\delta C_4\eta.\]
Hence, for every $0<\delta<1$, $\lim_{t\to \infty}\|\bb{u}(t)\|_{\mathbb{L}^3(\Omega)}\leq 2\delta C_4\eta$. This yields \eqref{sta-2} since $\delta$ is chosen arbitrarily.  So we  complete the proof of Theorem \ref{TH3}.
\end{proof}

%%%%%%%%%%%%%%%%%%%%%%%%%%%%%%%%%%%%%%%%%%%%%%%%%%%%
\section{Appendix}
\setcounter{section}{8}\setcounter{equation}{0}
\begin{lemma}\label{Lem-evolu}
Let $\bb{f}\in \mathbb{J}^p (\R^3)$, $p\in (1,\infty)$, $\varepsilon\in (0,\frac12)$, and $0<|\Rr|\leq \Rr^*$. Assume that $\widetilde{\bb{w}}(t)$ is a period function on $t\geq0$ with period $T>0$ satisfying
\begin{itemize}
\item \;There exist a constant $M>0$ independent of $t$ such that $\normmm{\widetilde{\bb{w}}(t)}_{\varepsilon,\R^3}\leq M$;
\item \;$\bb{w}(t)$ is H\"{o}lder continuous with respect to $t$.
\end{itemize}
Then,
$-\mathcal{P}_{\R^3}(-\Delta-\Rr\partial_1+B_{\widetilde{\bb{w}}(t)})$ and its dual operator  generate unique evolution operators $G(t,s)$ and $G^*(t,s)$ in $\mathbb{J}^p(\R^3)$, respectively, such that for $j\leq 2$
$$\|\nabla^j G(t,s)\bb{f}\|_{\LL^p(\R^3)},\,\|\nabla^j G^*(t,s)\bb{f}\|_{\LL^p(\R^3)}\leq C_{T}(t-s)^{-\frac{j}2}e^{C_T (t-s)}\|\bb{f}\|_{\LL^p(\R^3)}.$$
\end{lemma}

\begin{proof} Since
\[\|\nabla^j(\lambda I-\mathcal{P}_{\R^3}\Delta)^{-1}\bb{f}\|_{\mathbb{L}^p(\R^3)}\leq C_{\theta} \lambda^{-1+\frac j2}\|\bb{f}\|_{\mathbb{L}^p(\R^3)},\quad j\leq 2,\quad \lambda\in \Sigma_{\theta},\]
we have
$$\|\mathcal{P}_{\R^3}(\mathfrak{R}\partial_1+B_{\widetilde{\bb{w}}(t)})(\lambda I-\mathcal{P}_{\R^3}\Delta)^{-1}
\bb{f}\|_{\mathbb{L}^p(\R^3)}\leq C_{\theta,\Rr^*,M} (|\lambda|^{-1}+|\lambda|^{-\frac12})\|\bb{f}\|_{\mathbb{L}^p(\R^3)}.$$
This, together with the Neumann series expansion
\begin{align*}
&\big(\lambda I+\mathcal{P}_{\R^3}(-\Delta-\Rr\partial_1+B_{\widetilde{\bb{w}}(t)})\big)^{-1}\\
=&(\lambda I-\mathcal{P}_{\R^3}\Delta)^{-1}
\sum^{\infty}_{j=0}\big(-\mathcal{P}_{\R^3}(-\Rr\partial_1+B_{\widetilde{\bb{w}}(t)})
(\lambda I-\mathcal{P}_{\R^3}\Delta)^{-1}\big)^j,\end{align*}
gives
$$\|\nabla^j(\lambda+\mathcal{P}_{\R^3}(-\Delta-\Rr\partial_1+B_{\widetilde{\bb{w}}(t)})\big)^{-1}
\bb{f}\|_{\mathbb{L}^p(\R^3)}\leq C_{\theta}|\lambda|^{-1+\frac{j}2}\|\bb{f}\|_{\mathbb{L}^p(\R^3)},$$
for every $\lambda\in \Sigma_{\theta,\ell}$ where $\ell$ satisfies $C_{\theta,\Rr^*,M}(\ell^{-1}+\ell^{-1/2})<1$. Similarly, we have
$$\|\nabla^j(\lambda+\mathcal{P}_{\R^3}(-\Delta-\Rr\partial_1+B_{\widetilde{\bb{w}}(t)})\big)^{-1}\bb{f}\|_{\mathbb{L}^p(\R^3)}\leq C_{\theta}|\lambda|^{-1+\frac{j}2}\|\bb{f}\|_{\mathbb{L}^p(\R^3)},\;\; \lambda\in \Sigma_{\theta,\ell}.$$
Thus from the holomorphic semigroups theory in \cite{Paz83,RS75}, for every fixed $s\geq 0$, both $-\mathcal{P}_{\R^3}(-\Delta-\Rr\partial_1+B_{\widetilde{\bb{w}}(s)})$ and its dual operator generate analytic semigroups in $\mathbb{J}^p(\R^3)$. So we prove this lemma from the theory of parabolic evolution systems in \cite{Ama87,Lun95}.
\end{proof}

\begin{lemma}[\cite{Sog93}]\label{Lem.ST}
Let $X$, $Y$ be two measurable spaces,  and $T$ be an integral operator
\[Tf(x)=\int_{Y}K(x,y)f(y)\,\mathrm{d}y,\quad x\in X\]
with
\[\sup_{x\in X}\Big(\int_{Y}|K(x,y)|^r\,\mathrm{d}y\Big)^{\frac1{r}}+\sup_{y\in Y}\Big(\int_{X}|K(x,y)|^r\,\mathrm{d}x\Big)^{\frac1{r}}<C, \; r\ge 1.\]
 Then for every $1\leq p\leq q\leq \infty$ satisfying $\frac1r=\frac1p+\frac1q-1$
\begin{equation*}
\|Tf\|_{L^q(X)}\leq C \|f\|_{L^p(Y)}.
\end{equation*}
\end{lemma}
\begin{lemma}\label{Lem.In}
Let $\ell\geq 0$ and $0\leq\delta<1$ satisfy $\ell+\delta>3$. Then, for every $R>0$
\begin{equation}\label{est.In}
\int_{|x|\geq R}|x|^{-\ell}(1+s_{\mathfrak{R}}(x))^{-\delta}\,\mathrm{d}x\leq C_{\delta,\ell}R^{-\ell-\delta+3}.
\end{equation}
\end{lemma}
\begin{proof}
Using a change of variable:
\[y=Sx,\qquad S=\left(
      \begin{array}{ccc}
        \mathfrak{R}/|\mathfrak{R}| & 0 & 0 \\
        0 & 1 & 0 \\
        0 & 0 & 1 \\
      \end{array}
    \right)
\]
and the polar coordinate:
\[y_1=r\cos \theta,\quad y_2=r\sin\theta\cos\varphi,\quad y_3=r\sin\theta\sin\varphi,\quad r\in [0,\infty),\,\,\theta\in [0,\pi],\,\,\varphi\in [0,2\pi],\]
we have
\begin{align*}
\int_{|x|\geq R}|x|^{-\ell}(1+s_{\mathfrak{R}}(x))^{-\delta}\,\mathrm{d}x=&\int^{\infty}_R\int^{2\pi}_0\int^{\pi}_0 \frac{r^{-\ell+2}\sin \theta}{(1+r(1-\cos \theta))^{\delta}}\,\mathrm{d}\theta\mathrm{d}\varphi\mathrm{d}r\\
= &2\pi(1-\delta)^{-1}\int^{\infty}_R r^{-\ell+1}((1+2r)^{-\delta+1}-1)\,\mathrm{d}r
\end{align*}
where we have used
\begin{align*}
\int^{\pi}_0 \frac{\sin \theta}{(1+r(1-\cos \theta))^{\delta}}\,\mathrm{d}\theta=&\int^{\pi}_0 \frac{\partial_{\theta}(1+r(1-\cos \theta))^{-\delta+1}}{(1-\delta)r}\,\mathrm{d}\theta=\frac{((1+2r)^{-\delta+1}-1)}{(1-\delta)r}.
\end{align*}
Hence, we deduce \eqref{est.In} for $\ell+\delta>3$ and then complete the proof of Lemma \ref{Lem.In}.
\end{proof}

\begin{lemma}\label{Lem.Hardy}
Let $\alpha\in [0,1/2]$ and $p\in (1,3)$ satisfying $1-\alpha p\neq 0$. Then there exists a constant $C_{\alpha}$ such that
\begin{equation}\label{G-Hardy1}
\|\tfrac{f}{|x|^{1-\alpha}s_{\mathfrak{R}}(x)^{\alpha}}\|_{L^p(\R^3)}\leq C_{\alpha}\|\nabla f\|_{L^p(\R^3)}.
\end{equation}

\end{lemma}
\begin{proof}
It is well known that
\begin{equation}\label{eq-8-2}
\||x|^{-1}f\|_{L^p(\R^3)}\lesssim \|\nabla f\|_{L^p(\R^3)}.
\end{equation}
Similar to the  proof of Lemma \ref{Lem.In}, we have
\begin{align*}
\Big\|\frac{f}{|x|^{1-\alpha}s_{\mathfrak{R}}(x)^{\alpha}}\Big\|^p_{L^p(\R^3)}=
\int^{\infty}_{0}\int^{2\pi}_0\int^{\pi}_0\frac{r^{2-p}|f(S^{T}y)|^3\sin\theta}{(1-\cos\theta)
^{p\alpha}}\,\mathrm{d}\theta\mathrm{d}\varphi\mathrm{d}r.
\end{align*}
Let $\theta_0\in (0,\pi/4)$ and $\rho(\theta)\in C^{\infty}_0(\R)$ satisfying  $\rho(\theta)=1$ if $|\theta|\leq \theta_0$ and $\rho(\theta)=0$ if $|\theta|\geq 2\theta_0$, we deduce
\begin{align*}
\Big\|\frac{f}{|x|^{1-\alpha}s_{\mathfrak{R}}(x)^{\alpha}}\Big\|^p_{L^p(\R^3)}\le&
\frac{1}{(1-\cos \theta_0)^{p\alpha}}\int^{\infty}_{0}\int^{2\pi}_0\int^{\pi}_0r^{2-p}|f(S^{T}y)|^p\sin
\theta\,\mathrm{d}\theta\mathrm{d}\varphi\mathrm{d}r\\
&+\int^{\infty}_{0}\int^{2\pi}_0\text{I}(r,\varphi)\,\mathrm{d}\varphi\mathrm{d}r
\end{align*}
where
\[\text{I}(r,\varphi)=\int^{\pi}_0\frac{r^{2-p}\rho(\theta)|f(S^{T}y)|^p\sin\theta}{(1-\cos
\theta)
^{p\alpha}}\,\mathrm{d}\theta.\]
By integration by part, we have
\begin{align*}
&r^{p-2}(1-p\alpha)\text{I}(r,\varphi)-\int^{\pi}_0
\rho'(\theta)(1-\cos
\theta)
^{1-p\alpha}|f(S^{T}y)|^p\,\mathrm{d}\theta\\
=&\int^{\pi}_0\rho(\theta)(1-\cos
\theta)
^{1-p\alpha}\partial_{\theta}|f(S^{T}y)|^p\,\mathrm{d}\theta\\
\lesssim& r\int^{\pi}_0\rho(\theta)(1-\cos
\theta)
^{1-p\alpha}|f(S^{T}y)|^{p-1}|(\nabla f)(S^{T}y)|\,\mathrm{d}\theta\\
\lesssim& r\Big(\int^{\pi}_0\frac{\rho(\theta)(1-\cos
\theta)
^{p(1-p\alpha)/(p-1)}}{(\sin\theta)^{1/2}}|f(S^{T}y)|^p\,\mathrm{d}\theta\Big)^{\frac{p-1}{p}}
\Big(\int^{\pi}_0\rho(\theta)\sin\theta|(\nabla f)(S^{T}y)|^p\,\mathrm{d}\theta\Big)^{\frac1p}.
\end{align*}
Since
\begin{align*}
\frac{\rho(\theta)(1-\cos\theta)^{p(1-p\alpha)/(p-1)}}{(\sin \theta)^{1/2}}=\frac{\rho(\theta)\sin\theta}{(1-\cos\theta)^{p\alpha}}\Big(\frac{(1-\cos\theta)^{1-\alpha}}{\sin \theta}\Big)^{\frac{p}{p-1}}\lesssim \frac{\rho(\theta)\sin \theta}{(1-\cos\theta)^{p\alpha}}
\end{align*}
which follows from that $\frac{(1-\cos\theta)^{1-\alpha}}{\sin \theta}\leq 1$ for every $\theta\in [0,\frac{\pi}2]$ if $1-\alpha\geq \frac12$, we get
\begin{align*}
\text{I}(r,\varphi)\lesssim &_{\alpha}(\text{I}(r,\varphi))^{\frac{p-1}p}\Big(r^2\int^{\pi}_0\rho(\theta)
\sin\theta|\nabla f(S^T y)|^p\,\mathrm{d}\theta\Big)^{\frac1p}+\int^{\pi}_0 r^{2-p}\sin\theta |f(S^T y)|^p\,\mathrm{d}\theta.
\end{align*}
Hence we conclude
\[\big\|\tfrac{f}{|x|^{1-\alpha}s_{\mathfrak{R}}(x)^{\alpha}}\big\|^p_{L^p(\R^3)}\leq C_{\alpha}\big(\||x|^{-1}f\|^p_{L^p(\R^3)}+\|\nabla f\|^p_{L^p(\R^3)}\big).\]
This, together with \eqref{eq-8-2}, gives \eqref{G-Hardy1} and so finishes  the proof of  Lemma \ref{Lem.Hardy}.
\end{proof}

\section*{Acknowledgments}   This work is supported in part by the National Natural Science Foundation of China
 under grant  No.11831004, No.11771423 and No.11926303.

%%%%%%%%%%%%%%%%%%%%%%%%%%%%%%%%%%%%%%%%%%%%%%%%%%%%%%%%%%%%%%%%%%%%%%%%%%%%%%%%%%%%%%%%%%%%%%%%%%%%

\end{document}